\documentclass{uofcthesis}

\ifdim\overfullrule>0pt
\usepackage[notref, notcite]{showkeys}
\advance\topmargin 0.4in 
\fi

\input bartonsetup

\begin{document}

\title{Elliptic Partial Differential Equations with Complex Coefficients}
\author{Ariel Barton}
\date{June 2010}
\department{Mathematics}
\division{Physical Sciences}
\degree{Doctor of Philosophy}

{  \makeatletter
   \pagenumbering{roman}   
   \begin{titlepage}
   \begin{center}
   \doublespacing 
     \vspace*{0.0in} 
     \vfill
      \MakeTextUppercase{\@title}\\
     \vspace*{1in} 
      BY\\
      \MakeTextUppercase{\@author}\\
      \vspace*{1in} 
      \MakeTextUppercase{November 2009}\\
      \vfill
     \vfill
   \end{center}
   \end{titlepage}
   \markright{}
   \addtocounter{page}{1}
}



\raggedbottom

\setcounter{tocdepth}{1}
\tableofcontents

\setcounter{tocdepth}{2}






\topmatter{Abstract}
In this paper I investigate elliptic partial differential equations on Lipschitz domains in $\R^2$ whose coefficient matrices have small (but possibly nonzero) imaginary parts and depend only on one of the two coordinates.

I show that for Dirichlet boundary data in $L^p$ for $p$ large enough, solutions exist and are controlled by the $L^p$-norm of the boundary data. 

Similarly, for Neumann boundary data in $L^q$, or for Dirichlet boundary data whose tangential derivative is in $L^q$ (``regularity'' boundary data), for $q$ small enough, I show that solutions exist and are controlled by the $L^q$-norm of the boundary data. 

I prove similar results for Neumann or regularity boundary data in $H^1$, and for Dirichlet boundary data in $L^\infty$ or $BMO$. Finally, I show some converses: if the solutions are controlled in some sense, then Dirichlet, Neumann, or regularity boundary data must exist.


\topmatter{Acknowledgements} I would like to thank my advisor Carlos Kenig, and my parents.

\ifdim\overfullrule=0pt \mainmatter \fi


\chapter{Introduction}
Let $A$ be a uniformly elliptic matrix-valued function defined on $\R^2$. That is, assume that there exist constants $\Lambda>\lambda>0$ (called the ellipticity constants of $A$) such that
\begin{equation}\label{eqn:genelliptic}
\lambda|\eta|^2\leq \Re \bar \eta\cdot A(X)\eta,\quad |\xi\cdot A(X)\eta|\leq\Lambda|\eta| |\xi|
\end{equation}
for every $X\in\R^2$ and every $\xi,\eta\in\C^2$.

Let $1< p\leq\infty$. We say that $(D)^A_p$, $(N)^A_p$, or $(R)^A_p$ hold in the Lipschitz domain $\V$ if, for every $f\in L^p(\partial\V)$, there is a function $u$ such that
\begin{align*}
(D)_p^A&
\begin{cases}
\div A\nabla u=0&\text{ in }\V,\\
u=f&\text{ on }\partial\V,\\
\|Nu\|_{L^p}\leq C(p)\|f\|_{L^p}
\end{cases}
\\
(N)_p^A&
\begin{cases}
\div A\nabla u=0&\text{ in }\V,\\
\nu\cdot A\nabla u=g&\text{ on }\partial\V,\\
\|N(\nabla u)\|_{L^p}\leq C(p)\|g\|_{L^p}
\end{cases}
\\
(R)_p^A&
\begin{cases}
\div A\nabla u=0&\text{ in }\V,\\
u=f&\text{ on }\partial\V,\\
\|N(\nabla u)\|_{L^p}\leq C(p)\|\partial_\tau f\|_{L^p}
\end{cases}
\end{align*}
Furthermore, $u$ is unique among functions with $\div A\nabla u\equiv 0$ and either $Nu\in L^p$ (for $(D)_p^A$) or $N(\nabla u)\in L^p$ (for $(N)_p^A$ or $(R)_p^A$).

Here $\nu,\tau$ are the unit normal and tangent vectors to $\partial\V$. In the case $(N)_p^A$, $(R)_p^A$, we in addition require that $g\in H^1$ or $f\in W^{1,p}$, respectively.

In this paper, we will instead be concerned with complex matrix-valued functions $A$ on $\R^2$ which satisfy
\begin{equation}\label{eqn:elliptic}
\lambda|\eta|^2\leq \Re \bar \eta\cdot A(x,t)\eta,\quad |\xi\cdot A(x,t)\eta|\leq\Lambda|\eta| |\xi|,\quad A(x,t)=A(x,s)
\end{equation}
for all $x$, $t$, $s\in\R$ and all $\eta$, $\xi\in\C^2$.

In \cite{Rule} and \cite{rule2}, the following theorem was proven:

\begin{thm}\label{thm:rule} Suppose that $A_0:\R^2\mapsto\R^{2\times 2}$ is real-valued and satisfies (\ref{eqn:elliptic}). Let $\V$ be a Lipschitz domain which is either bounded or special, and let its Lipschitz constants be at most~$k_i$. (See \dfnref{domain}.) 

Let $1/p_0+1/p_0'=1$. Then if $(D)_{p_0'}^{A_0^t}$, $(D)_{p_0'}^{A_0/\det A_0}$ hold in $\V$ with constants at most $C_\V(p_0)$, then $(N)^{A_0}_{p_0}$ and $(R)^{A_0}_{p_0}$ hold in the Lipschitz domain $\V$, with constants depending only on $p_0,\lambda,\Lambda$, $k_i$, and $C_\V(p_0)$.
\end{thm}


In \cite{Dp'}, it was shown that
\begin{thm}\label{thm:Dp'} Let $A_0(x)$ be as in \thmref{rule}, and suppose that $\V$ is a bounded or special Lipschitz domain with Lipschitz constants at most~$k_i$. Then there is some (possibly large) $p_0'=p_0'(\lambda,\Lambda,k_i)$ and some $C(\lambda,\Lambda,k_i)$ such that $(D)^{A_0}_{p_0'}$ holds in $\V$ with constant at most~$C$.
\end{thm}

In this paper, we intend to prove the following theorem:

\begin{thm}\label{thm:big} Suppose that $A_0$ satisfies the conditions of \thmref{rule}. Suppose that $A$ satisfies the same conditions, except that $A(x)$ is allowed to be a complex-valued matrix instead of a real-valued matrix. Let $\V$ be a good Lipschitz domain, as defined in Definiton~\ref{dfn:domain}.

Then there is some $\epsilon>0$, $p_0>1$ depending only on $\lambda$, $\Lambda$ and the Lipschitz constants of $\V$, such that if $\|A-A_0\|_{L^\infty}<\epsilon$ and $1<p\leq p_0$, then $(N)^{A}_p$, $(R)^{A}_p$ and $(D)_{p'}^{A}$ hold in the Lipschitz domain~$\V$.

Furthermore, if $f\in H^1(\partial\V)$, then there are two functions $u_N, u_R$, unique among $u$ with $\div A\nabla u\equiv 0$ and $N(\nabla u)\in L^1$, such that
\begin{align}
(N)_1^A&\begin{cases}
\div A\nabla u_N=0&\text{ in }\Omega,\\
\nu\cdot A\nabla u_N=f&\text{ on }\partial\Omega,\\
\|N(\nabla u_N)\|_{L^1}\leq C\|f\|_{H^1},\end{cases}
\displaybreak[0]\\(R)_1^A&\begin{cases}
\div A\nabla u_R=0&\text{ in }\Omega,\\
\partial_\tau u_R=f&\text{ on }\partial\Omega,\\
\|N(\nabla u_R)\|_{L^1}\leq C\|f\|_{H^1},
\end{cases}
\end{align}
\end{thm}

We will also prove results for boundary data in $BMO$ and $L^\infty$:

\begin{thm}\label{thm:BMO} Suppose that $A$ and $\V$ satisfy the conditions of \thmref{big}. 
If $f\in BMO(\partial\V)$, the dual to $H^1(\partial\V)$, then there is a unique function $u$ such that $u=f$ on~$\partial\V$, $\div A\nabla u=0$ in $\V$, and such that $|\nabla u|^2\dist(\cdot,\partial\V)$ is a Carleson measure, that is, for any $X_0\in\partial\V$ and any $\V>0$,
\begin{equation}\label{eqn:uCarleson}
\frac{1}{\sigma(B(X_0,R)\cap\partial\V)}
\int_{B(X_0,R)\cap\V} |\nabla u(X)|^2\dist(X,\partial\V) \,dX \leq \tilde C
\end{equation}
where $\tilde C=C\|f\|_{BMO}$ and $C$ depends only on $\lambda$, $\Lambda$, the Lipschitz constants of $\V$, and the $H^1$ constants in \thmref{big}. Furthermore, this function is unique among functions $u$ with $\div A\nabla u=0$ in $\V$, $u=f$ on $\partial\V$, and which satisfy (\ref{eqn:uCarleson}) for some constant~$\tilde C$.
\end{thm}

In \cite{BMODp}, it was shown that for real coefficent matrices $A_0$ on bounded Lipschitz domains~$\V$, this theorem is equivalent to an $A_\infty$ condition on harmonic measure; this $A_\infty$ condition implies that $(D)^{A_0}_p$ holds in $\V$ for some $1<p<\infty$.

\begin{thm}\label{thm:maxprinciple} If $\V$ is bounded, then the maximum principle holds; that is, if $f\in L^\infty(\partial\V)$, then there exists a unique $u$ with $\div A\nabla u=0$ in~$\V$, $u=f$ on $\partial\V$ and $\|u\|_{L^\infty(\V)}\leq C \|f\|_{L^\infty(\partial\V)}$.
\end{thm}

Finally, we will prove some converses:
\begin{thm}\label{thm:converse} Suppose that $A,\V$ satisfy the conditions of \thmref{big}. Suppose further that $\div A\nabla u=0$ in $\V$.

If $N(\nabla u)\in L^1(\partial\V)$, then the boundary values $\nu\cdot A\nabla u$ and $\tau\cdot \nabla u$ exist and are in $H^1(\partial\V)$ with $H^1$-norm at most $C\|N(\nabla u)\|_{L^1}$.

If $|\nabla u(X)|^2\dist(X,\partial\V)$ satisfies (\ref{eqn:uCarleson}) for some $\tilde C$, then $u|_{\partial\V}$ exists and is in $BMO(\partial\V)$ with $BMO$ norm at most~$C\tilde C$. Furthermore, if $f=u|_{\partial\Omega}\in L^\infty$, then $u\in L^\infty$; functions in $L^\infty$ are unique, and so $(D)^A_\infty$ holds.
\end{thm}

\thmref{maxprinciple} will be proven in \autoref{chap:maximum}. Uniqueness and converses for Theorems~\ref{thm:big} and~\ref{thm:BMO}  will be proven in \autoref{chap:converse}. Existence for smooth coefficients for \thmref{BMO} will be proven in \autoref{chap:square}, and we will pass to arbitrary elliptic (rough) coefficients in Theorems~\ref{thm:apriori:BMOcpt} and~\ref{thm:apriori:BMO}.

We now outline the proof of existence for \thmref{big}; resolving the details will form the bulk of this work. 


\begin{proof}[Proof of \thmref{big}] 

If $f:\partial\V\mapsto\C$ is in $L^p$ for $1< p<\infty$, then we can define (\dfnref{layers}) functions $\D f,$ $\S^T f$ such that if $X\in\R^2-\partial\V$, then $\D f(X)$ and $\S^T f(X)$ are well-defined complex numbers, and $\div A\nabla(\D f)=0$, $\div A^T\nabla(\S^T f)=0$ in $\R^2-\partial\V$. (Here $A^T$ is the matrix transpose of~$A$.)

We will show that, if $\K$, $\J$ are bounded then
\[\|N(\D f)\|_{L^p(\partial\V)}\leq C(p)\|f\|_{L^p(\partial\V)},\quad \|N(\nabla S^T f)\|_{L^p(\partial\V)}\leq C(p)\|f\|_{L^p(\partial\V)}\]
for all $1<p<\infty$. (\thmref{NDf}.) Furthermore, if $f\in H^1_{at}$, then $\|N(\nabla \S^T f)\|_{L^1}\leq C\|f\|_{H^1}$, and so by density we may extend $\S$ to all of $H^1$. (\corref{NSf}.)

There exist operators $\K_\pm,\J$ on $L^p(\partial\V)$ such that
\begin{align*}
\K_\pm f&=\D f|_{\partial\V_\pm},\quad
   \J^t f=\tau\cdot \nabla \S^T f|_{\partial\V},\\
\K_\pm^t f&=\nu\cdot A^T\nabla\S^T f|_{\partial\V_{\mp}}
\end{align*}
in appropriate weak senses, where $\V_+=\V$, $\V_-=\bar\V^C$.

If $\K_+$ is invertible with bounded inverse on some $L^p(\partial\V)$, then for every $g\in L^p(\partial\V)$, we may let $u=\D (\K_+^{-1}g)$. Then  
\[\div A\nabla u=0,\quad u|_{\partial\V}=g,\quad \|Nu\|_{L^p}\leq C_p\|\K_+^{-1}g\|_{L^p}\leq C_p\|\K_+^{-1}\|_{L^p\mapsto L^p}\|g\|_{L^p}\]
and so $u$ is a solution to $(D)_p^A$.

Similarly, if $\K_-^t$ or $\J^t$ is bounded and invertible on $L^p$ or $H^1$, then $u=\S^T((\K_-^t)^{-1} g)$ or $u=\S^T((\J^t)^{-1} g)$ is a solution to $(N)_p^{A^T}$, $(R)_p^{A^T}$, $(N)_1^{A^T}$, or $(R)_1^{A^T}$.

So we need only show that the layer potentials $\K_\pm^t,$ $\J^t$ are bounded and invertible on $L^p$ or $H^1$. (Invertibility of $\K_+$ follows from invertibility of its transpose $\K_+^t$.)

If $A$ is smooth, and
\[\Omega=\{X\in\R^2:\phi(X\cdot\e^\perp)<X\cdot \e\}\]
for some Lipschitz function $\phi$ (which we may assume, a priori, to be in $C^\infty_0$) and some unit vector~$\e$, then we can prove that $\K,\J$ are bounded $H^1(\partial\Omega)\mapsto H^1(\partial\Omega)$ and $L^p(\partial\Omega)\mapsto L^p(\partial\Omega)$ for $1<p<\infty$.
We can then extend this to arbitrary Lipschitz domains $\V$. (\thmref{Tbounded}, \thmref{patching} and \thmref{H1patching}.)

There is some $p_0>1$ (in fact, the ${p_0}$ in \thmref{rule}) and some $\epsilon>0$ such that if $\|A-A_0\|<\epsilon$, then $\K_\pm$ has a bounded inverse on $L^{p_0'}(\partial\V)$ and $\J^t$, $\K^t_\pm$ have bounded inverses on $L^{p_0}(\partial\V)$ and $H^1(\partial\V)$. (Chapters~\ref{chap:invertibility} and~\ref{chap:H1}.)

Using standard interpolation techniques (\autoref{chap:interpolation}), we can show that $\K^t_\pm,\J^t$ have bounded inverses on $L^p$ for $1<p<p_0$, and therefore $\K_\pm$ has a bounded inverse on $L^{p'}$ for $p_0'<p'<\infty$.

Thus, we have that $(N)^A_p$, $(R)^A_p$, $(D)^A_{p'}$ hold if $A$ is smooth, for $p>1$ small. We pass to arbitrary (rough) $A$ in \autoref{chap:apriori}.
\end{proof}

\chapter{Definitions}
\label{chap:dfn}
\allowdisplaybreaks[1] If $u\in W^{1,2}_{loc}(\V)$, then we define
$\div A\nabla u=0$ in $\V$ in the weak sense, that is, 
\[\int A\nabla u\cdot\nabla\eta=0\]
for all $\eta\in C^\infty_0(\V)$. Similarly, we define $\nu\cdot A\nabla u =g$ on $\partial\V$ in the weak sense, that is, 
\[\int_\V A\nabla u\cdot\nabla \eta=\int_{\partial\V} g\eta,\]
for all $\eta\in C^\infty_0(\R^2)$.

We say that $u=f$ on $\partial\V$ if $f$ is the non-tangential limit of $u$ a.e., that is, if
\[\lim_{\eta\to 0}\sup\{|u(Y)-f(X)|: |X-Y|<(1+a)\dist(Y,\partial\V)<\eta\}\]
holds for almost every $X\in\partial\V$. 

We fix some notation. If $\U$ is a domain, then $\U=\U_+$, $\bar\U^C=\U_-$. The inner product between $L^p$ and $L^{p'}$ will be given by
\[\langle G, F\rangle=\int G(x)^t F(x)\,dx.\]
(This notation is more convenient than the usual inner product $\int \overline{G^t(x)} F(x)\,dx$.) A superscript of $t$ will denote the transpose of a matrix or the adjoint of an operator with respect to this inner product. (So if $f,g$ happen to be scalar-valued functions, then $f=f^t$ and $\langle f,g\rangle=\int fg$, and if $P$ is an operator, then $\langle F,PG\rangle=\langle G,P^t F\rangle^t$.)

If I have a function or operator defined in terms of $A$, then a superscript of $T$ will denote the corresponding function or operator for $A^t$. (So $A^t=A^T$; however, $P^t\neq P^T$ for most operators defined in terms of~$A$.)

If $\U$ is a domain, then if $X\in\partial\U$ and $f$ is a function defined on $\partial\U$, we let $Mf(X)$ be the generalization of the Littlewood-Paley maximal function given by $Mf(X)=\sup\{\dashint_I|g|:I\owns X, I\subset\partial\V$ connected$\}$.

If $f$ is defined on $\U$, we define the non-tangential cones $\gamma$ and non-tangential maximal function $N$ by
\begin{equation}\label{eqn:dfn:nontang}
\gamma_{\U,a}(X)=\{Y\in\U:|X-Y|<(1+a)\dist(Y,\partial\U)\},\quad
N_{\U,a} f(X)=\sup_{\gamma_{\U,a}(X)}|f(Y)|
\end{equation}
for some number $a>0$. When no ambiguity will arise we suppress the subscripts~$\U$ or~$a$; we let $\gamma_\pm(X)=\gamma_{\U_\pm}(X)$.

\begin{dfn} We say that the domain $\Omega$ is a \dfnemph{special Lipschitz domain} if, for some Lipschitz function $\phi$ and unit vector $\e$,\label{dfn:domain}
\[\Omega=\{X\in\R^2:\phi(X\cdot\e^\perp)<X\cdot\e\}.\]
We refer to $\|\phi'\|_{L^\infty}$ as the \dfnemph{Lipschitz constant of~$\Omega$.}

We say that $\U$ is a \dfnemph{Lipschitz domain} with \dfnemph{Lipschitz constant $k_1$} if there is some $k_1>0$ such that, for every $X\in\partial\U$, there is some neighborhood $\W$ of $X$ such that $\U\cap \W=\Omega_X\cap\W$ for some special Lipschitz domain $\Omega_X$ with Lipschitz constant at most~$k_1$.

If $\U$ is a special Lipschitz domain, then let $k_2=k_3=1$. Otherwise, let $k_2,$ $k_3>1$ be numbers such that $\partial\U$ may be covered (with overlaps) by at most $k_2$ such neighborhoods whose size varies by at most $k_3$.

That is, there is a constant $r$ such that
\[\partial\U\subset\bigcup_{j=1}^{k_2} B(X_j,r_j)\]
for some $X_j\in\partial\U$ and $r_j\in(r/k_3, rk_3)$. We further require that
there exist unit vectors $\e_j$ and Lipschitz functions $\phi_j$, with $\|\phi_j'\|_{L^\infty}\leq k_1$, such that if 
\begin{align*}
\Omega_j&=\{X\in\R^2:\phi_j((X-X_j)\cdot\e_j^\perp)<(X-X_j)\cdot\e_j\},
\\ 
R_j&=\{X\in\R^2:|(X-X_j)\cdot\e_j^\perp|<2r_j,|(X-X_j)\cdot\e_j|<(2+2k_1) r_j\},\end{align*}
then
\[\U\cap R_j=\Omega_j\cap R_j.\]

We refer to $k_1$, $k_2$, $k_3$ as the \dfnemph{Lipschitz constants} of $\U$. For simplicity, when we write $k_i$, we mean $k_1$,~$k_2$, $k_3$.

If the $k_i$ are finite, and $\partial\U$ is connected, then we call $\U$ a \dfnemph{good Lipschitz domain}.
\end{dfn}

We will reserve $\V$ for good Lipschitz domains and $\Omega$ for special Lipschitz domains. Note that if $\V$ is a good Lipschitz domain, then either $\V$ is special or $\partial\V$ is bounded.

Note that we do not care about the value of the positive constant $r$ in the definition of $k_2$,~$k_3$; this is because $(D)_p^A$, $(N)_p^A$, $(R)_p^A$ are scale-invariant.

A number depending on these constants is very important. Let $k_4$ be such that, if $X_0\in\partial\U$ and $r>0$, then $\sigma(\partial\U\cap B(X_0,r))\leq k_4 r$. We refer to $k_4$ as the \dfnemph{Ahlfors-David constant} of $\U$.


We can bound some integrals which will be needed later:
\begin{equation}\label{eqn:goodCZbound}
\int_{\partial\U\backslash B(X_0,r)}\frac{r^{p-1}}{|X-X_0|^p}\,d\sigma(X)
\leq \sum_{j=0}^\infty \frac{r^{p-1}}{(b^{j}r)^p}\sigma(B(X_0,rb^{j+1}))
\leq \frac{k_4}{p-1}\end{equation}
where $b=p^{1/(p-1)}$. In particular, if $X_0\notin\partial\U$ then
\[\int_{\partial\U}\frac{1}{|X-X_0|^p}\,d\sigma(X)\leq \frac{k_4}{p-1}\dist(X_0,\partial\U)^{1-p}.\]

The letter $C$ will always represent a positive constant, whose value may change from line to line, but which depends only on the ellipticity constants $\lambda$,~$\Lambda$ of $A$,~$A_0$, the positive constant $a$ in the definition of non-tangential maximal function, and the Lipschitz constants of whatever domain we are dealing with. If a particular constant depends on another parameter, it will be indicated explicitly. We will occasionally use the symbols $\lesssim$,~$\gtrsim$ to indicate inequality up to a multiplicative constant (e.g.\ $\|\T f\|\lesssim \|f\|$ as shorthand for $\|\T f\|\leq C\|f\|$); we will use $\approx$ to mean that $\lesssim$ and~$\gtrsim$ both hold.

\begin{lem}\label{lem:fundsoln:exists} Let $A:\R^2\mapsto \C^{2\times2}$ be an elliptic matrix-valued function. Then, for each $X\in\R^2$, there is a function $\Gamma_X\in W^{1,1}_{loc}(\R^2)$, unique up to an additive constant, such that
\[|\Gamma_X(Y)|\leq C+C|\log|X-Y| |\] for every $Y\in\R^2,$ and
\begin{equation}\label{eqn:fundsoln}
\int_{\R^2} A(Y) \nabla \Gamma_X(Y) \cdot \nabla \eta(Y) \, dy
=-\eta(X)\end{equation}
for every $\eta\in C^\infty_0(\R^2)$.\end{lem}

We refer to this function as the \dfnemph{fundamental solution for $\div A\nabla$ with pole at~$X$.}

This lemma will be proven in \autoref{chap:fundsoln}.
By $\nabla \Gamma_X(Y)$ we mean the gradient in $Y$. We will sometimes wish to refer to the gradient in $X$; we will then write $\nabla_X\Gamma_X(Y)$.

$H^1(\R)$ is defined to be $\big\{f:\R\mapsto\C\big| \|\sup_t f*\Phi_t(x)\|_{L^1}$ is finite$\big\}$ for some Schwarz function $\Phi$ with $\int\Phi=1$, where $\Phi_t(x)=\frac{1}{t}\Phi(x/t)$. It can be shown that if $f\in H^1$, then $f=\sum_k \lambda_k a_k$, where $\lambda_k\in\C$, $\int a_k=0,$ $\|a_k\|_{L^\infty}\leq 1/r_k$, $\supp a_k\subset B(x_k,r_k)$ for some $r_k>0$, and $\sum_k |\lambda_k|\approx \|f\|_{H^1}$. Functions $a$ satisfying these conditions are called \dfnemph{atoms}.

We may extend the definition of $H^1$ to $H^1(\partial\U)$, where $\U$ is a Lipschitz domain. We say that $f\in H^1(\partial\U)$ if $f=\sum_k \lambda_k a_k$, where the $\lambda_k$ are complex numbers, and $\int_{\partial\U}a_k\,d\sigma=0$, $\supp a_k\subset\Delta_k$ for some $\Delta_k\subset\partial\U$ is connected, and $\|a_k\|_{L^\infty}\leq 1/\sigma(\Delta_k)$. The norm is the smallest $\sum_k|\lambda_k|$ among all such representations of~$f$.

If $\U$ is a good Lipschitz domain, then this is equivalent to defining $H^1$ atoms to be functions $a$ which satisfy $\int_{\partial\U}a\,d\sigma=0$, $\supp a\subset B(X,R)\cap\partial\U$ for some $X\in\partial\U$, and $\|a\|_{L^\infty}\leq 1/R$.

\ifmulticonnect Note that this means that $f|_{\gamma}\in H^1(\gamma)$ for all connected components $\gamma$ of $\partial\U$. In particular, $\int_\gamma f\,d\sigma=0$. For the regularity problem, this condition is obviously necessary. For the Neumann problem, we might like to relax this condition to functions $f$ such that $\int_{\partial\U} f\,d\sigma=0$. This generalization will not be needed to prove $(N)_p$ for $p>1$.\fi

We consider $BMO(\partial\U)$ to be the dual of $H^1(\partial\U)$. This means that
\begin{align*}\|f\|_{BMO(\partial\U)}
&=\sup_{\Delta\subset\partial\U\text{ connected}}\frac{1}{\sigma(\Delta)}\int_{\Delta} \left|f-{\textstyle\dashint_\Delta} f\right|\,d\sigma
\\&\approx
\sup_{X\in\partial\V,\>R>0} \inf_C\frac{1}{R}\int_{B(X,R)\cap\partial\U} \left|f-C \right|\,d\sigma
.\end{align*}

\ifmulticonnect
If $\partial\U$ is not connected, then we require $H^1$ functions to integrate to 0 on each component of $\partial\U$, not on all of $\partial\U$. Consequently, if $f$ is constant on each connected component of $\partial\U$, then $\|f\|_{BMO}=0$, even if $f$ is not constant on all of $\partial\U$. 
\else

Multiply connected domains are beyond the scope of this paper. However, many lemmas and theorems in this paper have obvious generalizations to multiply connected domains. For the most part, these require that $H^1$ functions integrate to 0 on each connected component of~$\partial\V$.
\fi

\begin{dfn}\label{dfn:layers} Let $\V$ be a Lipschitz domain with unit outward normal and tangent vectors $\nu$ and~$\tau$.
If $f:\partial\V\mapsto \C$ is a function, we define
\begin{align}
\label{eqn:S}
\nabla\S f(X) &=\nabla \S_\V f(X)=\int_{\partial\V} \nabla_X\Gamma_{X}^T(Y) f(Y)\,d\sigma(Y)
\\
\label{eqn:D}
\D f(X)&= \D_\V f(X)=\int_{\partial\V} \nu(Y)\cdot A^T(Y)\nabla\Gamma_X^T(Y) f(Y)\,d\sigma(Y).
\end{align} 
This defines $\S f$ up to an additive constant.

We define the layer potentials $\K_\pm$, $\J$ via
\begin{align}
\label{eqn:K}
\K_\pm f(X)&=\K_{\V,\pm}^A f(X)=\lim_{Z\to X, \> Z\in \gamma_\pm(X)}\int_{\partial\V}
\nu(Y)\cdot A^T(Y) \nabla\Gamma_{Z}^{A^T}(Y) f(Y) \,d\sigma(Y)
\\
\label{eqn:J}
\J f(X)&=\J_\V^A f(X)=\lim_{Z\to X, \> Z\in\gamma_+(X)\cup\gamma_-(X)}\int_{\partial\V} 
\tau(Y)\cdot \nabla\Gamma_{Z}^{A^T}(Y) f(Y) \,d\sigma(Y).
\end{align}
\end{dfn}
When no confusion will arise we omit the subscripts and superscripts. By (\ref{eqn:gradfundsoln}), $|\nabla\Gamma_X(Y)|\leq C/|X-Y|$; hence, $\D f$ and $\nabla\S f$ converge for $X\notin\partial\V$ and $f\in L^p$, $1\leq p<\infty$.

We will show that the limits in the definition of $\K$ and $\J$ are well-defined for $f\in L^p$ smooth, $1\leq p<\infty$. (\lemref{Kwelldef}.)

$\K_+-\K_-=I$ as operators on $L^p$ (\lemref{Kpm}), so $\K_+\neq \K_-$; however, if $f\in L^p$ with $1<p<\infty$, then at any point where the obvious analogies $\J_{\pm}f$ exist, they are equal. (\lemref{Sfcts}).

It will be shown that
\[\K^t_\pm f = \nu\cdot A^T\nabla \S^T f|_{\partial\V_{\mp}},\quad
\J^t f = \tau\cdot \nabla\S^T f.\]
(\lemref{Ktranspose}).


We also need the following definitions:
\begin{align}
\label{dfn:bigdefstart}
A(X)&=\Matrix{a_{11}(X)&a_{12}(X)\\ a_{21}(X)&a_{22}(X)},\quad
A_0(X)=\Matrix{a_{11}^0(X)&a_{12}^0(X)\\ a_{21}^0(X)&a_{22}^0(X)}\\
\nabla\tilde\Gamma_X^T(Y)&=\Matrix{0&-1\\1&0}A(Y)\nabla \Gamma_X^T(Y)
\\
B_6^A(X)&=\Matrix{a_{11}(X)&a_{21}(X)\\0&1}\\
B_7^A(X)&=\left(B_6(X)^{t}\right)^{-1}\Matrix{A(X)\nu(X)&\tau(X)}\\
K^A(X,Y) &= \Matrix{B_6^A(Y)\nabla \Gamma^{A^T}_{X}(Y) 
     &B_6^A(Y)\nabla \Gamma^{A^T}_{X}(Y) }^t\label{dfn:kernelgen}\\
\tilde K^A(X,Y)&= B_6^A(X)\Matrix{
     \nabla_X\tilde \Gamma^{A^T}_{X}(Y)&
     \nabla_X\tilde \Gamma^{A^T}_{X}(Y)}
     \\
\T_\V^A F(X) &=\lim_{Z\to X\text{ n.t.},Z\in\V} \int_{\partial\V} K^A(Z,Y) F(Y)\,d\sigma(Y)\label{dfn:Tgen}
\\
\tilde\T_\V F(X) &=\lim_{Z\to X\text{ n.t.},Z\in\V} \int_{\partial\V} \tilde K(Z,Y) F(Y)\,d\sigma(Y)
\end{align}

When considering special Lipschitz domains, we will need some
terminology:
\begin{align}
\e&=\Matrix{e_1\\e_2},\quad \e^\perp=\Matrix{e_2\\-e_1}\\
\psi(x) &= x\e^\perp +\phi(x)\e\in\partial\Omega,\\
\label{eqn:psi}
\psi(x,h)&=\psi(x)+h\e=
\Matrix{xe_2+(\phi(x)+h)e_1\\-xe_1+(\phi(x)+h)e_2},\\
\nonumber
\lefteqn{\text{ so }(x,t) =\psi(e_2 x-e_1 t, e_1 x+ e_2 t -\phi(e_2 x-e_1 t)),}
\\
\langle G, T_\pm F\rangle &=\lim_{h\to 0^\pm} \int_{\R^2} G(x)^t K_h(x,y) F(y)\,dy\,dx 	\label{dfn:T}\\
\langle G, \tilde T_\pm F\rangle &= \lim_{h\to 0^\mp}
     \int_{\R^2} G(x)^t \tilde K_h(x,y) F(y)\,dy\,dx
\\
K_h(x,y) &=\Matrix{B_6(\psi(y))\nabla \Gamma^T_{\psi(x,h)}(\psi(y)) 
     &B_6(\psi(y))\nabla \Gamma^T_{\psi(x,h)}(\psi(y)) }^t\label{dfn:kernel}\\
\nonumber&= \Matrix{
	\nabla \Gamma^T_{{\psi(x,h)}}(\psi(y))^t \\
	\nabla \Gamma^T_{{\psi(x,h)}}(\psi(y))^t }B_6(\psi(y))^t
	= K(\psi(x,h),\psi(y))
\\
\tilde K_h(x,y)&= B_6(\psi(x))\Matrix{
     \nabla_X\tilde \Gamma^T_{{\psi(x)}}(\psi(y,h))&
     \nabla_X\tilde \Gamma^T_{{\psi(x)}}(\psi(y,h))}
\\
\nu(x)&=\frac{1}{\sqrt{1+\phi'(x)^2}}(\phi'(x)\e^\perp-\e)=\frac{1}{\sqrt{1+\phi'(x)^2}}\Matrix{-e_1+e_2\phi'(x)\\-e_2-e_1\phi'(x)}\\
\tau(x)&=\frac{1}{\sqrt{1+\phi'(x)^2}}(\e^\perp+\phi'(x)\e)=\frac{1}{\sqrt{1+\phi'(x)^2}}\Matrix{e_2+e_1\phi'(x)\\-e_1+e_2\phi'(x)}
\\
\nonumber 
B_1(x)&=B_1(\psi(x)),\\
B_1(\psi(x,h))&=\left(B_6(\psi(x,h))^{t}\right)^{-1} 
     \sqrt{1+\phi'(x)^2}
     \Matrix{A(\psi(x,h))\nu(x)&\tau(x)}.\label{eqn:B1def}
\end{align}
Note that $\Omega=\{X:\phi(X\cdot\e^\perp)<X\cdot\e\}=\{\psi(x,h):x\in\R,h>0\}.$

We will occasionally want slightly different forms of $K,T$:
\begin{align}
K'_h(x,y) &= \Matrix{B_6(\psi(y,h))\nabla \Gamma^T_{\psi(x)}(\psi(y,h)) 
     &B_6(\psi(y,h))\nabla \Gamma^T_{\psi(x)}(\psi(y,h)) }^t\\
\nonumber&= \Matrix{
	\nabla \Gamma^T_{{\psi(x)}}(\psi(y,h))^t \\
	\nabla \Gamma^T_{{\psi(x)}}(\psi(y,h))^t }B_6(\psi(y,h))^t\\
\tilde K'_h(x,y)&= B_6(\psi(x,h))\Matrix{
     \nabla_X\tilde \Gamma^T_{{\psi(x,h)}}(\psi(y))&
     \nabla_X\tilde \Gamma^T_{{\psi(x,h)}}(\psi(y))}
\\
\langle G, T'_\pm F\rangle &=\lim_{h\to 0^\pm} \int_{\R^2} G(x)^t K'_h(x,y) F(y)\,dy\,dx 
	\label{dfn:T'}\\
\langle G, \tilde T'_\pm F\rangle &= \lim_{h\to 0^\pm}
     \int_{\R^2} G(x)^t \tilde K'_h(x,y) F(y)\,dy\,dx
\label{dfn:bigdefend}
\end{align}
We will later show (\autoref{sec:analytic}) that if $A-I\in C^\infty_0(\R\mapsto \C^{2\times 2})$, then $T_\pm=T'_\mp$ on $C^\infty_0$. These requirements will be dealt with in \autoref{sec:buildup} and \autoref{chap:apriori}.

Note that, if $f$ is a function defined on $\partial\Omega$, we will often use $f(x)$ as shorthand for $f(\psi(x))$.

So $\T_{\Omega_\pm} F(\psi(x))=T_{\pm} (\sqrt{1+(\phi')^2}F\circ\psi)(x)$, and
\begin{align*}
T_\pm(B_1 f)(x)&=\T_{\Omega_\pm}(B_7 f)(\psi(x))=\lim_{h\to 0^\pm}
\int_\R K_h(x,y) B_1(y)f(y)\,dy
\\&=\lim_{h\to 0^\pm}\int_\R \Matrix{
	\nabla \Gamma^T_{{\psi(x,h)}}(\psi(\cdot))^t \\
	\nabla \Gamma^T_{{\psi(x,h)}}(\psi(\cdot))^t }
     \sqrt{1+(\phi')^2}
     \Matrix{A(\psi(\cdot))\nu&\tau}	
	f
\\&=\Matrix{\K_\pm f(\psi(x))&\J f(\psi(x))\\
\K_\pm f(\psi(x))&\J f(\psi(x))}
\end{align*}
and $B_1$ is bounded with a bounded inverse; thus, we need only show that $T_\pm$ is bounded from $L^p(\R\mapsto \C^{2\times 2})$ to itself to show that the layer potentials are bounded.
\allowdisplaybreaks[0]

\chapter{Useful theorems}
\label{chap:PDEbasic}
First, we start with a lemma about non-tangential maximal functions:
\begin{lem}\label{lem:NTM}
Recall that $N_a f(X)=\sup\{|f(Y)|:Y\in\V,\>|X-Y|\leq (1+a)\dist(Y,\partial\V)\}$ where we omit the $\V$ subscript.
Suppose that $0<a<b$ and $\V$ is a good Lipschitz domain. Then for all $1\leq p\leq\infty$,
\[\|N_b f\|_{L^p(\partial\V)}\leq C\|N_a f\|_{L^p(\partial\V)}\]
for some constant $C$ depending only on $a$, $b$ and the Lipschitz constants of~$\V$.
\end{lem}

\begin{proof} It is obvious for $p=\infty$. Suppose that $N_a f\in L^p(\partial\V)$ for $1\leq p<\infty$. Define \[\delta(Y)=\dist(Y,\partial\Omega),\quad
E_a(\alpha)=\bigcup_{|f(Y)|>\alpha} B(Y,(1+a)\delta(Y)).\]

Recall that
\[ \|N_a f\|_{L^p(\partial\V)}^p 
=\int_0^\infty p\alpha^{p-1} \sigma\{X\in\partial\V: N_a f(X)>\alpha\}\,d\alpha.
\]

So we need only show that $\sigma(E_b(\alpha)\cap\partial\V)\leq C\sigma(E_a(\alpha)\cap\partial\V)$ for all $\alpha>0$ to complete the proof. In fact, if $\partial\V$ is bounded, we need this result only for $\alpha>2\|N_a f\|_{L^p}\sigma(\partial\V)^{-1/p}$.

If $|f(Y)|>\alpha$, then $\alpha^p\sigma(B(Y,(1+a)\delta(Y))\cap\partial\V)\leq \|N_a f\|_{L^p}^p$. If $\V=\Omega$ is a special Lipschitz domain, then $\delta(Y)\leq \frac{1}{2a} \|Nf\|_{L^p}^p/\alpha^p$. If $\partial\V$ is bounded, and if $\sigma(\partial\V)\alpha^p>2\|N_a f\|_{L^p}^p$, we must have that some $X\in\partial\V$ is not in $B(Y,(1+a)\delta(Y))$; therefore, $\delta(Y)\leq \diam(\partial\V)/a$.

In either case $E_a(\alpha)$ is a union of balls with bounded radii. So by the Vitali lemma there is a countable set $\{Y_i\}_{i=1}^\infty$ of points such that the $B(Y_i,(1+a)\delta(Y))$ are pairwise disjoint and such that $E_a(\alpha)\subset \cup_i B(Y_i,C_1 (1+a)\delta(Y_i))$ where $C_1$ is a fixed constant (i.e.\ depends on nothing except the dimension of the ambient space $\R^2$).

Note that if $|f(Y)|>\alpha$, then $Y\in B(Y_i, C_1(1+a)\delta(Y_i))$ for some $i$. Then $\delta(Y_i)\geq \delta(Y)-|Y-Y_i| \geq \delta(Y)-C_1(1+a)\delta(Y_i)$, so $\delta(Y)\leq (C_1+C_1 a+1)\delta(Y_i)$. 

Thus, $B(Y,(1+b)\delta(Y))\subset B(Y_i,(1+b)\delta(Y)+|Y-Y_i|)$
and so
$E_b(\alpha)\subset \cup_i B(Y_i, C_2\delta(Y_i))$ where $C_2>1+a$ depends only on $a$, $b$ and~$C_1$.

Let $Y=Y_i$ for some $i$, and suppose that $Y^*\in\partial\V$ with $|Y-Y^*|=\dist(Y,\partial\V)$. 
\ifmulticonnect Let $\omega$ be the connected component of $\partial\V$ containing $Y^*$.\else \def\omega{\partial\V}\fi
Then either $\omega\not\subset B(Y,(1+a)\delta(Y)$, so 
\[\sigma\left(B\left(Y,C\delta(Y)\right)\cap\partial\V\right)
\leq
C\delta(Y)
\leq
C(2a\delta(Y))
\leq
C\sigma(B(Y,(1+a)\delta(Y))\cap \partial\V)
\]
or $\omega\subset B(Y,(1+a)\delta(Y))$, so
\[\sigma\left(B\left(Y,C\delta(Y)\right)\cap\partial\V\right)
\leq \sigma(\partial\V)
\ifmulticonnect\leq k_5\sigma(\omega)\fi
\leq C\sigma\left(B\left(Y,(1+a)\delta(Y)\right)\cap\partial\V\right)
\]

Thus,
\begin{align*}
\sigma(E_b(\alpha)\cap\partial\V)
&\leq 
\sum_i \sigma(B(Y_i, C\delta(Y_i))\cap\partial\V)
\leq 
\sum_i C\sigma(B(Y_i, (1+a)\delta(Y_i))\cap\partial\V)
\\&\leq 
C\sigma(E_a(\alpha)\cap\partial\V)
.\end{align*}
This completes the proof.
\end{proof}

We now turn to lemmas about solutions to elliptic PDE.
Suppose that $\div A\nabla u=0$ in~$\U$ where $A$ satisfies~(\ref{eqn:genelliptic}). Let $B_r\subset\R^n$ be a ball of radius $r$, $B_{r/2}$ be the concentric ball of radius $r/2$, such that either $u=0$ on $\partial\U\cap B_r$ or $\nu\cdot A\nabla u=0$ on $\partial\U\cap B_r$. (If $B_r\subset U$, then these conditions are both trivially true, and in fact it is this case which is used most often.)

In $\R^n$, the following lemmas are well-studied. They hold in all dimensions whenever $A$ is real; they all hold even for complex $A$ when working in $\R^2$ (but not in $\R^3$ and higher dimensions).

\begin{lem}\label{lem:PDE1} For some constant $C$ depending only on $\lambda,\Lambda$,
\[\int_{\U\cap B_{r/2}}|\nabla u|^2\leq \frac{C}{r^2}\int_{\U\cap B_r} |u|^2.\]
\end{lem}

\begin{lem}\label{lem:PDE2}
For some $C>0$, $p>2$ depending only on $\lambda,\Lambda$,
\[\left(\frac{1}{r^n}\int_{B_{r/2}\cap\U}|\nabla u|^p\right)^{1/p}
\leq C\left(\frac{1}{r^n}\int_{B_{r}\cap\U}|\nabla u|^2\right)^{1/2}
.\]
\end{lem}

\begin{lem}\label{lem:PDE3}
For all $p\geq 2$, there is a constant $C(p)$ depending only on $\lambda,\Lambda$ and $p$, such that
\[\sup_{B_{r/2}\cap\U}|u|\leq C(p)\frac{|B_r|}{|\U\cap B_r|} \left(\frac{1}{r^2}\int_{B_r\cap\U}|u|^p\right)^{1/p}.\]
\end{lem}

\begin{lem}\label{lem:PDE4}
If $\div A\nabla u=0$ in all of $B_r$, then for some $C,\alpha>0$ depending only on $\lambda,\Lambda$,
\[\sup_{X,Y\in B_{r/2}}|u(X)-u(Y)|\leq C\frac{|X-Y|^\alpha}{r^{1+\alpha}}\|u\|_{L^2(B_r)}.\]
\end{lem}

Now, assume that $A$ satisfies (\ref{eqn:elliptic}), so $A(x,t)=A(x)$. Then $u_t$ is also a solution, and so Lemmas~\ref{lem:PDE1}\mbox{--}\ref{lem:PDE4} hold with~$u$ replaced by~$u_t$. 



In \autoref{sec:conjugate}, we will construct for each such $u$ a function $\tilde u$ such that $u_x=\frac{1}{a_{11}}\tilde u_t+\frac{a_{12}}{a_{11}}u_t$ and $\div \frac{1}{\det A} A^T \nabla \tilde u=0$.

So
\begin{align}
\label{eqn:gradu}
\sup_{B_{r/2}}|\nabla u|
&\leq \sup_{B_{r/2}}|u_x|+\sup_{B_{r/2}}|u_t|
\leq C\sup_{B_{r/2}}|\tilde u_t|+C\sup_{B_{r/2}}|u_t|
\\\nonumber
&\leq C(p)\left(\dashint_{B_r}|\tilde u_t|^p\right)^{1/p} +C(p)\left(\dashint_{B_r}|u_t|^p\right)^{1/p}
\\\nonumber
&\leq C(p)\left(\dashint_{B_r}|\nabla u|^p\right)^{1/p} \end{align}
So \lemref{PDE3} holds for $\nabla u$ as well as $u$ if $A(x,t)=A(x)$ and $B_r\subset\U$.

\chapter{The fundamental solution}
\label{chap:fundsoln}
\section{A fundamental solution exists}
\label{sec:fundsolnexists}

Recall that \lemref{fundsoln:exists} states that there is a function $\Gamma_X^A(Y)$, unique up to an additive constant, called the fundamental solution of the operator
$L=\div A\nabla $, such that
\[\int_{\R^2} A(Y) \nabla \Gamma_X^A(Y) \cdot \nabla \eta(Y) \, dY
=-\eta(X)\]
for every $\eta\in C^\infty_0(\R^2)$, and where $|\Gamma_X(Y)|\leq C(1+|\log |X-Y| |)$ for some constant $C$. We will let $\Gamma=\Gamma^A,$ $\Gamma^0=\Gamma^{A_0}$, and $\Gamma^T=\Gamma^{A^T}$.

We now prove this lemma. From \cite[pp.~29--31]{AusTcha}, I know that there is a function\footnote{They refer to it as $K$; I use $\check K$ to differentiate it from the $K$ in (\ref{eqn:K}).} $\check K_t(X,Y)$ such that, for all $\eta\in C^\infty_0(\R^2)$,
\[\int \eta(X)\partial_t \check K_t(X,Y)\, dX = \int A(X) \nabla_X \check K(X,Y) \cdot
\nabla \eta(X) \, dX.\]
Furthermore, there is some $\beta$, $\mu$, $C>0$ depending only on $\lambda$,~$\Lambda$ such that
\begin{align*}
|\check K_t(X,Y)|&\leq \frac{C}{t}\exp\left\{-\frac{\beta|X-Y|^2}{t}\right\}
\\
|\check K_t(X,Y)-\check K_t(X',Y)|&\leq \frac{C}{t}\left(\frac{|X-X'|}{\sqrt{t}+|X-Y|}\right)^\mu\exp\left\{-\frac{\beta|X-Y|^2}{t}\right\}
\\
|\check K_t(X,Y)-\check K_t(X,Y')|&\leq \frac{C}{t}\left(\frac{|Y-Y'|}{\sqrt{t}+|X-Y|}\right)^\mu\exp\left\{-\frac{\beta|X-Y|^2}{t}\right\}
\end{align*}
whenever $|X-X'|$, $|Y-Y'|<\frac{1}{2}(\sqrt{t}+|X-Y|)$. This $\check K_t$ is called the Schwarz kernel of the operator $e^{-tL}$.

Let $J_t(X,Y)=\check K_t(X,Y)-\dashint_{r\leq |Z-Y|\leq 2r} \check K_t(Z,Y)\,dZ$, so that if $r<|X-Y|<2r$,
\begin{align*}|J_t(X,Y)|
&=\left|\check K_t(X,Y)-\dashint_{r\leq |Z-Y|\leq 2r} \check K_t(Z,Y)\,dZ\right|
\\&=\dashint_{r\leq |Z-Y|\leq 2r} \left|\check K_t(X,Y)-\check K_t(Z,Y)\right|\,dZ
\\&\leq \dashint_{r\leq |Z-Y|\leq 2r}
\frac{C}{t}\left(\frac{|X-Z|}{\sqrt{t}+|X-Y|}\right)^\mu\exp\left\{-\frac{\beta|X-Y|^2}{t}\right\}\,dZ
\\&\leq
\frac{C}{t}\left(\frac{|X-Y|}{\sqrt{t}+|X-Y|}\right)^\mu\exp\left\{-\frac{\beta|X-Y|^2}{t}\right\}
\end{align*}

So $J_t\in L^1_t$.

From  \cite[p.~54]{AusTcha}, there is some $\epsilon, \beta, c>0$,
such that
\[\left(\int_{r\leq |X-Y|\leq 2r} |\nabla_X K_t(X,Y)|^2\,dX\right)^{1/2}
\leq \frac{c}{t}\left(\frac{r^2}{t}\right)^\epsilon e^{-\beta r^2/t}.\]
So, by H\"older's inequality,
\begin{align*}\int_{r\leq |X-Y|\leq 2r} |\nabla_X K_t(X,Y)|\,dX
&\leq \sqrt{3\pi} r
\left(\int_{r\leq |X-Y|\leq 2r} |\nabla_X K_t(X,Y)|^2\,dX\right)^{1/2}
\\&\leq
\frac{Cr}{t}\left(\frac{r^2}{t}\right)^\epsilon e^{-\beta r^2/t}.
\end{align*}
So
\begin{align*}
\int_0^\infty \int_{r\leq |X-Y|\leq 2r} |\nabla_X K_t(X,Y)|\,dX\,dt
&\leq
 \int_0^\infty Cr\frac{r^{2\epsilon}}{t^{1+\epsilon}} e^{-\beta r^2/t}\,dt
\\&=
 \int_0^\infty Cr\frac{r^{2\epsilon}}{r^{2+2\epsilon}t^{1+\epsilon}} e^{-\beta /t}r^2\,dt
\\&=
Cr\int_0^\infty \frac{1}{t^{1+\epsilon}} e^{-\beta/t}\,dt
\end{align*}
Since that integral converges, we know that
\[\int_0^\infty \nabla_X K_t(X,Y)\,dt\]
converges almost everywhere in $B(Y,2r)-B(Y,r)$.

Define
\[\Gamma_Y(X) = C(r)+\int_0^\infty J_t(X,Y)\,dt\]
so
\[\nabla\Gamma_Y(X)=\int_0^\infty \nabla_X J_t(X,Y)\,dt=\int_0^\infty \nabla_X K_t(X,Y)\,dt\]
where $C(r)$ is chosen such that the values of $\Gamma$ on different annuli agree.
Then
\begin{align*}
{\int A\nabla\Gamma_Y \cdot \nabla \eta}
&=\int A(X)\int_0^\infty \nabla_X K_t(X,Y)\,dt\cdot \nabla\eta(X)\,dX\\
&=\int_0^\infty\int A(X)\nabla_X K_t(X,Y)\cdot \nabla\eta(X)\,dX\,dt\\
&=\int_0^\infty\int \partial_t K_t(X,Y)\eta(X)\,dX\,dt\\
&=\lim_{t\to\infty}\int K_t(X,Y)\eta(X)\,dX
-\lim_{t\to 0^+}\int K_t(X,Y)\eta(X)\,dX\\
&=0-\eta(Y)
\end{align*}
whenever $\eta\in C^\infty_0$.

So we have constructed a fundamental solution.

Furthermore, we have the following bound on its gradient:
\[\int_{|X-Y|\leq r}|\nabla \Gamma_Y(X)|\,dX\leq Cr. \]
This allows us to put an upper bound $C|\log|X-Y| |+C$ on $|\Gamma_Y(X)|$. 
Furthermore, by (\ref{eqn:gradu}), we have that
\begin{equation} \label{eqn:gradfundsoln}
|\nabla \Gamma_Y(X)|\leq\frac{C}{|X-Y|}.
\end{equation}

\section{Uniqueness of the fundamental solution} \label{sec:fundunique}

Let $u=\Gamma_X$, and assume that $|v(Y)|\leq C(X)(1+|\log |Y| |)$, and $\int A\nabla v\cdot\nabla\eta=-\eta(X)$ for all $\eta\in C^\infty_0$. (We will need $v$ to be this general later.)
Then $w=u-v$ is a solution to $Lw=0$
which satisfies $|w(Y)|\leq C(X)+C(X)|\log|Y| |$; \lemref{PDE4} allows us to conclude that
\[|w(Y)-w(Z)|\leq \frac{C|Y-Z|^\alpha}{r^{1+\alpha}} \|w\|_{L^2(B(X,r))}
\leq C(X)|Y-Z|^\alpha\frac{\log r}{r^\alpha}\]
for any $r$ large enough. This implies that $w$ is a constant; thus $\Gamma_X$ is unique up to an additive constant. (So $\nabla\Gamma_X$ is actually unique.)

If $A(y,s)=A(y)$, let $\zeta(X)=\eta(X-(0,t))$. Then
\begin{align*}
-\eta(X)&=-\zeta(X+(0,t))=\int A(y,s)\nabla \Gamma_{X+(0,t)}(y,s) \cdot \nabla\zeta(y,s)\,ds\,dy
\\&=\int A(y,s-t)\nabla \Gamma_{X+(0,t)}(y,s) \cdot \nabla\eta(y,s-t)\,ds\,dy
\\&=\int A(y,s)\nabla \Gamma_{X+(0,t)}(y,s+t) \cdot \nabla\eta(y,s)\,ds\,dy
\end{align*}

and so by uniqueness,
\begin{equation}\label{eqn:GammaTshift}
\nabla\Gamma_{X+(0,t)}(Y+(0,t))=\nabla\Gamma_X(Y).
\end{equation}

\section{Switching variables} \label{sec:symm}

We wish to show that
\begin{equation}
\label{eqn:symm}
\Gamma_X^T(Y)=\Gamma_Y(X).
\end{equation}

Let $\Gamma_X^T$ be the fundamental solution of the operator $\div A^T\nabla$. We follow the developement in \cite[Lemma~2.7]{Rule}.

Let $\eta\in C^\infty_0$ be 1 on a neighborhood of $\overline{B(0,R)}$. If $R\gg|X|$, $|Y|$, then
\begin{align*}
\Gamma^{T}_Y(X) - \Gamma_X(Y)
&=
\int_{\R^2} \nabla (\eta\Gamma_X)\cdot A^T\nabla \Gamma_Y^T
- \nabla (\eta\Gamma_Y^T) A\nabla \Gamma_X
\\&=
\int_{B(0,R)} \nabla \Gamma_X\cdot A^T\nabla \Gamma_Y^T
- \nabla \Gamma_Y^T\cdot A\nabla \Gamma_X
\\&\qquad
+\int_{B(0,R)^C} \nabla (\eta\Gamma_X) A^T\nabla \Gamma_Y^T
- \nabla (\eta\Gamma_Y^T) A\nabla \Gamma_X
\\&= 0+
\int_{\partial B(0,R)^C} \Gamma_X \nu\cdot A^T\nabla \Gamma_Y^T
- \Gamma_Y^T \nu\cdot A\nabla \Gamma_X\,d\sigma.
\end{align*}

So we need only prove that the last integral goes to zero. This will be easier if we work with $A^r$ and $\Gamma^r$ instead of $A$, $\Gamma$, where $A^r=I$ on $B(0,2r)^C$ and $A^r=A$ on $B(0,r)$; we will then show that $\lim_{r\to\infty}\Gamma^r=\Gamma$.

Consider only $r>2|X|+2|Y|$. Then $\Gamma_X^r$ is harmonic in $B(0,2r)^C$.

Think of $\R^2$ as the complex plane, and let $u(Z)=\Gamma_X^r(Z)$. There is some bounded harmonic function $\omega(Z)$ on $B(0,2r)^C$ such that $u(Z)=\omega(Z)$ on $\partial B(0,2r)$; let $v(Z)=\Gamma_X^r(e^Z)-\omega(e^Z)$.

Then $v$ is harmonic in a half-plane, and $v=0$ on the boundary of that half-plane; by the reflection principle we may extend $v$ to an entire function. Furthermore, $|v(Z)|\leq C+C\Re Z$, so $v$ is linear.

So $\Gamma_X^r(e^Z)=\omega(e^Z)+C_1+C_2 \Re Z$; thus,
$\Gamma_X^r(Z)=\omega(Z)+C_1+C_2 \log |Z|$ for some bounded $\omega$ and constants
$C_1, C_2$. Using a test function which is 1 on $B(0,3r)$, we see that
$C_2=\frac{1}{2\pi}$. We thus have a standard normalization: we choose additive constants such that
$\Gamma_X^r(Z)=\omega(Z)+\Gamma^I(Z)$, where $\Gamma^I(Z)=\frac{1}{2\pi}\log |Z|$ and $\lim_{|Z|\to\infty}\omega(Z)=0$.

Since $\omega$ is bounded and harmonic on $B(0,2r)^C$, the function
$f(Z)=\omega(1/Z)$ is bounded and harmonic in a disk; thus, so are its partial derivatives. By our normalization, $f(0)=\lim_{|Z|\to\infty}\omega(Z)=0$. Then
$|\omega(Z)|=|f(1/Z)|\leq C(r)/|Z|$ on $B(0,3r)^C$, and $\omega'(Z)=-f'(Z)/Z^2$, so
$|\nabla \omega(Z)|\leq C(r)/|Z|^2$ on $B(0,3r)^C$.

Let $R>3r$. Then on $\partial B(0,R)$,
\begin{align*}
\left|
\nu\cdot\Gamma^r_X\nabla\Gamma_Y^{r,T}
-
\nu\cdot\Gamma^I\nabla\Gamma^I\,d\sigma
\right|
&\leq
\left|\Gamma^r_X-\Gamma^I\right|
\left|\nabla\Gamma_Y^{r,T}\right|
+\Gamma^I\left|\nabla\Gamma_Y^{r,T}-\nabla\Gamma^I\right|
\\&\leq
\frac{C(r)}{R}\frac{C}{R}
+C\frac{\log R}{R} \frac{C}{R}
\end{align*}
and so
\[\left|\int_{\partial B(0,R)}
\nu\cdot\Gamma^r_X\nabla\Gamma_Y^{r,T}
-
\nu\cdot\Gamma^I\nabla\Gamma^I\,d\sigma
\right|\leq \frac{C(r)}{R}
+C\frac{\log R}{R}.\]

So since $A^r=A^{r,T}=I$ on $B(0,r)^C$,
\begin{align*}
|\Gamma^r_X(Y)-\Gamma^{r,T}_Y(X)|
&=\left|
\int_{\partial B(0,R)}
\nu\cdot\left(\Gamma^r_X\nabla\Gamma_Y^{r,T}
-\Gamma^{r,T}_Y\nabla\Gamma_X^{r}\right)
\,d\sigma\right|
\\&=\left|
\int_{\partial B(0,R)}
\nu\cdot(\Gamma^r_X\nabla\Gamma_Y^{r,T}-\Gamma^I\nabla\Gamma^I)
-\nu\cdot(\Gamma^{r,T}_Y\nabla\Gamma_X^{r}-\Gamma^I\nabla\Gamma^I)
\,d\sigma\right|
\\&\leq C\frac{1+\log R}{R}
\end{align*}
which goes to $0$ as $R\to\infty$.
So $\Gamma_X^{r,T}(Y)=\Gamma_Y^r(X)$.

Let $u_r(Z)=\Gamma_Y^r(Z)-\Gamma_Y(Z)$. Then
\[\int A(Z)\nabla u_r(Z)\cdot\nabla\eta(Z) \,dZ=0\]
for all $\eta\in C^\infty_0(B(0,r))$.

So if $|Z|,|Z_0|<r/3$, then by \lemref{PDE4}
\[
|u_r(Z)-u_r(Z_0)|\leq\frac{C}{r^{1+\alpha}}\|u_r\|_{L^2(B(0,2r/3))} |Z-Z_0|^\alpha
\]
and
\begin{align*}
\int_{B(0,r)} |u_r|^2 &\leq \int_{B(Y,2r)}|u_r|^2
\leq C\int_{B(Y,2r)}\left(1+|\log|X-Y| |\right)^2 \,dX
\leq  C^2r^2\log^2 r
\end{align*}
for $r$ large enough. (Note that this is independent of $Y$ provided $|Y|<r$.)

Thus,
\[|\Gamma_Y(Z)-\Gamma_Y^r(Z)-\Gamma_Y(Z_0)+\Gamma_Y^r(Z_0)|\leq C\frac{\log r}{r^{\alpha}}|Z-Z_0|^\alpha\]
provided $r$ is sufficiently large compared to $|Y|,|Z|,|Z_0|$.

Now, we choose a normalization that will set $\Gamma_X(Y)=\Gamma_Y^T(X)$.

Fix some choice of $\Gamma_0^T(X)$. Normalize each $\Gamma_X$ such that
$\Gamma_X(0)=\Gamma_0^T(X)$. Now, fix some $Y$. We want to show that $f(X)=\Gamma_X(Y)$ satisfies the conditions of
$\Gamma_Y^T$.

First, if $\eta\in C^\infty_0(\R^2)$,
\[\int A^T(X)\nabla_X(\Gamma_X(0))\cdot \nabla \eta(X)\,dX
=\int A^T(X)\nabla\Gamma_0^T(X)\cdot \nabla \eta(X)\,dX=-\eta(0)
\]
and so, since
$\Gamma^r_X(Y)-\Gamma^r_X(0)\to \Gamma_X(Y)-\Gamma_X(0)$ almost uniformly in $X$,
\begin{multline*}
{\int A^T(X)\nabla f(X)\cdot \nabla \eta(X)\,dX }
\\\begin{aligned}
&=
\int \nabla_X( \Gamma_X(Y)-\Gamma_X(0))\cdot A(X)\nabla \eta(X)\,dX
+\int \nabla_X\Gamma_X(0)\cdot A(X)\nabla \eta(X)\,dX
\\ &=
\int \lim_{r\to\infty}\nabla_X(\Gamma^r_X(Y)-\Gamma^r_X(0))
\cdot A(X)\nabla\eta(X)\,dX
-\eta(0)
\\ &=
\lim_{r\to\infty}\int \nabla_X( \Gamma^{r,T}_Y(X)-\Gamma^{r,T}_0(X))
\cdot A(X)\nabla \eta(X)\,dX
-\eta(0)
\\ &=
-\eta(Y)+\eta(0)-\eta(0)
\end{aligned}\end{multline*}
as desired.

Next,
\[|f(X)|=|\Gamma_X(Y)-\Gamma_X(0)+\Gamma_0^T(X)|\leq C(1+|\log |X| |+|\log |X-Y| |).\]
This growth condition suffices to establish the inequality in \autoref{sec:fundunique}, so by uniqueness $f(X)=\Gamma_Y^T(X)$.

\section{Conjugates to solutions}
\label{sec:conjugate}

Suppose that $u$ satisfies $\div A\nabla u=0$ in some simply connected domain $U\subset\R^2$. Then if $Y_0,Y\in\U$, the integral
\[\int_{Y_0}^Y \Matrix{0&-1\\1&0}A(Z)\nabla u(Z)\cdot dl(Z)
=\int_{Y_0}^Y \nu(Z)\cdot A(Z)\nabla u(Z)\, dl(Z)\]
is path-independent, where $\nu$ is the unit normal to the path from $Y_0$ to~$Y$.

There is then a family of functions $\{\tilde u\}$ which satisfy
\[\Matrix{0&1\\-1&0}\nabla \tilde u = A\nabla u.\]
We call such a function the conjugate to $u$ in $\U$; it is unique up to an additive constant. Any such $\tilde u$ may be written as
\[\tilde u(Y)=C+\int_{Y_0}^Y \nu(Z)\cdot A(Z)\nabla u(Z)\, dl(Z).\]

Now, $\div\tilde A \nabla\tilde u=0$, where
$\tilde A=\frac{1}{\det A}A^T$ is also an elliptic matrix.

In particular, we may define the conjugate $\tilde\Gamma_{X}$
to the fundamental solution $\Gamma_{X}$ on any simply connected domain not containing $X$.
Note that $\nabla\tilde\Gamma_X(Y)$ is defined on $\R^2\setminus\{X\}$, even though $\tilde\Gamma_X$ itself is necessarily undefined on a ray.

%
%

We now require that $A(x,t)=A(x)$, which allows us to use~(\ref{eqn:gradu}).

Letting $u=\partial_{x_i}\Gamma_X(Y)$, we have that by (\ref{eqn:gradfundsoln}) $|u(Y)|\leq C/|X-Y|$ and so
\begin{equation}
\label{eqn:doublegradfundsoln}
|\partial_{x_i}\partial_{y_j}\Gamma_X(Y)|\leq
|\nabla u(Y)|
\leq \frac{C}{|X-Y|}\left(\dashint_{(B(Y,|X-Y|/2)}|u|^2\right)^{1/2}
\leq \frac{C}{|X-Y|^2}.
\end{equation}

Now, since
\[\tilde\Gamma_X(Y)=\zeta(X)
+\int_{Y_0}^Y \nu(Z)\cdot A(Z)\nabla \Gamma_X(Z)\,dl(Z)\]
we have that
\[\nabla_X\tilde\Gamma_X(Y)
=\nabla \zeta(X)
+\int_{Y_0}^Y  \nu(Z)\cdot A(Z)\nabla (\nabla_X\Gamma_X(Z))\, dl(Z)\]
We can choose $\zeta$ so that
\[\nabla_X\tilde\Gamma_X(Y)
=\int_{\infty}^Y  \nu(Z)\cdot A(Z)\nabla (\nabla_X\Gamma_X(Z))\, dl(Z).\]
Since $|\partial_{x_i}\partial_{z_j}\Gamma_X(Z)|\leq C/|X-Z|^2$, this integral converges. Furthermore, it converges to the same value, no matter which direction we use to go off to infinity.


Now, if $\eta\in C^\infty_0(\W)$ for some bounded simply connected domain $\W$ with $Y\notin\W$,
\begin{multline*}
\int_\W A^T(X)\nabla_X\tilde\Gamma_X(Y)\cdot \nabla\eta(X)\,dX
\\\begin{aligned}
&=
\int_\W A^T(X)
\int_{\infty}^Y A(Z)\nu(Z)\cdot\nabla_Z(\nabla_X\Gamma_X(Z))\,dl(Z)
\cdot \nabla\eta(X)\,dX
\\ &=
\int_{\infty}^Y A(Z)\nu(Z)\cdot\nabla_Z
\int_\W A^T(X)\nabla\Gamma^T_Z(X)\cdot \nabla\eta(X)\,dX\,dl(Z)
\\ &= -\int_{\infty}^Y A(Z)\nu(Z)\cdot \nabla_Z \eta(Z)\,dl(Z)
= 0.
\end{aligned}
\end{multline*}
Thus,
\begin{equation} \label{eqn:LTtildeGamma}
\div{}_X A^T(X)\nabla_X\tilde\Gamma_X(Y)=0
\end{equation}
whenever $X\neq Y$.

\section{\CZ\ Kernels}

Recall that $B_6(y)=\Matrix{a_{11}(y)&a_{21}(y)\\0&1}$. We have that
\[\Matrix{\partial_s \tilde\Gamma_X^T(y,s)\\\rule{0pt}{2ex}\partial_s\Gamma_X^T(y,s)}
=B_6(y)\nabla \Gamma_X^T(y,s),\]
and because $\partial_s\tilde\Gamma_X^T(y,s)$ and $\partial_s\Gamma_X^T(y,s)$
are solutions to elliptic equations away from $X$, we have that for some
$\alpha>0$,
\begin{align*}
|B_6(Y)\nabla\Gamma_X^T(Y)-B_6(Y')\nabla\Gamma_X^T(Y')|
&\leq \frac{C|Y-Y'|^\alpha}{|X-Y|^{1+\alpha}}\|\nabla \Gamma_X^T\|_{L^2(B(Y,\frac{3}{4}|X-Y|))}
\\&\leq \frac{C|Y-Y'|^\alpha}{|X-Y|^{1+\alpha}}
\end{align*}
for $|Y-Y'|\leq |X-Y|/2$.

Since we have a bound on $\partial_{x_i}\partial_{y_j}\Gamma_{(x_1,x_2)}(y_1,y_2),$
we have that
\[|B_6(Y)\nabla\Gamma_X^T(Y)-B_6(Y)\nabla\Gamma_{X'}^T(Y)|\leq C\frac{|X-X'|^\alpha}{|X-Y|^{1+\alpha}}\]
for $|X-X'|\leq |X-Y|/2$.
(Actually, in this case we have it for $\alpha=1$.)

Recalling that
\[K(X,Y)=\Matrix{B_6(Y)\nabla \Gamma^T_{X}(Y)
   &B_6(Y)\nabla \Gamma_{X}^T(Y) }^t\]
we have that $K(X,Y)$ satisfies the \CZ\ kernel conditions
\begin{align}
\label{eqn:KisCZ}
|K(X,Y)|&\leq\frac{C}{|X-Y|},\\
\nonumber
|K(X,Y)-K(X',Y)|&\leq C\frac{|X-X'|^\alpha}
{|X-Y|^{1+\alpha}}
\\
\nonumber
|K(X,Y)-K(X,Y')|&\leq C\frac{|Y-Y'|^\alpha}
{|X-Y|^{1+\alpha}}
\end{align}
provided $|X-X'|,|Y-Y'|<\frac{|X-Y|}{2}$. (These conditions are clearly equivalent to those in \lemref{Tstar}.)

Now, recall that
\[\nabla_X\tilde\Gamma_X(Y)
=\int_{\infty}^Y  \nu(Z)\cdot A(Z)\nabla (\nabla_X\Gamma_X(Z))\, dl(Z).\]
Since $|\nabla_Z (\nabla_X \Gamma_X(Z))|\leq C/|X-Z|^2$ and the integral is path-independent, we have that
\[|\nabla_X\tilde\Gamma_X(Y)-\nabla_X\tilde\Gamma_{X}(Y')|
\leq \int_Y^{Y'}\frac{C}{|X-Z|^2}\,dl(Z)
\leq C\frac{|Y-Y'|}{|X-Y|^2}\]
whenever $|Y-Y'|\leq |X-Y|/2$.

Now, we wish to show that $\tilde K$ is $C^\alpha$ in $X$ as well as~$Y$. Let
\[u(Z)=B_6(X)\nabla_X\Gamma_X^T(Z)-B_6(X')\nabla_X\Gamma_{X'}^T(Z)
=B_6(X)\nabla\Gamma_Z(X)-B_6(X')\nabla\Gamma_{Z}(X').\]
Then $\div A^T\nabla u=0$ away from $X,X'$, and
$|u(Z)|\leq C\frac{|X-X'|^\alpha}{|X-Z|^{1+\alpha}}$ provided $|X-X'|\leq|X-Z|/2$.

So, by (\ref{eqn:gradu}),
\[|\nabla u(Z)|
\leq\frac{C}{|X-Z|}\left(\dashint_{B(Z,|X-Z|/4)}|u|^2\right)^{1/2}
\leq C\frac{|X-X'|^\alpha}{|X-Z|^{2+\alpha}}.\]

Then
\begin{align*}
|B_6(X)\nabla_X\tilde\Gamma_X^T(Y)-B_6(X')\nabla_X\tilde\Gamma_{X'}^T(Y)|
&=
\left|\int_{\infty}^Y \nu(Z)\cdot A(Z)\nabla u(Z)\,dl(Z)\right|
\\&\leq
\int_\infty^Y\frac{|X-X'|^\alpha}{|X-Z|^{2+\alpha}}
\leq C\frac{|X-X'|^\alpha}{|X-Y|^{1+\alpha}}.
\end{align*}
provided $|X-X'|$ is small compared to $|X-Y|$.

So, recalling that
\[
\tilde K(X,Y)= B_6(X)\Matrix{
     \nabla_X\tilde \Gamma_{X}(Y)&
     \nabla_X\tilde \Gamma_{X}(Y)}\\
\]
we have that, if $|X-X'|,|Y-Y'|<\frac{1}{2}|X-Y|$,
\begin{align}
\label{eqn:tildeKisCZ}
|\tilde K(X,Y)|&\leq\frac{C}{|X-Y|},\\
\nonumber
|\tilde K(X,Y)-\tilde K(X',Y)|&\leq C\frac{|X-X'|^\alpha}
{|X-Y|^{1+\alpha}}
\\
\nonumber
|\tilde K(X,Y)-\tilde K(X,Y')|&\leq C\frac{|Y-Y'|^\alpha}
{|X-Y|^{1+\alpha}}.
\end{align}

\section{Analyticity}\label{sec:analytic}

We will eventually want to compare the fundamental solutions (and related operators) for a real matrix $A_0$ and a nearby complex matrix~$A$. We can explore this using analytic function theory. Let $z\mapsto A_z$ be an analytic function from $\C$ to $L^\infty(\R^2\mapsto\C^{2\times 2}).$ (As a particularly useful example, take $A_z=A_0+z(\lambda/2\epsilon)(A-A_0).$) Assume that $A_z$ is uniformly elliptic in some neighborhood of $0$, say $B(0,1)$.

Let $L_z=\div A_z\nabla$. Since we are working in $\R^2$, we know from \cite{AusTcha} that the operator $e^{-tL_z}$ has a Schwarz kernel $\check K^z_t(X,Y)=\check K^{A_z}_t(X,Y)$. Furthermore, we have that (\cite[p.~57]{AusTcha}) the map $A\mapsto K^A$ is analytic.

Fix some $Y$. Recall that, as in \autoref{sec:fundsolnexists},
\[\nabla\Gamma_Y^{A_z}(X)=\int_0^\infty \nabla \check K_t^{z}(X,Y)\,dt.\]

Since $\nabla \check K_t^z\in L^1_t$, uniformly in $z$, we have that by e.g.\ the Cauchy integral formula, $\nabla\Gamma_Y^z(X)$ is analytic in~$z$.

This lets us compute many useful inequalities.

If $\omega$ is a simple closed curve lying in $B(0,1)$, then
\begin{align*}
\nabla \Gamma^{A_z}_Y(X)-\nabla \Gamma^{A_0}_Y(X)
&=
\frac{1}{2\pi i}\oint_\omega \nabla\Gamma^{A_\zeta}_Y(X)\left(\frac{1}{\zeta-z}-\frac{1}{\zeta}\right)\,d\zeta
\\&=
\frac{1}{2\pi i}\oint_\omega \nabla\Gamma^{A_\zeta}_Y(X)\left(\frac{z}{\zeta(\zeta-z)}\right)\,d\zeta
\end{align*}
which has norm at most $|z|\frac{C}{|X-Y|}$ if $|z|<1/2$; taking $A_z=A_0+z(\lambda/2\epsilon)(A-A_0)$ and then applying this equation to $z=2\epsilon/\lambda$, we get that
\begin{equation}\label{eqn:nearfundsolns}
|\nabla\Gamma_Y(X)-\nabla\Gamma_Y^0(X)|<\frac{C\epsilon}{|X-Y|}.
\end{equation}
Similarly,
\begin{align*}
|\nabla \Gamma_Y(X)-\nabla\Gamma_Y^0(X)-\nabla \Gamma_Y(X')+\nabla\Gamma_Y^0(X')|&\leq\frac{C\epsilon|X-X'|^\alpha}{|X-Y|^{1+\alpha}},\\
|\nabla \Gamma_Y(X)-\nabla\Gamma_Y^0(X)-\nabla \Gamma_{Y'}(X)+\nabla\Gamma_{Y'}^0(X)|&\leq\frac{C\epsilon|Y-Y'|^\alpha}{|X-Y|^{1+\alpha}},
\end{align*}
provided that $|X-X'|$, $|Y-Y'|$ are less than $\frac{1}{2}|X-Y|$.

Suppose that $\mathcal{J}^z$ is a \CZ\ operator whose kernel $K^z(X,Y)$ is analytic in $z$, and suppose that $\mathcal{J}^z$ is uniformly bounded on $L^p$ in some neighborhood of $z=0$. Then
\[\mathcal{J}^z f(X)-\mathcal{J}^0 f(X) =\frac{1}{2\pi i}\oint_\omega \mathcal{J}^\zeta f(X) \left(\frac{z}{\zeta(\zeta-z)}\right)\,d\zeta\]
and so
\begin{equation}\label{eqn:CZanalytic}
\|\mathcal{J}^z f-\mathcal{J}^0 f\|_{L^p}
\leq \frac{1}{2\pi}\oint_\omega \|\mathcal{J}^\zeta f\|_{L^p} \left|\frac{z}{\zeta(\zeta-z)}\right|\,d\zeta
\leq C|z|\, \|f\|_{L^p}.
\end{equation}

We now assume that $A$ is smooth, and complete the argument (referenced in \autoref{chap:dfn}) that $T_\pm=T'_\mp$.
Recall that $\Gamma_X(Y)^I=\frac{1}{2\pi} \ln |X-Y|$, and so $\Gamma_{X+\e}^I(Y)=\Gamma_X^I(Y-\e)$ for all $X$, $Y$ and~$\e$. We would like a similar result for more general $\Gamma$.
We know from (\ref{eqn:GammaTshift}) that $\Gamma_{X+(0,t)}(Y)=\Gamma_X(Y-(0,t))$ for any real~$t$.

If we define $A_z(X)=A(X)+z(A(X+\xi)-A(X))$, then $A_z$ is elliptic for all $|z| < \lambda/4\Lambda$, and so
\[|\nabla\Gamma_X^{A_1}(Y)-\nabla\Gamma_X(Y)|
\leq  \frac{1}{|X-Y|} \frac{C \|A'\|_{L^\infty}}{\lambda}|\xi|.\]

If $A(x)$ is smooth in $x$ and $A=I$ for large $x$, then $\|A'\|_{L^\infty(B(x,R))}$ is finite (if large).

But $\Gamma_{X+\xi}(Y+\xi) = \Gamma_{X}^{A_1}(Y)$ by uniqueness:
\begin{align*}
\eta(X) &= \int A(Y) \nabla \Gamma_{X+\xi}(Y)\nabla\eta(Y-\xi)\,dY
= \int A_1(Y-\xi) \nabla \Gamma_{X+\xi}(Y)\nabla\eta(Y-\xi)\,dY
\\&= \int A_1(Y) \nabla \Gamma_{X+\xi}(Y+\xi)\nabla\eta(Y)\,dY.
\end{align*}
So while we do not have that $\Gamma_{Y+\e}(X)=\Gamma_Y(X-\e)$, we do have that
\begin{equation}\label{eqn:smallshiftgamma}
|\nabla\Gamma_{Y+\e}(X)-\nabla\Gamma_{Y}(X-\e)|
\leq \frac{C \|A'\|_{L^\infty}}{|X-Y|} \|\e\|.
\end{equation}

This equation is not useful for establishing bounds on the constants in the definitions of $(D)_p$, $(N)_p$, $(R)_p$, since we wish those bounds to be independent of $A'$; however, this equation is useful in showing that certain limits are equal to other limits.

Recall that
\begin{align*}
T_\pm F(x) &= \lim_{h\to 0^\pm}\int_\R\Matrix{\nabla\Gamma^T_{\phi(x,h)}(\phi(y))^t\\\nabla\Gamma^T_{\phi(x,h)}(\phi(y))^t} B_6(\phi(y))^t F(y)\,dy,\\
T'_\pm F(X) &= \lim_{h\to 0^\pm}\int_\R\Matrix{\nabla\Gamma^T_{\phi(x)}(\phi(y,h))^t\\\nabla\Gamma^T_{\phi(x)}(\phi(y,h))^t} B_6(\phi(y,h))^t F(y)\,dy.
\end{align*}

If $A$ is smooth, then $\lim_{h\to 0^\pm}\nabla\Gamma^T_{\phi(x)}(\phi(y,h))=\lim_{h\to 0^\mp}\nabla\Gamma^T_{\phi(x,h)}(\phi(y))$; and so we have that $T_\pm F=T_\mp' F$ for sufficiently well-behaved~$F$.

\chapter{Layer potentials}
The general idea is as follows. Suppose we wish to solve the  Neumann, Dirichlet, or regularity problem in some domain~$\V$. In the special case where $A\equiv I$, $\V=\R^2_+$, it turns out that 
\[\lim_{t\to 0} \D f(x,t)=\frac{1}{2}f(x),\quad\lim_{t\to 0} \partial_t \S g(x,t)=\frac{1}{2}g(x)\]
at any Lipschitz point $x$ of $f$ or $g$, provided that $f$,~$g\in L^p$ for some $1\leq p\leq\infty$. Thus, $u=\S(2g)$ and $u=\D(2f)$ solve the Neumann and Dirichlet or regularity problem.

More generally, we solve the Neumann, Dirichlet, or regularity problem by showing that there exist some functions $a,$ $b,$ $c$ defined on $\partial\V$ such that $\|a\|_{L^p}\leq C\|g\|_{L^p}$, $\|b\|_{L^p}\leq C\|f\|_{L^p}$, $\|c\|_{L^p}\leq C\|\partial_\tau f\|_{L^p}$ and for each $X\in\partial\V$,
\[\lim_{Z\to X\text{n.t.}}\D b(Z)=f(X),\quad\lim_{Z\to X\text{n.t.}} \S c(Z)= f(X),\quad\nu\cdot A\nabla \S a=g.\]

We will want the following properties of layer potentials. From (\ref{eqn:fundsoln}),~(\ref{eqn:gradfundsoln}), and~(\ref{eqn:symm}), we have that
\begin{itemize}

\item If $X\notin\partial\V$, then $\D f(X)$ and $\nabla\S f(X)$ are well-defined (the integrals converge) for $f\in L^p(\partial\V)$, $1\leq p<\infty$.

\item If $u=\D f$ or $u=\S f$, with $f$ as above, then $\div A\nabla u=0$ in $\R^2\backslash\partial\V$.

\item If $\partial\V$ is compact and $f\in L^1(\partial\V)$, then $\lim_{|X|\to\infty} \D f(X)=0$. Furthermore, $\lim_{|X|\to\infty} \S f(X)-\Gamma_X^T(0)\int_{\partial\V} f\,d\sigma=0$, so if $f\in H^1$ then $\lim_{|X|\to\infty} \S f=0$.

\item
If $\V$ is bounded, $\div A\nabla u=0$ inside $\V$, and $u$ is continuous and $\nabla u$ is bounded on $\R^2_+$, then if we define $f=u$, $g=\nu\cdot A\nabla u$ on $\partial\V$, letting $\eta\in C^\infty_0(\R^2)$ with $\eta\equiv 1$ near $\V$ gives us that
\begin{equation}
\label{eqn:Df-Sa}
\begin{split} 
u(X)&=-\int_{\R^2} A^T\nabla\Gamma_X^T\cdot\nabla(u\eta)
=
-\int_{\V} \nabla\Gamma_X^T\cdot A\nabla u
-\int_{\V^C} A^T\nabla\Gamma_X^T\cdot\nabla(u\eta)
\\
&=
-\int_{\partial\V}\Gamma_X^T\nu\cdot A\nabla u\,d\sigma
+\int_{\partial\V}\nu\cdot A^T\nabla\Gamma_X^T u\,d\sigma
=
\D f(X)-\S g(X)
\end{split}
\end{equation}
for all $X\in\V$. (Some of these conditions may be relaxed; however, it is obvious when they all hold.)
In particular, $\int_{\partial\V}\nu\cdot A^T\nabla\Gamma_X^T\,d\sigma=1$.

\item If $\V$ is bounded, then $\D 1(X)=1$ in $\V$ and $0$ in $\V^C$. If $\V=\Omega$ is a special Lipschitz domain and $X$, $X_0\in\Omega$, then we may define $\D 1(X)-\D 1(X_0)$ in the obvious way:
\begin{align*}
\D 1(X)-\D 1(X_0) &= \int_{\partial\Omega} \nu\cdot A^T (\nabla\Gamma_X^T-\nabla\Gamma_{X_0}^T) \,d\sigma
\\&=\int_{\partial\Omega\setminus B(X_0,R)}
+\int_{\partial(\Omega\cap B(X_0,R))}
+\int_{(\partial B(X_0,R))\cap\Omega}
\end{align*}
If $R>2|X-X_0|$, the first and third integrals are at most
$C|X-X_0|^\alpha/R^\alpha$, and the second integral is equal to 0 since $\Omega\cap B(X_0,R)$ is bounded.

So again, $\D 1(X)$ is a constant on $\Omega$.

\item If $f\in BMO(\partial\V)$, I claim that $\D f$ is well-defined up to a constant. Pick some $X_0$, $X\in\V$. Let $R=2|X-X_0|+2\dist(X_0,\partial\V)$, and let $X^*\in\partial\V$ with $\dist(X_0,\partial\V)=|X_0-X^*|$. Then 
\[\D f(X)-\D f(X_0) = \D f(X) -\D f_R(X)-\D f(X_0)+\D f_R(X_0),\] where $f_R$ is any constant, so without loss of generality $\dashint_{B(X^*,R)\cap\partial\V} f=0$. By basic $BMO$ theory, this means that $\dashint_{\partial\V\cap B(X^*,2^kR)}|f|\leq C(k+1)\|f\|_{BMO}$.
Then
\begin{align*}
|\D f(X)-\D f(X_0)|
&\leq\int_{\partial\V}|\nu\cdot A(\nabla\Gamma_X^T-\nabla\Gamma_{X_0}^T)| |f|\,d\sigma
\\&\leq
\int_{\partial\V\cap B(X^*,R)} \left(\frac{C}{|X-Y|}+\frac{C}{|X_0-Y|}\right) |f|\,d\sigma
\\&\qquad
+\int_{\partial\V\setminus B(X^*,R)}\frac{C|X-X_0|^\alpha}{|X-Y|^{1+\alpha}} |f|\,d\sigma
\end{align*}
The first term is at most
\[C R\|f\|_{BMO}
\left(\frac{C}{\dist(X,\partial\V)}+\frac{C}{\dist(X_0,\partial\V)}\right) \]
while the second term is at most
\begin{multline*}\sum_{k=0}^\infty\int_{\partial\V\cap B(X^*,2^{k+1}R)\setminus B(X^*,2^k R)}\frac{C|X-X_0|^\alpha}{|X-Y|^{1+\alpha}} |f|\,d\sigma
\\\begin{aligned}
&\leq
\sum_{k=0}^\infty
\frac{C|X-X_0|^\alpha}{2^{k+k\alpha}R^{\alpha}}
\dashint_{\partial\V\cap B(X^*,2^{k+1}R)} |f|\,d\sigma
\\&\leq
\frac{C|X-X_0|^\alpha}{R^{\alpha}} \|f\|_{BMO}
\sum_{k=0}^\infty
\frac{k+1}{2^{k+k\alpha}}
\end{aligned}\end{multline*}

These are both finite, so $\D f$ is well-defined up to a constant.

\end{itemize}

We will need a number of lemmas:

\begin{lem}\label{lem:Kwelldef} If  $f\in L^p(\partial\V)$ for some $1\leq p\leq\infty$, or if $f\in BMO\supset L^\infty$, and if $M(\partial_\tau f)(X)$ is finite, then the limits in the definitons of $\K$, $\J$ exist at~$X$. (If $f\in BMO$, the limits exist once we have fixed our choice of additive constant for $\D f$.)

Similarly, if $F\in L^p(\partial\V\mapsto\C^{2\times 2})$ and $M(\partial_\tau B_7^{-1}F)(X)$ is finite, then $\T F(X)$ is well-defined.
\end{lem}

Recall that $Mf$ is the maximal function.


\begin{proof}
Recall the definitions of $\K,\J,\T$:
\begin{align*}
\K_\pm f(X)&=\lim_{Z\to X,Z\in\V_\pm} \D f(Z)
=\lim_{Z\to X,Z\in\V_\pm}\int_{\partial\V}
\nu(Y)\cdot A^T(Y) \nabla\Gamma_{Z}^T(Y) f(Y) \,d\sigma(Y)
\\
\J f(X)&=\lim_{Z\to X}\int_{\partial\V} 
\tau(Y)\cdot \nabla\Gamma_Z^T(Y) f(Y) \,d\sigma(Y)
\\
\T_\pm F(X)&=\lim_{Z\to X,Z\in\V_\pm}\int_{\partial\V} 
K(Z,Y) F(Y) \,d\sigma(Y)
\end{align*}
Note that if $F=\Matrix{f_1&f_2\\f_3&f_4}$, then
\[\T_\pm B_7F (X) = \Matrix
{\K_\pm f_1(X)+\J f_3(X)&\K_\pm f_2(X)+\J f_4(X)\\
 \K_\pm f_1(X)+\J f_3(X)&\K_\pm f_2(X)+\J f_4(X)}.\]
So our result for $\T_\pm$ will follow from our results for $\K_\pm$,~$\J$.

By (\ref{eqn:KisCZ}), if $|Z-Z'|<\frac{1}{2}|Z-Y|$, then
\[|\nabla\Gamma_Z(Y)-\nabla\Gamma_{Z'}(Y)|\leq C\frac{|Z-Z'|^\alpha}
{|Z-Y|^{1+\alpha}}.\]

Suppose $\e,$ $\check\e$ are two small vectors such that $X+\e$, $X+\check\e$ are both in $\V$, and that $\dist(X+\e,\partial\V)<(1+a)|\e|$, $\dist(X+\check\e,\partial\V)<(1+a)|\check\e|$. Define $\Delta(\rho)=B(\rho)\cap\partial\V$.
Then, if $\rho>2|\e|+2|\check\e|$,
\begin{align*}
\left|\D f(X+\e)
-\D f(X+\check\e)\right|
&=
\left|\int_{\partial\V}
\nu \cdot A^T (\nabla\Gamma_{X+\e}^T -\nabla\Gamma_{X+\check\e}^T ) f \,d\sigma \right|
\\&\leq
\left|\int_{\Delta(\rho)}
\nu\cdot A^T (\nabla\Gamma_{X+\e}^T-\nabla\Gamma_{X+\check\e}^T) (f-f(X)) \,d\sigma\right|
\\&\qquad
+\left|f(X)\int_{\Delta(\rho)}
\nu \cdot A^T (\nabla\Gamma_{X+\e}^T -\nabla\Gamma_{X+\check\e}^T ) \,d\sigma \right|
\\&\qquad
+\left|\int_{\partial\V\setminus\Delta(\rho)}
\nu \cdot A^T (\nabla\Gamma_{X+\e}^T -\nabla\Gamma_{X+\check\e}^T ) f \,d\sigma \right|.
\end{align*}

If $\V$ is a good Lipschitz domain and $f\in L^p$ for $p<\infty$, then by (\ref{eqn:goodCZbound}), the third term is at most
\begin{align*}
\int_{\partial\V\setminus\Delta(\rho)} \frac{C|\e-\check\e|^\alpha}{|Y-X|^{1+\alpha}}
|f|\,d\sigma
\leq 
\|f\|_{L^p}
\left\|\frac{C|\e-\check\e|^\alpha}{|Y-X|^{1+\alpha}}\right\|_{L^{p'}({\partial\V\setminus\Delta(\rho)})}
&\leq \|f\|_{L^p} \frac{C|\e-\check\e|^\alpha}{\rho^{\alpha+1/p}}.
\end{align*}

If $f\in BMO$, and $X$, $X_0\in \V$, then since $\D 1=0$ we may assume without loss of generality that $\dashint_{\zeta(\rho)} f=0$.
Then
\begin{align*}
\int_{\partial\V\setminus\Delta(\rho)} \frac{C|\e-\check\e|^\alpha}{|Y-X|^{1+\alpha}}
|f|\,d\sigma
&\leq 
\sum_{k=1}^\infty
\int_{\Delta(2^k\rho)\setminus\Delta(2^{k-1}\rho)}
\frac{C|\e-\check\e|^\alpha}{|X-Y|^{1+\alpha}} |f(Y)|\,d\sigma
\\&\leq 
\sum_{k=1}^\infty\frac{C|\e-\check\e|^\alpha}{2^{k+k\alpha}\rho^{1+\alpha}}
\int_{\Delta(2^k\rho)\setminus\Delta(2^{k-1}\rho)}
\left|f(Y)\right|\,d\sigma
\\&\leq 
\sum_{k=1}^\infty\frac{C|\e-\check\e|^\alpha}{2^{k+k\alpha}\rho^{1+\alpha}}
2^k\rho k\|f\|_{BMO}
\\&\leq 
\frac{C|\e-\check\e|^\alpha}{\rho^{\alpha}}
\|f\|_{BMO}\sum_{k=1}^\infty\frac{k}{2^{k\alpha}}.
\end{align*}

To control the first term, recall that $M(\partial_\tau f)(X)$ is finite, so there is some $C(X)$ such that, if $|X-Y|$ is small, then $|f(X)-f(Y)|< C(X) |X-Y|$. So
\begin{align*}
{\left|\int_{\Delta(\rho)}
\nu\cdot A^T \nabla\Gamma_{X+\e}^T (f-f(X)) \,d\sigma\right|}
&\leq \int_{\Delta(\rho)} \frac{C}{|X+\e-Y|} C(X)|X-Y|\,d\sigma(Y)
\end{align*}
Since $\dist(X+\e,\partial\V)\leq (1+a)|\e|$, we have that $|\e|\leq(1+a)|X+\e-Y|$ and so
\[{\left|\int_{\Delta(\rho)}
\nu\cdot A^T \nabla\Gamma_{X+\e}^T (f-f(X)) \,d\sigma\right|}
\leq \int_{\Delta(\rho)} \frac{C}{|X-Y|}  C(X)|X-Y|\,d\sigma(Y)
\leq C(X)\rho .
\]

To control the middle term, note that 
\[\int_{\partial (B(X,\rho)\cap\V)}\nu\cdot A^T \nabla\Gamma_{X+\e}^T  \,d\sigma = 1
=\int_{\partial (B(X,\rho)\cap\V)}\nu\cdot A^T \nabla\Gamma_{X+\check\e}^T  \,d\sigma
\]
and so
\begin{multline*}
{\left|f(X)\int_{B(X,\rho)\cap\partial\V}
\nu \cdot A^T (\nabla\Gamma_{X+\e}^T -\nabla\Gamma_{X+\check\e}^T ) \,d\sigma \right|}
\\\begin{aligned}
&=
\left|f(X)\int_{\partial B(X,\rho)\cap\V}
\nu \cdot A^T (\nabla\Gamma_{X+\e}^T -\nabla\Gamma_{X+\check\e}^T ) \,d\sigma \right|
\\&\leq
\left|f(X)\right|\int_{\partial B(X,\rho)}
C\frac{|\e-\check \e|^\alpha}{|X-Y|^{1+\alpha}} \,d\sigma 
\leq C|f(X)|\frac{\|\e-\check \e\|^\alpha}{\rho^\alpha}.
\end{aligned}
\end{multline*}

We may control all three terms by first making $\rho$ small and then making $|\e|$ and $|\check\e|$ small.

Similarly, 
\[\int_{\partial (B(X,\rho)\cap\V)} \tau\cdot\nabla \Gamma_{X+\e}^T\,d\sigma=\int_{\partial (B(X,\rho)\cap\V)} \tau\cdot\nabla \Gamma_{X+\check\e}^T\,d\sigma=0,\]
and so
\begin{align*}
\left|\int_{\partial\V}
\tau\cdot 
\nabla\Gamma_{X+\e}^T f \,d\sigma
-\int_{\partial\V}
\tau\cdot \nabla\Gamma_{X+\check\e}^T f \,d\sigma\right|
\end{align*}
goes to 0 as $|\e|,|\check\e|\to 0$.
\end{proof}

\begin{lem}\label{lem:Ktranspose} If $f\in L^p$, $1<p<\infty$, then we have the following equations for the transposes of $\K$ and $\J$:
\begin{align*}
\K^t_\pm f &= \mp\nu\cdot A^T\nabla\S^T f|_{\partial\V_\mp},\qquad
\J^t f(X) = \partial_\tau\S^T f(Z).
\end{align*}
\end{lem}

\begin{proof} Consider $\K^t_+$ first. Pick some $\eta\in C^\infty_0(\R^2)$. By the weak definition of $\nu\cdot A\nabla$, we need to show that $\int_{\partial\V} \eta \K^t_+ f \,d\sigma = -\int_{\V^C}\nabla\eta \cdot A^T\nabla\S^T f$.

If $f\in L^p$, $1<p<\infty$, then by (\ref{eqn:goodCZbound}),
\[|\nabla \S^T f(X)|\leq \int_{\partial\V} |\nabla_X\Gamma_X f|\,d\sigma
\leq \|f\|_{L^p}\left\|\frac{C}{\cdot-X} \right\|_{L^{p'}(\partial\V)}
\leq C_p\|f\|_{L^p} \dist(X,\partial\V)^{-1/p}\]
which is in $L^1_{loc}$, so the right-hand integral converges.

By definition,
\begin{align*}
\int_{\partial\V} \eta \K^t_+ f \,d\sigma 
&= \int_{\partial\V} \K_+\eta f \,d\sigma
\\&= \int_{\partial\V} \lim_{Z\to Y,\>Z\in\gamma_+} \int_{\partial\V} \nu\cdot A^T \nabla\Gamma_Z^T(X) \eta(X)\,d\sigma(X)\, f(Y) \,d\sigma(Y)
\\&= -\int_{\partial\V} \lim_{Z\to Y,\>Z\in\gamma_+} 
\int_{\V^C} \nabla\eta(X)\cdot A^T \nabla\Gamma_Z^T(X)\,dX
\, f(Y) \,d\sigma(Y)
\end{align*}
while
\begin{align*}
\int_{\V^C}\nabla\eta \cdot A^T\nabla\S^T f
&=\int_{\V^C}\nabla\eta(X) \cdot A^T(X)\int_{\partial\V}\nabla_X\Gamma_X(Y) f(Y)\,d\sigma(Y)\,dX
\\&=\int_{\partial\V}\int_{\V^C}\nabla\eta(X) \cdot A^T(X)\nabla\Gamma_Y^T(X) \,dX \,f(Y)\,d\sigma(Y).
\end{align*}

I claim that
\[\lim_{Z\to Y,\>Z\in\gamma_+} 
\int_{\V^C} \nabla\eta(X)\cdot A^T \nabla\Gamma_Z^T(X)\,dX
=
\int_{\V^C} \nabla\eta(X)\cdot A^T \nabla\Gamma_Y^T(X)\,dX.\]

Let $\delta=2|X-Y|$, $\supp\eta\subset B(Y,R)$. Then
\begin{multline*}
\left|\int_{\V^C} \nabla\eta(X)\cdot A^T \nabla\Gamma_Z^T(X)\,dX
-
\int_{\V^C} \nabla\eta(X)\cdot A^T \nabla\Gamma_Y^T(X)\,dX\right|
\\\begin{aligned}
&=
\left|\int_{\V^C} \nabla\eta(X)\cdot A^T (\nabla\Gamma_Z^T(X)- \nabla\Gamma_Y^T(X))\,dX\right|
\\&\leq C(\eta)
\int_{|X-Y|\leq R} |\nabla\Gamma_Z^T(X)- \nabla\Gamma_Y^T(X)|\,dX
\\&\leq 
C(\eta)\int_{\delta<|X-Y|\leq R} \frac{|Z-Y|^\alpha}{|X-Y|^{1+\alpha}}\,dX
+C(\eta)\int_{|X-Y|\leq \delta} \frac{C}{|X-Y|}+\frac{C}{|X-Z|}\,dX
\\&\leq 
C(\eta)|Z-Y|^\alpha R^{1-\alpha}+C(\eta)\delta
=C(\eta)|Z-Y|^\alpha R^{1-\alpha}+C(\eta)|Z-Y|.
\end{aligned}
\end{multline*}
This clearly goes to 0 as $Z\to Y$; thus $\K^t_+ f=\nu\cdot A^T\nabla \S^T f|_{\partial\V_-}$. Similarly, $\K^t_- f=\nu\cdot A^T\nabla \S^T f|_{\partial\V_+}$.

Now we come to $\J f$. Let $\eta\in C^\infty_0(\R^2)$ again. Then
\begin{align*}
\int_{\partial\V} \eta \J^t f\,d\sigma
&=\int_{\partial\V} f\J \eta \,d\sigma
=\int_{\partial\V} f(X)\lim_{Z\to X}\int_{\partial\V} \tau(Y)\cdot\nabla\Gamma_Z^T(Y)\eta(Y)\,d\sigma(Y) \,d\sigma(X).
\end{align*}
If $\partial\V$ is bouned, let $\W=\V$. If $\V$ is a special Lipschitz domain, then let $\W=\psi((-R,R)\times(0,R))$ where $R$ is large enough that $\partial\V\cap\supp\eta=\partial\W\cap\supp\eta$. Then $\W$ is bounded, and $Z\notin\partial\W$.

So $\int_{\partial\W}\tau(Y)\cdot\nabla\Gamma_Z^T(Y)\,d\sigma(Y)=0$, and so
\begin{align*}
\int_{\partial\V} \eta \J^t f\,d\sigma
&=\int_{\partial\V} f(X)\lim_{Z\to X}\int_{\partial\W} \tau(Y)\cdot\nabla\Gamma_Z^T(Y)(\eta(Y)-\eta(X))\,d\sigma(Y)\,d\sigma(X) .
\end{align*}
But
\[\lim_{Z\to X}\int_{\partial\W} \tau(Y)\cdot\nabla\Gamma_Z^T(Y)(\eta(Y)-\eta(X))\,d\sigma(Y)
=\int_{\partial\W} \tau(Y)\cdot\nabla\Gamma_X^T(Y)(\eta(Y)-\eta(X))\,d\sigma(Y)\]
since for all $X\in \partial\V$ and all $Z\in \gamma_\pm(X)$ with $|X-Z|=\delta/2$,
\begin{multline*}
\int_{\partial\W} |\tau(Y)\cdot(\nabla\Gamma_Z^T(Y) -\nabla\Gamma_X^T(Y))(\eta(Y)-\eta(X))|\,d\sigma(Y)
\\
\begin{aligned}
&\leq
\int_{\partial\W\cap B(X,\delta)} \left(\frac{C}{|X-Y|} +\frac{C}{|Z-Y|}\right)\|\eta'\|_{L^\infty} |X-Y|\,d\sigma(Y)
\\&\qquad+
\int_{\partial\W\setminus B(X,\delta)} \frac{C|X-Z|^\alpha}{|X-Y|^{1+\alpha}}\|\eta\|_{L^\infty}\,d\sigma(Y)
\\&\leq
C\|\eta'\|_{L^\infty}\delta + C|X-Z|^\alpha\|\eta\|_{L^\infty}\delta^\alpha
\end{aligned}
\end{multline*}
which goes to 0 as $|X-Z|\to 0$. So
\begin{align*}
\int_{\partial\V} \eta \J^t f\,d\sigma
&=\int_{\partial\V} f(X)\int_{\partial\W} \tau(Y)\cdot\nabla\Gamma_X^T(Y)(\eta(Y)-\eta(X))\,d\sigma(Y)\,d\sigma(X) \\&=\int_{\partial\W}(\eta(Y)-\eta(X)) \tau(Y)\cdot
\int_{\partial\V} f(X)\nabla_Y\Gamma_Y(X)\,d\sigma(X) \,d\sigma(Y)
\\&=\int_{\partial\W}(\eta(Y)-\eta(X)) \tau(Y)\cdot
\nabla\S^T f(Y)\,d\sigma(Y)
\\&=\int_{\partial\W}\eta(Y) \tau(Y)\cdot
\nabla\S^T f(Y)\,d\sigma(Y)=\int_{\partial\V}\eta(Y) \tau(Y)\cdot
\nabla\S^T f(Y)\,d\sigma(Y).
\end{align*}
This concludes the proof.
\end{proof}

\begin{lem}\label{lem:Tstar} Assume that $K(X,Y)$ satisfies the \CZ\ kernel conditions
\begin{align*}
|K(X,Y)|&\leq\frac{C}{|X-Y|},
\\
|K(X,Y)-K(X',Y)|&\leq \frac{C|X-X'|^\alpha}{\min(|X-Y|,|X'-Y|)^{1+\alpha}},
\\
|K(X,Y)-K(X,Y')|&\leq \frac{C|Y-Y'|^\alpha}{\min(|X-Y|,|X-Y'|)^{1+\alpha}}
\end{align*}
for some $C,\alpha>0$.

Then if we define
\[\T F(X)=\int_{\partial\V} K(X,Y) F(Y)\,d\sigma(Y),\]
then for any $X\in\V,$ $X^*\in\partial\V$ with $|X-X^*|\leq (1+a)\dist(X,\partial\V)$,
\[|\T F(X)|\leq C MF(X^*) + \T_* F(X^*).\]
\end{lem}

Here if $\T$ is an operator on $L^2(\partial\V)$ with \CZ\ kernel $K(X,Y)$, then 
\[\T_* f(X)=\sup_{\epsilon>0}\left| \int_{Y\in\partial\V,|X-Y|>\epsilon} K(X,Y) f(Y)\,d\sigma(Y)\right|\]
is the standard truncated maximal operator associated with $\T$. Note that $\T_+$, $\T_-$ have the same truncated maximal operator.

It is well-known (see, for example, \cite{grafakos} and \cite[p.~34]{stein}) that if $\partial\V=\R$, then $\T_*$ is bounded $L^2\mapsto L^2$ if and only if $\T_\pm$ is. It is easy to see that this remains true if $\V$ is any good Lipschitz domain.

\begin{proof}
Define 
\[
\T_h F(X)=\int_{|Y-X|>h,Y\in\partial\V} K_0(X,Y) F(Y)\,d\sigma(Y).\]

Now, if $h=\dist(X,\partial\Omega)$, then
\begin{align*}
|\T F(X)-\T_{h} F(X^*)|&\leq
\left|\int_{|Y-X^*|>h} (K(X,Y)-K(X^*,Y)) F(Y)\,d\sigma(Y)
\right|
\\&\qquad
+\left|
\int_{|Y-X^*|<h} K(X,Y) F(Y)\,d\sigma(Y)\right|
\\
&\leq \int_{|Y-X^*|>h} \frac{Ch^\alpha}{|Y-X^*|^{1+\alpha}} |F(Y)|d\sigma(Y)
\\&\qquad
+\int_{|Y-X^*|<h} \frac{C}{h} |F(Y)|d\sigma(Y)
\\&\leq C MF(X^*)
\end{align*}
where $M$ is the standard maximal function and our constants depend on the Lipschitz constants of $\V$.

But $|\T_h F(X)|\leq \T_* F(X)$, and so we are done.\end{proof}

\begin{thm}\label{thm:NDf} Let $\V$ be a good Lipschitz domain. If $\T_\pm$ is bounded from $L^p$ to itself and $L^{p'}$ to itself, for $1<p<\infty$, and the conditions on $K$ in \lemref{Tstar} hold, then 
\[\|N_{\V_\pm}(\D f)\|_{L^{p}}\leq C(p) \|f\|_{L^{p}},\quad
\|N_{\V_\mp}(\nabla\S^T g)\|_{L^{p}}\leq C(p) \|g\|_{L^{p}}.\]
If $\tilde\T_\pm$ is bounded on $L^p$ and $L^{p'}$, then
\[\|N_{\V_\pm}(\nabla \D f)\|_{L^p}\leq C(p)\|\partial_\tau f\|_{L^p}.\]
\end{thm}

\begin{proof}
Since $\T_\pm$ is bounded on $L^p$, so is $\T_*$; if $1<p<\infty$, then $M$ is bounded on $L^p$. Therefore, by \lemref{Tstar}, we have that $F\mapsto N(\T F)$ is bounded $L^p\mapsto L^p$ and $L^{p'}\mapsto L^{p'}$.

But
\[\T (B_7 f)(Z)
=\Matrix{
\D f(Z)& \S (\partial_\tau f)(Z)\\
\D f(Z)& \S (\partial_\tau f)(Z)}.
\]
This completes the proof for $N(\D f)$.

Note that
\begin{align*}
\nabla \D f(X) &=\nabla_X\int_{\partial\V} \nu\cdot A^T\nabla\Gamma^T_X f\,d\sigma
=\nabla_X\int_{\partial\V} \tau\cdot \nabla\tilde\Gamma^T_X f\,d\sigma
\\&=-\nabla_X\int_{\partial\V} \tilde\Gamma^T_X \partial_\tau f\,d\sigma
=-\int_{\partial\V} \nabla_X\tilde\Gamma^T_X \partial_\tau f\,d\sigma
\end{align*}
and so, taking our kernel to be $\tilde K(X,Y)=B_6(X)\Matrix{\nabla_X\Gamma_X^T(Y) & \nabla_X\Gamma_X^T(Y)}$
we may use \lemref{Tstar} to bound $\|N(\nabla \D f)\|_{L^p}$.

We now turn to $N(\nabla \S^T g)$. Recall that
\[\nabla \S^T g(Z)=\int_{\partial\V} \nabla_Z\Gamma_Z(Y) g(Y)\,d\sigma(Y).\]

Define $\T_h$ as before, by
\begin{align*}
\T_h F(X) &= \int_{|X-Y|>h} K(X,Y) F(Y)\,d\sigma(Y)
\\&=\int_{|X-Y|>h} \Matrix{\nabla\Gamma_X^T(Y) &\nabla\Gamma_X^T(Y)}^t B_6(Y)^t F(Y)\,d\sigma(Y).\end{align*}
Then
\begin{align*}
\T_h^t F(X) &= \int_{|X-Y|>h}K^t(Y,X) F(Y)\,d\sigma(Y)
\\&=B_6(X)\int_{|X-Y|>h} \Matrix{\nabla_X\Gamma_X(Y) &\nabla_X\Gamma_X(Y)}  F(Y)\,d\sigma(Y)
\end{align*}
If $\T$ is bounded on $L^p$, then so is $\T_*$, and so $\T_*^t$ is bounded on $L^{p'}$. But by our expression for $\T_h$ above and by \lemref{Tstar}, $N(\nabla\S^T f)(X)\leq CMf(X) + C\T^t_* f(X)$. This completes the proof.
%
%
%
\end{proof}

\begin{cor} \label{cor:NDminusD0} We have that $\|N( \D f-\D^0 f)\|_{L^p}\leq C\|f\|_{L^p}\|A-A_0\|_{L^\infty}$.
\end{cor}

\begin{proof}
Consider $\check\T = \T-\T^0$, and recall that $\|A-A_0\|_{L^\infty}=\epsilon$. Then by (\ref{eqn:nearfundsolns}) and \lemref{Tstar}, we know that
\[|\check \T f(Z)|\leq C\epsilon Mf(X)+C\check T_* f(X)\]
for any $X\in\partial\V$ with $|X-Z|<(1+a)\dist(Z,\partial\V)$.

But by (\ref{eqn:CZanalytic}), $\|\check T\|_{L^p\mapsto L^p}\leq C_p \epsilon$. So as in the proof of \thmref{NDf}, $\|N(\D f-\D^0 f)\|_{L^p}\leq C\epsilon \|f\|_{L^p}$, as desired.
\end{proof}

\begin{cor}\label{cor:NSf} If $g$ is a $H^1$ atom, then $\|N(\nabla \S g)\|_{L^1}\leq C$.
\end{cor}

\begin{proof} Suppose that $g$ is a $H^1$ atom supported in some connected set $\Delta\subset\partial\V$ with $\sigma(\Delta)=R$ and $X_0\in\Delta$. Then \[\|g\|_{L^\infty}\leq \frac{1}{\sigma(\Delta)}=\frac{1}{R},\] and so $\|g\|_{L^2}\leq 1/\sqrt R$.

Therefore, 
\begin{align*}
\int_{B(X_0,bR)\cap\partial\V} |N(\nabla\S g)|
&\leq b k_4 R \left(\dashint_{B(X_0,2R)\cap\partial\V} |N(\nabla\S g)|^2\right)^{1/2}
\\&\leq b k_4 \sqrt{R} C\|g\|_{L^2}
\leq bC 
\end{align*}
by H\"older's inequality.

We need to bound $\int_{\partial\V\setminus B(X_0,bR)}|N(\nabla\S g)|$.
If $|X-X_0|>2R$, and $Y\in \gamma(X)$, then either $|X-Y|\leq |X_0-X|/2$ and so $|X_0-Y|\geq |X_0-X|/2$, or $|X-Y|>|X_0-X|/2$ and so 
\[|X_0-Y|\geq\dist(Y,\partial\V)\geq\frac{1}{1+a}|Y-X|> \frac{|X_0-X|}{2+2a}.\] 
In any case, 
\[N(\nabla \S g)(X)\leq \sup\left\{|\nabla\S g(Y)|:|Y-X_0|>\frac{|X_0-X|}{2+2a}\right\}.\]

If $|X-X_0|>(2+2a)R$, then we care only about $|Y-X_0|>R$. In this case,
\begin{align*}
|\nabla\S g(Y)|
&\leq \left|\int_{B(X_0,R)\cap\V}\nabla_Y\Gamma_Y(Z) g(Z)\,d\sigma(Z)\right|
\\&\leq \left|\int_{B(X_0,R)\cap\V}(\nabla_Y\Gamma_Y(Z)-\nabla_Y\Gamma_Y(X_0)) g(Z)\,d\sigma(Z)\right|
\\&\leq \int_{B(X_0,R)\cap\V}\frac{|Z-X_0|^\alpha}{(|Y-X_0|-R)^{1+\alpha}}\frac{C}{R}\,d\sigma(Z)
\\&\leq C\frac{R^\alpha}{(|Y-X_0|-R)^{1+\alpha}}.
\end{align*}
Therefore, if $|X-X_0|>(2+2a)R$, then
\[N(\nabla \S g)(X)\leq 
C\frac{R^\alpha}{(|X-X_0|-R(2+2a))^{1+\alpha}}.\]
and so if we choose $b=4+4a$, then
\[\|N(\nabla \S g)\|_{L^1(\partial\V\setminus B(X_0,bR))}\leq C.\]
This completes the proof.
\end{proof}

\begin{cor}\label{cor:Kmorewelldef} If $\T$ and $\T^t$ are bounded on $L^p$ and $L^{p'}$ for some $1<p<\infty$, so \thmref{NDf} holds, then the limits in the definition of $\K_\pm f$, $\J f$ exist pointwise a.e.\ for $f\in L^p$, even if $f$ is not smooth.\end{cor}

\begin{proof} We work with $\J$ only; the proof for $\K_\pm$ is identical. Let $f_n\in L^p$ be smooth and such that $\|f_n-f\|_{L^p}\leq 4^{-n}$. Then for each $X\in\partial\V$, $\lim_{Y\to X\text{ n.t.}} \D f_n (Y)$ exists. Since $\J$ is bounded on $L^p$, $\lim_{n\to\infty} \J f_n =\J f$ exists in $L^p$.

Let $E_n=\{X\in\partial\V:|\J f(X)-\J f_n(X)|>2^{-n}$ or $N(\D (f_n-f))(X)>2^{-n}\}$; then $\sigma(E_n)\leq C 2^{-n}$. Thus, $\sigma\left(\cup_{m=n}^\infty E_n\right)\leq C2^{-n}$, so
\[E=\bigcap_{n=1}^\infty \bigcup_{m=n}^\infty E_n\]
has measure 0.

Suppose $X\in\partial\V$, $X\notin E$. So there is some $N>0$ such that, if $n>N$, then
\begin{align*}
|\D f(Y)- \J f(X)|
&\leq |\D f(Y)-\D f_n(Y)|+|\D f_n(Y)-\J f_n(X)|+|\J f_n(X)-\J f(X)|
\\&\leq N(\D(f-f_n))(X) +|\D f_n(Y)-\J f_n(X)|+|\J f_n(X)-\J f(X)|
\\&\leq C 2^{-n} +|\D f_n(Y)-\J f_n(X)|.
\end{align*}
So for every $\epsilon>0$, there is some $n>N$ such that $C 2^{-n}<\epsilon/2$, and some $\delta>0$ such that $|\D f_n(Y)-\J f_n(X)|<\epsilon/2$ provided $Y\in\gamma(X)$ and $|X-Y|<\delta$; thus, $|\D f(Y)-\J f(X)|<\epsilon$ if $|X-Y|<\delta$, and so the non-tangential limit exists at $X$, as desired.\end{proof}

Similar techniques may be used if we know that we can solve $(D)^A_p$ for smooth boundary data and would like to show that solutions exist for arbitrary $L^p$ (or $BMO$) boundary data.

\chapter{Boundedness of layer potentials on~$L^p$}
\label{chap:Tbounded}

\begin{thm}\label{thm:Tbounded} 

Let $A_0,$ $A$ be matrices defined on $\R^2$ which are independent of the $t$-variable, that is, $A(x,t)=A(x,s)=A(x)$, $A_0(x,t)=A_0(x,s)=A_0(x)$ for all $x,t,s\in\R$.

Assume that $A_0$ is uniformly elliptic (that is, satisfies (\ref{eqn:elliptic})) and $A_0(x)\in\R^{2\times 2}$ for all~$x$. We make the a priori assumptions that $A$,~$A_0$ are smooth.

Then there exists an $\epsilon_0=\epsilon_0(\lambda,\Lambda)>0$ such that, if $\|A-A_0\|_{L^\infty}\leq\epsilon_0$, then if $\V$ is a bounded or special Lipschitz domain, the layer potentials $\T$ and~$\tilde\T$ defined by (\ref{dfn:Tgen}) are $L^p$-bounded for any $1<p<\infty$, with bounds depending only on $\lambda$,~$\Lambda$,~$p$, and the Lipschitz constants of~$\V$ (as defined in \dfnref{domain}).
\end{thm} 

We assume $\epsilon_0<\lambda/2$; this will ensure that $A$ is elliptic as well.

If I can prove that $\T$ is bounded $L^2\mapsto L^2$, then I will know that $\T$ is bounded $L^p\mapsto L^p$ for $1<p<\infty$. This follows from basic \CZ\ theory (e.g.\
\mbox{\cite[I.7]{stein}}). 

We will start with special cases:

\begin{thm}\label{thm:special} Suppose that we are working in a special Lipschitz domain $\Omega$ (so that we may consider $T$ instead of $\T$). We may write that $\Omega=\{X:\phi(X\cdot\e^\perp)<X\cdot\e\}$ for some $\e,\e^\perp$ and some Lipschitz function $\phi$.

Then \thmref{Tbounded} holds.
\end{thm} 

We will pass from \thmref{special} to \thmref{Tbounded} in \autoref{sec:patchingT}.

We will want to start with an even more specialized case:
\begin{thm} \label{thm:flat} Suppose that the $\phi$ in the definition of $\Omega$ above is smooth and compactly supported. Further assume that 
\begin{equation}\label{eqn:Atriangular}
A_0(x)=\Matrix{1&a_{12}^0(x)\\0&a_{22}^0(x)}
\end{equation}
for some functions $a_{12}^0$, $a_{22}^0$ which leave $A_0$ uniformly elliptic, and that for some $R_0$ large,
\begin{equation}\label{eqn:aprioriint}
A(x)=A_0(x)=I\text{ for }|x|>R_0,\quad 
\dashint_{-R_0}^{R_0} \frac{a_{21}(y)}{a_{11}(y)}\,dy=0,\text{ and }\enspace \dashint_{-R_0}^{R_0} \frac{1}{a_{11}(y)}\,dy=1.
\end{equation}

Then there is an $\delta_0=\delta_0(\lambda,\Lambda)>0$ such that if $\|\phi'\|_{L^\infty}<\delta_0$, then $T$ and $\tilde T$ are $L^2$-bounded. 
\end{thm}

We will pass from this theorem to \thmref{special} in \autoref{sec:buildup}, and will prove this theorem in Sections~\mbox{\ref{sec:proveflat}--\ref{sec:Tfinite}}.

We let $O(p_1,p_2,\ldots p_n)$ denote a term which, while not a constant, may be bounded by a constant depending only on $p_1,p_2,\ldots p_n$; thus, for example, $\Lambda \sin (x)+\lambda^2=O(\lambda,\Lambda)$. 

\section{Buildup to arbitrary special Lipschitz domains}\label{sec:buildup}

We will prove \thmref{flat} in Sections~\ref{sec:B1accretive}--\ref{sec:Tfinite}; in this section, we will assume it. We wish to prove \thmref{special}. We work only with $T$; the proof for $\tilde T$ is identical. 

\begin{thm} \label{thm:small} \thmref{flat} holds if we relax the condition (\ref{eqn:aprioriint}) on $A$ and the requirement that $\phi\in C^\infty_0$, and replace the requirement that $\|\phi'\|_{L^\infty}<\delta_0$ with the requirement that $\|\phi'-\gamma\|_{L^\infty}<\delta_0$ for some $\gamma\in\R$, and permit $\delta_0$ to depend on $\gamma$ as well as $\lambda,\Lambda$.
\end{thm}

\begin{proof} We first look at the requirements (\ref{eqn:aprioriint}) and $\phi\in C^\infty_0$. Recall that $T_\pm F(x)=\lim_{h\to 0^\pm} T_h F(x)$, where
\[T_h F(x)=\int_\R K(x\e^\perp+(\phi(x)+h)\e,y\e^\perp+\phi(y)\e) F(y)\,dy.\]
We need only show that the $T_h$ are uniformly bounded on $L^2$, so we may consider some fixed choice of~$h$.

Pick some $F\in L^2(\R)$. Then for each $\mu>0$, there is some positive number $R>0$ with $\|F\|_{L^2(\R\setminus(-R,R))}<\mu$. If $T_h^\delta$ is an operator similar to $T$, but based on a different elliptic matrix $A$ or Lipschitz function $\phi$, then
\begin{align*}
|T_h^\delta F(x)-T_h F(x)| 
&\leq
\left|\int_{-R}^{R}F(y)(K_h^\delta(x,y)-K_h(x,y)),dy\right|
+
\int_{|y|>R} |F(y)|\frac{C}{|x-y|+|h|}\,dy
\end{align*}

Then
\begin{align*}
|T_h^\delta F(x)-T_h F(x)| &\leq
\|F\|_{L^2(\R\setminus(-R,R))}\left\|\frac{C}{y}\right\|_{L^2_y(\R\setminus(-R,R))}
+\left|(T_h^\delta-T_h)(F\chi_{(-R,R)})\right|
\\&\leq
\mu \frac{C}{\sqrt{R}}
+ \left|(T_h^\delta-T_h)(F\chi_{(-R,R)})\right|
\end{align*}
provided that $|x|<R$.

If we choose $T^\delta$ to satisfy the requirements of \thmref{flat}, and such that \[\|T_h^\delta-T_h\|_{L^2(-R,R)\mapsto L^2(-R,R)}<\theta,\] where $\theta\to 0$ as $\delta\to 0$ and $R\to\infty$, then we will have that by \lemref{Tstar},
\[\|T_h F\|_{L^2(-R,R)}\leq \|T_h^\delta F\|_{L^2(-R,R)}+C\mu + \theta\|F\|_{L^2}
\leq (C+\theta) \|F\|_{L^2}+C\mu.\]
By letting $R\to\infty$ and $\delta$,~$\mu\to 0$, we will be able to achieve our desired result.

We first remove the requirement (\ref{eqn:aprioriint}). Pick some $A$ satisfying the other a priori assumptions (i.e.\ smooth and elliptic). 
For $R>1$, let $A^\delta(x)=A(x)$ on $(-R^2,R^2)$, $A^\delta(x)=I$ if $|x|>2R^2$, and assume that $A^\delta$ is smooth and satisfies~(\ref{eqn:aprioriint}).
Let $\Gamma^\delta$, $K^\delta$, $T_h^\delta$ be $\Gamma$, $K$ and $T_h$ for the elliptic matrix $A^\delta$ instead of $A$.

Then $\div A\nabla(\Gamma_X(Y)-\Gamma_X^\delta(Y))=0$ for $Y=(y,s)$ with $|y|<R^2$; therefore, $|\Gamma_X(Y)-\Gamma_X^\delta(Y)|\leq C\log R$ if $|X|$,~$|Y|<R^2/2$, and so 
\[|\nabla \Gamma_X(Y)-\nabla\Gamma_X^\delta(Y)|<\frac{C\log R}{R^2}\text{ if }|X|, |Y|<R/4.\]

So $K_h^\delta(x,y)-K_h(x,y) <C\frac{\log R}{R^2}$ if $|x|,|y|<R/4$ and $h$ is small. Therefore, 
\[\|K_h^\delta(x,y)-K_h(x,y)\|_{L^2(-R,R)}\leq \frac{C\log R}{R^{3/2}}\] and so $\theta=\frac{C\log R}{R}$ which goes to 0 as $R\to\infty$; thus we are done.

We next remove the requirement that $\phi\in C^\infty_0$. Assume instead that $\|\phi'\|_{L^\infty}<\delta_0$. Choose $\phi_\delta$ compactly supported and smooth with $\|\phi_\delta-\phi\|_{L^\infty}<\delta$ on $(-R,R)$. Then 
\[\left|K_h(x,y)-K_h^\delta(x,y)\right|\leq \frac{|\phi(x)-\phi_\delta(x)|^\alpha+|\phi(y)-\phi_\delta(y)|^\alpha}{|x-y|^{1+\alpha}+|h|^{1+\alpha}}\]
which has $L^2_y(-R,R)$ norm at most $\frac{\delta^\alpha}{h^{\alpha+1/2}}$ if $x\in(-R,R)$; thus by forcing $\delta$ very small in comparison with $h$,~$R$, we are done.

Finally, we relax from $\|\phi'\|<\delta_0$ to $\|\phi'-\gamma\|<\delta_0(\gamma)$ for some real number $\gamma$. Fix some choice of $\e$, $\phi$ and $\gamma$. Then $\Omega=\{X\in\R^2:\phi(X\cdot \e^\perp)<X\cdot\e\}$.

Define $\check \e=\frac{\e-\gamma\e^\perp}{\sqrt{1+\gamma^2}}$. If $\|\phi'-\gamma\|_{L^\infty}$ is small enough, relative to $|\gamma|$, then there is some $\check\phi$ such that
\[\phi(X\cdot\e^\perp) < X\cdot\e \text{ if and only if }
\check\phi(X\cdot\check \e^\perp) < X\cdot\check\e.\]

Applying \thmref{flat} to $\check\phi$, we see that the layer potentials $\check T_\pm$ are bounded $L^2\mapsto L^2$. Therefore by \lemref{Kwelldef}, so are the potentials $\check\T_\pm$. But $\T_\pm=\check\T_\pm$, and therefore $T_\pm$ is bounded $L^2\mapsto L^2$.
\end{proof}

Now we wish to remove the assumption that $\phi'-\gamma$ must be small. This may be done exactly as in \cite[Section~5]{Rule}.
Let 
\[\Lambda^k(\delta_0)=\big\{\phi:\R\mapsto\R \big|
\text{there is a constant }\gamma\in(-k,k)\text{ such that }\|\phi'-\gamma\|_{L^\infty}<\delta_0\big\}.\]
We require $0<\delta_0\leq k$.

\begin{lem}\label{lem:buildup} Suppose that for some fixed choice of $\e$ and $\e^\perp$, we have that for every $k>0$ there is a $\delta_0(k)>0$ such that $T_\pm$ is bounded for every $\phi\in\Lambda^k(\delta_0)$ with bounds depending only on $\lambda,\Lambda$ and~$k$.

Then $T_\pm$ is bounded on $L^2$ for any Lipschitz function $\phi$ with bounds depending on $\lambda,\Lambda$ and $\|\phi'\|_{L^\infty}$.
\end{lem}

\begin{proof}

We have the following useful theorems:
\begin{thm}\label{thm:buildup} \cite[Theorem~5.2]{Rule}
Suppose that $\delta$, $k>0$, $\phi\in\Lambda^k(\delta)$, and $I\subset \R$ is an interval. Then there is a compact subset $E\subset I$ and a function $\psi\in\Lambda^{k+\frac{\delta}{10}}\left(\frac{9}{10}\delta\right)$ such that
\begin{itemize}
\item $|E|\geq \frac{1}{3\sqrt{1+(k+\delta)^2}} |I|,$
\item $\phi(x)=\psi(x)$ for all $x\in E,$ and
\item Either $-\frac{4}{5}\delta\leq\psi'(x)\leq\delta$ a.e., or $-\delta\leq\psi'(x)\leq\frac{4}{5}\delta$ a.e.
\end{itemize}
\end{thm}
\begin{thm}\label{thm:buildup:CZ} \cite[Theorem~5.3]{Rule}
Suppose that $K:\R^2\mapsto \C^{2\times 2}$ is a \CZ\ kernel with constants no more than $C_6$. Suppose that there is a constant $\theta\in(0,1]$ such that for any interval $I\subset\R$, there is a \CZ\ kernel $K_I$ and a compact subset $E\subset I$ such that
\begin{itemize}
\item $|E|>\theta |I|,$
\item For all $x,y\in E$ we have that $K_I(x,y)=K(x,y)$, and
\item $\|T_I^*\|_{L^2\mapsto L^2}\leq C_6$
\end{itemize}
where $T_I^*$ is the maximal singular integral operator associated to $K_I$. Then \[\|T^*\|_{L^2\mapsto L^2}\leq C(\theta) C_6.\]
\end{thm}
From \lemref{Tstar} and following remarks, $T$ is bounded from $L^2$ to $L^2$ if and only if $T_*$ is.

Pick some $\phi\in\Lambda^{k-\delta/9}\left(\frac{10}{9}\delta\right)$.
If $I\subset\R$ is an interval, then by \thmref{buildup} there is an $E$ and a $\phi_I\in \Lambda^{k}(\delta)$ such that if $T_I$ is as in (\ref{dfn:T}) with $\phi$ replaced by $\phi_I$, then
\begin{itemize}
\item $|E|\geq \frac{1}{3\sqrt{1+(k+\delta)^2}} |I|$,
\item $K_I(x,y)=K(x,y)$ for all $x,y\in E$.
\end{itemize}
So by \thmref{buildup:CZ}, if $T_I^*$ is bounded on $L^2$ for all $I$, then so is $T^*$. 

Thus, if \thmref{special} holds for all functions $\phi\in\Lambda^k(\delta)$, then it holds for all $\phi\in \Lambda^{k-\delta/9}\left(\frac{10}{9}\delta\right)$.

Let $a_n=\frac{1}{9}(1+(10/9)+\ldots+(10/9)^{n-1})=(10/9)^n-1$.
So if \thmref{special} holds for all $\phi\in\Lambda^{k-\delta a_{n}}((10/9)^{n}\delta)$, it also holds for all $\phi\in\Lambda^{k-\delta a_{n+1}}((10/9)^{n+1}\delta)$, provided that $k-\delta a_{n+1}\geq (10/9)^{n+1} \delta$.

Pick any $k>0$. Let $\delta(k)$ be the $\delta_0$ given by \thmref{small}, so
\thmref{special} holds for all $\phi\in\Lambda^k(\delta_0)$. Let 
$n(k)$ be the positive integer such that 
$k\geq 2\delta(k) (10/9)^{n(k)}> (9/10) k$. Then \thmref{special} 
holds for all
\[\phi\in\Lambda^{k-\delta a_{n(k)}}((10/9)^{n(k)}\delta)
\supset \Lambda^{k/2}\left(\frac{9}{20}k\right).\]

Thus, for \emph{any} $k$, \thmref{special} holds for any 
$\|\phi'\|_{L^\infty}\leq \frac{9}{20}k$ (with constants depending on~$k$). Thus, it holds for any Lipschitz function $\phi$, as desired.
\end{proof}

Finally, we deal with the assumption that $A_0$ is upper triangular with $a_{11}^0\equiv 1$. 

Consider the mapping $J:(x,t)\mapsto (f(x),t+g(x))$ where $\lambda<f'<\Lambda$. Then $\check\Gamma_X(Y)=\Gamma_{J(X)}(J(Y))$ is the fundamental solution with pole at $X$ associated with the elliptic matrix
\[\check A(y)=\frac{1}{f'(y)}\Matrix{1 &0 \\ -g'(y) & f'(y)} A(f(y))\Matrix{1 & -g'(y)\\0 & f'(y)}.\]

If we choose $f'(y)=\frac{1}{a_{11}^0(f(y))}$ (or, put another way, choose $f^{-1}(y)=\int_0^y a_{11}^0$), then we will have $\check a_{11}^0=1$, and if we choose $g'(y)=\frac{f(y)a_{21}^0(f(y))}{a_{11}^0(f(y))}$, then we will have $\check a_{21}^0=0$.

But if $\phi\in \Lambda^k(\delta_0)$ for $\delta_0$ sufficiently small (depending on $\lambda$, $\Lambda$ and~$k$), where $\phi$ is the Lipschitz function in the definition of $\Omega$, then $J (\Omega)$ will also be a special Lipschitz domain, and so $\check T_\pm$ will be bounded on it; thus, $T_\pm$ wil be bounded on $\Omega$, and so by \lemref{buildup} we may build up to arbitrary Lipschitz domains.

\section{Patching: \texorpdfstring{$\T$}{T} is bounded on good Lipschitz domains}
\label{sec:patchingT}

Recall that we wish to prove \thmref{Tbounded}:
\begin{thm} \label{thm:patching} If\/ $\V$ is a Lipschitz domain with Lipschitz constants~$k_i$, then $\T$ and~$\tilde\T$ are bounded $L^p(\partial\V)\mapsto L^p(\partial\V)$ for $1<p<\infty$, with bounds depending only on $\lambda,\Lambda,p,$ and the Lipschitz constants of~$\V$. 
\end{thm}

In the following sections, we will show that $\T$ and $\tilde\T$ are bounded on special Lipschitz domains~$\Omega$ (that is, we will prove \thmref{special}). In this section, we pass to arbitrary good Lipschitz domains.

As in \autoref{sec:buildup}, we work with $\T$; the proof for $\tilde\T$ is identical.

\begin{proof}
Note that if $\supp F\subset\partial\U\cap\partial\V$, then $\T_\U F=\T_\V F$ on $\partial\U\cap\partial\V$.

From \dfnref{domain}, we may partition $\partial\V$ as follows: there are $k_2$ points $X_i\in\partial\V$ with associated numbers $r_i>0$, such that $\partial\V\subset\cup_i B(X_i,r_i)$ and $B(X_i,2r_i)\cap\V=B(X_i,2r_i)\cap\Omega_i$ for some special Lipschitz domains $\Omega_i$. Let $\sum_i\eta_i$ be a partition of unity with $\supp\eta_i\subset\partial\V\cap B(X_i,r_i)$, and let $F_i=F\eta_i$.

Then
\[\|\T_\V F\|_{L^p}\leq \sum_{i=1}^{k_2} \|\T_\V F_i\|_{L^p}.\]

But
\[\|\T_\V F_i\|_{L^p(\partial\V)}^p=
\|\T_\V F_i\|^p_{L^p(\partial\V\cap B(X_i,2r_i))}
+\|\T_\V F_i\|^p_{L^p(\partial\V\setminus B(X_i,2r_i))}\]
and
\[\|\T_\V F_i\|_{L^p(\partial\V\cap B(X_i,2r_i))}
=\|\T_{\Omega_i} F_i\|_{L^p(\partial\V\cap B(X_i,2r_i))}
\leq\|\T_{\Omega_i} F_i\|_{L^p(\partial\Omega_i)}
\leq C \|F_i\|_{L^p}.
\]

But if $|X_i-Y|>2r_i$, then
\begin{align*}
|\T F_i(Y)|
&=\left|\int_{\partial\V} K(Y,Z) F_i(Z)\,d\sigma(Z)\right|
\\&\leq \frac{C}{|X_i-Y|}\int_{\partial\V} |F_i(Z)|\,d\sigma(Z)
\leq \frac{Cr_i^{1/p'}}{|X_i-Y|}\|F_i\|_{L^{p}}.
\end{align*}


But then
\[\|\T F_i\|_{L^p}^p\leq \int_{|X_i-Y|>2r_i}\frac{C r_i^{p-1}}{|X_i-Y|^p}\|F_i\|_{L^p}^p\,d\sigma + \int_{|X_i-Y|<2r_i}|\T_{\Omega_i}F_i|^p\,d\sigma
\leq C \|F_i\|_{L^p}^p,\]
where $C$ depends on the Lipschitz constants of~$\V$.

Therefore,
\[\|\T_\V F_i\|_{L^p}^p\leq C_p^p\|F_i\|_{L^p}^p\] and so
\[\|\T_\V F\|_{L^p}^p\leq C_{p,k_2}\sum_{i=1}^{k_2} \|\T_\V F_i\|_{L^p}^p
\leq \sum_i C_{p,k_2}\|F_i\|_{L^p}^p\leq C_{p,k_2}\|F\|_{L^p}^p.\]
\end{proof}

Note that $C_{p,k_2}$ does \emph{not} depend on $\diam \V$.

\section{Proving \thmref{flat}: preliminary remarks} \label{sec:proveflat}

From \cite[p.~42]{DJS}, we have the following useful theorem:
\begin{thm}\label{thm:useful} 
Suppose that $B_1$, $B_2:\R\mapsto \C^{2\times 2}$ are invertible matrices at all points, and suppose that $\|B_1\|_{L^\infty},$ $\|B_2\|_{L^\infty}\leq C_1.$ 

Assume that 
there exist nonnegative real smooth functions $v_i$ with $\supp v_i\subset[-1,1]$, $\int v_i=1$, and $\|v_i\|_{L^\infty}$, $\|v_i'\|_{L^\infty}\leq C_2$, and such that for all $x\in\R$ and all $t>0$,
\begin{equation}\label{eqn:accretive}
\left|\left(\int \frac{1}{t} v_i\left(\frac{x-y}{t}\right)B_i(y)\,dy\right)^{-1}\right|\leq C_3.
\end{equation}

Suppose that the operator $T:B_1\S\mapsto (B_2\S)'$ has a \CZ\ kernel, 
the oper\-ator $f\mapsto B_2^tT(B_1f)$ is weakly bounded, and that the constants in the definition of weak boundedness and \CZ\ kernel are no more than 
$C_4$. 

Suppose finally that $T(B_1)$ and $T^t(B_2)$ have BMO norm no more than 
$C_5$. 

Then $T$ has a continuous extension to $L^2$, and its norm may be bounded by a constant depending only on $C_1$,~$C_2$, $C_3$, $C_4$ and~$C_5$.\end{thm} 

Note that if $B(x)=\beta(x)I$ for some scalar-valued function $\beta$, and if $\nu\leq\Re\beta(x)\leq|\beta(x)|\leq N$ for some constants $\nu$,~$N>0$, then $B$ satisfies~(\ref{eqn:accretive}). 

$T B\in BMO$ is defined as follows: if $M_0$ is a smooth $H^1$ atom (that is, $\supp M_0 \subseteq [x_0-R,x_0+R]$, 
$\|M_0\|_{L^\infty}\leq \frac{1}{R}$, and $\int M_0=0$), then
\[
\langle M_0, T B\rangle = \langle M_0, T(\eta B)\rangle
+\int_\R \int_\R M_0(x) (K_0(x,y)-K_0(x_0,y))\, dx \, (1-\eta(y)) B(y) \, dy
\]
whenever $\eta\in C^\infty_0$, $0\leq\eta\leq 1$, and $\eta\equiv 1$ on $[x_0-2R,x_0+2R]$.

If $B\in L^\infty$, then assuming $\supp \eta\subset 3R$, and taking $x_0=0$ for convenience, we see that
\begin{multline}
\left|\int_\R \int_\R M_0(x) (K_0(x,y)-K_0(0,y))\, dx \, (1-\eta(y)) B(y) \, dy\right|
\\\leq
\int_\R |M_0(x)|  \int_{|y|>2R}\frac{|x|^\alpha}{|y|^{1+\alpha}}\, dx \, \|B\|_{L^\infty} \, dy
\leq
C\|B\|_{L^\infty}
\label{eqn:TonLinfty}
\end{multline}

If $T$ is bounded $L^2\mapsto L^2$, then since $\|M_0\|_{L^2}\leq \sqrt {2/R}$, $\|\eta B\|_{L^2}\leq \|B\|_{L^\infty} \sqrt{6R}$,
\begin{equation}
|\langle M_0,TB\rangle|
\leq \leq
(C+C\|T\|_{L^2\mapsto L^2})\|B\|_{L^\infty}.
\end{equation}

\begin{proof}[Proof of \thmref{flat}]
In \autoref{sec:B1accretive}, we will find a matrix $B_1$ which is  bounded, invertible, and satisfies~(\ref{eqn:accretive}). 
By Lemmas~\ref{lem:Tcts} and~\ref{lem:tildeTcts}, $f\mapsto B_0^t T(B_1 f)$ and $f\mapsto B_0^t \tilde T^t(B_1 f)$ are weakly bounded for \emph{any} bounded matrix~$B_0$.
By \corref{Tdirect}, $\|TB_1\|_{BMO}$ and $\|\tilde T^t B_1\|_{BMO}$ are bounded.

By \thmref{useful}, if $B$ satisfies~(\ref{eqn:accretive}), then
\begin{align*}
\|T\|_{L^2\mapsto L^2}&\leq C+C\|TB_1\|_{BMO}+C\|T^t B\|_{BMO}\leq C+C\|T^t B\|_{BMO}, \\
\|\tilde T\|_{L^2\mapsto L^2}&\leq C+C\|\tilde T^tB_1\|_{BMO}+C\|\tilde T B\|_{BMO}\leq C+C\|\tilde T B\|_{BMO}.
\end{align*}

In \autoref{sec:Tfinite}, we will find a $B_8$ such that $\|T^t(B_8)\|_{BMO}$ is finite, and so we will know that $\|T\|_{L^2\mapsto L^2}$ is finite. Unfortunately, our bound on $T^t(B_8)$ will depend on such terms as $A'$; therefore, we will seek a better bound on $\|T\|_{L^2\mapsto L^2}$.

In \autoref{sec:Ttranspose}, we will find bounded matrices $B_3,\ldots,B_6$ such that $\|\tilde TB_3\|_{BMO}\leq C\|T^t B_4\|_{BMO}$ and $\|T^t B_6\|_{BMO}\leq C\|\tilde T B_5\|_{BMO}$, where $B_3$,~$B_6$ satisfy~\ref{eqn:accretive} and $B_5$ is small (depending on $\|A-A_0\|_{L^\infty}=\epsilon_0$, $\|\phi'\|_{L^\infty}=\delta_0$). This implies that
\begin{align*}
\|T\|_{L^2\mapsto L^2}&\leq C+C\|T^t B_6\|_{BMO}
\leq C+C\|\tilde T B_5\|_{BMO}
\leq C+C\|\tilde T\|_{L^2\mapsto L^2}\| B_5\|_{L^\infty}
\\&\leq C+\left(C+C\|\tilde T B_3\|_{BMO}\right)\| B_5\|_{L^\infty}
\leq C+\left(C+C\| T^t B_4\|_{BMO}\right)\| B_5\|_{L^\infty}
\\&\leq C+\left(C+C\|T\|_{L^2\mapsto L^2} \| B_4\|_{L^\infty}\right)\| B_5\|_{L^\infty}
\\&\leq C + C\|T\|_{L^2\mapsto L^2}\|B_4\|_{L^\infty}\|B_5\|_{L^\infty}.
\end{align*}
Since $\|T\|_{L^2\mapsto L^2}$ is finite, if $\|B_5\|$ is small enough then $\|T\|_{L^2\mapsto L^2}\leq C$.

Furthermore,
\begin{align*}
\|\tilde T\|_{L^2\mapsto L^2}
&\leq C+C\|\tilde T B_3\|_{BMO}
\leq C+C\|T^t\|_{L^2\mapsto L^2} \|B_4\|_{L^\infty}
\leq C
\end{align*}
as desired.
\end{proof}
The remainder of this chapter will be devoted to establishing the facts used in this proof. We will consider only $T_+$; the case for $T_-$ is similar.

\section{\texorpdfstring{A $B$ for our $TB$ theorem}{A B for our TB theorem}}\label{sec:B1accretive}

Recall (\ref{eqn:B1def}):
\[B_1(y)=\Matrix{a_{11}(\psi(y))&0\\a_{21}(\psi(y))&1}^{-1}\Matrix{A(\psi(y))\nu(y) &\tau(y)}\sqrt{1+\phi'(y)^2}\]
where $\psi(y)=y\e^\perp +\phi(y)\e$, 
\[\tau(y)=\frac{1}{\sqrt{1+\phi'(y)^2}}\Matrix{\psi_1'(y)\\\psi_2'(y)}, \quad
\nu(y)=\frac{1}{\sqrt{1+\phi'(y)^2}}\Matrix{-\psi_2'(y)\\\psi_1'(y)}.\]
In this section, we will show that $B_1$ satisfies~(\ref{eqn:accretive}).

Recall that we have changed variables such that $a_{11}^0\equiv 1$, $a_{21}^0\equiv 0$. Let 
\[B_0 = \Matrix{A_0(\psi(y))\nu(y) &\tau(y)}\sqrt{1+\phi'(y)^2}
=\Matrix{-\psi_2'(y)+a_{12}^0(\psi_1(y))\psi_1'(y) &\psi_1'(y)\\
a_{22}^0(\psi_1(y)) \psi_1'(y) &\psi_2'(y)}
.\]
Then $B_0$ is near $B_1$. We first show that, for any interval~$I$, $\dashint_I B_0$ is invertible with bounded inverse. If $\|A-A_0\|$ is sufficiently small, this will imply that $\dashint_I B_1$ is as well.

If $A$ is elliptic, then
\[\Re a_{11}>\lambda,\quad\Re a_{22}>\lambda,\quad
|a_{12}+\overline{a_{21}}|\leq 2\sqrt{(\Re a_{11}-\lambda)(\Re a_{22}-\lambda)}\]
and so
\[|a_{12}^0|\leq2\sqrt{(1-\lambda)(a_{22}^0-\lambda)}.\]

Then, if $I=(a,b)$, we have that
\begin{align*}
\det \int_I  B_0&= \det\int_a^b \Matrix{-\psi_2'(y)+a_{12}^0(\psi_1(y))\psi_1'(y) &\psi_1'(y)\\
a_{22}^0(\psi_1(y)) \psi_1'(y) &\psi_2'(y)}\,dy
\\&=
\int_a^b\psi_2'(y)\,dy\int_a^b a_{12}^0(\psi_1(y))\psi_1'(y)\,dy
-\left(\int_a^b \psi_2'(y)\,dy\right)^2
\\&\phantom{=}
-\int_a^b \psi_1'(y)\,dy\int_a^b a_{22}^0(\psi_1(y))\psi_1'(y)\,dy
\displaybreak[0]
\\&=
\left(\psi_2(b)-\psi_2(a)\right)\int_{\psi_1(a)}^{\psi_1(b)} a_{12}^0(y)\,dy
-\left(\psi_2(b)-\psi_2(a)\right)^2
\\&\phantom{=}
-\left(\psi_1(b)-\psi_1(a)\right)\int_{\psi_1(a)}^{\psi_1(b)} a_{22}^0(y)\,dy
\displaybreak[0]
\\&\leq
\frac{1}{2\theta}\left(\int_{\psi_1(a)}^{\psi_1(b)} a_{12}^0(y)\,dy\right)^2
-\left(1-\frac{\theta}{2}\right)\left(\psi_2(b)-\psi_2(a)\right)^2
\\&\phantom{=}
-\lambda\left(\psi_1(b)-\psi_1(a)\right)^2
-\left(\psi_1(b)-\psi_1(a)\right)\int_{\psi_1(a)}^{\psi_1(b)} (a_{22}^0(y)-\lambda)\,dy
\displaybreak[0]
\\&\leq
-\left(1-\frac{\theta}{2}\right)\left(\psi_2(b)-\psi_2(a)\right)^2
-\lambda\left(\psi_1(b)-\psi_1(a)\right)^2
\\&\phantom{=}
+\frac{1}{2\theta}\left(\int_{\psi_1(a)}^{\psi_1(b)} a_{12}^0(y)\,dy\right)^2
-\left(\int_{\psi_1(a)}^{\psi_1(b)} \sqrt{a_{22}^0(y)-\lambda}\,dy\right)^2
\end{align*}
We have that
\[
\int_{\psi_1(a)}^{\psi_1(b)} |a_{12}^0(y)|\,dy
\leq  
\int_{\psi_1(a)}^{\psi_1(b)}2\sqrt{(1-\lambda)(a_{22}^0-\lambda)}\,dy
\leq 
{\sqrt{4-4\lambda}}\int_{\psi_1(a)}^{\psi_1(b)}\sqrt{a_{22}^0-\lambda}\,dy
\]
and so choosing $2\theta={4-4\lambda}$, we have that
\begin{align*}
\det \int_I  B_0&\leq
-\lambda\left(\psi_2(b)-\psi_2(a)\right)^2
-\lambda\left(\psi_1(b)-\psi_1(a)\right)^2 = -\lambda|\psi(b)-\psi(a)|^2\leq -\lambda |I|^2.
\end{align*}

Note that
$\det(M+N) = \det M+\det N + O(|M|*|N|)$. So if $\|A-A_0\|_{L^\infty}<\epsilon_0$, then $\Re \det \dashint_I B_1\leq -\lambda+C\epsilon_0$.

So $\left|\det\dashint_I B_1\right|\geq \lambda-C\epsilon_0$. But (\ref{eqn:accretive}) requires a \emph{smooth} truncation of $B_1$. 

\begin{lem} \label{lem:Bisaccretive}
Suppose that $B:\R\mapsto\C^{2\times 2}$ satisfies $|\xi\cdot B(x)\eta|\leq \Lambda|\eta| |\xi|$ for all $x\in\R^2$, $\eta,\xi\in C^{2}$. Assume further that
\begin{equation}\label{eqn:matrixaccretive}
\left|\det\dashint_I B(x)\,dx\right|\geq \lambda
\end{equation}
for all intervals $I\subset\R$.

Then there is a smooth real function $v$, with $0\leq v\leq 1$, $\int v = 1,$ $\supp v\subset\left[-1,1\right]$, and $\|v'\|\leq C(\lambda,\Lambda)$, such that if $v_t(x)=\frac{1}{t} v\left(\frac{x}{t}\right)$, then $\left|\det B* v_t (x)\right|\geq \lambda/2$ for all $t>0$ and all $x\in\R$.

\end{lem}

\begin{proof}

Choose some $v$ such that
$v\equiv 1$ on $\left(-\frac{1}{2},\frac{1}{2}\right)$, $\supp v\subset \left[-\frac{1}{2}-\mu,\frac{1}{2}+\mu\right]$, and $\int v=1$, $\|v'\|_{L^\infty}\leq\frac{2}{\mu}$ for some positive real number $\mu<1/2$ to be determined.

Then
\begin{align*}
v_t* B(x)
&=\int v_t(x-y) B(y)\,dy
\\&=
\frac{1}{t}\int_{x-t/2}^{x+t/2} B(y)\,dy
+\frac{1}{t}\int_{t/2<|x-y|<t/2+t\mu} B(y)v\left(\frac{x-y}{t}\right)\,dy
\end{align*}
and so 
\begin{align*}
\left|v_t* B(x)-\dashint_{x-t/2}^{x+t/2} B(y)\,dy\right|
&\leq 2\mu \|B\|_{L^\infty}.
\end{align*}
Therefore, 
\[|\det v_t* B(x)|\geq \lambda - C \mu\|B\|_{L^\infty}^2.\]
So if $\mu < \lambda/C\Lambda^2$, then $\Re \det v_t* B(x)\geq \lambda/2$ for all~$x$.
\end{proof}

\section{\texorpdfstring{$T$ and $\tilde T^t$ are continuous and weakly bounded on $B_1\S$}{T is continuous and weakly bounded}} 
\label{sec:Tcts} 

For real $A$, this is shown in \cite{Rule}, in Lemmas~4.3, 4.7, 4.8, and~4.10. 
We now prove weak boundedness in the complex case.

\begin{dfn}\label{dfn:weakbdd} A function $F$ is called a \dfnemph{normalized bump function} if $\supp F\subset B(X_0,10)$ for some $X_0$, and 
\[|\partial^\alpha F(X)|\leq 1\] 
for any multiindex $\alpha$ with $|\alpha|\leq 2$. For such a function $F$, 
let $F_R(X)=\frac{1}{R}F(X/R)$. 

If there is a constant $C$ such that, for any $R>0$ and any normalized bump functions $F$ and $G$, the equation 
\[|\langle G_R, PF_R\rangle| \leq \frac{C}{R}\] 
holds, then we say the operator $P$ is \dfnemph{weakly bounded}.
\end{dfn} 

\begin{lem}\label{lem:Tcts} The operator $T$ is a continuous linear operator from $B_1\S$ to $(B_0\S)'$ for any bounded $B_0$. The operator $F\mapsto B_0^t T(B_1F)$ is weakly bounded; in fact, $\|T(B_1F_R)\|_{L^\infty}\leq C/R$ for any normalized bump function~$F$.
\end{lem} 

\begin{proof} We know that $T_\pm F(x)$ is pointwise well-defined from \lemref{Kwelldef} and that if $f\in L^\infty$, $f'\in L^\infty$ and $f\in L^2$ then the limits in the definition converge uniformly. Recall that \[B_1(y)=(B_6(\psi(y))^t)^{-1}\Matrix{A(\psi(y))\nu(y)&\tau(y)}\sqrt{1+\phi'(y)^2}.\]
Choose some $F,G\in\S^{2\times 2}$ 
(that is, matrix-valued functions whose components are all in $\S$). 
We wish to show that $\langle B_0G, T(B_1F)\rangle$ 
exists, and that 
\[|\langle B_0G, T(B_1F)\rangle|\leq C\|F\|_\S\|G\|_\S.\] 

Recall that 
\begin{align*}{\langle B_0G, T(B_1F)\rangle} 
=
\lim_{h\to 0^+}\int_\R \int_\R G(x)^t B_0(x)^t K_h(x,y)B_1(y)F(y)\,dy\,dx 
\end{align*} 
Now, the components of 
\[K_h(x,y)B_1(y)=\Matrix 
{\nabla\Gamma^T_{\psi(x,h)}(\psi(y))^t\\ 
 \nabla\Gamma^T_{\psi(x,h)}(\psi(y))^t} 
\Matrix{A(\psi(y))\nu(y)&\tau(y)}\sqrt{1+\phi'(y)^2}\] 
are 
\[\frac{d}{dy}\Gamma^T_{\psi(x,h)}(\psi(y))\text{ and } 
\frac{d}{dy}\tilde\Gamma^T_{\psi(x,h)}(\psi(y)).\] 

Let $f$ be a component of $F$. We need only show that 
\[\lim_{h\to 0^+} \int_\R f(y)\frac{d}{dy}\Gamma_{\psi(x,h)}(\psi(y))\,dy,\quad
\lim_{h\to 0^+} \int_\R f(y)\frac{d}{dy}\tilde\Gamma_{\psi(x,h)}(\psi(y))\,dy\] 
are bounded by $C\|F\|_\S$ and converge uniformly in $x$, and that if 
$F=\bar F_R$ for some normalized bump function $\bar F$, then those integrals are at most $C/R.$ 

Now, for any number $R>0$, we have that
\begin{multline*}
{\left|\int_\R f(y)\frac{d}{dy} \Gamma_{\psi(x,h)}(\psi(y))\,dy\right|}
\\\begin{aligned}
&\leq
\left|\int_{|x-y|>2R}f(y) \frac{d}{dy} \Gamma_{\psi(x,h)}(\psi(y))\,dy\right|
+
\left|\int_{x-2R}^{x+2R}f(y) \frac{d}{dy} \Gamma_{\psi(x,h)}(\psi(y))\,dy\right|
\\&\leq
\left|\int_{|x-y|>2R}f(y) \frac{d}{dy} \Gamma_{\psi(x,h)}(\psi(y))\,dy\right|
+\left|f(x)\int_{x-2R}^{x+2R} \frac{d}{dy} \Gamma_{\psi(x,h)}(\psi(y))\,dy\right|
\\&\phantom{\leq}
+\left|\int_{x-2R}^{x+2R} (f(y)-f(x))\frac{d}{dy} \Gamma_{\psi(x,h)}(\psi(y))\,dy\right|
\\&\leq
\|f\|_{L^\infty}
\left|\Gamma_{\psi(x,h)}(\psi(x+2R)) -\Gamma_{\psi(x,h)}(\psi(x-2R))\right|
\\&\phantom{\leq}
+\|f'\|_{L^\infty}\int_{x-2R}^{x+2R} \left|\frac{d}{dy} \Gamma_{\psi(x,h)}(\psi(y))\right| |x-y|\,dy
+\|f\|_{L^1} \frac{C}{R}
\\&\leq
C\|f\|_{L^\infty}
+C R \|f'\|_{L^\infty}
+\|f\|_{L^1} \frac{C}{R}
\end{aligned}\end{multline*}
Picking $R=1$, we see that this is bounded by the Schwarz norm of $f$, and if $F=\check F_R$  where $\check F$ is a normalized bump function, this is clearly at most $C/R$.


Similarly, 
\[\lim_{h\to 0^+} \int_\R 
\frac{d}{dy}\tilde\Gamma_{\psi(x,h)}(\psi(y))f(y)\,dy\] 
exists, converges uniformly in $x$, and is at most $C/R$ provided $F=\check F_R$ for some normalized bump function~$\check F$.
\end{proof}

\begin{lem}\label{lem:tildeTcts} 
The operator $\tilde T^t$ is a continuous linear operator from 
$B_1\S$ to $(B_0\S)'$ for any bounded $B_0$. Also, 
$F\mapsto B_0^t \tilde T^t B_1F$ is weakly bounded; in fact, $\|\tilde T^t(B_1F_R)\|_{L^\infty}\leq C/R$ for any normalized bump function~$F$. 
\end{lem} 

\begin{proof} Fix $F,G\in\S^{2\times 2}$. 
Again, we wish to show that $\langle B_0G, \tilde T^t(B_1F)\rangle$ exists and is bounded, and that if $F=\bar F_R,G=\bar G_R$ for some normalized bump functions $\bar F,\bar G$, then 
$|\langle B_0G, \tilde T^t(B_1F)\rangle|\leq 1/R.$ 

Now, 
\begin{align*} 
\langle B_0G, \tilde T_\pm^t(B_1F)\rangle 
&= \lim_{h\to 0^\pm} 
\int_\R G(x)B_0(x)\int_\R \tilde K_h(y,x)^t B_1(y) F(y)\,dy\,dx 
\end{align*} 
But 
\[\tilde K_h(y,x)^t B_1(y) 
=\Matrix{\nabla_Y\tilde\Gamma^T_{\psi(y)}(\psi(x,h))^t 
\\\nabla_Y\tilde\Gamma^T_{\psi(y)}(\psi(x,h))^t} 
\Matrix{A(\psi(y))\nu(y)&\tau(y)}\sqrt{1+\phi'(y)^2}.\] 

We may deal with the 
$\sqrt{1+\phi'(y)^2}\nabla_Y\tilde\Gamma^T_{\psi(y)}(\psi(x,h))\cdot\tau (y)
=\frac{d}{dy}\tilde\Gamma^T_{\psi(y)}(\psi(x,h))$ as before. We are left trying to show that 
\[ 
\lim_{h\to 0^\pm} 
\int_\R 
\nu\cdot A^T(\psi(y))\nabla_Y\tilde\Gamma^T_{\psi(y)}(\psi(x,h)) f(y) 
\sqrt{1+\phi'(y)^2}\,dy 
\] 
is bounded and converges uniformly in $x$, for any $f$ a component of~$F$. 

But this integral is 
\[\int_{\partial \Omega}\nu(Y)\cdot A^T(Y) 
\nabla_Y\tilde\Gamma_{Y}^T(\psi(x,h)) f(\psi^{-1}(Y))d\sigma(Y).\] 

Let $m\in C^\infty_0(\R)$ with $m\equiv 1$ on $(-R-R\|\phi'\|_{L^\infty},R+R\|\phi'\|_{L^\infty})$,
and $0\leq m\leq 1$,
$\supp m\subset (-CR,CR)$,
and $|m'|<C/R$.
Let $u(\psi(y,t))=f(y)m(t+\phi(y)),$ so $u(y\e^\perp+t\e)=f(y)m(t)$. Then if $\pm h>0$, then 
\begin{multline*} 
{\int_{\partial \Omega}\nu(Y)\cdot A^T(Y) 
\nabla_Y\tilde\Gamma^T_{Y}(\psi(x,h)) f(\psi^{-1}(Y))d\sigma(Y)}
\\\begin{aligned}
&=  
\mp\int_{\Omega_\mp}\nabla\cdot(A^T(Y)\nabla_Y\tilde\Gamma^T_Y(\psi(x,h))u(Y))\,dY 
\\ &=  
\mp\int_{\Omega_\mp}A^T(Y)\nabla_Y\tilde\Gamma^T_Y(\psi(x,h))\cdot\nabla u(Y)\,dY 
\end{aligned}\end{multline*}
Since $|\nabla u|$ is bounded and in $L^1$, and since $\nabla_Y \Gamma_Y^T(X)\in L^1(B(X,R))$, $\nabla_Y \Gamma_Y^T(X)\in L^\infty(B(X,R)^C)$, this integral is also bounded, and the limit as $h\to 0^+$ converges uniformly in $x$. 
Thus, we have shown that $\tilde T^t$ is continuous on $B_1\S$. 

Furthermore, if $f$ is a component of a normalized bump function, then $|\supp f|\leq CR$, $|f|\leq C/R$, $|f'|\leq C/R^2$, and so
\begin{multline*} 
{\left|\int_{\partial \Omega}\nu(Y)\cdot A^T(Y) 
\nabla_Y\tilde\Gamma_{Y}(\psi(x,h)) f(\psi^{-1}(Y))d\sigma(Y)\right|}
\\\begin{aligned}
&=
\left|\int_{\Omega}A^T(Y) 
\nabla_Y\tilde\Gamma_{Y}(\psi(x,h)) \cdot \nabla u(Y)\,dY\right|
\\
&\leq
\int_{\supp \nabla u}\frac{C}{|Y-\psi(x,h)|}\frac{C}{R^2}\,dY
\leq
\int_{0}^{CR}\frac{C}{R^2\rho}2\pi\rho \,d\rho
\leq C/R
\end{aligned}\end{multline*}
and so if $F=\check F_R$ where $\check F$ is a normalized bump function, then $|\tilde T^t (B_1 F)(x)|\leq C/R$.

So $F\mapsto B_0^t \tilde T^t(B_1 F)$ is weakly bounded as well.
\end{proof}

\begin{cor}\label{cor:Tdirect}
$\|T(B_1)\|_{BMO}\leq C,$ $\|\tilde T^t(B_1)\|_{BMO}\leq C$.
\end{cor}

\begin{proof}
Recall that
\begin{align*}
T_\pm F(x) &= \lim_{h\to 0^\pm}\int \Matrix{\nabla\Gamma^T_{\psi(x,h)}(\psi(y))^t \\ \nabla\Gamma^T_{\psi(x,h)}(\psi(y))^t} B_6(\psi(y))^t F(y)\,dy,
\end{align*}
and that if $M_0$ is a smooth $H^1$ atom and $\eta\in C^\infty_0$ is 1 in a neighborhood of its support, then
\[
\langle M_0, T B\rangle = \langle M_0, T(\eta B)\rangle
+\int_\R \int_\R M_0(x) (K_0(x,y)-K_0(x_0,y))\, dx \, (1-\eta(y)) B(y) \, dy.
\]

By (\ref{eqn:TonLinfty}), the second term on the right-hand side is bounded in terms of $\lambda$ and $\Lambda$, so we need only bound $\langle M_0, T(\eta B_1)\rangle$ and $\langle M_0, \tilde T^t(\eta B_1)\rangle$. 

In fact, we need only bound $\|T(\eta B_1)\|_{L^\infty}$ and $\|\tilde T^t(\eta B_1)\|_{L^\infty}$. If $\supp\eta\subset B(x_0,R)$, then $\eta=RF_R$ for some normalized bump function $F$, and so this follows immediately from \lemref{Tcts} and \lemref{tildeTcts}.
\end{proof}

\section{The transpose inequalities\texorpdfstring{: controlling $\|T_\pm^t(B_2)\|_{BMO}$}{}} 
\label{sec:Ttranspose}

In the proof of \thmref{flat} in \autoref{sec:proveflat},
we promised to produce functions $B_3$, $B_4$, $B_5$, $B_6$ such that
$B_3$,~$B_6$ satisfy~(\ref{eqn:accretive}), $B_4$ is bounded, and $B_5$ is small provided $\epsilon_0$, $\delta_0$ are, and such that
\[\|\tilde TB_3\|_{BMO}\leq C\|T^t B_4\|_{BMO},\quad
\|T^t B_6\|_{BMO}\leq C\|\tilde T B_5\|_{BMO}.\]
 
Recall that
\begin{align*}
T_\pm F(x) &= \lim_{h\to 0^\pm}\int \Matrix{\nabla\Gamma^T_{\psi(x,h)}(\psi(y))^t \\ \nabla\Gamma^T_{\psi(x,h)}(\psi(y))^t} B_6(\psi(y))^t F(y)\,dy,\\
T_\pm' F(x) &= \lim_{h\to 0^\pm}\int \Matrix{\nabla\Gamma^T_{\psi(x)}(\psi(y,h))^t \\ \nabla\Gamma^T_{\psi(x)}(\psi(y,h))^t} B_6(\psi(y,h))^t F(y)\,dy, 
\end{align*}
By (\ref{eqn:smallshiftgamma}), $T_\pm = T'_\mp$ and $\tilde T_\pm=\tilde T'_\mp$; it will be easier to work with $(T')^t$ and $\tilde T'$, not~$T^t$.

Suppose that $B_2(x)=\beta(x)I$ for some scalar-valued function $\beta$ with $1/C<|\beta|<C$. 
Then to show that $\|(T')^t (B_2)\|_{BMO}\leq C$, we need only show that for each $R>0$, there is some $\eta\in C^\infty_0(\R)$ with $\eta\equiv 1$ on $(-R,R)$, such that
\[\int K_{-h}'(y,x)^t \beta(y)\eta(y)\,dy\]
is bounded for arbitrary $h\neq 0$ and $|x|<R/2$.

But $K_h'(y,x)^t=B_6(\psi(x,h))
\Matrix{\nabla\Gamma_{\psi(y)}^T(\psi(x,h))& \nabla\Gamma_{\psi(y)}^T(\psi(x,h))}$
and so we need only bound
\[\int \nabla\Gamma_{\psi(y)}^T(\psi(x,-h)) \beta(y)\eta(y)\,dy = \int\nabla_X\Gamma_{\psi(x,-h)}(\psi(y)) \beta(y)\eta(y)\,dy.\]

Similarly, to bound $\|\tilde T'(\beta I)\|_{BMO}$, we need only bound
\[\int \nabla_X \tilde\Gamma^T_{\psi(x,h)}(\psi(y)) \beta(y) \eta(y)\,dy.\]

If $X\in\U$ and $f\in C^\infty_0(\R^2)$, then the divergence theorem tells us that for any $\U\subset\R^2$,
\begin{align*}
{-f(X)+\int_{\U} 
\Gamma_X \nabla\cdot (A^T \nabla f)}
&=  
\int_{\U} 
\Gamma_X \nabla\cdot (A^T \nabla f)
+\int_{\U^C}f\nabla\cdot A\nabla\Gamma_X
\\&\phantom{=} 
+\int_{\U}\nabla f\cdot A\nabla\Gamma_X
+\int_{\U^C}\nabla f\cdot A\nabla\Gamma_X
\\  &=  
\int_{\U} \nabla\cdot(\Gamma_X A^T\nabla f)
+\int_{\U^C} \nabla\cdot(fA \nabla\Gamma_X )
\end{align*}
and so
\begin{equation}\label{eqn:green1} 
{-f(X)+\int_{\U} 
\Gamma_X \nabla\cdot (A^T \nabla f)}=  
\int_{\partial \U} \Gamma_X\nu\cdot A^T \nabla f\, d\sigma 
-\int_{\partial \U} f \nu\cdot A \nabla \Gamma_X\, d\sigma .
\end{equation}

Let $X=\psi(x,-h)$, with $|x|<R$. 
Then, by taking the gradient in $X$ of (\ref{eqn:green1}), 
we have that 
\begin{multline*} 
{\nabla_X 
\int_{\Omega^C} \Gamma_X(Y) \nabla\cdot (A^T(Y) \nabla f(Y))\, dY 
-\nabla f(X)}
\\=  
\nabla_X\int_{\partial \Omega} \Gamma_X(Y) \nu\cdot A^T(Y) \nabla f(Y) 
\,d\sigma(Y) 
-\nabla_X\int_{\partial \Omega} f(Y)\nu\cdot (A(Y) \nabla \Gamma_X(Y)) 
\,d\sigma(Y) 
\end{multline*}

We may simplify this: 
\begin{align*} 
{\nabla_X\int_{\partial \Omega} f\,\nu\cdot A \nabla \Gamma_X \, d\sigma} 
&= 
\nabla_X\int_{\R} 
f(\psi(y))\nu(y)\cdot(A(\psi(y)) \nabla \Gamma_X(\psi(y))\sqrt{1+\phi'(y)^2}\, dy 
\\ &= -
\nabla_X\int_{\R} 
f(\psi(y))\tau(y)\cdot\nabla \tilde\Gamma_X(\psi(y))\sqrt{1+\phi'(y)^2}\, dy 
\\&=  
\int_{\R} \nabla_X\tilde\Gamma_X(\psi(y))\frac{d}{dy}f(\psi(y))\, dy 
\end{align*} 
and 
\begin{align*} 
\int_{\partial \Omega} \nabla_X\Gamma_X \nu\cdot A^T \nabla f\, d\sigma
=  
\int_{\R} 
\nabla_X\Gamma_X(\psi(y)) \nu(y)\cdot A^T(\psi(y)) \nabla f(\psi(y))\sqrt{1+\phi'(y)^2}\, dy 
\end{align*} 

So 
\begin{multline}\label{eqn:TBinBMO:useful} 
\int_{\Omega^C}
\nabla_X  \Gamma_X(Y) \nabla\cdot (A^T(Y) \nabla f(Y))\, dY 
-\nabla f(X)
\\\begin{aligned}&=
\int_{\R} 
\nabla_X\Gamma_X(\psi(y)) \nu(y)\cdot A^T(\psi(y)) \nabla f(\psi(y))\sqrt{1+\phi'(y)^2}\, dy 
\\&\phantom{=}
- 
\int_{\R} \nabla_X\tilde\Gamma_X(\psi(y))\frac{d}{dy}f(\psi(y))\, dy 
\end{aligned}\end{multline}

Let $f(y,s)=(\xi s+m(y))\rho(y)\rho(s)$, where $m'(y)=\frac{\zeta-\xi a_{21}(y)}{a_{11}(y)}$, for some constants $\xi$,~$\zeta$, where $\rho\equiv 1$ on $(-R,R)$, $\rho\in C^\infty(-R-1,R+1)$, and $|\rho'|<C$, $|\rho''|<C$.

Then
\[\nabla f(y,s)=\Matrix{
\frac{\zeta-\xi a_{21}(y)}{a_{11}(y)}\\
\xi}
+
(\xi s+m(y))\Matrix{
\rho'(y)\rho(s)
\\
\rho(y)\rho'(s)
}
\]
and so
\begin{align*}
A^T(y)\nabla f(y,s)&=
\Matrix{
\zeta
\\
a_{12}(y)\frac{\zeta-\xi a_{21}(y)}{a_{11}(y)}+\xi a_{22}(y)
}\rho(y)\rho(s)
\\&\phantom{=}
+(\xi s+m(y))\Matrix{
a_{11}(y)\rho'(y)\rho(s)+a_{21}(y)\rho(y)\rho'(s)
\\
a_{12}(y)\rho'(y)\rho(s)+a_{22}(y)\rho(y)\rho'(s)
}
\end{align*}
and 
\begin{align*}
f(\psi(y))&=f(ye_2+\phi(y)e_1,-ye_1+\phi(x)e_2)
\\&=\big(\xi\phi(y)e_2-\xi ye_1+m(ye_2+\phi(y)e_1)\big) \eta(ye_2+\phi(y)e_1)\eta(\phi(y)e_2-ye_1).
\end{align*}

Thus, $\nabla\cdot A^T(Y)\nabla f(Y)=0$ except on $\U_R$, 
where $\U_R=\{(y,s):R<|s|<R+1$ or $R<|y|<R+1\}$.

So 
\begin{align*}
\lefteqn{\int_{\Omega^C} \partial_{x_i} \Gamma_X(Y) \nabla\cdot (A^T(Y)\nabla f(Y))\,dY}
\hspace{3em}&\\&=
\int_{\U_R\cap\Omega^C} \partial_{x_i} \Gamma_X(Y) \Matrix{
\zeta
\\
a_{12}(y)\frac{\zeta-\xi a_{21}(y)}{a_{11}(y)}+\xi a_{22}(y)
}\cdot
\Matrix{\rho'(y)\rho(s)\\\rho(y)\rho'(s)}
\,dy\,ds
\\&+
\int_{\U_R\cap\Omega^C} \partial_{x_i} \Gamma_X(Y)  
\Matrix{\frac{\zeta-\xi a_{21}(y)}{a_{11}(y)}\\\xi}\cdot
A^T(y)\Matrix{\rho'(y)\rho(s)\\\rho(y)\rho'(s)}
\,dy\,ds
\\&+
\int_{\U_R\cap\Omega^C} \partial_{x_i} \Gamma_X(Y) 
(\xi s+m(y))\nabla\cdot A^T(y)\Matrix{\rho'(y)\rho(s)\\\rho(y)\rho'(s)}
\,dy\,ds
\displaybreak[0]\\&=
O(\lambda,\Lambda,\zeta,\xi)
\\&\phantom{=}+
\int_{\U_R\cap\Omega^C} 
\nabla\cdot\left(
\partial_{x_i} \Gamma_X(Y) 
(\xi s+m(y))A^T(y)\Matrix{\rho'(y)\rho(s)\\\rho(y)\rho'(s)}\right)
\,dy\,ds
\\&\phantom{=}
-\int_{\U_R\cap\Omega^C} 
\nabla\left(
\partial_{x_i} \Gamma_X(Y) (\xi s+m(y))\right)
\cdot
A^T(y)\Matrix{\rho'(y)\rho(s)\\\rho(y)\rho'(s)}
\,dy\,ds
\displaybreak[0]\\&=
O(\lambda,\Lambda,\zeta,\xi)
\\&\phantom{=}
-
\int_{\U_R\cap\partial\Omega} 
\partial_{x_i} \Gamma_X(Y) 
(\xi s+m(y))\nu\cdot A^T(y)\Matrix{\rho'(y)\rho(s)\\\rho(y)\rho'(s)}\,dy\,ds
\\&\phantom{=}
-\int_{\U_R} 
\nabla\left(
\partial_{x_i} \Gamma_X(Y) (\xi s+m(y))\right)
\cdot
A^T(y)\Matrix{\rho'(y)\rho(s)\\\rho(y)\rho'(s)}
\,dy\,ds
\\&=O(\lambda,\Lambda,\zeta,\xi)
\end{align*}
since $|\nabla_X\Gamma_X(Y)|\leq C/|X-Y|$, $|\nabla_Y\partial_{x_i}\Gamma_X(Y)|\leq C/|X-Y|^2$.

Let $\eta(y)=\rho(\psi_1(y))\rho(\psi_2(y))$.

Then
\begin{align*}
\Matrix{
\frac{\zeta-\xi a_{21}(y)}{a_{11}(y)}\\
\xi}
&=O(\lambda,\Lambda,\zeta,\xi)
+\int_{\R} \nabla_X\tilde\Gamma_X(\psi(y))\frac{d}{dy}f(\psi(y))\, dy
\\&\phantom{=}
- \int_{\R} 
\nabla_X\Gamma_X(\psi(y)) \nu(y)\cdot A^T(\psi(y)) \nabla f(\psi(y))\sqrt{1+\phi'(y)^2}\, dy 
\\&=
O(\lambda,\Lambda,\zeta,\xi)
+\int_{\R} \nabla_X\tilde\Gamma_X(\psi(y))
\beta_1(y,\zeta,\xi)
\, dy 
\\&\phantom{=}
+ \int_{\R} 
\nabla_X\Gamma_X(\psi(y)) 
\beta_2(y,\zeta,\xi)
\, dy 
\end{align*}
where
\begin{align*}
\beta_1(y,\zeta,\xi)&=\xi \phi'(y)e_2-\xi e_1+\frac{\zeta-\xi a_{21}(\psi(y))}{a_{11}(\psi(y))}(\phi'(y)e_1+e_2),\\
\beta_2(y,\zeta,\xi)&=\zeta(e_1-\phi'(y)e_2)+
\frac{\zeta a_{12}(\psi(y))+\xi \det A(\psi(y))}{a_{11}(\psi(y))}(\phi'(y)e_1+e_2)
\end{align*}

Define $\beta_3,\ldots,\beta_6$ as folows:
\begin{align*}
\beta_3(y)&=\beta_1\left(y,\frac{\Lambda^4}{\lambda^3}e_2,-e_1\right)=
(e_1)^2
-\phi'(y)e_1 e_2
+\frac{\frac{\Lambda^4}{\lambda^3} e_2+e_1 a_{21}(\check y)}{a_{11}(\check y)}
(\phi'(y)e_1+e_2)
\\
\beta_4(y)&=\beta_2\left(y,\frac{\Lambda^4}{\lambda^3}e_2,-e_1\right)
\\&=
-\frac{\Lambda^4}{\lambda^3} \phi'(y)(e_2)^2
+\frac{\Lambda^4}{\lambda^3} e_1e_2
+
\frac{\frac{\Lambda^4}{\lambda^3} e_2 a_{12}(\check y)-e_1 \det A(\check y)}{a_{11}(\check y)}(\phi'(y)e_1+e_2)
\\
\beta_5(y)&=\beta_1(y,e_1,e_2)=
\phi'(y)e_2^2-e_2 e_1+\frac{e_1-e_2 a_{21}(\check y)}{a_{11}(\check y)}(\phi'(y)e_1+e_2)
\\
\beta_6(y)&=\beta_1(y,e_1,e_2)=
e_1^2-\phi'(y)e_1e_2
+
\frac{e_1 a_{12}(\check y)+e_2 \det A(\check y)}{a_{11}(\check y)}(\phi'(y)e_1+e_2)
\end{align*}
where $\check y=\psi_1(y)$.

Then if $B_i=I\beta_i$,
\[\|\tilde T'B_3\|_{BMO}\leq C+C\|(T^t)' B_4\|_{BMO},
\quad
\|(T^t)' B_6\|_{BMO}\leq C+C\|\tilde T'B_5\|_{BMO}\]
as desired; we need only show that $B_3$,~$B_6$ satisfy~(\ref{eqn:accretive}), $B_5$ is small and all the $B_i$ are bounded. 

Note the following:
\begin{itemize}
\item $\beta_3,\ldots,\beta_6$ are all bounded by a constant depending only on $\lambda,\Lambda$.

\item If $\|\phi'\|_{L^\infty}$ is small enough, then 
\begin{align*}
\beta_3(y)\approx
(e_1)^2
+\frac{\frac{\Lambda^4}{\lambda^3} (e_2)^2+e_1e_2 a_{21}(\check y)}{a_{11}(\check y)}
\end{align*}
and so 
\begin{align*}
\Re \beta_3(y)&\gtrsim 
(e_1)^2 
+\frac{\Lambda^2}{\lambda^2}(e_2)^2
-|e_1e_2|\frac{\Lambda}{\lambda}
\geq \frac{1}{2}
\end{align*}
since $e_1^2+e_2^2=1$.

\item Similarly, if $\|\phi'\|_{L^\infty}$ is small enough, then
\begin{align*}
\beta_6(y)\approx \frac{1}{a_{11}(\check y)}
\left(
a_{11}(\check y)(e_1)^2
+a_{12}(\check y)(e_1e_2)
+(e_2)^2\det A(\check y)
\right)
\end{align*}
and if $\|A-A_0\|$ is small enough, so that $a_{11}\approx 1$, $a_{21}\approx 0$, then 
\begin{align*}
\beta_6(y)\approx \frac{1}{a_{11}(\check y)}
\Matrix{e_1 &e_2} A(\check y) \Matrix{e_1\\e_2}
\end{align*}
so $\Re \beta_6(y)\geq \frac{\lambda^2}{\Lambda^2}$.

\item Finally, if $\|\phi'\|_{L^\infty}$ is small, and if $\|A-A_0\|$ is small, so that $a_{11}\approx 1$, $a_{21}\approx 0$, then 
\[\beta_5(y)=
\frac{1}{{a_{11}}}\left(
e_1e_2 (1-a_{11})
+\phi'(e_2^2a_{11}+e_1^2-e_1e_2a_{21})
-e_2^2 a_{21}
\right)
\]
is also small.

\end{itemize}


\section{\texorpdfstring{$\|T\|_{L^2\mapsto L^2}$ is finite}{T is bounded on L2}} \label{sec:Tfinite}

In this section, we establish that $\|T\|_{L^2\mapsto L^2}$ is finite; this allows us to use results of the form
$\|T\|_{L^2\mapsto L^2}\leq C(\lambda,\Lambda)+\epsilon\|T\|_{L^2\mapsto L^2}$ to show that $\|T\|_{L^2\mapsto L^2}\leq C(\lambda,\Lambda).$

Note that we can only prove this under our a priori assumptions that $A$ and $\phi$ are smooth, and that for some (large) $R_0$, $A(y)=I$ and $\phi(y)=0$ for $|y|>R_0$. While the following analysis will yield a bound on $\|T\|_{L^2\mapsto L^2}$, the bound will depend on $A'$, $\phi''$, $R_0$, and (if $e_2\neq 0$) $1/e_2$; thus, the previous analysis was necessary (it yields a bound independent of these ugly quantities). 

We assume $\delta_0<1/2$, so that $\|\phi\|_{L^\infty}\leq R_0/2$.

First, we consider the case where $e_2=0$ (that is, we are working in a domain that looks like the left or right half-planes). 
Then either $e_1=1$ or $e_1=-1$; we consider only the case where $e_1=1$. Let
\[f(x,t)=\int_{\eta(x)\phi(t)}^{x}\frac{1}{a_{11}(w)}dw
\]
where $\eta\equiv 1$ on $(-R_0,R_0)$, $\eta\equiv 0$ on $(-2R_0,2R_0)^C$, and $|\eta'|,|\eta''|\leq C(R_0)$.
Then
\[f(\psi(y,s))=f(\phi(y)+s,y)=
\int_{\eta(\phi(y)+s)\phi(y)}^{\phi(y)+s}\frac{1}{a_{11}(w)}dw
\]
and
\[\nabla f(x,t)=
\Matrix{
\frac{1}{a_{11}(x)}-\frac{\eta'(x)\phi(t)}{a_{11}(\eta(x)\phi(t))}
\\
-\frac{\eta(x)\phi'(t)}{a_{11}(\eta(x)\phi(t))}
}
\]
\[
\nabla f(\psi(y,s))=
\Matrix{
\frac{1}{a_{11}(\phi(y)+s)}-\frac{\eta'(\phi(y)+s)\phi(y)}{a_{11}(\eta(\phi(y)+s)\phi(y))}
\\
-\frac{\eta(\phi(y)+s)\phi'(y)}{a_{11}(\eta(\phi(y)+s)\phi(y))}
}
\]

Note that if $|t|>R_0$ or $|x|>2R_0$, then $\nabla f(y,s)=\Matrix{1/a_{11}(y) &0}$.

So 
\[A^T(x)\nabla f(x,t)=
\Matrix{a_{11}(x)&a_{21}(x)\\a_{12}(x)&a_{22}(x)}
\Matrix{
\frac{1}{a_{11}(x)}-\frac{\eta'(x)\phi(t)}{a_{11}(\eta(x)\phi(t))}
\\
-\frac{\eta(x)\phi'(t)}{a_{11}(\eta(x)\phi(t))}
}
\]
Therefore, $|\nabla\cdot A^T\nabla f|\leq C(R_0,\|\phi''\|_{L^\infty},\|A'\|_{L^\infty}),$ and if $|s|>R_0$ or $|y|>2R_0$, then $\nabla\cdot A^T(y)\nabla f(y,s)=0$.
So
\begin{align*}{ 
\left|\int_{\Omega_h^C} \nabla_X\Gamma_X(Y) \nabla\cdot (A^T(Y) \nabla f(Y))\, dY 
-\nabla f(X) \right|
} \leq C(R_0,\|\phi''\|_{L^\infty},\|A'\|_{L^\infty})
\end{align*} 

Now, recall that 
\begin{multline*} 
{\int_{\Omega} 
\nabla_X\Gamma_X(Y) \nabla\cdot (A^T(Y) \nabla f(Y))\, dY 
-\nabla f(X)}
\\\begin{aligned}&=
\int_{\R} 
\nabla_X\Gamma_X(\psi(y)) \nu(y)\cdot A^T(\psi(y)) \nabla f(\psi(y))\sqrt{1+\phi'(y)^2}\, dy 
\\&\phantom{=}
- 
\int_{\R} \nabla_X\tilde\Gamma_X(\psi(y))\frac{d}{dy}f(\psi(y))\, dy 
\end{aligned}\end{multline*}
Then
\begin{align*} 
\int_{\R} \nabla_X\tilde\Gamma_X(\psi(y))\frac{d}{dy}f(\psi(y))\, dy 
 &=  0
\end{align*} 
and  
\begin{multline*}
\int_{\R} 
\nabla_X\Gamma_X(\psi(y)) 
\nu(y)\cdot A^T(\psi(y)) \nabla f(\psi(y))\sqrt{1+\phi'(y)^2}\, dy 
\\=  
\int_{\R} 
\nabla_X\Gamma_X(\psi(y)) \frac{1}{a_{11}(\phi(y))}\Matrix{-1 \\ \phi'(y)}\cdot 
A^T(\psi(y)) 
\Matrix{
1
\\
-\phi'(y)
}
\, dy 
\end{multline*}

So
\[\|T^t(B_8)\|_{BMO}\leq C(\lambda,\Lambda,R_0,\|A'\|_{L^\infty},\|\phi''\|_{L^\infty})\] 
where \[B_8(y)= \Matrix{1&0\\0&1}
\frac{1}{a_{11}(\phi(y))}\Matrix{-1 \\ \phi'(y)}\cdot 
A^T(\psi(y)) \Matrix{1\\-\phi'(y)}
.\] 
Thus, under our a priori smoothness assumptions, $\|T\|_{L^2\mapsto L^2}$ is finite. 

Now, consider the case where $e_2\neq 0$. Note that if $|x|>R_0/|e_2|$, then $\phi(x)=0$ and $\psi(x)=(xe_2,-xe_1)$ and $|\psi_1(x)|=|xe_2|>R_0$, so $A(\psi(x))=I$.

Assume that $R_1\geq R_0/|e_2|$ is so large that if $|x|<R_0/|e_2|$ and $|t|<R_0$, then $|\psi(x,t)|<R_1$, and also that if $|\psi(x,t)|>R_1$ for some $|t|<1$, then $|\psi_1(x)|>R_0.$

Choose $\eta\in C^\infty$ such that $\eta\equiv 1$ on $(-R_1,R_1)$, $\eta \equiv 0 $ on $(-2R_1,2R_1)^C$, and $|\eta'|,|\eta''|<C(R_1)$.

Let
\begin{align*}
f(x,t)&=e_2 t
+(1-\eta(t))\left(\int_{-R_0}^x \frac{e_1-e_2 a_{21}(w)}{a_{11}(w)}dw\right)
\\&\phantom{=}+\eta(t) \left(e_1 x-\phi(e_2 x-e_1 t)\right)
+\eta(t)e_1R_0.
\end{align*}
Note that if $|x|>R_0$, our a priori assumption $\int_{-R_0}^{R_0}\frac{e_1-e_2 a_{21}}{a_{11}}=2R_0 e_1$ means that
\[f(x,t)=e_2 t+e_1 x+e_1 R_0
-\eta(t) \phi(e_2 x-e_1 t)
\]

So if $|y|\geq R_0/|e_2|$, then $\phi(y)=0$, so
\[f(\psi(y))=f(-ye_2,ye_1)=e_1 R_0-\eta(ye_1)\phi(y)=e_1 R_0.\]

Conversely, if $|y|< R_0/|e_2|$, then $|\psi(y)|<R_1$, and so $\eta(\psi_2(y))=1$, and so
\begin{align*}
f(\psi(y))&=f(ye_2+\phi(y)e_1,-ye_1+\phi(y) e_2)
\\&=
e_2 (-ye_1+\phi(y) e_2)
+e_1 (ye_2+\phi(y)e_1)-\phi(y)
+e_1R_0
=e_1 R_0.
\end{align*}

So $f$ is constant on $\partial\Omega$.

Now, consider 
\begin{align*}
\nabla f(x,t)&=
\Matrix{
(1-\eta(t))\frac{e_1-e_2 a_{21}(x)}{a_{11}(x)}
+e_1\eta(t) 
+e_2\eta(t)\phi'(e_2 x-e_1 t)
\\
e_2
-e_1\eta(t)\phi'(e_2 x-e_1 t)
}
\\&\phantom{=}+
\eta'(t)\Matrix{
0\\
-\int_{-R_0}^x \frac{e_1-e_2 a_{21}(w)}{a_{11}(w)}dw
+e_1 x-\phi(e_2 x-e_1 t)+e_1 R_0
}
\end{align*}

If $|t|>2R_1$, or if $|x|,|e_1 t-e_2 x|>R_0$, then 
\begin{align*}
\nabla f(x,t)&=
\Matrix{
\frac{e_1-e_2 a_{21}(x)}{a_{11}(x)}
\\
e_2
}
\end{align*}
and so $\div A^T(x)\nabla f(x,t)=0$. So $\div A^T\nabla f$ is zero outside a bounded set. Clearly, it is bounded in this set by a constant which depends on $A',\phi'', R_0,\lambda,\Lambda.$

If $|t|<R_1$ or $|x|,|e_1 t-e_2 x|>R_0$, then
\[\nabla f(x,t)
=\Matrix
{
e_1+e_2\phi'(e_2 x-e_1 t)
\\
e_2-e_1\phi'(e_2 x-e_1 t)
}=\e+\phi'(e_2 x-e_1 t)\e^\perp.
\]
Note that by definition of $R_1$ if $(x,t)=\psi(y,s)$ for any $|s|<1$, then either $|t|<R_1$ or $|x|,|e_1 t-e_2 x|>R_0$. So this is true near $\partial\Omega$.

Now, recall that
\begin{multline*}
{\int_{\Omega^C} 
\nabla_X\Gamma_X(Y) \nabla\cdot (A^T(Y) \nabla f(Y))\, dY 
-\nabla f(X)}
\\\begin{aligned}&=  
\int_{\R} 
\nabla_X\Gamma_X(\psi(y)) \nu(y)\cdot A^T(\psi(y)) \nabla f(\psi(y))\sqrt{1+\phi'(y)^2}\, dy 
\\&\phantom{=}
- 
\int_{\R} \nabla_X\tilde\Gamma_X(\psi(y))\frac{d}{dy}f(\psi(y))\, dy 
\end{aligned}\end{multline*}

Since $|\nabla_X \Gamma_X(Y)|\leq C/|X-Y|$, $\frac{d}{dy} f(\psi(y))=0$, and $\div A^T\nabla f$ is bounded with compact support, we have that
\begin{multline*}
{\left|\int_\R \nabla_X \Gamma_X(\psi(y))
(\e+\phi'(y)\e^\perp)
\cdot 
A^T(\psi(y))(\e+\phi'(y)\e^\perp)\,dy\right|}
\\
\leq C(\lambda,\Lambda, R_0, 1/|e_2|, \|A'\|_{L^\infty},\|\phi''\|_{L^\infty})
\end{multline*}

Therefore, \[\|T^t(B_8)\|_{BMO}\leq C(\lambda,\Lambda, R_0, 1/|e_2|, \|A'\|_{L^\infty},\|\phi''\|_{L^\infty})\]
where $B_8(y)=I((\e+\phi'(y)\e^\perp)
\cdot 
A^T(\psi(y))(\e+\phi'(y)\e^\perp))$ is bounded with bounded inverse for $\|\phi'\|_{L^\infty}$ small enough. Thus, $\|T\|_{L^2\mapsto L^2}$ is finite.

\chapter[Layer potentials on $H^1$]{Useful theorems involving layer potentials on\texorpdfstring{~$H^1$}{ H1}}
\label{chap:H1useful}

\ifmulticonnect
	\newcommand\boundary{\omega}
	\newcommand\boundaryO{{\omega_0}}
\else
	\newcommand\boundary{{\partial\V}}
	\newcommand\boundaryO{{\partial\V}}
\fi

We have now shown that $\T$ is bounded on $L^p(\partial\Omega)$. By \thmref{NDf}, 
\[\|N (\D f)\|_{L^p}\leq \|f\|_{L^p},\quad \|N(\nabla\S g)\|_{L^p}\leq \|g\|_{L^p}.\]

We will eventually find ourselves wanting to interpolate from $L^p$ to $H^1$; thus, we would also like to show that $\T$, $\K$ and $\J$ are bounded on~$H^1$.

\section{Boundedness of layer potentials on\texorpdfstring{~$H^1$}{H1}}

We begin by proving a lemma:
\begin{lem} \label{lem:inH1}
Suppose that $f\in L^1(\boundary)$, \ifmulticonnect where $\boundary$ is a connected component of $\partial\V$, \fi $\V$ is good Lipschitz domain, and assume that for some $C_3,R,\alpha>0$, $p>1$,
\[\int_{\partial\V} f\,d\sigma=0,
\quad
\|f\|_{L^p(\partial\V)}\leq \frac{C_3}{R^{1-1/p}},
\quad
\int_{\partial\V} f(X) (1+|X|/R)^\alpha\,d\sigma(X)\leq C.
\]
Then $f$ is in $H^1(\partial\V)$ with $H^1$ norm depending only on $C_3$, $p$, $\alpha$ and the Lipschitz constants of~$\V$. (Specifically, not on~$R$.)
\end{lem}

\begin{proof} Since we can parameterize $\boundary$ by arc length, it suffices to prove this in the case where $\partial\V=\R$.

Let $\varphi$ be a Schwarz function with $\int\varphi\neq 0$. Recall that $f\in H^1$ if
\[\int \sup_t \left|\int f(y)\frac{1}{t}\varphi\left(\frac{x-y}{t}\right)\,dy\right|\,dx\]
is finite, and that its $H^1$ norm is comparable to the value of this integral (with comparability constants depending only on~$\varphi$).

Now, note that the inner integral is at most $C_\varphi Mf(x)$. So
\begin{align*}
\int_{|x|<R} \sup_t \left|\int f(y)\frac{1}{t}\varphi\left(\frac{x-y}{t}\right)\,dy\right|\,dx
&\leq
C\int_{|x|<R} Mf(x)\,dx
\leq C(2R)^{1/p'}\|Mf\|_{L^p}
\\&\leq C_p R^{1-1/p} \|f\|_{L^p}
\leq C_p.
\end{align*}

For simplicity, assume that $\varphi(x)\geq 0$ for all $x$. Then if $|x|>R$,
\begin{align*}
\int f(y)\frac{1}{t}\varphi\left(\frac{x-y}{t}\right)\,dy
&=
\int f(y)\frac{1}{t}\left(\varphi\left(\frac{x-y}{t}\right) -\varphi\left(\frac{x}{t}\right)\right)\,dy
\\&= 
\int_{|y|<|x|/2} 
+
\int_{|x|/2<|y|<2|x|} 
+
\int_{2|x|<|y|} 
\end{align*}
Now,
\begin{multline*}
\left|\int_{|y|<|x|/2} f(y)\frac{1}{t}\left(\varphi\left(\frac{x-y}{t}\right) -\varphi\left(\frac{x}{t}\right)\right)\,dy\right|
\leq
\int_{|y|<|x|/2} |f(y)|\frac{1}{t}\frac{|y|}{t}\frac{C}{(|x|/t)^2}
\,dy
\\=
C\int_{|y|<|x|/2} |f(y)| |y|^\alpha\frac{|y|^{1-\alpha}}{|x|^2}
\,dy
\leq
\frac{C}{|x|^{1+\alpha}}\int |f(y)| |y|^\alpha
\,dy
\leq
\frac{C R^\alpha}{|x|^{1+\alpha}}
\end{multline*}
and
\begin{multline*}
\left|\int_{2|x|<|y|} f(y)\frac{1}{t}\left(\varphi\left(\frac{x-y}{t}\right) -\varphi\left(\frac{x}{t}\right)\right)\,dy\right|
\\\leq
\int_{2|x|<|y|} |f(y)|\frac{1}{t}\left(\frac{Ct}{|y-x|}+\frac{Ct}{|x|}\right)\,dy
\leq
\int_{2|x|<|y|} |f(y)| |y|^\alpha\frac{C}{|x|^{1+\alpha}}\,dy
\leq
\frac{C R^\alpha}{|x|^{1+\alpha}}
\end{multline*}
and similarly
\begin{align*}
\left|\int_{|x|/2<|y|<2|x|} f(y)\frac{1}{t}\varphi\left(\frac{x}{t}\right)\,dy\right|
&\leq
\frac{C R^\alpha}{|x|^{1+\alpha}}.
\end{align*}

But $\int_{|x|>R} \frac{C R^\alpha}{|x|^{1+\alpha}}\,dx\leq C$.

So we need only bound
\[\int_{|x|>R}\sup_t\left|\int_{|x|/2<|y|<2|x|} f(y)\frac{1}{t}\varphi\left(\frac{x-y}{t}\right)\,dy\right|\,dx.\]
But this is equal to
\begin{multline*}
{\sum_{k=1}^\infty
\int_{2^{k-1}R<|x|<2^k R}\sup_t\left|\int_{|x|/2<|y|<2|x|} f(y)\frac{1}{t}\varphi\left(\frac{x-y}{t}\right)\,dy\right|\,dx}
\\\begin{aligned}&
\leq\sum_{k=1}^\infty
\int_{2^{k-1}R<|x|<2^k R}\sup_t\int_{2^{k-2}R<|y|<2^{k+1}R} |f(y)|\frac{1}{t}\varphi\left(\frac{x-y}{t}\right)\,dy\,dx
\\&
\leq\sum_{k=1}^\infty
\int_{2^{k-1}R<|x|<2^k R}
C M f_k(x)
\,dx
\end{aligned}\end{multline*} 

where $f_k(x)=f(x)$ if $2^{k-2}R<|x|<2^{k+1}R$ and $f_k(x)=0$ otherwise.


Now,
\[\int_{2^{k-1}R<|x|<2^k R}
M f_k(x)
\,dx
=
\int_0^\infty \lambda(\beta)\,d\beta\]
where $\lambda(\beta)=|\{x:2^{k-1}R<|x|<2^k R, Mf_k(x)<\beta\}|$.

Note the following three upper bounds on $\lambda(\beta)$:
\begin{itemize}
\item $\lambda(\beta)\leq 2^{k+1}R$.
\item Since $f\mapsto Mf$ is weak $(1,1)$-bounded,
$\lambda(\beta)\leq \frac{C\|f_k\|_{L^1}}{\beta} \leq\frac{C}{2^{k\alpha}\beta}.$
\item Since $f\mapsto Mf$ is strong $(p,p)$-bounded,
$\lambda(\beta)\leq \frac{C\|f_k\|_{L^p}^p}{\beta^p}\leq\frac{C}{R^{p-1}\beta^p}.$
\end{itemize}
So
\begin{align*}
\int_0^\infty \lambda(\beta)\,d\beta
&=
\int_0^{1/2^{k+k\alpha}R} \lambda(\beta)\,d\beta
+
\int_{1/2^{k+k\alpha}R}^{2^{k\alpha/(p-1)}/R} \lambda(\beta)\,d\beta
+
\int_{2^{k\alpha/(p-1)}/R}^\infty \lambda(\beta)\,d\beta
\\&\leq
\int_0^{1/2^{k+k\alpha}R} 2^{k+1}R\,d\beta
+
\int_{1/2^{k+k\alpha}R}^{2^{k\alpha/(p-1)}/R} \frac{C}{2^{k\alpha}\beta}\,d\beta
\\&\qquad+
\int_{2^{k\alpha/(p-1)}/R}^\infty \frac{C}{R^{p-1}\beta^p}\,d\beta
\\&=
\frac{2}{2^{k\alpha}}
+
\frac{C}{2^{k\alpha}} \ln(2^{k+k\alpha+k\alpha/(p-1)})
+
\frac{C}{p-1}2^{-k\alpha}
\leq C_p (1+k)2^{-k\alpha}
\end{align*}
Summing from $k=1$ to infinity gives a finite number depending only on $p$ and $\alpha$, as desired.
\end{proof}

We use this lemma to show that $\K^t$, $\J^t$ are bounded on $H^1$.

\begin{thm} \label{thm:H1patching} If $\V$ is a good Lipschitz domain, then $\K^t_\pm$ and $\J^t$ are bounded $H^1(\boundary)\mapsto H^1(\partial\V)$, \ifmulticonnect where $\boundary$ is any connected component of $\partial\V$, \fi with bounds depending only on $\lambda,\Lambda$ and the Lipschitz constants of~$\V$.
\end{thm}

\begin{proof}
Suppose that $f$ is an atom in $H^1(\boundaryO)$. 

Then $\int f\,d\sigma=0$, $\supp f\subset B(X_0,R)$, $\sigma(\supp f)<2R$, $\|f\|_{L^\infty}\leq 1/R$ for some $X_0\in\R^2$ and $R>0$; without loss of generality we let $X_0=0$.

We intend to use \lemref{inH1} to show that $\K_\pm^t f,$ $\J^t f$ is in $H^1(\boundary)$\ifmulticonnect\ for any connected component of $\partial\V$ (including~$\boundaryO$)\fi.

First: if $1<p<\infty$, then
\[\|\K_{\V}^t f\|_{L^p(\V)},\>\|\K_{\V_-}^t f\|_{L^p(\V)},\>\|\J_{\V}^t f\|_{L^p(\V)}
\leq C_p \|f\|_{L^p}\leq C_p R^{1/p-1}.\]

Next, if $|X|>2R$, then
\begin{align*}|\nabla \S^T f(X)|
&=\left|\int_{\boundaryO} \nabla_X \Gamma_X(Y) f(Y)\,d\sigma(Y)\right|
\\&=\left|\int_{\boundaryO} \nabla_X (\Gamma_X(Y)-\Gamma_X(0)) f(Y)\,d\sigma(Y)\right|
\leq \frac{C R^\alpha}{|X|^{1+\alpha}}\end{align*}

So \ifmulticonnect if $\boundary$ is a connected component of $\partial\V$, then\fi
\begin{align*}
{\int_\boundary |\K^t_\V f(X)| (1+|X|/R)^{\alpha/2} \,d\sigma(X)}
&\leq
\int_{\boundary\setminus B(0,2R)}\frac{C R^\alpha}{|X|^{1+\alpha}} (|X|/R)^{\alpha/2} \,d\sigma(X)
\\&\qquad
+C\int_{\boundary\cap B(0,2R)}|\K^t_\V f(X)| \,d\sigma(X)
\\&\leq
C+C R^{1-1/p} \|\K^t_\V f\|_{L^p}\leq C.
\end{align*}
Similar results hold for $\K^t_{\V_-},\J^t$.

We need only show that $\int \K^t f\,d\sigma=\int \J^t f\,d\sigma=0$. For bounded domains, this is simple:
\[\int_\boundary \J^t_{\V} f(X)\,d\sigma(X)=\int_\boundary\tau\cdot\nabla \S^T f(X)\,d\sigma(X).\]
By continuity of $\S^T f$ this must be zero.

Also,
\[\int_\boundary \K^t_{-} f(X)\,d\sigma(X)=-\int_\boundary\nu\cdot A^T\nabla \S^T f(X)\,d\sigma(X).\]
But this must be zero, since $\div A^T\nabla\S^T f=0$ on $\R^2-\boundaryO\supset\V$. 

By \lemref{Kpm}, $\K^t_+f-\K^t_-f=f$, so 
$\int_\boundary \K^t_+f \,d\sigma = 0$ if $f\in H^1(\partial\V)$. A similar argument holds if $\V^C$ is bounded.

Conversely, if $\V=\Omega$ is a special Lipschitz domain, then
\[\S^T f(X)
=\int_{\partial\Omega}\Gamma_X (Y) f(Y)\,d\sigma(Y)
=\int_{\partial\Omega}(\Gamma_X (Y)-\Gamma_{X}(X_0)) f(Y)\,d\sigma(Y)\]
and so if $|X-X_0|$ is large enough, then
\[|\S^T f(X)|\leq 
\int_{\supp f}\frac{|X-Y|}{|X-X_0|} \frac{1}{R}\,d\sigma(Y)
\leq
\frac{CR}{|X-X_0|}\]
and so $\int_{\partial\Omega}\J^t f(X)\,d\sigma(X)=\int_{\partial\Omega}\tau\cdot\nabla\S^T f(X)\,d\sigma(X)=0$.

Now, let $\eta\in C^\infty_0(B(X_0,2r))$ with $\eta\equiv 1$ on $B(X_0,r)$ and $|\nabla\eta|<C/r$. Then
\begin{align*}
\int_{\partial\Omega} \eta(X)\K^t_{\pm} f(X)\,d\sigma(X) 
&= \mp\int_{\partial\Omega_\mp}\eta(X)\nu\cdot A^T\nabla \S^T f(X)\,d\sigma(X)
\\&= \mp\int_{\Omega_\mp}\nabla \eta(X)\cdot A^T\nabla \S^T f(X)\,dX
\\&= \mp\int_{\Omega_\mp\cap B(X_0,2r)\setminus B(X_0,r)}\nabla \eta(X)\cdot A^T\nabla \S^T f(X)\,dX
\end{align*}
So
\begin{align*}
\left|\int_{\partial\Omega} \eta(X)\K^t_{\pm} f(X)\,d\sigma(X) \right|
&\leq \int_{B(X_0,2r)\setminus B(X_0,r)} \frac{CR^\alpha}{r|X|^{1+\alpha}}\,dX
\leq \frac{CR^\alpha}{r^{\alpha}}
\end{align*}
since $|\nabla \S^T f(X)|\leq \frac{CR^\alpha}{|X|^{1+\alpha}}$ for $|X|>2R$. Since this goes to $0$ as $r\to\infty$, we know $\int_{\partial\Omega}\K^t_\pm f\,d\sigma=0$.

So, applying \lemref{inH1}, we are done.
\end{proof}

\section{Nearness to the real case}

\begin{thm} \label{thm:nearT} Let $\Gamma^0$ be the fundamental solution for $A_0$, and let $K^0,$ $\T^0=\T^0_\V$ be defined as in (\ref{dfn:kernelgen}) and~(\ref{dfn:Tgen}), but with $\Gamma,A$ replaed with $\Gamma^0,A_0$.

Then if $\|A-A_0\|_{L^\infty}$ is small enough, then 
$\|\T -T^0 \|_{L^p\mapsto L^p}\leq C_p\|A-A_0\|_{L^\infty},$ $\|\T^t -(\T^0)^t\|_{H^1\mapsto L^1}\leq C\|A-A_0\|_{L^\infty},$ where the constants $C_p,C$ depend only on $p$ and the ellipticity constants of $A,A_0$.
\end{thm}

\begin{proof}
If $F\in L^p(\partial\V)$, then by \thmref{NDf}, we have that so is 
$N(\T F)$.

Let $A_z(x)=(1-z)A_0+z A_1(x)=A_0(x)+z(A_1(x)-A_0(x))$, where $\|A_0-A_1\|_{L^\infty}=\frac{1}{2}\epsilon_0$, where $\epsilon_0$ is the constant in \thmref{Tbounded}, and $A=A_\epsilon$ for some real number $\epsilon\in(0,1)$. It is easy to see that $A_z(x)$ is uniformly elliptic in both $z$ and $x$ for $x\in\R$, $|z|<2$. Note also that $A_z$ is analytic in $z$.

We may define $\Gamma^z,K^z,\T^z$ by analogy with $\Gamma^0,K^0,\T^0$.

In \autoref{sec:analytic}, we showed that $\Gamma^z_Y(X)$ is analytic in $z$ provided $A_z$ is. Therefore, so is $K^z$, and therefore, if $F\in L^2$, so is
$\T^z(X)=\int_{\partial\V} K^z(X,Y) F(Y)\,d\sigma(Y).$


%

So
\[(\T^z F(X)-\T^w F(X))
=(z-w)\frac{1}{2\pi i}\oint_\boundary \T^\zeta F(X)\left(\frac{1}{(\zeta-z)(\zeta-w)}\right)\,d\zeta
\]
and so
\begin{multline*}
{\left\|\sup_{X\in\gamma(Z)}|\T^z F(X)\vphantom{\sup_X}-\T^w F(X)|\right\|_{L^p_Z(\partial\V)}}
\\\begin{aligned}
&=\left\|\sup_{X\in\gamma(Z)}
\left|(z-w)\frac{1}{2\pi i}\oint_\boundary \T^\zeta F(X)\left(\frac{1}{(\zeta-z)(\zeta-w)}\right)\,d\zeta\right|
\right\|_{L^p_Z(\partial\V)}
\\&\leq|z-w|
\left\|
\frac{C}{2\pi}\oint_\boundary \sup_{X\in\gamma(Z)}|\T^\zeta F(X)|\,d\zeta
\right\|_{L^p_Z(\partial\V)}
\\&\leq|z-w|
\frac{C}{2\pi}\oint_\boundary \left\|\sup_{X\in\gamma(Z)}|\T^\zeta F(X)|
\right\|_{L^p_Z(\partial\V)}\,d\zeta
\\&\leq|z-w| C_p \|F\|_{L^p}.
\end{aligned}
\end{multline*}

Therefore, $\|\T^z-\T^w\|_{L^p\mapsto L^p}\leq C|z-w|$. But $A=A_{2\|A-A_0\|_{L^\infty}/\epsilon_0}$, so
\[\|\T-\T^0\|_{L^p\mapsto L^p}\leq C\left|\frac{2\|A-A_0\|_{L^\infty}}{\epsilon_0}\right|.\]

Similarly, $\int (T_h-T_h^0) F(x-x_0) (1+(x-x_0)/R)^\alpha\,dx\leq C$ if $F$ is an atom supported on $(x_0-R,x_0+R)$, and so as in the proofs of \thmref{patching} and \thmref{H1patching}, we have that
\begin{multline*}\|\K-\K_0\|_{L^p\mapsto L^p},\ 
\|\J-\J_0\|_{L^p\mapsto L^p},\ 
\|\K^t-\K_0^t\|_{H^1\mapsto H^1},\
\|\J^t-\J_0^t\|_{H^1\mapsto H^1}
\\\leq C \|A-A_0\|_{L^\infty}
\end{multline*}
provided $\|A-A_0\|_{L^\infty}$ is small enough.
\end{proof}

\chapter{Invertibility of layer potentials}
\label{chap:invertibility}
\newcommand\p{p} 

Recall our goal: to prove that if $p$ is small enough, then for any $g\in L^p\cap H^1(\partial\V)$, there is some $u$ such that
\begin{equation}\label{eqn:goal}
\div A\nabla u=0,\quad
\|N (\nabla u)\|\leq C_p \|g\|,
\quad\text{and}\quad
\nu\cdot A\nabla u=g\text{ or }\tau\cdot\nabla u=g\text{ on }\partial\V
\end{equation}
with norms taken in either $L^p$ or $H^1$.

By \thmref{NDf}, \corref{NSf}, \thmref{patching} and \thmref{H1patching}, if $1<p<\infty$, then
\[\|N(\D f)\|_{L^p}\leq C_p \|f\|_{L^p},\quad \|N(\nabla \S^T f)\|_{L^p} \leq C_p\|f\|_{L^p},\quad \|N(\nabla \S^T f)\|_{L^1} \leq C_p\|f\|_{H^1}.\]

Recall that $\nu\cdot A^T\nabla \S^T f  = -\K^t_- f$, $\tau\cdot\nabla \S^T f =\J^t f$. So if we could find a $f\in L^p\cap H^1$ such that $\|f\|\leq C\|g\|$ and $g=-\K^t_- f$ or $g=\J^t f$, then we would have a $u=\S f$ which satisfied~(\ref{eqn:goal}).

Thus, if $\K^t_\pm$, $\J^t$ are invertible on $L^p$, then solutions to $(N)^A_p$, $(R)^A_p$ in $\V$ and $\bar\V^C$ exist.

By duality this will imply that $\K_\pm$ is invertible on $L^p$ for sufficiently {large} $p$, which will imply that solutions to $(D)^A_p$ exist.

This chapter will be devoted to the proof of the following theorem:
\begin{thm} \label{thm:invertibility}
Let $1<p<\infty$, and let $\epsilon_0>0$. Let $\V$ be a Lipschitz domain with Lipschitz constants $k_i$. Let $A_0$ be a real elliptic matrix, and let $A_s=sI + (1-s)A_0$. Assume that
\begin{myenumerate}


\item \label{invertcondI} $(\K^A_\pm)^t$ and $(\J^A)^t$ are bounded linear operators on $H^1(\partial\V)$ and $L^{p}(\partial\V)$, for all matrix-valued functions $A:\R\mapsto\C^{2\times 2}$ such that either $A$ is a complex-valued matrix with $\|A-A_0\|_{L^\infty}\leq\epsilon_0$, or $A=A_s$, $0\leq s\leq 1$,
\item \label{invertcondII} The layer potential $(\K^I_+)^t:L^p\cap H^1(\partial\V)\mapsto L^p\cap H^1(\partial\V)$ is onto,
\item \label{invertcondIII}  $\smash{(N)^{A_s}_{p}}$, $\smash{(R)^{A_s}_{p}}$ hold in $\V=\V_+$ and $\bar\V^C=\V_-$ with constants at most $C_1$, for all $0\leq s\leq 1$.
\end{myenumerate}

Then there is some $\epsilon>0$ such that, if $\|A-A_0\|_{L^\infty}<\epsilon$, then $\K^t_\pm,\J^t$ are invertible on $L^{p}_0\cap H^1(\partial\V)$, where $L^p_0=H^1\cap L^p$ and $\|f\|_{L^p_0}=\|f\|_{L^p}$. Furthermore,
\[1/\epsilon,\ \|(\K^t_\pm)^{-1}\|,\ \|(\J^t)^{-1}\|\leq C(\lambda,\Lambda,k_i,C_1,p,\epsilon_0).\] 

If in addition
\begin{myenumerate}
\item \label{invertcondIV} There exists a number $C_1>0$ such that, if $a$ is a $H^1(\partial\V)$ atom with support in some $B(X_0,R)\cap\partial\V$ and $H^1$ norm 1, then
\begin{equation}\label{eqn:NinL1}\int_{\partial\V} N(\nabla u)(X)(1+|X|/R)^\alpha\,d\sigma(X)\leq C_1\end{equation}
provided $\div A_0\nabla u=0$ in $\V$, and $\nu\cdot A_0\nabla u=a$ or $\partial_\tau u=a$ on $\partial\V$
\end{myenumerate}
then $\K^t_\pm$, $\J^t$ are invertible with bounded inverse on $H^1(\partial\V)$ as well.
\end{thm}

The $\epsilon$ produced by this theorem depends on~$p$. Therefore, we intend to use this theorem only for some fixed $p_0>1$ and $H^1$, and interpolate to get an $\epsilon$ that will work for all $p$ with $1< p\leq p_0$. 

We use the notational shorthand that $\K^s=\K^{A_s}$, $\J^s=\J^{A_s}$ for any $0\leq s\leq 1$.

\begin{proof} By \lemref{Kpm}, $\K_{+}^A-\K_{-}^A=I$ on $L^\p$ and on $H^1$ for any matrix $A$; so in $L^\p$ or $H^1$,
\[\|f\|=\|(\K^0_+)^t f -(\K^0_-)^t f\|
\leq
\|(\K^0_+)^t f\|
+\|(\K^0_-)^t f\|.\]

By \lemref{compareLp} and \lemref{compareH1}, if (\ref{invertcondIII}) or (\ref{invertcondIV}) holds then
$\|(\K^0_\V)^t f\|\approx
\|(\J^0_\V)^t  f\|$ where norms are in $L^p$ or $H^1$, with constants that may depend on~$p$. By \lemref{Sfcts}, $\J_{\V_+}^t =\J_{\V_-}^t $ on $L^\p$ and $H^1$.

In either case,
\begin{equation}
\label{eqn:f<Kf}
\|f\|\leq \|(\K^0_+)^t f\|+\|(\K^0_-)^t f\|
\leq C(\p)\|\J_{0}^t f\|\leq  C\|(\K^0_\pm)^t f\|.
\end{equation}

This implies the following:
\begin{itemize}
\item $(\J^0)^t$ and $(\K^0_\pm)^t $ are one-to-one.
\item If $(\J^0)^t$ and $(\K^0_\pm)^t $ are onto, then their inverses have norms at most $C(\p)$.
\end{itemize}

Recall \thmref{nearT}: if $A_0$ is a real uniformly elliptic matrix and the layer potential $\T^A_\V$ is bounded uniformly for all $A$ near $A_0$, then 
\[\|\T^A-\T^{A_0}\|_{L^p\mapsto L^p}\leq C_p \|A-A_0\|_{L^\infty}, \quad \|(\T^A)^t-(\T^{A_0})^t\|_{H^1\mapsto L^1}\leq C \|A-A_0\|_{L^\infty}\]
for all $A$ sufficiently near $A_0$. Since the components of $\T^A(B_1 f)$ are $\K^A f$ and $\J^A f$, these inequalities apply to $\K^A$, $\J^A$ as well.

In \thmref{nearonto}, we show that if $G$ is onto and satisfies $\|f\|\leq C\|G f\|$ and $\|G-G'\|$ is small, then $G'$ is also onto and satisfies $\|f\|\leq C\|G' f\|$. So to show that $(\K^A)^t$ or $(\J^A)^t$ is invertible for all $A$ near $A_0$, we need only show that $(\K^0)^t$ or $(\J^0)^t$ is onto.

Consider $\K^0$ first. Let $A_s=(1-s)I+sA_0$. Then $\|f\|\leq C\|(\K^s)^t f\|$ uniformly in~$s$; applying \thmref{nearonto} to $A_s$ and $A_{s-\eta}$ for sufficiently small $\eta$ and repeating, we see that since $(\K^1)^t=(\K^I)^t$ is onto, $(\K^0)^t$ is as well.

Now consider $\J^0$. Let $f\in L^\p\cap H^1$; this set is dense in $L^\p$ and~$H^1$.
Since $(R)^{A_0}_\p$ holds, there is some $u$ with $\div A_0^T\nabla u=0$ in $\V$, $u=\int f$ on $\partial\V$ and $\|N(\nabla u)\|_{L^\p}\leq C\|f\|_{L^\p}$.

Then there is some $g\in L^\p$ with $g=\nu\cdot A_0^T\nabla u$. Let $h=((\K^0_+)^t)^{-1} g\in L^\p$.

Then by uniqueness $\S_0^T h=u$, and so $(\J^0)^t h=f$, as desired. Thus $(\J^0)^t$ as well as $(\K^0)^t$ is invertible on $L^\p$ and on~$H^1$.
\end{proof}

\section{Domains to which \thmref{invertibility} applies}

\thmref{invertibility} has three conditions.

By \thmref{Tbounded}, if $\V$ is a good Lipschitz domain, then~(\ref{invertcondI}) holds.

By \thmref{h1}, which we will prove in the next chapter, if $\V$ is a good Lipschitz domain and there exists a $p>1$ such that (\ref{invertcondIII}) holds in all domains with Lipschitz constants at most $C=C(\V)$, then (\ref{invertcondIV}) holds. (We will need the results of this chapter to prove it.) 

Conditions (\ref{invertcondII}) and (\ref{invertcondIII}) are more complicated. I claim that, for any $K$, we can find a (possibly small) $p=p(K)>1$ such that, if $\V$ is a good Lipschitz domain with constants at most $K$, then (\ref{invertcondIII}) holds in~$\V$. 

By \cite{Rule} and \cite{rule2}, if $\V$ is a special or bounded Lipschitz domain, then $(N)_{\p}^{A_0}$,~$(R)^{A_0}_{\p}$ hold in~$\V$ for some $\p>1$. $\p$ and the constants in the definition of $(N)^{A_0}_{\p}$, $(R)^{A_0}_{\p}$ depend only on $\lambda$, $\Lambda$ and the Lipschitz constants of~$\V$. Since the complement of a special Lipschitz domain is also a special Lipschitz domain, we need only show that $(N)^{A_0}_\p$ or $(R)^{A_0}_\p$ holds in the complement of a bounded Lipschitz domain to have (\ref{invertcondIII}) hold for all good Lipschitz domains. Before doing this, we will deal with Condition~(\ref{invertcondII}).

From \cite[Theorem~4.2 and Corollary~4.4]{verchota}, Condition~(\ref{invertcondII}) holds if $\V$ is a bounded, simply connected Lipschitz domain. It also holds if $\V=\Omega$ is a special Lipschitz domain:

\begin{lem} If $\Omega$ is a special Lipschitz domain, then $(\K^I)^t$ and $(\J^I)^t$ are surjective on $H^1(\partial\Omega)$ and $L^p(\partial\Omega)$ for $p$ small enough.
\end{lem}

\begin{proof}
$\Gamma^I_X(Y)=\frac{1}{2\pi}\log |X-Y|$. If $B=\Matrix{1&-\phi'(x)\\-\phi'(x)&1+\phi'(x)^2}$ for some Lipschitz function $\phi$, then it is easy to check that
\[\Gamma^B_{(x,t)}(y,s)=\Gamma^I_{(x,t-\phi(x))}(y,s-\phi(y)).\]

So if $\Omega=\{(x,t):t>\phi(x)\}$, then
\begin{multline*}
\K^{B}_\Omega f(x,\phi(x))
=\lim_{t\to 0^+}
\int_{\partial\Omega} \nu(Y)\cdot B\nabla\Gamma_{(x,t+\phi(x))}^B(Y) f(Y)\, d\sigma(Y)
\\\begin{aligned}
&=\lim_{t\to 0^+}
\int_{\R} \Matrix{\phi'\\-1}\cdot \Matrix{1&\phi'(y)\\\phi'(y)&1+\phi'(y)^2}\nabla\Gamma_{(x,t+\phi(x))}^B(y,\phi(y)) f(y)\,dy
\\&=\lim_{t\to 0^+}
\int_{\R} \Matrix{\phi'\\-1}\cdot \Matrix{1&\phi'(y)\\\phi'(y)&1+\phi'(y)^2}
\Matrix{1&-\phi'(y)\\0&1}\nabla\Gamma_{(x,t)}^I(y,0)f(y)\,dy
\\&=\lim_{t\to 0^+}
\int_{\R} \Matrix{0\\-1}\cdot \nabla\Gamma_{(x,t)}^I(y,0)f(y)\,dy
\\&=\lim_{t\to 0^+}
\int_{\R} \frac{t}{2\pi(t^2+(x-y)^2)}f(y)\,dy=\frac{1}{2}f(x)
\end{aligned}
\end{multline*}
for any good function $f$; thus, if $\Omega$ is a special Lipschitz domain with $\e=(0,1)$, then for some real matrix $B$, $\K^B_\Omega=\frac{1}{2}$, and so $(\K_\Omega^B)^t=\frac{1}{2}$ is also invertible. 

Since (\ref{invertcondI}) and (\ref{invertcondIII}) hold for special Lipschitz domains, we may proceed as in the proof of \thmref{invertibility} to see that $(\K^I)^t$ and $(\J^I)^t$ are invertible on $L^p(\partial\Omega)\cap H^1(\partial\Omega)$ for all special Lipschitz domains $\Omega$ with $\e=(0,1)$; by symmetry, this must hold for any special Lipschitz domain~$\Omega$, as desired.
\end{proof}

We now complete Condition (\ref{invertcondIII}) for bounded simply connected domains.
This is fairly simple:

\begin{thm}\label{thm:complement} 
Let $\V$ be a good Lipschitz domain with compact boundary. There is some $C_2$ depending only on $\lambda$, $\Lambda$ and the Lipschitz constants of $\V$ such that, if $(N)^A_p$ holds in all bounded simply connected Lipschitz domains $\U$ with Lipschitz constants at most $C_2$, then 
$(N)^A_p$, $(R)^{\tilde A}_p$ hold in $\V$ with constants depending on the same quantities.
\end{thm}

Note that in this lemma, we do not assume that $\partial\V$ is connected.

\begin{proof} Uniqueness (with conditions) follows from \thmref{NRunique}.

First, we show that $(N)^A_p$ holds in~$\V$.
Given a $g\in C^\infty(\partial\V)$ with $\int_\omega g\,d\sigma=0$ for every connected component $\omega$ of $\partial\V$, we can construct a $u$ in $\V$ such that $\nu\cdot A\nabla u=g$ on $\partial\V$ and $\div A\nabla u=0$ in $\V$ in the weak sense, as is done in \cite[Lemma~1.2]{Rule}. Then $\nabla u\in L^2$ with $\|\nabla u\|_{L^2}\leq C\|g\|_{H^1}.$

By (\ref{eqn:gradu}), we have that if $1<p<\infty$, then
\begin{align*}
|\nabla u(Y)|
&\leq \frac{C}{\dist(Y,\partial\V)}\|\nabla u\|_{L^2}
\leq \frac{C}{\dist(Y,\partial\V)}\|g\|_{H^1(\partial\V)}
\leq \smash{\frac{C_p\|g\|_{L^p(\partial\V)}
\sigma(\partial\V)^{1/p'}}{\dist(Y,\partial\V)}}
.\end{align*}

Now, let $X\in\partial\V$ and let $Z\in\gamma(X)$, so that $|X-Z|\leq(1+a)\dist(Z,\partial\V)$.
If $|X-Z|>\sigma(\partial\V)/C$, then $\dist(Z,\partial\V)\geq \sigma(\partial\V)/C(1+a)$, so 
\[|\nabla u(Z)|\leq {{C_p\|g\|_{L^p(\partial\V)}
\sigma(\partial\V)^{-1/p}}}\]
which has $L^p(\partial\V)$ norm at most $C_p\|g\|_{L^p}$. Let $\tilde N(\nabla u)(X)=\sup\{|\nabla u(Z)|:|X-Z|\leq(1+a)\dist(Z,\partial\V)\leq\sigma(\partial\V)/C\}$ for some $C$ to be chosen later; to bound $\|N(\nabla u)\|_{L^p}$ we need only bound $\|\tilde N(\nabla u)\|_{L^p}$.

As in Definiton~\ref{dfn:domain}, $\partial\V$ may be covered by balls $B_j=B(X_j,r_j)$. Choose one such $j$ and take $X_j=0$. Define
\begin{align*}
\Q(r)&=\psi((-r,r)\times(0,(1+k_1)r))
\\&=\{X\in\R^2:|X\cdot\e^\perp|<r, \phi(X\cdot\e^\perp) <X\cdot\e <\phi(X\cdot\e^\perp)+(1+k_1)r \}.
\end{align*}
So for $\frac32 r_j<r<2r_j$, we have that $\Q_c(r)\subset \V$ and $B(X,r_j/2)\cap\V\subset\Q_c(r)$. Furthermore, we have that $\sigma(\partial\V)\leq Cr_j$.

Then
\[\int_{B_j\cap\partial\V} \tilde N(\nabla u)^p\,d\sigma
\leq \int_{\partial\Q(r)}  N_{\Q(r)}(\nabla u)^p\,d\sigma
\]
for $\frac32 r_j<r<2r_j$. But $\Q(r)$ is a simply connected bounded Lipschitz domain; let $p$ be small enough that $(N)^A_p$ holds in all the $\Q(r)$s, and assume $p\leq 2$. So
\begin{align*}
\int_{B_j\cap\partial\V} \tilde N(\nabla u)^p\,d\sigma
&
\leq \frac{C}{r_j}\int_{3r_j/2}^{2r_j}\int_{\partial\Q(r)}  N_{\Q(r)}(\nabla u)^p\,d\sigma\,dr
\\&
\leq \frac{C}{r_j}\int_{3r_j/2}^{2r_j}\int_{\partial\Q(r)}  |\nu\cdot A\nabla u|^p\,d\sigma\,dr
\\&
\leq 
C \|g\|_{L^p}^p+
\frac{C}{r_j}\int_{\Q(2r_j)}  |\nabla u|^p
\leq 
C \|g\|_{L^p}^p+
\frac{Cr_j}{r_j^p}\left(\int_{\Q(2r_j)}  |\nabla u|^2\right)^{p/2}
\\&\leq 
C \|g\|_{L^p}^p+
\frac{C_p}{r_j^{p-1}}\|g\|_{L^p}^{p}\sigma(\partial\V)^{p-1}
\leq C \|g\|_{L^p}^p.
\end{align*}
Since there are at most $k_2$ such balls, we have that $\int_{\partial\V} \tilde N(\nabla u)^p\,d\sigma\leq C\|g\|_{L^p}^p$, as desired. So $(N)^A_p$ holds in $\V$ for $\V$ any good Lipschitz domain, for $p=p(k_1)$ small enough.

We now pass to the regularity problem $(R)^{\tilde A}_p$. Pick some $g\in L^p(\partial\V)$ with $\int_\omega g=0$ for every connected component $\omega$ of $\partial\V$; if $p>1$ and $\sigma(\omega)<\infty$ this condition is equivalent to requiring $g\in H^1$.

If $p>1$ is small enough, then there is some $u$ such that $\div A\nabla u=0$ in $\V$ and $\nu\cdot A\nabla u=g$ on $\partial\V$.

I claim that $\tilde u$ is defined on $\V$. (\autoref{sec:conjugate} will only guarantee this if $\V$ is simply connected.) We need only show that $\int_\chi \nu\cdot A\nabla u=0$ for all Jordan curves $\chi\subset\V$; we may assume $\chi=\partial\U$ for some simply connected domain $\U$.

But 
\[\int_{\partial\U\cup(\U\cap\partial\V)} \nu\cdot A\nabla u = \int_{\partial(\U\cap\V)} \nu\cdot A\nabla u = \int_{\U\cap \V} \nabla 1\cdot A\nabla u=0\]
by the weak definition of $\nu\cdot A\nabla u$. But $\U\cap\partial\V$ is the union of one or more entire components of $\partial\V$; therefore,
\[\int_{\partial\U} \nu\cdot A\nabla u = -\int_{\U\cap\partial\V} \nu\cdot A\nabla u = -\int_{\U\cap\partial\V} g =0.\]
So $\tilde u$ is well-defined on $\V$. By \autoref{sec:conjugate}, $\div \tilde A\nabla\tilde u=0$ in $\V$.

Recall that $\nabla\tilde u=\Matrix{0&-1\\1&0}A\nabla u$.
Thus, if $\tau=\Matrix{0&-1\\1&0}\nu$, we have that
\begin{align*}
\tau\cdot\nabla\tilde u
&=
\Matrix{0&-1\\1&0}\nu\cdot \Matrix{0&-1\\1&0}A\nabla u
=
\nu\cdot \Matrix{0&1\\-1&0}\Matrix{0&-1\\1&0}A\nabla u
=
\nu\cdot A\nabla u
\end{align*}
and so if $\nu\cdot A\nabla u=g$ on $\partial\V$, then $\partial_\tau\tilde u=g$ on $\partial\V$. Clearly $N(\nabla u)(X)\approx N(\nabla\tilde u)(X)$. So if $(N)^A_p$ holds in $\V$, then there exist solutions to $(R)^{\tilde A}_p$.

So if $(N)^A_p$ holds in $\V$, so does $(R)^{\tilde A}_p$.
\end{proof}

\section{Comparing layer potentials on the two sides of a boundary}

\begin{lem}\label{lem:Kpm} If $f$ is a $L^p$ Lipschitz function, then $\K_+ f(X)-\K_-f(X)=f(X)$.
\end{lem}

\begin{proof}
Let $\V_{\pm}^\rho=\V_{\pm}\backslash B(X,\rho)$, $\Psi_\pm^\rho=\V_\pm\cap B(X,\rho)$.
Recall that if $\U$ is a bounded domain, and $X\in \U$, then
\[\int_{\partial \U}\nu\cdot A^T\nabla\Gamma_X^T \,d\sigma=1.\]

So, fixing some $\rho>0$ small, letting $\e$ be a vector such that $X\pm t\e$ is in a nontangential cone in $\V_\pm$ for all sufficiently small positive $t$, and extending $f$ in some reasonable fashion to $\R^2$, we have
\begin{align*}
\K_+ f(X)-\K_-f(X) &=
\lim_{t\to 0^+}
\int_{\partial\V} \nu\cdot A^T\nabla\Gamma^T_{X+t\e}f\,d\sigma
-\lim_{t\to 0^+}\int_{\partial\V} \nu\cdot A^T\nabla\Gamma^T_{X-t\e}f\,d\sigma
\\&=
\lim_{t\to 0^+}
 \int_{\partial\V^\rho_+} \nu\cdot A^T\nabla\Gamma^T_{X+t\e}f\,d\sigma
+\lim_{t\to 0^+}\int_{\partial\V^\rho_-} \nu\cdot A^T\nabla\Gamma^T_{X-t\e}f\,d\sigma
\\&\quad
+\lim_{t\to 0^+}\int_{\partial\Psi^\rho_+} \nu\cdot A^T\nabla\Gamma^T_{X+t\e}f\,d\sigma
+\lim_{t\to 0^+}\int_{\partial\Psi^\rho_-} \nu\cdot A^T\nabla\Gamma^T_{X-t\e}f\,d\sigma
\end{align*}
and so
\begin{align*}
\K_+ f(x)-\K_-f(x) &=
-\int_{\partial B(X,\rho)} \nu\cdot A^T\nabla\Gamma^T_{X}f\,d\sigma
\\&\phantom{=f(X)\,}
+\lim_{t\to 0^+}\int_{\partial\Psi^\rho_+} \nu\cdot A^T\nabla\Gamma^T_{X+t\e}f\,d\sigma
+\int_{\partial\Psi^\rho_-} \nu\cdot A^T\nabla\Gamma^T_{X-t\e}f\,d\sigma
\\&=
f(X)
-
\int_{\partial B(X,\rho)} \nu(Y)\cdot A^T(Y)\nabla\Gamma^T_{X}(Y)(f(Y)-f(X))\,d\sigma
\\&\phantom{=f(X)\,}
+\lim_{t\to 0^+}
\int_{\partial\Psi^\rho_+} \nu(Y)\cdot A^T(Y)\nabla\Gamma^T_{X+t\e}(Y)(f(Y)-f(X))\,d\sigma
\\&\phantom{=f(X)\,}
+\lim_{t\to 0^+}\int_{\partial\Psi^\rho_-} \nu(Y)\cdot A^T(Y)\nabla\Gamma^T_{X-t\e}(Y)(f(Y)-f(X))\,d\sigma
\end{align*}
and the integrals are each at most $C\rho\|f'\|_{L^\infty}$. Taking the limit as $\rho\to 0$ yields the desired result.
\end{proof}

Since $\K_\pm$ are bounded as operators on $L^p$ and $H^1$, and Lipschitz functions are dense in those spaces, we have that $\K_+-\K_-$ is the identity on those spaces as well.

\begin{lem}\label{lem:Sfcts} Suppose that $N(\nabla u)\in L^p(\partial\V)$ for $1<p\leq\infty$. Then $u$ is H\"older continuous on $\overline\V$. If in particular $u=\S f$ for some $f\in L^p$, then $u$ is H\"older continuous on all of $\R^2$.\end{lem}

This implies that $\J^t_+ f=\tau\cdot\nabla\S^T f=\J^t_- f$ for all $f\in L^p$; since $L^p$ functions are dense in $H^1$, this must hold in $H^1$ as well.

\begin{proof}
If $N(\nabla u)\in L^p(\partial\V)$, then
\begin{align*}
\min(\sigma(\partial\V),2a\dist(Y,\partial\V))|\nabla u(Y)|^p
&\leq \int_{|Y-X|<(1+a)\dist(Y,\partial\V)}N(\nabla u)^p
\leq\|N(\nabla u)\|_{L^p}^p
\end{align*}
and so
\begin{equation}
\label{eqn:nablau}
|\nabla u(Y)|\leq \frac{C\|N(\nabla u)\|_{L^p}} {\min(\sigma(\partial\V),\dist(Y,\partial\V))^{1/p}}.
\end{equation}

If $\Omega$ is a {special} Lipschitz domain, and $N(\nabla u)\in L^p(\partial\Omega)$, and $0\leq\tau<t$ or $0\geq\tau>t$, then
\begin{align*}
|u(\psi(x,t))-u(\psi(x,\tau))|&\leq
\int_\tau^t|\nabla u(\psi(x,s))|\,ds
\leq
\int_\tau^t C s^{-1/p}\|N(\nabla u)\|_{L^p}\,ds
\\&\leq
C(p) \|N(\nabla u)\|_{L^p}(t^{1/p'}-\tau^{1/p'})
\\&\leq
C(p) \|N(\nabla u)\|_{L^p}|t-\tau|^{1/p'}
\end{align*}
and if $t\neq 0$, then
\begin{align*}
|u(\psi(x,t))-u(\psi(y,t))|
&\leq \left|\int_x^y N(\nabla u)(\psi(z))\sqrt{1+\phi'(z)^2} \,dz\right|
\\&\leq C(p)|x-y|^{1/p'}\|N(\nabla u)\|_{L^p}
\end{align*}
So if $X,Y\in\overline{\Omega}$, then
\begin{equation}
\label{eqn:uholder}
|u(X)-u(Y)|\leq C(p)|X-Y|^{1/p'}\|N(\nabla u)\|_{L^p}.
\end{equation}
Thus, $u$ is H\"older continuous on $\overline{\Omega}$. 

This result also holds locally near the boundary of an arbitrary Lipschitz domain; so $u$ is H\"older continuous on $\overline{\V}$ near the boundary. So, in particular, $u|_{\partial\V}$ exists pointwise (and not just in $L^p$).

By \thmref{NDf}, if $f\in L^p$, then $N(\nabla \S f)\in L^p$ and so $\S f$ is continuous on each of $\overline{\V_+}=\overline\V$ and $\overline{\V_-}=\V^C$; I would like to show that $\S f$ is continuous on $\overline{\V_+}\cup\overline{\V_-}=\R^2$.

Pick some $X\in \partial\V$, $t>0$ small, $\e$ as in the proof of \lemref{Kpm}. Then
\begin{align*}
|\S f(X+t\e)-\S f(X-t\e)|&=
\left|\int_{\partial\V}\left(\Gamma_{Y}(X+t\e)-\Gamma_{Y}(X-t\e)\right)
f(Y)\,d\sigma(Y)\right|
\end{align*}
But
\begin{align*}
|\Gamma_{Y}(X+t\e)-\Gamma_Y(X-t\e)|
&\leq\left|\int_{-t}^t\nabla\Gamma_Y(X+r\e) dr\right|
\leq \left|\int_{0}^{Ct} \frac{C}{\sqrt{|X-Y|^2+r^2}}\,dr\right|
\end{align*}
But that integral is at most $\frac{Ct}{|X-Y|}$, and if $|X-Y|<t$, then
\[\int_{0}^{Ct} \frac{C}{\sqrt{|X-Y|^2+r^2}}\,dr
=\int_{0}^{Ct/|X-Y|} \frac{C}{\sqrt{1+r^2}}\,dr
\leq C\ln \frac{Ct}{|X-Y|}.
\]

So
\begin{multline*}
{|\S f(X+t\e)-\S f(X-t\e)|}
\\\begin{aligned}
&=
\left|\int_{\partial\V}\left(\Gamma_{Y}(X+t\e)-\Gamma_{Y}(X-t\e)\right)
f(Y)\,d\sigma(Y)\right|
\\
&\leq \int_{|X-Y|>|t|} |f(Y)|\frac{C|t|}{|X-Y|}\,d\sigma(Y)
+C\int_{|X-Y|<|t|}|f(Y)| \ln \frac{Ct}{|X-Y|}d\sigma(Y)
\\&\leq
C\|f\|_{L^p}\left(
\left\|\frac{t}{|X-Y|}\right\|_{L^{p'}_Y(\{|X-Y|>t\})}
+
\left\|\ln \left(\frac{Ct}{|X-Y|}\right) \right\|_{L^{p'}_Y(\{|X-Y|<t\})}
\right)
\\&\leq C(p)t^{1/p'}\|f\|_{L^p}
\end{aligned}
\end{multline*}
provided $1<p<\infty$.

So  $u$ is continuous across the boundary, and so is continuous on~$\R^2$.
\end{proof}


\section{Comparing norms of layer potentials}\label{sec:compareKL}

\begin{lem}\label{lem:compareLp} If $(N)^{A_0}_\p$ and $(R)^{A_0}_\p$ hold in $\V$ and~$\overline\V^C$, then for all $f\in L^\p(\partial\V)$,
\[\|(\K^0_+)^t f\|_{L^\p}\approx\|(\J^0)^t f\|_{L^\p}\approx \|(\K^0_-)^t f\|_{L^\p}.\]
\end{lem}

\begin{proof} By definition of $(N)_\p^{A_0}$, $(R)_\p^{A_0}$, there is some constant $C$ depending on $k_i$,~$\lambda$,~$\Lambda$, such that if $\div A_0\nabla u=0$ in $\V$ then
\begin{align*}
\| N(\nabla u)\|_{L^\p(\partial\V)}&\leq
C(\p)\|\nu\cdot A_0^T\nabla u\|_{L^\p(\partial\V)},\text{ and }
\\
\| N(\nabla u)\|_{L^\p(\partial\V)}&\leq
C(\p)\|\tau\cdot\nabla u\|_{L^\p(\partial\V)}.
\end{align*}

Since $|\nu\cdot A_0^T\nabla u|\leq \Lambda N(\nabla u)$ and $|\tau\cdot \nabla u|\leq N(\nabla u)$, this means that
\[\|\nu\cdot A_0^T\nabla u\|_{L^\p(\partial\V_\pm)}\approx
\|\tau\cdot\nabla u\|_{L^\p(\partial\V_\pm)}.\]
But by \lemref{Sfcts}, $\S^T_0 f$ is continuous on $\R^2$, so $\tau\cdot\nabla \S^T_0 f$ must be the same on $\partial\V_+$ and $\partial\V_-$, and so
\[\|(\K^0_+)^t f\|_{L^\p}\approx\|(\J^0)^t f\|_{L^\p}\approx \|(\K^0_-)^t f\|_{L^\p}.\]
\end{proof}


\begin{lem}\label{lem:compareH1} If (\ref{eqn:NinL1}) holds, then for all $f\in H^1$,
\[\|(\K^0_+)^t f\|_{H^1}\approx\|(\J^0)^t f\|_{H^1}\approx \|(\K^0_-)^t f\|_{H^1}.\]
\end{lem}

\begin{proof} Let $f\in H^1$.
By \thmref{H1patching}, $(\K^0)^t f\in H^1$, so $(\K^0)^t f = \sum_i \lambda_i a_i$ for $H^1$ atoms $a_i$ and constants $\lambda_i$ with $\sum |\lambda_i|=\|(\K^0)^t f\|_{H^1}$.

Let $u=\S^T f$, so that $(\K^0)^t f = \nu\cdot A^T\nabla u$. Then $(\J^0)^t f=\tau\cdot \nabla u$.

Since $(N)^{A_0}_\p$ holds in $\V$, we can write $u=\sum_i \lambda_i u_i$ where $\nu\cdot A_0^T\nabla u_i=a_i$. Then by assumption,
\[\int_{\partial\V} N(\nabla u_i)(X)(1+|X|/r_i)^\alpha\,d\sigma(X)\leq C,
\quad \|N(\nabla u_i)\|_{L^\p}\leq Cr_i^{1/\p-1}.\]


If $\partial\V$ is bounded, then $\int_{\partial\V}\tau\cdot \nabla u_i=0$. Otherwise, we are working in a special Lipschitz domain. Since $N(\nabla u)\in L^1$, we must have that $\lim_{R\to\infty} \int_{|X|<R} \tau\cdot\nabla u\,d\sigma$ exists. But since $N(\nabla u_i)(X) (1+|X|/r_i)^\alpha\in L^1$, for any $\epsilon>0$ and any $R>0$ there must be some $x_1,x_2$ with $x_1<-R,$ $R<x_2$, $|x_1|\approx|x_2|$, and $N(\nabla u)(\psi(x_i))<\epsilon/|x_i|^{1+\alpha}$. 

So by integrating $\tau\cdot\nabla u$ on the boundary of the domain $\psi((x_1,x_2)\times(0,C|x_i|))$, we see that $|u(\psi(x_2))-u(\psi(x_1))|<C |x_i| \epsilon/ |x_i|^{1+\alpha}$; this may be made arbitrarily small by making $\epsilon$ small or $R$ large.

So by \lemref{inH1}, $\|\tau\cdot \nabla u_i\|_{H^1}\leq C$; therefore, 
\[\|\J^t f\|_{H^1}=\left\|\tau\cdot\sum_i \lambda_i u_i\right\|_{H^1}\leq C\sum_i\lambda_i = C\|\K^t f\|_{H^1}.\]

Similarly, we may say that $(\J^0)^t f=\sum_i \lambda_i a_i$; as in the proof of \thmref{H1patching}, $\int \nu\cdot A_0^T\nabla u_i=0$ and so $\|\K^t f\|_{H^1}\leq C\|\J^t f\|_{H^1}$.
\end{proof}

%
%
%

\section{Some elementary analysis}

\begin{thm}\label{thm:nearonto}
Let $G,G'$ be two bounded linear operators from a Banach space $B$ to itself, and suppose that $\|f\|_B\leq C \|Gf\|_B$, $G$ is a bijection, and $\|G-G'\|_{B\mapsto B}\leq\epsilon$. If $\epsilon<1/C$, then $G'$ is also a bijection, and its inverse has norm at most $C/(1-C\epsilon)$.
\end{thm}

\begin{proof}
First, if $f\neq 0$ and $G'f=0$, then
\[\|f\|\leq C\|Gf\|=C\|(G-G')f\|\leq C\epsilon\|f\|\]
and so $\epsilon\geq 1/C$; conversely, if $\|G-G'\|<1/C$, then $G'$ is one-to-one.

Next, if $\epsilon<1/C$, then
\[\|f\|=\frac{C}{1-C\epsilon} \left(\frac{1}{C}-\epsilon\right)\|f\|
\leq \frac{C}{1-C\epsilon}\left( \|G f\|-\|(G-G')f\|\right)
\leq \frac{C}{1-C\epsilon} \|G'f\|\]
and so $G'$ satisfies the same sort of useful inequality as $G$; in particular, if $(G')^{-1}$ exists it has norm at most $C/(1-C\epsilon)$. Furthermore, if $G'g_n\to f$ then $\{g_n\}$ is a Cauchy sequence; since $B$ is a Banach space $g_n\to g$ for some $g$, and so $G'g=f$. Thus, $G'$ has closed range.

Suppose $G'$ is not onto. Let $f_0\in B-G'B$, and let  $\eta=\inf\left\{\|f_0-G'b\|:b\in B\right\}$; since $G'B$ is closed, $\eta$ is positive. Let $b_0$ be such that $\|f_0-G'b_0\|\leq \rho\eta$, $\rho>1$. Let $f_1=f_0-G'b_0$, $f_2=f_1/\|f_1\|$. Then $\|f_1\|\leq \rho\eta$, and if $b\in B$, then \[\|f_1-G'b\|=\|f_0-G'(b+b_0)\|\geq \eta\geq \frac{1}{\rho}\|f_1\|,\] so
\[\|f_2-G'b\|\geq \frac{1}{\rho}\]
for all $b\in B$.

Let $G h=f_2$. Then $\|G'h-Gh\|=\|f_2-G'h\|\geq \frac{1}{\rho}$
and so
\[\|G'-G\|\geq \frac{1}{\rho\|h\|}\geq \frac{1}{\rho C\|G h\|}
=\frac{1}{\rho C\|f_2\|}=\frac{1}{C \rho}.\]

Thus, if $\|G'-G\|<\frac{1}{C}$, then $G'$ is also one-to-one and onto. If I further suppose that $\|G'-G\|<\frac{1}{2C}$, then $\|(G')^{-1}\|\leq 2C$.
\end{proof}


\chapter{Boundary data in \texorpdfstring{$H^1$}{H1} for the real case}
\label{chap:H1}
We have shown that $\J^t$, $\K^t$ are invertible on $H^1$ provided that the following theorem holds:

\begin{thm}\label{thm:h1} Let $\V$ be a good Lipschitz domain. Suppose that $a$ is an atom of $H^1(\partial\Omega)$, that is, $\|a\|_{L^\infty}\leq 1/r,$ $\int a=0$, and $\supp a\subset B(X_0,r)\cap\partial\Omega$ for some $X_0\in\partial\Omega$, $r>0$.

Suppose that $\div A\nabla u=0$ in $\V$, and $A$ satisfies the conditions of \thmref{big}.

Assume that there is some $1<p<\infty$ such that 
$(N)^{A}_{p}$, $(R)^{A}_{p}$ $(D)^A_{p'}$, $(N)^{\tilde A}_{p}$, $(R)^{\tilde A}_{p}$ and $(D)^{\tilde A}_{p'}$ hold in all good Lipschitz domains whose constants are no bigger than some $C=C(\lambda,\Lambda,\V)$ (to be chosen later).

Assume that $u$ satisfies the conditions of \thmref{NRunique}, and that either 
\begin{enumerate}

\item\label{thm:Neumann} $\nu\cdot A\nabla u=a$ on $\partial\V$, or 
\item\label{thm:regularity} $\partial_\tau u=a$ on $\partial\V$.

\end{enumerate}
Then there is some $C,\alpha$ depending only on the ellipticity constants of $A$ such that for some $X_0\in\supp a$
\[\int_{\partial\V} N (\nabla u)(X)(1+|X-X_0|/r)^\alpha \, d\sigma(X)\leq C.\]
\end{thm}

Throughout this section, we assume without loss of generality that $X_0=0$.

\section{A priori bounds on \texorpdfstring{$u$}{u}} \label{sec:H1apriori}

Suppose that $u$ is a regularity solution in $\V$, $u=f$ on $\partial\V$, with $a=\partial_\tau f$ a $H^1$ atom with support in $B(0,r)\cap\partial\V$.

So $\|Nu\|_{L^{\smash{p'}}}\leq C\|f\|_{L^{\smash{p'}}}\leq C r^{1/p'}$. As in the proof of \lemref{Sfcts}, this implies that
$|u(X)|\leq Cr^{1/p'}\min(\dist(X,\partial\V),\sigma(\partial\V))^{-1/p'}$.
By \lemref{PDE1} and (\ref{eqn:gradu}),
\[|\nabla u(X)|\leq Cr^{1/p'}\dist(X,\partial\V)^{-1}\min(\dist(X,\partial\V),\sigma(\partial\V))^{-1/p'}.\]

Now, suppose that $\div A\nabla u=0,$ $\nu\cdot A\nabla f=a$.
Then as in the proof of \thmref{complement}, $\tilde u$ exists on $\V$ and satisfies $\div \tilde A\nabla\tilde u=0$, $\partial_\tau \tilde u=a$; so $|\nabla u(X)|\approx|\nabla \tilde u(X)|$.

So if $\nu\cdot A\nabla u=a$ or $\partial_\tau u=a$, then
\begin{equation}\label{eqn:H1gradu}
|\nabla u(X)|\leq Cr^{1/p'}\dist(X,\partial\V)^{-1}\min(\dist(X,\partial\V),\sigma(\partial\V))^{-1/p'}\end{equation}

If $\V=\Omega$ is special, we may make some statements about $u$ as well as $\nabla u$. First, $\lim_{t\to\infty} u(\psi(x,t))$ exists for every $x\in\R$; furthermore, they must all be equal. Without loss of generality, we may assume that $\lim_{t\to\infty} u(\psi(x,t))=0$. Then 
\[|u(\psi(x,t))|\leq \int_t^\infty Cr^{1/p'}t^{-1-1/p'}\,dt
\leq C_p r^{1/p'} t^{-1/p'}.\]

If $|x|>2r$ or $t>r$, let $R=|\psi(x,t)|/2$; by applying \lemref{PDE3}, we see that 
\begin{align*}
|u(\psi(x,t))|
&\leq C\left(\dashint_{B(\psi(x,t),R)\cap\Omega} |u|^2\right)^{1/2}
\leq C\left(\dashint_{B(\psi(x,t),R)\cap\Omega} C_p r^{2/p'} s^{-2/p'} dy\,ds\right)^{1/2}
\\&\leq C\left(\frac{C}{R}\dashint_{x-R}^{x+R}\int_{0}^{Cr} C_p r^{2/p'} |s|^{-2/p'}\,ds\,dy\right)^{1/2}
\leq
C_pr^{1/p'}R^{-1/p'}.
\end{align*}

Summarizing, we have that in a special Lipschitz domain,
\begin{equation}\label{eqn:H1special}
|\nabla u(X)|\leq C_p r^{1/p'} \dist(X,\partial\Omega)^{-1/p'},
\quad
|u(X)|\leq C_p r^{1/p'} |X|^{-1/p'}.\end{equation}

\section{A bound on our integral in terms of itself}
\label{sec:realh1:bddiffinite}

\begin{lem} Suppose that $\V$, $p$, $A$, $a$, $r$ satisfy the conditions of \thmref{h1}.

Then for any $h>0$, there is some constant $C$ depending only on $p$, $\lambda$, $\Lambda$ and the Lipschitz constants of $\V$ such that if
\[I=\int_{\partial\V} N (\nabla u)(X)(1+|X-X_0|/r)^\alpha \, d\sigma(X)\]
then $I$ satisfies
\[I\leq C h^{1/p} I + \frac{C}{h^{2-1/p}}.\]
\end{lem}

If $\partial\V$ is bounded, we know $N(\nabla u)\in L^p(\partial\V)\supset L^1(\partial\V)$, so $I$ is finite. So by choosing $h$ small enough this lemma will give us \thmref{h1}. If $\V=\Omega$ is a special Lipschitz domain, we will show first that $I$ is finite if $\phi$ is compactly supported, which will yield our desired bound for such $\phi$, and then pass to the case where $\phi$ is an arbitrary Lipschitz function.

\begin{proof} We begin with the case $\V=\Omega$ a special Lipschitz domain.

It suffices to prove this for $r=1$; we may rescale to get the general result.

\allowdisplaybreaks[2]

We will need the following definitions.
\begin{align*}
N_\U f(X)&=\sup\,\{|f(Y)|:|X-Y|<(1+a)\dist(Y,\partial\U)\},
\\
\gamma(X)&=\{Y:|Y-X|<(1+a)\dist(Y,\partial\V)\}\\
\gamma_1(X,R)&=\{Y\in\gamma(X): |Y-X|<R/8\}\\
\gamma_2(X,R)&=\{Y\in\gamma(X): |Y-X|\geq R/8\} \\
m_i(X,R)&=\sup\,\{|\nabla u(Y)|:Y\in\gamma_i(X)\}\displaybreak[0]\\
N_\U f(X)&=\sup\,\{|f(Y)|:|X-Y|<(1+a)\dist(Y,\partial\U)\},
\\
Nf(x)&=N_{\Omega}f(\psi(x))
=\sup_{\gamma(\psi(x))}|f|\\
\Lambda^+(R) &= (R,2R),\quad \Lambda^-(R) = (-2R,-R)\\
\Lambda^\pm_\tau(R) &=\cup_{x\in\Lambda^\pm(R)}(x-\tau R,x+\tau R)\\
\Delta^\pm_\tau(R) &=\psi(\Lambda^\pm_\tau(R))\\
\U^\pm_\tau(R)&=\left\{\psi(x,s):x\in\Lambda^\pm_\tau(R),0<s<R\left(1+\tau\|\phi'\|_{L^\infty}\right)\right\}
\end{align*}

\allowdisplaybreaks[0]

We will usually omit the $R$. Note the following:
\begin{itemize}
\item if $x\in\Lambda^\pm$ then $m_1(\psi(x))\leq N_{\U^\pm_\tau}(\nabla u)(x)$ for $\tau\geq\frac{1}{4}$,
\item if $Y\in\gamma_2(X,R)$, then $\dist(Y,\partial\Omega)\geq \frac{1}{1+a} |Y-X|\geq \frac{R}{8+8a}$.
\end{itemize}

Let $I=
\int_{\partial\Omega}N(\nabla u)(X) (|X|+1)^\alpha \,d\sigma(X)$. Then
\[I\leq C\int_{\psi((-2,2))} N(\nabla u)(X)\,d\sigma(X)+C\sum_{j=1}^\infty 2^{j\alpha} \int_{\Delta^+(2^j)\cup\Delta^-(2^j)} N(\nabla u)(X)\,d\sigma(X).\]

We can bound the first term easily:
\[\int_{\psi((-2,2))}N(\nabla u)(X)\,d\sigma(X)
\leq C \|N(\nabla u)\|_{L^{p}}
\leq C.\]

We just need to bound $\sum_j 2^{j\alpha} \int_{\Delta^\pm(2^j)} N(\nabla u)(X)\,d\sigma(X).$ Pick some $2^j=R$; we seek a bound on $\int_{\Delta^\pm(R)} N(\nabla u).$ We consider only $\Delta=\Delta^+(R)$ and the Neumann problem; the case $\Delta=\Delta^-(R)$ and the regularity problem is similar.

By our a priori bounds on $u$, if $X\in\Delta=\Delta^+(R)$, then $m_2(X)\leq \frac{C}{R^{1+1/p'}}.$
%
If $\tau\geq\frac{1}{4}$, then 
\begin{align*}
\int_{\Delta(R)}m_1&\leq \int_{\partial \U_\tau} N_{\U_\tau}(\nabla u) \,d\sigma
\leq C R \dashint_{\partial \U_\tau} N_{\U_\tau}(\nabla u) \,d\sigma
\\&\leq  C R \left(\dashint_{\partial \U_\tau} N_{\U_\tau}(\nabla u)^{p} \,d\sigma\right)^{1/p}
\leq  C R \left(\dashint_{\partial \U_\tau} 
|\nu\cdot A\nabla u|^{p} \,d\sigma\right)^{1/p}
\\&=  C R^{1-1/p} \left(\int_{\partial \U_\tau} 
|\nu\cdot A\nabla u|^{p} \,d\sigma\right)^{1/p}
\end{align*}
But if we let $\eta_-(t)=\psi(R-\tau R,t)$, $\eta_+(t)=\psi(2R+\tau R,t)$, then by (\ref{eqn:H1gradu})
\begin{align*}
\int_{\partial \U_\tau} |\nu\cdot A\nabla u|^{p} \,d\sigma
&=
\int_{\eta_+(t)\cup\eta_-(t), t<hR} |\nu\cdot A\nabla u|^{p} \,d\sigma
+\int_{\eta_+(t)\cup\eta_-(t), t\geq hR} |\nu\cdot A\nabla u|^{p} \,d\sigma
\\&\leq
hR
N(\nabla u)(\psi(R-\tau R))^{p}
+hR
N(\nabla u)(\psi(2R+\tau R))^{p}
+\frac{CR}{(hR)^{p+p/p'}}.
\end{align*}
 
So, if we take $h<1/2$ and $1/4\leq\tau\leq 1/2$, we have that
\begin{multline*}
{\left(\int_{\partial\U_\tau}|\nu\cdot A\nabla u(y,s)|^{p} \right)^{1/p}}
\\\begin{aligned}
&\leq
h^{1/p}R^{1/p} \bigl(N(\nabla u)(\psi(R-\tau R))+N(\nabla u)(\psi(2R+\tau R))\bigr)
+\frac{C}{h^{1+1/p'}R^{2/p'}}
\end{aligned}\end{multline*}
Taking the integral from $\tau=1/4$ to $\tau=1/2$, we get that
\begin{align*}
\int_{\Delta} m_1
&\leq \int_{1/4}^{1/2}C R^{1-1/p}\left(\int_{\partial\U_\tau} |\nu\cdot A\nabla u|^{p} d\sigma\right)^{1/p}\,d\tau
\\
&\leq \int_{1/4}^{1/2}C R\,h^{1/p}
\bigl(N(\nabla u)(\psi(R-\tau R))+N(\nabla u)(\psi(2R+\tau R)) \bigr)\,d\tau
+\frac{C}{h^{1+1/p'}R^{2/p'}}
\\
&\leq Ch^{1/p}\int_{1/4}^{1/2} \bigl(N(\nabla u)(\psi(R-\tau R))+N(\nabla u)(\psi(2R+\tau R)) \bigr) R \,d\tau
+\frac{C}{h^{1+1/p'}R^{2/p'}}
\\
&\leq Ch^{1/p}\int_{\Delta_{1/2}}N(\nabla u)
+\frac{C}{h^{1+1/p'}R^{2/p'}}.
\end{align*}

Therefore,
\begin{align*}
\int_{\Delta}N(\nabla u)&\leq \int_{\Delta} m_1+\int_{\Delta} m_2
\leq
\frac{C}{R^{1/p'}}
+
Ch^{1/p}\int_{\Delta_{1/2}}N(\nabla u)
+
\frac{C}{h^{1+1/p'}R^{2/p'}}
\\&
\leq
Ch^{1/p}\int_{\Delta_{1/2}}N(\nabla u)
+
\frac{C}{h^{1+1/p'}R^{1/p'}}
\end{align*}

So
\begin{align*}
\int N(\nabla u)(1+|x|^\alpha)
&\leq C +\sum_{j=0}^\infty 2^{j\alpha}\int_{\Delta(2^j)} N(\nabla u)
\\&\leq C 
+\sum_{j=0}^\infty Ch^{1/p}\int_{\Delta_{1/2}(2^j)}2^{j\alpha}N(\nabla u) 
+\sum_{j=0}^\infty 2^{j\alpha}\frac{C}{h^{1+1/p'}2^{j/p'}}
\\&\leq C 
+Ch^{1/p}\int(1+|x|)^\alpha N(\nabla u) 
+\sum_{j=0}^\infty \frac{C}{h^{1+1/p'}} 2^{j(\alpha-1/p')}.
\end{align*}

This completes the argument for $\V=\Omega$ a special Lipschitz domain.
Suppose $\partial\V$ is bounded.

If $\dist(X,\partial\Omega)\geq \sigma(\partial\V)/C$, then  $|\nabla u(X)|\leq C r^{1/p'}\sigma(\partial\V)^{-1-1/p'}.$ So 
\begin{align*}
\int_{\partial\V} m_2(X,\sigma(\partial\V)/C) (1+|X|/r)^\alpha\,d\sigma
&\leq
\int_{\partial\V} C r^{1/p'}\sigma(\partial\V)^{-1-1/p'} (1+|X|/r)^\alpha\,d\sigma
\\&\leq C r^{1/p'-\alpha}\sigma(\partial\V)^{-1-1/p'}
\int_{\partial\V} (1+|X|)^\alpha\,d\sigma
\\&\leq C\left(\frac{r}{\sigma(\partial\V)}\right)^{1/p'-\alpha}
\end{align*}
which is bounded provided $\alpha<1/p'$.

If $r>\sigma(\partial\V)/C$, then $\|a\|_{L^p}\leq C\sigma(\partial\V)^{1/p-1}$. Then
\begin{multline*}
\int_{\boundary} N(\nabla u)(X) (1+|X|/r)^\alpha\,d\sigma(X)
\leq \int_{\boundary} N(\nabla u)(X) \,d\sigma(X)
\\\leq \sigma(\boundary)^{1-1/p}\left(\int_{\boundary} N(\nabla u)^p \,d\sigma\right)^{1/p}
\leq \sigma(\boundary)^{1-1/p} C\sigma(\partial\V)^{1/p-1}\leq C
\end{multline*}
and so we may assume that $\supp a$ is small.

We can cover $\partial\V$ with $k_2$ balls of radius $r_j$, where $\sigma(\partial\V)\leq C r_j$, as in \dfnref{domain}. We have that $\supp a$ is contained in one of the boundary cylinders; we may write $\supp a\subset \psi((-r,r))$ and require that $\psi(-2r,2r)\subset\partial\Omega_j\cap\partial\V$.

The above argument for special Lipschitz domains also shows that, if $\div A\nabla u=0$ in $\psi((-2R_0,2R_0)\times(0,2R_0))$, with $\|N(\nabla u)\|_{L^p}\leq C r^{-1/p'}$ and $\supp \nu\cdot A\nabla u\cap \psi(-2R_0,2R_0)\subset\psi(-r,r)$, then 
\[\int_{\psi(-R_0,R_0)} N(\nabla u)(X) (1+|X|/r)^\alpha\,d\sigma
\leq
C+Ch\int_{\psi(-2R_0,2R_0)} N(\nabla u)(X) (1+|X|/r)^\alpha\,d\sigma.
\]
We may apply this argument to each of the boundary cylinders and sum to get a bound for all of $\partial\V$, as desired.
\end{proof}

\section{Finiteness}

Recall that $\Omega=\{X\in\R^2:\phi(X\cdot\e^\perp)<X\cdot \e\}$. We made the a priori assumption that there is som

Redefine $A$ on $\Omega^C$ so that $A(\psi(x,-t))=A(\psi(x,t))$.
In $\Omega^C$, define
\begin{itemize}
\item $u(\psi(x,t))=u(\psi(x,-t))$, for the Neumann problem
\item $u(x,t)=-u(\psi(x,-t))$, for the regularity problem.
\end{itemize}

Then $\div A\nabla u=0$ in $\Omega$ and in $\{\psi(x,t):|x|>R_0,t\neq 0\}$.

If the regularity problem holds, $u=0$ on $\partial\Omega\cap B(0,r)^C$ and $\partial\Omega^C\cap B(0,r)^C$. If the Neumann problem holds, $\nu\cdot A\nabla u=0$ on $\partial\Omega\cap B(0,r)^C$ and on $\partial\Omega^C\cap B(0,r)^C$. In either case, we have that $\div A\nabla u=0$ in $\{\psi(x,t):|x|>R_0$ or $t>0\}$.

If $|X|>CR_0$, then by \autoref{sec:H1apriori},
\begin{align*}|\nabla u(X)|
&\leq C\left(\dashint_{B(X,|X|/4)}|\nabla u|^{2}\right)^{1/2}
\leq \frac{C}{|X|}\left(\dashint_{B(X,|X|/2)}|u|^{2}\right)^{1/2}
\leq \frac{C}{|X|^{1+1/p'}}.
\end{align*}

Since $N(\nabla u)\in L^{p}$, we have that 
\[\int_{|X|< CR_0,X\in\partial\Omega} N(\nabla u)(X) (1+|X|/r)^\alpha\,d\sigma(X)\]
is finite. But $N(\nabla u)(X)\leq C/|X|^{1+1/p'}$; so if $\alpha<1/p'$, then 
\[\int_{|X|\geq  CR_0,X\in\partial\Omega} N(\nabla u)(X) (1+|X|/r)^\alpha\,d\sigma(X)\]
is finite as well and we are done.

We now remove the assumption on $\phi$.
Recall that $\Omega=\{X\in\R^2:\phi(X\cdot\e^\perp)<X\cdot \e\}$. Assume $\phi(0)=0$, and let $\phi_R=\phi$ on $(-R,R)$ and let $\phi_R=0$ on $(-3R,3R)^C$.
Let $\Omega_R=\{X\in\R^2:\phi_R(X\cdot\e^\perp)<X\cdot \e\}$.

Let $\nu\cdot A\nabla u_R=a$ or $\tau\cdot\nabla u_R=a$ on $\partial\Omega\cap\partial\Omega_R$, $0$ on $\partial\Omega_R\backslash\partial\Omega$.

Suppose that $|Y|$ is small compared to $R,$ $S$ with $R<S$. Let $\delta=\dist(Y,\partial\Omega)$.
Then $\div A\nabla (u_S-u_R)=0$ in $\Omega_S\cap\Omega_R$, and $\nu\cdot A\nabla (u_S-u_R)=0$ or $\tau\cdot A\nabla (u_S-u_R)=0$ on $\psi((-R,R))$.

By \autoref{sec:H1apriori}, if $|X|>2r$ then $|u_R(\psi_R(x,t))|<Cr^{1/p'}t^{-1/p'}$.
So 
\begin{align*}
|u_S(Y)-u_R(Y)|^2&\leq \frac{C}{R^2}\int_{B(Y,R/2)} |u_S-u_R|^2
\leq \frac{C}{R}\int_0^R Cr^{2/p'}t^{-2/p'} dt
\leq C_p \frac{r^{2/p'}}{R^{2/p'}}.
\end{align*}
Therefore, by (\ref{eqn:gradu}), if $R\gg |Y|$ then
\[|\nabla u_S(Y)-\nabla u_R(Y)|\leq \frac{C_p r^{2/p'}} {R^{2/p'}\dist(Y,\partial\Omega)}.\]

Define $\check u=\lim_{R\to \infty} u_R$. Then clearly $\div A\nabla\check u=0$ in $\Omega$, $\nu\cdot A\nabla\check u=a$ or $\tau\cdot\nabla \check u=a$ on $\partial\Omega$, and $\int N(\nabla \check u)(1+|X|/R_0)^\alpha\leq C$. By \thmref{NRunique}, we have $u=\check u$, and so we are done.

%
%
%
%
%
%
%

\chapter{Interpolation}
\label{chap:interpolation}
We have established that for $\V$, $A$ and $A_0$ as in \thmref{big}, there exist some $\epsilon_0>0$, $p_0>1$ depending only on $\lambda$, $\Lambda$ and $k_i$ such that, if $\|A-A_0\|_{L^\infty}<\epsilon_0$, then
\begin{itemize}
\item $\K_\pm,\J,\K_\pm^t,\J^t$, are bounded $L^p(\partial\V)\mapsto L^p(\partial\V)$ for all $1<p<\infty$.
\item $\K_\pm^t,\J_\pm^t$ are invertible on $L^{p_0}(\partial\V)$ with bounded inverse.
\item $\K_\pm^t,\J_\pm^t$ are bounded and invertible on $H^1(\partial\V)$ with bounded inverse.
\end{itemize}

We also know that, for any given $p$, if $\K_\pm^t$ is invertible with bounded inverse on $L^p(\partial\V)$ then $(N)_p^A$ holds in~$\V$, and if $\J^t$ is invertible with bounded inverse on $L^p(\partial\V)$ then $(R)_p^A$ holds in~$\V$. We now interpolate to complete the proof of \thmref{big}.

\begin{thm} Suppose that $\|A-A_0\|_{L^\infty}<\epsilon_0$, where $\epsilon_0$ is as above. Then if $1<p<p_0$, then 
\[\|f\|_{L^p}\leq C_p\|\J^t f\|_{L^p},\quad\|f\|_{L^p}\leq C_p\|\K_\pm^t f\|_{L^p}\]
for all $f\in L^p(\partial\V)$.
\end{thm}

Note that $\epsilon_0$ is independent of~$p$. Since $L^{p_0}\cap L^p$ is dense in $L^{p}$, we know that $\K^t,\J^t$ are onto; thus, this suffices to establish that they are invertible.

\begin{proof}
We prove this only for $\J^t$ (the proof for $\K_\pm^t$ is identical). It suffices to prove that $\|(\J^t)^{-1} g\|_{L^p}\leq \|g\|_{L^p}$ for all $g\in \C^\infty_0$.

Let $E_\alpha=\{x:M g(x)>\alpha\}$ where $M g(x)$ is the maximal function. (The ``intervals'' $I$ are taken to be arbitrary bounded, connected open subsets of~$\partial\V$.)

Then $E_\alpha$ is an open subset of $\partial\V$, which is one-dimensional; it is therefore a union of countably many disjoint open intervals $Q_i$, and $\sigma(E_\alpha)=\sum_i\sigma(Q_i)<C\|g\|_{L^p}^p/\alpha^p$. Note that $\dashint_{Q_i}|g|=\alpha$ for all~$i$.

Then define $a_i(x)$, $b(x)$ by
\[a_i(x)=\begin{cases}
0,&x\notin Q_i;\\
g(x)-\dashint_{Q_i} g,&x\in Q_i
\end{cases}
\quad
b(x)=\begin{cases}
g(x),&x\notin E_\alpha;\\
\dashint_{Q_i} g,&x\in Q_i
\end{cases}\]
Then $g=b+a$ where $a=\sum_i a_i$.

Now, $b\in L^{p_0}$ with $\|b\|_{L^{p_0}}^{p_0}\leq\|g\|_{L^p}^p \alpha^{p_0-p}$. Define $r=(1+p)/2$, so $1<r<p$. 

We know that $\int a_i=0$, $\supp a_i\subset Q_i$, and 
\[\int |a_i|^r=\int_{Q_i} a_i^r=\int_{Q_i} \left|g-{\textstyle\dashint_{Q_i}} g\right|^r
\leq 2^{r-1}\int_{Q_i} |g|^r+\left|{\textstyle\dashint_{Q_i}}\right|^r
\leq 2^r \int_{Q_i} |g|^r.\]

So
\[\|a_i\|_{H^1}\leq C_r |Q_i|^{1-1/r} \|a_i\|_{L^r}\leq 
C_p|Q_i|^{1/r'}\|g\|_{L^r(Q_i)}.\]
Therefore,
\begin{align*}
\|a\|_{H^1}&\leq C_p\sum_i|Q_i|^{1/r'}\|g\|_{L^r(Q_i)}
\leq C_p\left(\sum_i|Q_i|^{r'/r'}\right)^{1/r'}\left(\sum_i\|g\|_{L^r(Q_i)}^r\right)^{1/r}
\\&\leq
C_p|E_\alpha|^{1/r'} \|g\|_{L^r(E_\alpha)}
\leq
C_p \left(\frac{\|Mg\|_{L^r(E_\alpha)}^r}{\alpha^r}\right)^{1/r'}\|g\|_{L^r(E_\alpha)}
\leq C_p\frac{\|Mg\|_{L^r(E_\alpha)}^r}{\alpha^{r-1}}.
\end{align*}

So
\[\int_{\partial\V} |(\J^t)^{-1} g|^p\,d\sigma=
\int_0^\infty p\alpha^{p-1}
\sigma\{X\in\partial\V:|(\J^t)^{-1} g(X)|>\alpha\}
\,d\alpha.\]
But if $|(\J^t)^{-1} g|>\alpha,$ then either $|(\J^t)^{-1} a|>\alpha/2$ or $|(\J^t)^{-1} b|>\alpha/2$. So
\begin{align*}
\int_{\partial\V} |(\J^t)^{-1} g|^p\,d\sigma
&\leq
\int_0^\infty p 2^{p}\alpha^{p-1}
\left(\sigma\{|(\J^t)^{-1} a(X)|>\alpha\}+\sigma\{|(\J^t)^{-1} b(X)|>\alpha\}\right)
\,d\alpha
\\&\leq
p 2^{p}\int_0^\infty \alpha^{p-1}
\left(\frac{\|(\J^t)^{-1} a\|_{L^1}}{\alpha}
+\frac{\|(\J^t)^{-1} b\|_{L^{p_0}}^{p_0}}{\alpha^{p_0}}\right)
\,d\alpha.
\end{align*}
Since $(\J^t)^{-1}$ is bounded $H^1\mapsto L^1$ and $L^{p_0}\mapsto L^{p_0}$, this means that
\begin{align*}
\int_{\partial\V} |(\J^t)^{-1} g|^p\,d\sigma
&\leq
C_p\int_0^\infty \alpha^{p-1}\left(\frac{\|a\|_{H^1}}{\alpha}+\frac{\|b\|_{L^{p_0}}^{p_0}}{\alpha^{p_0}}\right)d\alpha
\\&\leq
C_p\int_0^\infty 
\alpha^{p-1}\left(\frac{\|Mg\|_{L^r(E_\alpha)}^r}{\alpha^{r}}
+\frac{\|g\|_{L^{p_0}(E_\alpha^C)}^{p_0}}{\alpha^{p_0}}+\sigma(E_\alpha)\right)d\alpha.
\end{align*}
Since $E_\alpha=\{X\in\partial\V:Mg(X)>\alpha\}$, so $\sigma(E_\alpha)\leq {\|Mg\|_{L^r(E_\alpha)}^r}/{\alpha^{r}}$, and since $r<p<p_0$, we can rewrite this as
\begin{align*}
\int_{\partial\V} |(\J^t)^{-1} g|^p\,d\sigma
&\leq
\int_0^\infty \frac{C_p}{\alpha^{r-p+1}}
\int_{M g(X)>\alpha} Mg(X)^r\,d\sigma(X)\,d\alpha
\\&\qquad
+\int_0^\infty \frac{C_p}{\alpha^{p_0-p+1}}
\int_{M g(X)\leq\alpha} |g(X)|^{p_0} \,d\sigma(X)\,d\alpha
\\&=
C_p\int_{\partial\V}\int_0^{M g(X)}
 \alpha^{p-r-1}\,d\alpha\,Mg(X)^r\,d\sigma(X)
\\&\qquad
+C_p\int_{\partial\V}\int_{M g(X)}^\infty 
\alpha^{p-p_0-1} \,d\alpha\, |g(X)|^{p_0} \,d\sigma(X)
\\&=
C_p\int_{\partial\V}
\frac{1}{r-p}M g(X)^{p-r}
Mg(X)^r\,d\sigma(X)
\\&\qquad
+C_p\int_{\partial\V}
\frac{1}{p_0-p}M g(X)^{p-p_0}
|g(X)|^{p_0} \,d\sigma(X)
\\&\leq 
C_p \|M g\|_{L^p}^p\leq C_p\|g\|_{L^p}^p
\end{align*}
as desired.
\end{proof}

\chapter[$BMO$ and square-function estimates]{Boundary data in \texorpdfstring{$BMO$}{BMO} and square-function estimates}
\label{chap:square}
\begin{thm} \label{thm:BMOcarleson}
Suppose that $g\in BMO(\partial\V)$ for some good Lipschitz domain $\V$. Then there exists a unique function $u$ with $\div A\nabla u=0$ in $\V$, $u=g$ on $\partial\V$, and such that
\[
\frac{1}{\sigma(B(X_0,R)\cap\partial\V)}\int_{\V\cap B(X_0,R)}|\nabla u(X)|^2 \dist(X,\partial\V)\,dX
\leq C  \|g\|_{BMO}^2.
\]

\end{thm}

This $u$ will turn out to be a layer potential. Recall that $\D f|_{\partial\V}=\K f$, and that $\K^t$ is bounded and invertible on $H^1(\partial\V)$ (\thmref{invertibility}). Therefore by duality, $\K$ is bounded with bounded inverse on $BMO(\partial\V)$.

So if $g\in BMO$, then $g=\D f$ for some $f$ with $\|g\|_{BMO(\partial\V)}\approx \|f\|_{BMO(\partial\V)}$. Thus, to show that there exists a $u$ with $u|_{\partial\V}=g$ and
\[\frac{1}{\sigma(B(X_0,R)\cap\partial\V)}\int_{\V\cap B(X_0,R)}|\nabla u(X)|^2 \dist(X,\partial\V)\,dX
\leq C  \|g\|_{BMO}^2
\]
we need show only that
\begin{equation}\label{eqn:squareBMOpot}
\frac{1}{\sigma(B(X_0,R)\cap\partial\V)}\int_{\V\cap B(X_0,R)}|\nabla \D f(X)|^2 \dist(X,\partial\V)\,dX
\leq C  \|f\|_{BMO}^2
\end{equation}
for all~$f$.

\section{\texorpdfstring{$L^2$ estimates imply $BMO$ estimates}{L2 estimates imply BMO estimates}}
\label{sec:L2toBMO}

In this section, we show that if 
\begin{equation}\label{eqn:squareL2pot}
\int_{\V} |\nabla \D f(X)|^2 \dist(X,\partial\V)\,dX \leq C\|f\|_{L^2}^2
\end{equation}
holds for all $f\in L^2$, then (\ref{eqn:squareBMOpot}) holds for all $f\in BMO$. 

For most of this chapter, it will be convenient to work with more general operators. We have that
\[\nabla \D f(X) = \nabla_X\int_{\partial\V}\left(\nu\cdot A^T \nabla\Gamma_X^T(Y)\right) f(Y)\,d\sigma(Y).\]
So if $TF(X)=\int_{\partial\V} J(X,Y)F(Y)\,d\sigma(Y)$ where \[J(X,Y)=\Matrix{\nabla_X\left(\nu\cdot A^T \nabla\Gamma_X^T(Y)\right)
&\nabla_X\left(\nu\cdot A^T \nabla\Gamma_X^T(Y)\right)},\]
then results for $T$ will imply results for $\nabla \D$.

We have some useful conditions on this $J$. By (\ref{eqn:gradfundsoln}), $|\partial_{y_i}\Gamma^T_X(Y)|\leq\frac{C}{|X-Y|}$; since $\Gamma^T_X(Y)=\Gamma_Y(X)$, we have that $|\partial_{x_i}\Gamma^T_X(Y)|\leq\frac{C}{|X-Y|}$.

But $\partial_{x_i}\Gamma^T_X(Y)$ is a solution in $Y$ to $\div A^T\nabla (\partial_{x_i}\Gamma^T_X)=0$; hence, by (\ref{eqn:gradu}),
\[|\partial_{x_i}\partial_{y_j}\Gamma^T_X(Y)|\leq \frac{C}{|X-Y|^2}.\]
Thus, $|J(X,Y)|\leq \frac{C}{|X-Y|^2}$.

Also, $\D 1(X)$ is constant on each connected component of $\R^2\setminus\partial\V$, so $\nabla \D 1 (X)\equiv 0$; thus, $TI(X)\equiv 0$.

So we prove a lemma for $T$ and $J$ having these nice properties:
\begin{lem}\label{lem:squaretoBMO}
Suppose that $\V\subset\R^{n+1}$ is a good Lipschitz domain, that $TF(X)=\int_{\partial\V} J(X,Y)F(Y)\,d\sigma(Y)$ for 
some (matrix-valued function)~$J$, and that
\[|J(X,Y)|\leq C\frac{\dist(X,\partial\V)^{\alpha-1}}{|X-Y|^{n+\alpha}},\quad
TI(X)\equiv 0,\quad
\int_{\V} |TF(X)|^2 \dist(X,\partial\V)\,dX \leq C\|F\|_{L^2}^2\]
for some $\alpha>0$, for all $X\in\V$, $Y\in\partial\V$, and all $F$ in $L^2(\partial\V\mapsto \C^{2\times 2})$.

Then if $B\in BMO(\partial\V\mapsto \C^{2\times2})$, then $TB(X)$ converges for each $X\in\partial\V$ and
\[\frac{1}{\sigma(B(X_0,R)\cap\partial\V)}\int_{B(X_0,R)\cap\V} |T B(X)|^2 \dist(X,\partial\V)\,dX\leq C\|B\|_{BMO}^2\]
for any $R>0$, $X_0\in\partial\V$.
\end{lem}

Note that if $X\in\V$, then $Y\mapsto J(X,Y)\in L^\infty(\partial\V)$. Furthermore, 
\begin{align*}\int_{\partial\V} |J(X,Y)|\,d\sigma(Y)
&\leq \dist(X,\partial\V)^{\alpha-1}\int_{\partial\V} \frac{1}{|X-Y|^{n+\alpha}}\,d\sigma(Y)
\leq C\dist(X,\partial\V)^{\alpha-1}.
\end{align*}
So we know $TF(X)$ converges for all $X\in\V$ and all $F\in L^p$, $1\leq p\leq\infty$.

\begin{proof}
Recall from basic $BMO$ theory that, if $B\in BMO$ and $\Delta=B(X_0,R)\cap\partial\V$ is any surface ball, then 
\begin{align*}
\dashint_{2^{k+1}\Delta} |B-B_{\Delta}|\,d\sigma
&\leq C (k+1)\|B\|_{BMO}
\end{align*}
where $B_\Delta=\dashint_\Delta B$, and $2^k\Delta(X_0,r)=\Delta(X_0,2^kr)$.

Furthermore, by the John-Nirenberg inequality,
\[\int_{\Delta} |B-B_{\Delta}|^2\leq C\sigma(\Delta) \|B\|_{BMO}^2.\]

Fix some surface ball $\Delta=B(X_0,R)\cap\partial\V$. Now, $TB=T(B-B_{\Delta})$, and so without loss of generality $B_{\Delta}=0$. So
\[ T B=T(B\chi_{\Delta})+\sum_{k=0}^\infty T(B\chi_{2^{k+1}\Delta\setminus 2^k \Delta}).
\]
But
\[\int_{\V} | T (B\chi_{\Delta})(X)|^2 \dist(X,\partial\V)\,dX \leq C\|B\chi_\Delta\|_{L^2(\partial\V)}^2\leq C \sigma(\Delta) \|B\|_{BMO}^2.
\]

If $X\in B(X_0,R)$, then
\begin{align*}
| T(B\chi_{2^{k+1}\Delta\setminus2^k \Delta})(X)| 
&= \left|\int_{2^{k+1}\Delta \setminus 2^k \Delta} J(X,Y) B(Y)\,d\sigma(Y)\right|\\
&\leq 
\int_{2^{k+1}\Delta \setminus 2^k \Delta} \frac{\dist(X,\partial\V)^{\alpha-1}}{|X-Y|^{n+\alpha}} |B(Y)|\,dy\\
&\leq 
\frac{\dist(X,\partial\V)^{\alpha-1}}
{|2^k R|^{n+\alpha}}
\int_{2^{k+1}\Delta} |B(Y)|\,d\sigma(Y)
\\&\leq 
\frac{\dist(X,\partial\V)^{\alpha-1}}
{|2^k R|^{n+\alpha}}
C (k+1)\sigma(2^{k+1}\Delta) \|B\|_{BMO}
.
\end{align*}
If $\sigma(\Delta)<c(n)R^n$, then $B(X_0,R)$ contains \ifmulticonnect an entire boundary component of \fi $\partial\V$; so $\sigma(2^k\Delta)\leq \sigma(\partial\V)\leq C\sigma(\Delta)$ for all $k$. Otherwise, $\sigma(2^{k+1}\Delta)\leq C (2^k R)^n\leq C 2^{kn}\sigma(\Delta)$. In either case,
\[\sigma(2^{k+1}\Delta)\leq C 2^{kn}\sigma(\Delta)\]
and so
\begin{align*} 
\left|\sum_{k=0}^\infty
 T(B(\chi_{2^{k+1}\Delta\setminus2^k \Delta}))(X)
\right|
\leq 
C \|B\|_{BMO}
\frac{\dist(X,\partial\V)^{\alpha-1}\sigma(\Delta)}{R^{n+\alpha}}
\sum_{k=0}^\infty
\frac{(k+1)}{2^{k\alpha}}
.
\end{align*}

Therefore, if $\Q=B(X_0,R)\cap\V$ and $\beta=\sigma(\Delta) \|B\|_{BMO}^2$, we have that
\begin{align*}
\int_{\Q} \left|
 T B(X)\right|^2 \dist(X,\partial\V)\,dX
&\leq 
C\beta
+C \beta\sigma(\Delta)
\int_{\Q} 
\frac{\dist(X,\partial\V)^{2\alpha-1}}{R^{2n+2\alpha}}
\,dX
\\&\leq 
C\beta
+C \beta\sigma(\Delta)
\int_{B(X_0,R)} 
\frac{|X|^{2\alpha-1}}{R^{2n+2\alpha}}
\,dX
\leq 
C\beta
+C \beta
\frac{\sigma(\Delta)}{R^{n}}
\\&\leq C\beta = C\sigma(\Delta) \|B\|_{BMO}^2.
\end{align*}
This concludes the proof.
\end{proof}

Therefore, to prove \thmref{BMOcarleson}, we need only  establish~(\ref{eqn:squareL2pot}).

\section{Layer potentials on the half-plane} 

In this section, we establish a sufficient condition for \lemref{squaretoBMO} to hold if $\V=\R^{2}_+$. In the next two sections, we will work with general operators $T$ and kernels $J$; in \autoref{sec:BMObacktoD}, we will return to the Dirichlet problem and~$\nabla\D$.

From \cite[Theorem~1.1]{hofmann}, we have
\begin{thm}\label{thm:hofmann}
Let $\theta_t f(x) = \int_{\R^n} \psi_t(x,y) f(y)\,dy$, where
\begin{align}
|\psi_t(x,y)|&\leq C\frac{t^\alpha}{(t+|x-y|)^{n+\alpha}}\\
|\psi_t(x,y)-\psi_t(x+h,y)|&\leq C\frac{|h|^\alpha}{(t+|x-y|)^{n+\alpha}}\\
|\psi_t(x,y)-\psi_t(x,y+h)|&\leq C\frac{|h|^\alpha}{(t+|x-y|)^{n+\alpha}}
\end{align}
whenever $|h|\leq t/2$. 

Suppose that there exists a function $b:\R^n\mapsto\C$ such that for some constants $\lambda$,~$\Lambda$, $C_0>0$,
\begin{equation}\label{eqn:CZscalar}
\lambda\leq \Re b(x)\leq|b(x)|\leq\Lambda\quad\text{and}\quad
\dashint_{Q}\int_0^\infty |\theta_t b(x)|^2 \, \frac{dt}{t} dx\leq C_0
\end{equation}
for every $x\in\R^n$ and every dyadic cube $Q\subset\R^n$.

Then for all $f\in L^2(\R^n)$ we have that
\begin{equation}
\int_{\R^{n+1}_+}|\theta_t f(x)|^2\frac{dx\,dt}{t}\leq C(\lambda,\Lambda,C_0)\|f\|_{L^2(\R^n)}^2.
\end{equation}
\end{thm}

(In \cite{hofmann}, this is presented in more generality; the test function $b$ is replaced by a system of test functions $b_Q$ indexed by dyadic cubes~$Q$. We do not need that generality here.)


%
%

We will need this theorem to hold if our kernel and test function take  values in $\C^{2\times 2}$ rather than in $\C$. It is trivial to generalize to the following simple matrix-valued version:
\begin{thm}\label{thm:CZmatrixI}
Let $T_t F(x) = \int_{\R^n} J_t(x,y) F(y)\,dy$, where $J_t(x,y)$, $F(y)$, and $T_t F(x)$ are all square matrices. Assume that $J$ satisfies
\begin{align}
|J_t(x,y)|&\leq C\frac{t^\alpha}{t(t+|x-y|)^{n+\alpha}}\label{eqn:Jsize}\\
|J_t(x,y)-J_t(x+h,y)|&\leq C\frac{|h|^\alpha}{t(t+|x-y|)^{n+\alpha}}\label{eqn:Jx}\\
|J_t(x,y)-J_t(x,y+h)|&\leq C\frac{|h|^\alpha}{t(t+|x-y|)^{n+\alpha}}\label{eqn:Jy}
\end{align}
whenever $|h|\leq t/2$.

Suppose that there exists a function $b:\R^n\mapsto\C$ such that for some constants $\lambda$,~$\Lambda$, $C_0>0$,
\begin{equation}\label{eqn:CZmatrixI}
\lambda\leq \Re b(x)\leq|b(x)|\leq\Lambda\quad\text{and}\quad
\dashint_{Q}\int_0^\infty |T_t (Ib)(x)|^2 t\, {dt}\, dx\leq C_0
\end{equation}
for every $x\in\R^n$ and every dyadic cube $Q\subset\R^n$.

Then
\begin{equation}\label{eqn:squareL2}
\iint_{\R^{n+1}_+}|T_t F(x)|^2t\,dx\,dt\leq C(\lambda,\Lambda,C_0)\|F\|_{L^2(\R^n)}^2\end{equation}
for any $L^2$ matrix-valued function~$F$.
\end{thm}
Note that $t\,dx\,dt=\dist(X,\partial\R^2_+)\,dX$; this is why we use it instead of $\frac{dx\,dt}{t}$.

So if $T(Ib)\in BMO$ for some good scalar function~$b$, and a few other conditions hold, then (\ref{eqn:squareL2}) holds. The problem is to generalize, so that we need only show $T(B)\in BMO$ for $B$ in some larger class of test matrices.

Fortunately, \cite{hofmann} proves \thmref{hofmann} from a similar theorem with the identity function replacing the $b_Q$s; we may prove a useful matrix-valued theorem from \thmref{CZmatrixI} using the same techniques.

The conditions on $B$ analagous to (\ref{eqn:CZmatrixI}) will turn out to be the same as the conditions on $B_1$ required by \thmref{useful}, that is, there exist constants $C_0$, $C_1>0$ and a smooth, compactly supported function $v$ with $\int v=1$ such that, if we define $v_t(x)=\frac{1}{t^n} v\left(\frac{x}{t}\right)$, then
\begin{equation}\label{eqn:matrixintinverse}
\sup_{x\in\R^n}|B(x)|\leq C_0,\quad
\sup_{x\in\R^n,\>t\in\R^+}\left|\left(v_t\star B(x)\right)^{-1}\right|\leq C_1.
\end{equation}

\begin{thm}\label{thm:CZmatrixB} Suppose that $B:\R^n\mapsto \C^{2\times 2}$ is a bounded matrix-valued function which satisfies~(\ref{eqn:matrixintinverse}).

Let $T_t F(x) = \int_{\R^n} J_t(x,y) F(y)\,dy$, where $J_t(x,y)$, $F(y)$, and $T_t F(x)$ are all square matrices, and $J$ satisfies (\ref{eqn:Jsize}--\ref{eqn:Jy}).

Suppose that for each dyadic cube $Q$,
\begin{equation}
\int_Q \int_0^{l(Q)} |T_t B(x)|^2t\,dt\,dx
\leq C_2|Q|.\end{equation}
\label{eqn:CZmatrixB}
Then
\[\iint_{\R^{n+1}_+}|T_t F(x)|^2t\,dx\,dt\leq C(C_0,C_1,C_2)\|F\|_{L^2(\R^n)}^2.\]
\end{thm}

\begin{proof} Without loss of generality take $C_2=1$.
It suffices to establish (\ref{eqn:CZmatrixI}) for $b\equiv 1$.

As before, $TF(X)$ converges for $X\in\V$ and $F\in L^p(\partial\V)$, $1\leq p\leq\infty$. In particular, $\left|\int J_t(x,z)\,dz\right|\leq C/t$.

Let $A_t B(x)=v_t* B(x)$. Then
\[|T_t I(x)|=|T_t I(x) A_t B(x) (A_t B(x))^{-1}|\leq C|T_t I(x) A_t B(x)|.\]

So we may write
\begin{align*}
T_t I(x) A_t B(x)
&=\int J_t(x,z)\,dz \int \frac{1}{t^n}v\left(\frac{x-y}{t}\right) B(y)\,dy
\\&
=\int J_t(x,z)\,dz \int \frac{1}{t^n}v\left(\frac{x-y}{t}\right) B(y)\,dy
-T_t B(x)
+T_t B(x)
\\&
= \int 
\left(\frac{\int J_t(x,z)\,dz}{t^n}v\left(\frac{x-y}{t}\right)-J_t(x,y)\right) B(y)\,dy
+T_t B(x).
\end{align*}

Let $\check T_t F(x)=\int \check J_t(x,y) F(y)\,dy$, where 
\[\check J_t(x,y)=\left(\frac{\int J_t(x,z)\,dz}{t^n}v\left(\frac{x-y}{t}\right)-J_t(x,y)\right).\] 
We have that
\begin{align*}
{ \int_0^{l(Q)} |T_t I(x)|^2 t\,dt}
&\leq C\int_0^{l(Q)} |T_t I(x)A_t B(x)|^2 t\,dt
\leq C \int_0^{l(Q)} |\check T_t B(x)+T_t B(x)|^2 t\,dt
\\&\leq C \int_0^{l(Q)} |\check T_t B(x)|^2 t\,dt
       +C \int_0^{l(Q)} |T_t B(x)|^2 t\,dt
\end{align*}
and so
\begin{align*}
\int_Q\int_0^{l(Q)} |T_t I(x)|^2 t\,dtdx
&\leq C \int_Q\int_0^{l(Q)} |\check T_t B(x)|^2 t\,ddx
+C |Q|.
\end{align*}
So we need only show that 
\[\int_Q \int_0^{l(Q)} |\check T_t B(x)|^2 t\,dtdx\leq C|Q|.\] 
We will do this by applying \thmref{CZmatrixI} and \lemref{squaretoBMO} to~$\check T$.

By assumption, $J_t$ satisfies (\ref{eqn:Jsize}--\ref{eqn:Jy}). We need to show that $\check J_t$ does as well; it suffices to prove this for $J_t(x,y)+\check J_t(x,y)=\frac{\int J_t(x,z)\,dz}{t^n}v\left(\frac{x-y}{t}\right)$.

Recall that $v_t(x-y)=0$ if $|x-y|>Ct$.
If $|x-y|<Ct$, then
\begin{align*}
|\check J_t(x,y)+J_t(x,y)|
&\leq\left|\frac{C}{t^{n+1}}v\left(\frac{x-y}{t}\right)\right|
\leq \frac{C}{t^{n+1}}
\\&
\leq \frac{C}{t(t+|x-y|)^n}
\leq \frac{C}{t(t+|x-y|)^n}\left(\frac{Ct}{t+|x-y|}\right)^\alpha
.\end{align*}

If $|y-y'|<t/2$, then since $0<\alpha\leq 1$,
\begin{multline*}
\left|\frac{\int J_t(x,z)\,dz}{t^{n}}v\left(\frac{x-y}{t}\right)-\frac{\int J_t(x,z)\,dz}{t^{n}}v\left(\frac{x-y'}{t}\right)\right|
\\
\begin{aligned}&\leq \frac{C}{t^{n+1}}\left|v\left(\frac{x-y}{t}\right)-v\left(\frac{x-y'}{t}\right)\right|
\leq \frac{1}{t^{n+2}}\|\nabla v\|_{L^\infty}|y-y'|
\\&\leq C\frac{|y-y'|/(t+|x-y|)}{t(t+|x-y|)^n}
\leq C\frac{\big(|y-y'|/(t+|x-y|)\big)^\alpha}{t(t+|x-y|)^n}
= C\frac{|y-y'|^\alpha}{t(t+|x-y|)^{n+\alpha}}.
\end{aligned}
\end{multline*}
If $|x-x'|<t/2$, then 
\begin{multline*}
\left|\frac{\int J_t(x,z)\,dz}{t^{n}}v\left(\frac{x-y}{t}\right)-\frac{\int J_t(x',z)\,dz}{t^{n}}v\left(\frac{x'-y}{t}\right)\right|
\\
\begin{aligned}
&\leq
\left|\frac{\int J_t(x,z)\,dz}{t^{n}}v\left(\frac{x-y}{t}\right)-\frac{\int J_t(x',z)\,dz}{t^{n}}v\left(\frac{x-y}{t}\right)\right|
\\&\qquad+\left|\frac{\int J_t(x',z)\,dz}{t^{n}}v\left(\frac{x-y}{t}\right)-\frac{\int J_t(x',z)\,dz}{t^{n}}v\left(\frac{x'-y}{t}\right)\right|
\end{aligned}
\end{multline*}
The second term is at most $\frac{C|x-x'|^\alpha}{t(t+|x-y|)^{n+\alpha}}$ as before. The first term is $0$ if $|x-y|>Ct$; otherwise, it is at most
\[\frac{\|v\|_{L^\infty}}{t^n}\int |J_t(x,z)-J_t(x',z)|\,dz
\leq \frac{C}{t^n}\int \frac{|x-x'|^\alpha}{t(t+|x-z|)^{n+\alpha}}\,dz
\leq \frac{C|x-x'|^\alpha}{t^{n+1+\alpha}}
\]
which is at most $\frac{C|x-x'|^\alpha}{t (t+|x-y|)^{n+\alpha}}$ if $|x-y|<Ct$. Thus $\check J_t$ satisfies (\ref{eqn:Jsize}--\ref{eqn:Jy}).

Also, since $\int v=1$, we have that
\begin{align*}
\check T_t I(x)&=\int 
\left(\frac{\int J_t(x,z)\,dz}{t^n}v\left(\frac{x-y}{t}\right)-J_t(x,y)\right)\,dy
\\&=\int J_t(x,z)\,dz \int \frac{1}{t^n}v\left(\frac{x-y}{t}\right)\,dy
-\int J_t(x,y)\,dy
=0.
\end{align*}

So by \thmref{CZmatrixI},
\[\int_{\R^{n+1}_+} |\check T_t F(x)|^2 t\,dx\,dt \leq C\|F\|^2_{L^2(\R^n)}\]
for any $L^2$ matrix-valued function $F$, and so by \lemref{squaretoBMO}, we have that
\[\int_Q \int_0^{l(Q)} |\check T_t B(x)|^2 t\,dt\,dx
\leq C\|B\|_{BMO}^2 |Q|.\]
This completes the proof.
\end{proof}

\section{Layer potentials on Lipschitz domains}

We now wish to generalize the results of the previous section to arbitrary good Lipschitz domains~$\V$; this will provide a sufficient condition for (\ref{eqn:squareL2pot}) to hold.

We begin with special Lipschitz domains.

\begin{thm}\label{thm:CZmatrixLip}
Supppose that $\Omega\subset\R^2$ is a special Lipschitz domain with Lipschitz constant $k_1$, and that $J:\R^2\times\R^2\mapsto \C^{2\times 2}$ satisfies
\begin{align*}
|J(X,Y)| &\leq C\frac{\dist(X,\partial\Omega)^{\alpha}}{\dist(X,\partial\Omega)|X-Y|^{1+\alpha}},\\
|J(X,Y)-J(X',Y)| &\leq C\frac{|X-X'|^\alpha}{\dist(X,\partial\Omega)|X-Y|^{1+\alpha}},\\
|J(X,Y)-J(X,Y')| &\leq C\frac{|Y-Y'|^\alpha}{\dist(X,\partial\Omega)|X-Y|^{1+\alpha}}.
\end{align*}
Define $T F(X) = \int_{\partial\Omega} J(X,Y) F(Y)\,d\sigma(Y).$

Let $B:\partial\Omega\mapsto \C^{2\times 2}$ satisfy
\begin{equation}\label{eqn:generalintinverse}
\sup_{x\in\partial\Omega}|B(x)|\leq C_0,\quad
\sup_{\Delta\subset\partial\Omega\text{ connected}}\left|\left(\dashint_{\Delta} B\right)^{-1}\right|\leq C_1.
\end{equation}

Suppose that for any $X_0\in\partial\Omega$ and any $R>0$,
\begin{equation}\label{eqn:TBmatrixLip}
\int_{B(X_0,R)\cap\Omega} |T B(X)|^2 \dist(X,\partial\Omega)\,dX
\leq C_2\sigma(B(X_0,R)\cap\partial\Omega).\end{equation}
Then for every $F\in L^2(\partial\Omega\mapsto \C^{2\times 2})$,
\[\int_{\Omega}|T F(X)|^2 \dist(X,\partial\Omega)\,dX\leq C(C_0,C_1,C_2,k_1)\|F\|_{L^2(\partial\Omega)}^2.\]
\end{thm}

\begin{proof}
This follows from \thmref{CZmatrixB} by a change of variables.
Define the matrix-valued function $\check B(x)=B(\psi(x))\sqrt{1+\phi'(x)^2}$; then $\sup_{x\in\partial\V} |B(x)|\leq C_0(1+\|\phi'\|_{L^\infty})$, and $\sup_{a,b\in\R}\left|\left(\dashint_{a}^b B\right)^{-1}\right|\leq C_1(1+\|\phi'\|_{L^\infty})$.
By \lemref{Bisaccretive}, this means that $\check B$ satisfies (\ref{eqn:matrixintinverse}). 

Let $J_t(x,y)=J(\psi(x,t),\psi(y))$. Then $J_t$ satisfies \mbox{(\ref{eqn:Jsize}--\ref{eqn:Jy})}. If we let $T_t F(x)=\int_\R J_t(x,y) F(y)\,dy$, then 
\begin{align*}
T_t \check B(x)
&=\int_\R J_t(x,y) \check B(y)\,dy
=\int_\R J(\psi(x,t),\psi(y)) B(\psi(y))\sqrt{1+\phi'(y)^2}\,dy
\\&=\int_{\partial\Omega}J(\psi(x,t),Y) B(Y)\,d\sigma(Y)
=T B(\psi(x,t))
\end{align*}
and so if $Q\subset\R$ is an interval and $x_0\in Q$, $X_0=\psi(x_0)$, then
\begin{align*}
\int_Q \int_0^{l(Q)} |T_t B(x)|^2 t\,dt\,dx
&=
\int_Q \int_0^{l(Q)} |T B(\psi(x,t))|^2 t\,dt\,dx
\\&\leq
\int_{\psi(Q\times(0,l(Q)))} |T B(X)|^2 C\dist(X,\partial\Omega)\,dX
\\&\leq
\int_{B(X_0,Cl(Q))\cap\Omega} |T B(X)|^2 C\dist(X,\partial\Omega)\,dX
\\&\leq C C_2 l(Q)
\end{align*}
Thus, by \thmref{CZmatrixB}, we must have that
\[\iint_{\R^2_+} |T_t F(x)|^2 t\,dt\,dx\leq C\|F\|_{L^2}\]
for all $F\in L^2(\R\mapsto\C^{2\times 2})$. But since $T_t (\sqrt{1+(\phi')^2}F)(x)= T F(\psi(x,t))$, we must have that
\[\int_\Omega |T F(X)|^2 \dist(X,\partial\Omega)\,dX
\leq C\|F\|_{L^2(\partial\Omega)}\]
for all $F\in L^2(\partial\Omega\mapsto \C^{2\times 2})$, as desired.
\end{proof}

We now wish to move to good Lipschitz domains with compact boundary.
\begin{thm}\label{thm:CZmatrixBdd}
Supppose that $J:\R^2\times\R^2\mapsto \C^{2\times 2}$ satisfies
\begin{align}
|J(X,Y)| &\leq C\frac{1}{|X-Y|^{2}},\label{eqn:JsizeLip}\\
|J(X,Y)-J(X',Y)| &\leq C\frac{|X-X'|^\alpha}{|X-Y|^{2+\alpha}},\label{eqn:JxLip}\\
|J(X,Y)-J(X,Y')| &\leq C\frac{|Y-Y'|^\alpha}{|X-Y|^{2+\alpha}}.\label{eqn:JyLip}
\end{align}
Define $T_\U F(X) = \int_{\partial\U} J(X,Y) F(Y)\,d\sigma(Y).$

Assume that for each special Lipschitz domain $\Omega\subset\R^2$ with Lipschitz constant $k_\Omega$, there exists some $B_\Omega:\partial\Omega\mapsto \C^{2\times 2}$ which satisfies (\ref{eqn:generalintinverse}) and (\ref{eqn:TBmatrixLip}) with constants depending only on~$k_\Omega$.

Then if $\V\subset \R^2$ is a good Lipschitz domain with Lipschitz constants $k_i$, and if $F\in L^2(\partial\V\mapsto \C^{2\times 2})$, we have that
\[\int_{\V}|T_\V F(X)|^2 \dist(X,\partial\V)\,dX\leq C(k_i)\|F\|_{L^2(\partial\V)}^2.\]
\end{thm}

\begin{proof} This follows trivially from the previous theorem if $\V$ is a special Lipschitz domain; thus we may assume that $\partial\V$ is bounded. By \dfnref{domain}, there are $k_2$ special Lipschitz domains $\Omega_i$, with Lipschitz constant at most $k_1$, such that 
\[\partial\V\subset \bigcup_{i=1}^{k_2} \partial\Omega_i\cap B(X_i,r_i)\]
where $r_i>0$ and $X_i\in\partial\V$, with $\Omega_i\cap B(X_i,2r_i) = \V\cap B(X_i,2r_i)$.

So we may write $F=\sum_{i=1}^{k_2} F_i$, where $F_i(X)=0$ outside of $B(X_i,r_i)$, and $\sum_{i}|F_i(X)|=|F(X)|$ for all $X\in\partial\V$.


Pick some $i$ and note that $T_\V F_i\equiv T_i F_i$.
By \thmref{CZmatrixLip}, 
\[\int_{\Omega_i}|T_i F_i(X)|^2 \dist(X,\partial\Omega_i)\,dX\leq C(k_1)\|F_i\|_{L^2(\partial\Omega_i)}^2.\]

Note that if $\dist(X,\partial\Omega_i)\neq \dist (X,\partial\V)$, then since $\partial\Omega_i\cap B(X_i,2r_i)=\partial\V\cap B(X_i,2r_i)$ we must have that $\dist(X,B(X_i,3r_i/2)^C)\leq\dist(X,\partial\Omega_i)$. So if $X\in B(X_i,3r_i/2)$ then
\[\dist(X,\partial\V)\leq \frac{3}{2}r_i\leq 3\dist(X,B(X_i,3r_i/2)^C)\leq 3\dist(X,\partial\Omega_i).\] So
\begin{multline*}
\int_{\V\cap B(X_i,3r_i/2)}|T_i F_i(X)|^2 \dist(X,\partial\V)\,dX
\\
\begin{aligned}
&\leq \int_{\Omega_i\cap B(X_i,3r_i/2)}|T_i F_i(X)|^2 3\dist(X,\partial\Omega_i)\,dX
\\
&\leq 3\int_{\Omega_i}|T_i F_i(X)|^2 \dist(X,\partial\Omega_i)\,dX
\leq C(k_1)\|F_i\|_{L^2(\partial\Omega_i)}^2.
\end{aligned}
\end{multline*}

Conversely, suppose that $X\notin B(X_i,\frac{3}{2}r_i)$. Recall that $r_i\approx \sigma(\partial\V)$. Then
\begin{align*}
|T_\V F_i(X)| 
&= \left|\int_{B(X_i,r_i)\cap\partial\V} J(X,Y) F_i(Y)\,d\sigma(Y)\right|
\leq \int_{B(X_i,r_i)\cap\partial\V} \frac{C}{|X-Y|^2}\left|F_i(Y)\right|\,d\sigma(Y)
\\&\leq \frac{C}{|X-X_i|^2} \int_{\partial\V} \left|F_i(Y)\right|\,d\sigma(Y)
\leq \frac{C}{|X-X_i|^2}\sqrt{\sigma(\partial\V)} \|F_i\|_{L^2(\partial\V)}.
\end{align*}

So
\begin{multline*}
\int_{\V\setminus B(X_i,3r_i/2)} |T_\V F_i(X)|^2\dist(X,\partial\V)\,dX \\
\leq 
\int_{\V\setminus B(X_i,3r_i/2)} \frac{C}{|X-X_i|^3}{\sigma(\partial\V)} \|F_i\|_{L^2(\partial\V)}^2\,dX 
\\
\leq 
C\|F_i\|_{L^2(\partial\V)}^2\sigma(\partial\V)
\frac{C}{r_i}
=C\|F_i\|_{L^2(\partial\V)}^2.
\end{multline*}
Putting these together, we see that 
\[\int_\V |T_\V F_i(X)|^2\dist(X,\partial\V)\,dX \leq C \|F_i\|_{L^2}^2.\]

So
\begin{align*}
\int_\V |T_\V F(X)|^2\dist(X,\partial\V)\,dX
&=\int_\V \Big|\sum_{i=1}^{k_2} T_\V F_i(X)\Big|^2\dist(X,\partial\V)\,dX
\\&\leq k_2\sum_{i=1}^{k_2} \int_\V |T_\V F_i(X)|^2\dist(X,\partial\V)\,dX
\\&\leq C \sum_{i=1}^{k_2}\|F_i\|_{L^2}^2 \leq C\|F\|_{L^2}^2.
\end{align*}
This completes the proof.
\end{proof}

\section{The Dirichlet problem on a Lipschitz domain}\label{sec:BMObacktoD}

We now return to the Dirichlet problem. We want to show that
\begin{equation}
\label{eqn:nablaD}
\int_\V |\nabla \D f(X)|^2 \dist(X,\partial\Omega)\,dX\leq C\|f\|_{L^2(\partial\Omega)}^2
\end{equation}
for $\V$ a good Lipschitz domain.

Recall that $\D f(X) = \int_{\partial\Omega} \nu\cdot A^T\nabla\Gamma_X^T f\,d\sigma$. So
\begin{align*}
\nabla\D f(X) 
&=\nabla_X\int_{\partial\V} \nu\cdot A^T\nabla\Gamma_X^T f\,d\sigma
\\&=\int_{\partial\V} 
\Matrix{
	\partial_{x_1}\partial_{y_1}\Gamma^T_X(Y) & \partial_{x_1}\partial_{y_2}\Gamma^T_X(Y) \\
	\partial_{x_2}\partial_{y_1}\Gamma^T_X(Y) & \partial_{x_2}\partial_{y_2}\Gamma^T_X(Y) }
( A(Y) \nu(Y) ) f(Y)\,d\sigma(Y)
\end{align*}

We would like to apply \thmref{CZmatrixBdd} to this expression. The obvious candidate for $J(X,Y)$ is
\[\Matrix{
	\partial_{x_1}\partial_{y_1}\Gamma^T_X(Y) & \partial_{x_1}\partial_{y_2}\Gamma^T_X(Y) \\
	\partial_{x_2}\partial_{y_1}\Gamma^T_X(Y) & \partial_{x_2}\partial_{y_2}\Gamma^T_X(Y) }.\]
Unfortunately, this matrix is not H\"older continuous; however, a nearby matrix is.

Recall that
\begin{align*}
B_6(Y)=B_6^A(Y)=\Matrix{a_{11}(Y)&a_{21}(Y)\\0&1},\quad 
B_6^T(X)=\Matrix{a_{11}(X)&a_{12}(X)\\0&1}.
\end{align*}
Define
\[
J(X,Y)=
B_6^T(X)\Matrix{
	\partial_{x_1}\partial_{y_1}\Gamma^T_X(Y) & \partial_{x_1}\partial_{y_2}\Gamma^T_X(Y) \\
	\partial_{x_2}\partial_{y_1}\Gamma^T_X(Y) & \partial_{x_2}\partial_{y_2}\Gamma^T_X(Y) }B_6(Y)^t\]
and let $TF(X)=\int_{\partial\Omega} J(X,Y) F(Y)\,d\sigma(Y)$ as usual. Then \[\nabla \D f(X) =
B_6^T(X)^{-1}\int_{\partial\Omega} J(X,Y)
B_6(Y)^{-1}( A(Y) \nu(Y) ) f(Y)\,d\sigma(Y)
\]
so if we let $F(Y)= B_6(Y)^{-1}A(Y)\Matrix{\nu(Y) &\nu(Y)} f(Y)$, then $|\nabla\D f(X)|\leq C |TF(X)|$ and $|F(Y)|\leq C|f(Y)|$; thus, we need only show that $T$ satisfies the conditions of \thmref{CZmatrixBdd} to establish (\ref{eqn:squareL2pot}), which by \lemref{squaretoBMO} will prove \thmref{BMOcarleson}.

We begin with the conditions on $J$.
As in \autoref{sec:L2toBMO}, $|J(X,Y)|\leq C/|X-Y|^2$.
We may write
\[
J(X,Y)=
\Matrix{
	\partial_{y_1}\partial_{x_2} \tilde\Gamma_X^T(Y)&
	\partial_{y_2}\partial_{x_2} \tilde\Gamma_X^T(Y)
\\
	\partial_{x_2}\partial_{y_1}\Gamma^T_X(Y) & \partial_{x_2}\partial_{y_2}\Gamma^T_X(Y) }B_6(Y)^t.\]
Since all of the $X$-derivatives are now in terms of $x_2$, and $A(X)$ is independent of $x_2$, we have that each component is a solution to an elliptic equation in $X$. Thus, $J(X,Y)$ must be H\"older continuous in~$X$, and in fact by \lemref{PDE4} must satisfy
\[|J(X',Y)-J(X,Y)|\leq C\frac{|X-X'|^\alpha}{|X-Y|^{2+\alpha}}\]
for all $|X-X'|<\frac{1}{2}|X-Y|$.

Similarly,
\[|J(X,Y')-J(X,Y)|\leq C\frac{|Y-Y'|^\alpha}{|X-Y|^{2+\alpha}}\]
for all $|Y-Y'|<\frac{1}{2}|X-Y|$. 

So (\ref{eqn:JsizeLip}--\ref{eqn:JyLip}) hold. Fix some special Lipschitz domain $\Omega$; we need to find a $B=B_\Omega$ which satisfies (\ref{eqn:generalintinverse}) and~(\ref{eqn:TBmatrixLip}).

Let 
\[B(Y)=(B_6(Y)^t)^{-1}\Matrix{A(Y)\nu(Y) &\tau(Y)}.\] 
(This is a slight adjustment of the $B_1$ of~(\ref{eqn:B1def}).)
From \autoref{sec:B1accretive}, we know that $B$ satisfies (\ref{eqn:generalintinverse}). Furthermore,
\begin{align*}
\left(B_6^T(X)^t\right)^{-1}T B (X)
&=\left(B_6^T(X)^t\right)^{-1}\int J(X,Y) B(Y)\,d\sigma(Y)
\\
&=
\int_{\partial\Omega}
\Matrix{
\partial_{x_1}\partial_{y_1} \Gamma^T_{X}&
\partial_{x_1}\partial_{y_2} \Gamma^T_{X}\\
\partial_{x_2}\partial_{y_1} \Gamma^T_{X}&
\partial_{x_2}\partial_{y_2} \Gamma^T_{X}
}
\Matrix{\vphantom{\nabla\Gamma^T_(}A\nu &\tau}\,d\sigma(Y)
\\&=
\int
\Matrix{
\nabla_X\nu\cdot A^T\nabla\Gamma^T_{X}(Y)&
\nabla_X\partial_\tau\Gamma^T_{X}(Y)}
\,d\sigma(Y)
\\&=
\int
\Matrix{
\nabla_X\partial_\tau\tilde\Gamma^T_{{X}}(Y)&
\nabla_X\partial_\tau\Gamma^T_{{X}}(Y)}
\,d\sigma(Y)
\\&=
\int\partial_\tau
\Matrix{
\nabla_X\tilde\Gamma^T_{{X}}(Y)&
\nabla_X\Gamma^T_{{X}}(Y)}
\,d\sigma(Y)
\end{align*}
which equals 0, since $\nabla_X \Gamma^T_X(Y),\nabla_X \tilde\Gamma^T_X(Y)$ go to zero as $|X-Y|\to\infty$. Thus, $B$ satisfies (\ref{eqn:TBmatrixLip}) as well and (\ref{eqn:nablaD}) is proven.

\chapter{Removing the smoothness assumptions}
\label{chap:apriori}
We have proven (the existence halves of) Theorems~\ref{thm:big} and~\ref{thm:BMO} only under the assumption that $A$, $A_0$ are smooth.

\begin{thm}\label{thm:apriori:NR} Fix some $\Lambda$, $\lambda$ and~$\epsilon_0$. Let $\V$ be a Lipschitz domain, and suppose that solutions to $(N)^A_p$, $(R)^A_p$, $(N)^A_1$, and~$(R)^A_1$ exist for all matrix-valued functions~$A$ satisfying (\ref{eqn:elliptic}), $\|\Im A\|_{L^\infty}<\epsilon_0$, and which in addition are smooth and satisfy $A(x)\equiv I$ for large $|x|$. 

Let $A$ be a matrix-valued function that satisfies (\ref{eqn:elliptic}) and $\|\Im A\|_{L^\infty}<\epsilon_0$, but not smoothness or $A(x)\equiv I$ for large~$x$. Then there exist solutions to $(N)_p^A$, $(R)_p^A$, $(N)_1^A$, and~$(R)_1^A$ in~$\V$.
\end{thm}

\begin{proof}
First, by \cite[Lemma~1.1]{Rule}, if $f\in W^{1,2}\cap L^{7/6}\cap L^{17/6}(\partial\V)$, then there is some $u$ with $\div A\nabla u=0$, $\Tr u=f$ and 
\[\|\nabla u\|_{L^2} + \left\|\frac{u(X)}{1+|X|}\right\|_{L^2}\leq C(\|f\|_{W^{1,2}}+\|f\|_{L^{7/6}}+\|f\|_{L^{17/6}}).\]

Similarly, by \cite[Lemma~1.2]{Rule}, if $g_0\in H^1(\partial\V)$ then there exists some $u$ with $\div A\nabla u=0$, $\nu\cdot A\nabla u=g_0$ and $\|\nabla u\|_{L^2(\V)} \leq C\|g_0\|_{H^1(\partial\V)}$.

So there exist regularity or Neumann solutions for boundary data in (dense subspaces of) $H^1$ or $L^p$. We need only show that they have the desired bounds on their non-tangential maximal functions.

Let $u$ be the Neumann or regularity solution given above to $\nu\cdot A\nabla u=g_0$ or $\partial_\tau u=g_0$, where $g_0\in H^1_{at}$.

Assume $\eta$ is small, and let $A^\eta(x)$ be a smooth matrix-valued function with $\|A-A^\eta\|<\eta$, except for a set $E_\eta$ with $|E_\eta\cap (-1/\eta,1/\eta)|<\eta$, where here $|\cdot |$ denotes Lebesgue measure in $\R^1$. (We are still working with matrices which are independent of the second variable.) Let $\V_\eta=\{(x,t): x\in E_\eta,\allowbreak\>(x,t)\in\V\}$. 

For convenience, let $A_\eta=A^\eta$, and assume that $A^\eta$ satisfies all of the conditions of the theorem, so $(N)_p^{A_\eta}$, $(R)_p^{A_\eta}$, $(N)_1^{A_\eta}$, and~$(R)_1^{A_\eta}$ hold in~$\V$.

Take $u^\eta$ to be the solution to $(N)_p^{A_\eta}$ and $(N)_1^{A_\eta}$, or $(R)_p^{A_\eta}$ and $(R)_1^{A_\eta}$, developed in the rest of this paper with the same boundary data as~$u$.

I claim that
\[\dashint_{B(Y,r)} |\nabla u-\nabla u^\eta|^2\leq \frac{o(\eta)}{r^2}\] where $o(\eta)\to 0$ as $\eta\to 0$.

Suppose my claim is true. By (\ref{eqn:gradu}), if $Y\in\gamma(X)$ then
\begin{align*}
|\nabla u(Y)|^2& \leq C\dashint_{B(Y,\dist(Y,\partial\V)/2)} |\nabla u|^2
\\&\leq C\dashint_{B(Y,\dist(Y,\partial\V)/2)} |\nabla u-\nabla u^\eta|^2+C\dashint_{B(Y,\dist(Y,\partial\V)/2)} |\nabla u^\eta|^2
\\&\leq C\frac{o(\eta)}{\dist(Y,\partial\V)^2}+CN(\nabla u^\eta)(X)^2.
\end{align*}

Therefore, if we let 
\[N_\delta f (X) =\sup\{|f(Y)|: (1+|X|^2)\delta<|X-Y|<(1+a)\dist(Y,\partial\V)\},\]
then $N_\delta(\nabla u)(X)\leq \frac{o(\eta)}{\delta(1+|X|)^2}+C N (\nabla u^\eta)(X)$. So $\|N_\delta(\nabla u)\|_{L^p}\leq C\frac{o(\eta)}{\delta} + C\|N(\nabla u^\eta)\|_{L^p}$ for all $1\leq p\leq \infty$. (The decay in $X$ is necessary to ensure that this is true for $\partial\V$ not compact, that is, $\V=\Omega$ a special Lipschitz domain.)

But 
$\|N(\nabla u^\eta)\|_{L^1}\leq C\|g_0\|_{H^1}$, $\|N(\nabla u^\eta)\|_{L^p}\leq C\|g_0\|_{L^p}$ for $1<p<p_0$. By taking the limit as $\eta\to 0$, we see that $\|N_\delta(\nabla u)\|_{L^p}\leq C_p \|g_0\|_{L^p}$ and $\|N_\delta(\nabla u)\|_{L^1}\leq C\|g_0\|_{H^1}$ uniformly in $\delta$; thus these inequalities must hold for $N(\nabla u)$ as well.

So I need only show that
\[\dashint_{B(Y,r)} |\nabla u-\nabla u^\eta|^2\leq \frac{o(\eta)}{r^2}.\] 

Fix some choice of $\eta$. 
We have that $u^\eta$ is a Neumann or regularity solution with boundary data in $H^1$ or $H^1\cap L^p$. Since our solutions are of the form $\S((\K^t_-)^{-1} g_0)$, the $(N)^{A_\eta}_1$ and the $(N)^{A_\eta}_p$ solutions are the same. So $N(\nabla u^\eta)\in L^1(\partial\V)$.

If $\V$ is bounded or special, then by \lemref{L2loc},  $\nabla u^\eta\in L^2(\V)$, uniformly in~$\eta$.

If $\V^C$ is bounded, recall that $u^\eta(X)=\S^\eta(P_\eta^{-1} g_0)$ where $P_\eta = (\J^{A_\eta})^t$ or $(\K^{A_\eta}_-)^t$. Since both these operators are invertible on $H^1$ we have that $u^\eta=\S^\eta f$ for some $f\in H^1$. But then $f=\partial_\tau F$ for some $F$ with $\|F\|_{L^\infty}\leq \|f\|_{L^1}\leq C\|g_0\|_{H^1}$; then
\[|u^\eta(X)| = |\S^\eta(\partial_\tau F)(X)|
=\left|\int_{\partial\V} \partial_\tau\Gamma_X^T F\,d\sigma\right|
\leq C\|g_0\|_{H^1}\sigma(\partial\V)/\dist(X,\partial\V)\]
and so by \lemref{PDE1} and (\ref{eqn:gradu}) we have that $|\nabla u(X)|\leq C\|g_0\|_{H^1}\sigma(\partial\V)/\dist(X,\partial\V)^2$. So in particular $\nabla u^\eta$ is still in $L^2$ uniformly in $\eta$.

Pick some $R_0$ is large enough that $B(Y,r)\subset B(0,R_0)$ and let $R\geq R_0$. Then
\begin{align*}
\int_{B(0,R_0)\cap\V} 
|\nabla u-\nabla  u^\eta|^2
&\leq
C\Re \int_{\V\cap B(0,R)} (\nabla \overline u-\nabla \overline{u^\eta})\cdot A^\eta(\nabla u-\nabla u^\eta)\\
&=
C\Re \int_{\V\cap B(0,R)} \nabla (\overline u-\overline{u^\eta})\cdot ( A\nabla u- A^\eta\nabla u^\eta)
\\&\qquad+\frac{C}{r^2}\Re \int_{\V\cap B(0,R)} 
(\nabla \overline u-\nabla \overline{u^\eta})(A^\eta-A)\nabla u
\\
&=
C\Re 
\int_{B(0,R)\cap \partial\V} \Tr (\overline{u-u^\eta})(\nu\cdot A\nabla u -\nu\cdot A^\eta\nabla u^\eta )\,d\sigma
\\&\qquad+
C\Re 
\int_{\V\cap \partial B(0,R)} \Tr (\overline{u-u^\eta})(\nu\cdot A\nabla u -\nu\cdot A^\eta\nabla u^\eta )\,d\sigma
\\&\qquad+
C\Re \int_{\V\cap B(0,R)} 
(\nabla \overline u-\nabla \overline{u^\eta})(A^\eta-A)\nabla u.
\end{align*}
But
\begin{align*}
\left|\Re \int_{\V} 
(\nabla \overline u-\nabla \overline{u^\eta})(A^\eta-A)\nabla u\right|
&\leq
C\eta (\|\nabla u\|_{L^2(\V)}^2+\|\nabla u^\eta\|_{L^2(\V)}^2)
\\&\qquad
+C\|\nabla u\|_{L^2(V_\eta)}(\|\nabla u\|_{L^2(\V)}+\|\nabla u^\eta\|_{L^2(\V)})
\end{align*}

So the third term is at most $o(\eta)$, independently of $R$. Since $u$, $u^\eta$ are Neumann or Dirichlet solutions with the same boundary data on $\partial\V$, the first term is 0 for all~$R$.

So for $R_0$ large enough and $R\geq R_0$,
\begin{multline*}
\int_{B(0,R_0)\cap\V} |\nabla u-\nabla u^\eta|^2
\leq
o(\eta) + C\Re \int_{\V\cap\partial B(0,R)}
\overline{(u-u^\eta)}(\nu\cdot A\nabla u - \nu\cdot A^\eta\nabla u^\eta)\,d\sigma 
\end{multline*}
By averaging our inequality over a range of $R$, we get that
\begin{align*}
\int_{B(0,R_0)\cap\V} |\nabla u-\nabla u^\eta|^2
&\leq
o(\eta) + \frac{C}{R_0}\int_{R_0}^{2R_0}\int_{\V\cap\partial B(0,R)}
|u-u^\eta| |A\nabla u - A^\eta\nabla u^\eta|\,d\sigma \,dR
\\&\leq
o(\eta) + \frac{C}{R_0}\int_{\V\cap B(0,2R_0)\setminus B(0,R_0)}
|u-u^\eta| |A\nabla u - A^\eta\nabla u^\eta|
\end{align*}
But $A\nabla u-A^\eta\nabla u^\eta\in L^2(\V)$. So 
\[\frac{C}{R_0}\int_{\V\cap B(0,2R_0)\setminus B(0,R_0)}
|A\nabla u - A^\eta\nabla u^\eta|
\leq C\|\nabla u\|_{L^2(\V\setminus B(0,R_0))}+ C\|\nabla u^\eta\|_{L^2(\V\setminus B(0,R_0))}
\]
which goes to 0 as $R_0\to\infty$; so we need only show that $u$, $u^\eta$ are bounded on $\V\setminus B(0,R_0)$ for large $R_0$. 

This is trivially true if $\V$ is bounded. If $\V^C$ is bounded, then $u^\eta(X) = \S^\eta f^\eta(X)$ where $\|f^\eta\|_{H^1}\leq C \|g_0\|_{H^1}$; then $|u^\eta(X)|\leq C \|f^\eta\|_{L^1} \diam\partial\V/\dist(X,\partial\V)$ and so $u^\eta$ is bounded (and in fact goes to 0).

If $\V=\Omega$ is special, then $N(\nabla u^\eta)\in L^1(\partial\Omega)$. So 
\[|u^\eta(\psi(x,t))-u(\psi(y,t))|\leq \int_x^y N(\nabla u^\eta)(\psi(z))\,dz \leq \|N(\nabla u^\eta)\|_{L^1}\]
for any $0<t$ and any $x<y$. Furthermore, for any $t>0$, there is some $x_t$ such that $N(\nabla u^\eta)(\psi(x_t))<\|N(\nabla u^\eta)\|_{L^1}/(1+t)$.

Then
\begin{align*}
|u^\eta(\psi(x,t))-u^\eta(\psi(0,1))|
&\leq 
|u^\eta(\psi(x,t))-u^\eta(\psi(x_t,t))|
+|u^\eta(\psi(x_t,t))-u^\eta(\psi(x_t,1))|
\\&\qquad+|u^\eta(\psi(x_t,1)-u^\eta(\psi(0,1)|
\\&\leq 2\|N(\nabla u^\eta)\|_{L^1}+ |t-1| N(\nabla u^\eta)(\psi(x_t))
\leq 3\|N(\nabla u^\eta)\|_{L^1}.\end{align*}
So $u^\eta$ is in fact bounded on $\Omega$. We need only bound $u(X)$ for large~$X$.

If we are dealing with the regularity problem, then $u(X)/(1+|X|)\in L^2$ and so by \lemref{PDE3} if $X$ is large enough then $|u(X)|\leq 1$.

If we are dealing with the Neumann problem, then by the Poincar\'e inequality
\[\dashint_{\V\cap B(0,2R)\setminus B(0,R/2)} |u-u_R|^2\leq C\|\nabla u\|_{L^2(\V\setminus B(0,R/2))}^2\] where $u_R=\dashint_{\V\cap B(0,2R)\setminus B(0,R/2)} u$; by applying \lemref{PDE3}, since $\nu\cdot A\nabla u=0$ on $\partial\V\setminus B(0,R)$, we see that $|u(X)-u_R|<\epsilon$ for all $X\in\V\cap B(0,2R)\setminus B(0,R)$, for all $R$ sufficently large.

But since Neumann solutions are only defined up to additive constants we may take $u_R=0$; this lets us write
\[
\int_{B(0,R_0)\cap\V} |\nabla u-\nabla u^\eta|^2
\leq o(\eta)
+ C(u,u^\eta)
\left(\|\nabla u\|_{L^2(\V\setminus B(0,R_0))}+ \|\nabla u^\eta\|_{L^2(\V\setminus B(0,R_0))}\right)
.\]
By taking the limit as $R_0\to \infty$ we see that
\begin{equation}\label{eqn:nabladifference}
\int_{\V} |\nabla u-\nabla u^\eta|^2
\leq o(\eta).
\end{equation}
This completes the proof.
\end{proof}

\begin{thm}\label{thm:apriori:Dp} Fix some $\Lambda$, $\lambda$ and~$\epsilon_0$. Let $\V$ be a Lipschitz domain, and suppose that solutions to $(D)^A_p$ and~$(R)^A_{p'}$, $(R)^A_1$ exist for all matrix-valued functions~$A$ satisfying (\ref{eqn:elliptic}), $\|\Im A\|_{L^\infty}<\epsilon_0$, and which in addition are smooth and satisfy $A(x)\equiv I$ for large $|x|$. 

Let $A$ be a matrix-valued function that satisfies (\ref{eqn:elliptic}) and $\|\Im A\|_{L^\infty}<\epsilon_0$, but not smoothness or $A(x)\equiv I$ for large~$x$. Then there exist solutions to $(D)_p^A$ in~$\V$.
\end{thm}

\begin{proof}
Suppose that $f$ is smooth and compactly supported on $\partial\V$. Construct $u$, $A^\eta$, $u^\eta$ as in the proof of \thmref{apriori:NR}, with $\Tr u=\Tr u^\eta=f$, $\div A\nabla u = \div A^\eta\nabla u^\eta=0$.

As in the proof of \thmref{apriori:NR}, by using \lemref{PDE3} instead of (\ref{eqn:gradu}) it suffices to show that 
\[\int_{B(Y,r)} |u-u^\eta|^2\leq o(\eta).\]
In fact, it suffices to show that $\int_{B(Y,r)} |u-u^\eta|^2\leq C(R)o(\eta)$ for any $B(Y,r)\subset B(0,R)$; we then get a uniform bound on the $L^p$ norm of $N_{\delta,R}u(X)=\sup\{|u(Y)|:\delta<|X-Y|<(1+a)\dist(Y,\partial\V)<R\}$
which becomes a bound on $\|Nu\|_{L^p}$ as before.

Fix $\W\subset\V$ compact with $B(0,2R)\cap\V\subset\W$ for $R$ large. Then by the Poincar\'e inequality,
\[\left\|v-{\textstyle\dashint_{\partial\W\cap\partial\V} \Tr v}\right\|_{L^2(\W)}\leq C(\W)\|\nabla v\|_{L^2(\W)}\]
for all functions $v\in {W}^{1,2}(\W)$.

$u$ is constructed in this theorem the same way it was in \thmref{apriori:NR}. We will show (\thmref{mixedunique}) that Dirichlet solutions and regularity solutions with the same boundary data are equal. So (\ref{eqn:nabladifference}) applies, and so since $\Tr u=\Tr u^\eta$ on $\partial\V$,
\[\|u-u^\eta\|_{L^2(\W)}
\leq C(\W)\|\nabla u-\nabla u^\eta\|_{L^2(\W)}
\leq C(\W)\|\nabla u-\nabla u^\eta\|_{L^2(\V)}
\leq o(\eta)
.\]
This completes the proof.
\end{proof}

\begin{thm}\label{thm:apriori:BMOcpt} Fix some $\Lambda$, $\lambda$ and~$\epsilon_0$. Let $\V$ be a good Lipschitz domain, and suppose that solutions to $(R)^A_p$ and \thmref{BMO} exist for all matrix-valued functions~$A$ satisfying (\ref{eqn:elliptic}), $\|\Im A\|_{L^\infty}<\epsilon_0$, and which in addition are smooth and satisfy $A(x)\equiv I$ for large $|x|$. 

Let $A$ be a matrix-valued function that satisfies (\ref{eqn:elliptic}) and $\|\Im A\|_{L^\infty}<\epsilon_0$, but not smoothness or $A(x)\equiv I$ for large~$x$. Then solutions to \thmref{BMO} with compactly supported boundary data exist for $A$ in~$\V$.
\end{thm}

\begin{proof}
First take $f$ to be a smooth, compactly supported function on~$\partial\V$. We may construct a $u$ with $\div A\nabla u=0$ in~$\V$, $\Tr u = f$ and $\|N(\nabla u)\|_{L^{p}}\leq C\|\partial_\tau f\|_{L^{p}}$. We need only show that (\ref{eqn:uCarleson}) holds.

Define $A^\eta$, $u^\eta$ as before. Then as in \thmref{apriori:NR},
\[|\nabla u(Y)-\nabla u^\eta(Y)| \leq \frac{o(\eta)}{\dist(Y,\partial\V)}.\]

So if $X\in\partial\V$ and $R>0$, and if $\V_\delta=\{X\in\V:\dist(X,\partial\V)>\delta\}$, then
\begin{multline*}
\frac{1}{R}\int_{B(X,R)\cap\V_\delta} |\nabla u(Y)|^2\dist(Y,\partial\V)\,dY
\\\begin{aligned}
&\leq
\frac{2}{R}\int_{B(X,R)\cap\V_\delta}
\left(|\nabla u(Y)-\nabla u^\eta(Y)|^2
+ |\nabla u^\eta(Y)|^2\right)
\dist(Y,\partial\V)\,dY
\\&\leq
\frac{2}{R}\int_{B(X,R)\cap\V_\delta} \frac{o(\eta)}{\dist(Y,\partial\V)}\,dY
+C\|f\|_{BMO}^2
\end{aligned}
\end{multline*}
and so by letting $\eta\to 0$,
\[\frac{1}{R}\int_{B(X,R)\cap\V_\delta} |\nabla u(Y)|^2\dist(Y,\partial\V)\,dY\leq C\|f\|_{BMO}^2\]
uniformly in $\delta$; by letting $\delta\to 0$ we recover \thmref{BMO} for smooth, compactly supported boundary data.

We now consider moving to nonsmooth, compactly supported boundary data. (I did not do this explicitly when proving Theorems~\ref{thm:apriori:NR} and~\ref{thm:apriori:Dp} because smooth functions are dense in $L^p$; they are not dense in $BMO$.) 

Since $f$ is compactly supported, $f\in L^p$ for $1\leq p<\infty$. Let $p$ be large enough that $(D)^A_p$ holds in~$\V$. Let $f_n\to f$ in $L^p$, $f_n$ smooth and compactly supported, with $\|f_n\|_{BMO}\leq C\|f\|_{BMO}$. Let $u_n$ be the solution with boundary data $f_n$ as above.

%
%
%

Then $\|N(u_n-u_m)\|_{L^p}\leq C\|f_n-f_m\|_{L^p}$, so \[|u_n(X)-u_m(X)|\leq \frac{C\|f_n-f_m\|_{L^p}}{\dist(X,\partial\V)^{1/p}},\quad|\nabla u_n(X)-\nabla u_m(X)|\leq \frac{C\|f_n-f_m\|_{L^p}}{\dist(X,\partial\V)^{1+1/p}}.\]
Thus, $\{u_n\}$ converges almost uniformly to some $u$ with the right boundary data; we see that for any $X\in\partial\V$ and any $R>0$,
\begin{multline*}
\frac{1}{R}\int_{B(X,R)\cap\V_\delta} |\nabla u(Y)|^2\dist(Y,\partial\V)\,dY
\\\leq
\frac{2}{R}\int_{B(X,R)\cap\V_\delta}
\left(|\nabla u(Y)-\nabla u_n(Y)|^2
+ |\nabla u_n(Y)|^2\right)
\dist(Y,\partial\V)\,dY
\leq C\|f\|_{BMO}
\end{multline*}
for $n$ large enough; by letting $\delta\to 0$ we complete the proof.

We now want to pass to non-compactly supported $f$.
\end{proof}

\chapter{Converses and uniqueness}
\label{chap:converse}
We have show that, if $\V$ is a bounded or special Lipschitz domain and $f$ is a function defined on $\partial\V$, then there is some $u$ with $\div A\nabla u=0$ in $\V$ and such that
\begin{itemize}
\item If $f\in H^1(\partial\V)$, then $\nu\cdot A\nabla u=f$ (or $\tau\cdot\nabla u=f$) on $\partial\V$ and $N(\nabla u)\in L^1(\partial\V)$.
\item If $f\in L^p(\partial\V)\cap H^1(\partial\V)$ for $p>1$ small enough,  then $\nu\cdot A\nabla u=f$ (or $\tau\cdot\nabla u=f$) on $\partial\V$ and $N(\nabla u)\in L^p(\partial\V)$. 
\item If $f\in L^p(\partial\V)$ for $p<\infty$ large enough,  then $u=f$ on $\partial\V$ and $N(u)\in L^p(\partial\V)$. 
\item If $f\in BMO(\partial\V)$, then $u=f$ on $\partial\V$ and (\ref{eqn:uCarleson}) holds.
\end{itemize}

We wish to prove the converses, and to show that such $u$ are unique. 

Recall \lemref{NTM}: if $Nu$ (or $N(\nabla u)$) is in $L^p$ for $1\leq p\leq\infty$, then we may widen the apertures of our nontangential cones as much as we like; in particular we may widen them so that
$\psi(x,t)\in\gamma(\psi(x))$, where $\psi$ is as in (\ref{eqn:psi}) for any of the special Lipschitz domains $\Omega_i$ in \dfnref{domain}.

We begin with a lemma for functions whose gradients are in $L^p$:

\begin{lem}\label{lem:L2loc} Suppose that $N(\nabla u)\in L^p(\partial\V)$ for some $1\leq p\leq \infty$. Then $\nabla u\in L^{2}_{loc}(\bar\V)$ and $u\in L^\infty_{loc}(\bar\V)$. 

Furthermore, if $\U\subset\bar \V$ is compact, we may bound $\|u\|_{L^\infty(\U)}$ or $\|\nabla u\|_{L^2(\U)}$ depending only on $\U$, $\V$ and $\|N(\nabla u)\|_{L^p}$. In particular, if $\V=\Omega$ is a special Lipschitz domain, then
\[\int_{B(X_0,R)\cap\Omega}|\nabla u|^2 \leq C_p R^{2/p'}\|N(\nabla u)\|_{L^p},\]
\[\int_{B(X_0,R)\cap\Omega}|\nabla u|^2 \leq C \|N(\nabla u)\|_{L^1}.\]
\end{lem}

\begin{proof}
First, note that if $N(\nabla u)\in L^p(\partial\V)$ for $p\geq 1$, then by (\ref{eqn:nablau}),
\[|\nabla u(X)|\leq C\|N(\nabla u)\|_{L^p} \min(\sigma(\partial\V),\dist(X,\partial\V))^{-1/p}\]
and so $\nabla u\in L^\infty_{loc}(\V)\supset L^2_{loc}(\V)$. We need this result for~$\overline\V$.

By looking at sufficiently small neighborhoods of points in $\partial\V$, we see that it suffices to establish that $\nabla u\in L^2_{loc}(\Omega)$ for special Lipschitz domains $\Omega$.

If $p>1$, then
\begin{align*}
\int_{0}^R\int_{x_0}^{x_0+R} |\nabla u(\psi(x,t))|^2\,dx\,dt
&\leq C\int_{0}^R\int_{x_0}^{x_0+R} |\nabla u(\psi(x,t))|t^{-1/p}\,dx\,dt
\\&\leq C_p R^{1/p'}\int_{x_0}^{x_0+R} |N(\nabla u)(\psi(x))|\,dx
\\&\leq C_p R^{2/p'}\|N(\nabla u)\|_{L^p}
\end{align*}
and so $\nabla u\in L^2_{loc}(\bar \Omega)$.

If $N(\nabla u)\in L^1(\partial\Omega)$, define
\[
E(\alpha)=\{X\in\Omega:|\nabla u|>\alpha\},
\quad
e(\alpha)=\{X\in\partial\Omega:N(\nabla u)>\alpha\}.
\]
Note that $\alpha\,\sigma(e(\alpha))\leq \|N(\nabla u)\|_{L^1}$.

If $X\in E(\alpha)$, then $X\notin\gamma(Y)$ for any $Y\in\partial\Omega\setminus e(\alpha)$, and if $X^*\in\partial\Omega$ is the closest point to $X$, then $X^*\in e(\alpha)$. So
\[\dist(X,\partial\Omega)+\frac{1}{2}\sigma(e(\alpha))
\geq \dist(X,\partial\Omega\setminus e(\alpha))
\geq (1+a)\dist(X,\partial\Omega)\]
and so $\dist(X,\partial\Omega)\leq C\sigma(e(\alpha))$.
So $E(\alpha)\subset\psi(\R\times(0, C\sigma(e(\alpha)) ))$.

Now, we want to bound $|E(\alpha)|$. But since $\psi(x,t)\in E(\alpha)$ implies that $\psi(x)\in e(\alpha)$, we have that $E(\alpha)\subset\psi(e(\alpha)\times(0, C\sigma(e(\alpha)) ))$
which implies that
\[|E(\alpha)|\leq  C\sigma(e(\alpha))^2\]
where $|\cdot|$ is used for Lebesgue measure in $\R^2$ and $\sigma$ is used for surface measure.

Now, 
\begin{align*}
\int_{\Omega} |\nabla u|^2
&=\int_0^\infty 2\alpha |E(\alpha)|\,d\alpha
\leq C\int_0^\infty 2\alpha \,\sigma(e(\alpha))^2\,d\alpha
\\&\leq C\|N(\nabla u)\|_{L^1}\int_0^\infty  \sigma(e(\alpha))\,d\alpha
= C\|N(\nabla u)\|_{L^1}^2.
\end{align*}

We now must establish that $u\in L^\infty_{loc}$. If $p>1$ then $u$ is H\"older continuous in $\bar\V$. Otherwise, $u$ is continuous on compact subsets of $\V$ because $\nabla u$ is bounded; we need only look at a small neighborhood of the boundary, and so we need only consider $\V=\Omega$ a special Lipschitz domain.

For some $X_0=\psi(x_0)\in\partial\Omega$, $N(\nabla u)(X_0)$ is finite. Then for any $t$, $|u(\psi(x_0,t))-u(X_0)|\leq tN(\nabla u)(X_0)$ is finite.

Now, for any $y\in\R$ and any $s$,
\begin{align*}
|u(\psi(y,s))-u(X_0)|
&\leq |u(\psi(y,s))-u(\psi(x_0,s))|+|u(\psi(x_0,s))-u(\psi(x_0))|
\\&\leq s N(\nabla u)(X_0) + \int_{\partial\Omega+s\e} |\nabla u|\,d\sigma
\\&\leq s N(\nabla u)(X_0) + \|N(\nabla u)\|_{L^1}.
\end{align*}
Thus we are done.
\end{proof}

\section{Uniqueness for the Neumann and regularity problems}

\begin{thm} \label{thm:NRunique} Suppose that $\div A\nabla u=0$ in $\V$ for some Lipschitz domain~$\V$. Assume that either $\nu\cdot A\nabla u=0$ on $\partial\V$ or $u\equiv C$ on $\partial\V$ for some constant~$C$. 

If $\V$ is a bounded Lipschitz domain and $N(\nabla u)\in L^p(\partial\V)$ for $1\leq p\leq \infty$, then $u$ is a constant.

If $\V=\Omega$ is a special Lipschitz domain, then there is some $p_0>1$ depending only on ellipticity and the Lipschitz constant of $\Omega$ such that, if $N(\nabla u)\in L^p(\partial\Omega)$ for some $1\leq p\leq p_0$, then $u$ is a constant.

If $\R^2\setminus\V$ is bounded, $N(\nabla u)\in L^p(\partial\V)$ for $1\leq p\leq \infty$ and $\lim_{|X|\to\infty} u(X)$ exists, then $u$ is a constant.
\end{thm}

In the regularity case we may assume without loss of generality that $u\equiv 0$ on $\partial\V$. 

\begin{proof} We begin with domains $\V$ with compact boundary.

Fix some $\zeta>0$ and define $R_0$ as follows. If $\V$ is bounded, simply let $R_0$ be large enough that $\V\subset B(0,R_0)$.
Otherwise, $u_\infty=\lim_{|X|\to\infty} u(X)$ exists, and so there is some $R_0>0$ such that if $|X|>R_0$, then $|u(X)-u_\infty|<\zeta$ and so if $R>4R_0$, then 
\[\int_{B(0,2R)\setminus B(0,R)}|\nabla u|^2
\leq\frac{C}{R^2}\int_{B(0,4R)\setminus B(0,R/2)}|u-u_\infty|^2
\leq C\zeta^2.\]

Consider the Neumann case $\nu\cdot A\nabla u=0$ first. By \lemref{L2loc} $\nabla u\in W^{1,2}(\bar\V\cap B(0,R))$. So for any $\epsilon>0$ there is some $\eta\in C^\infty_0(B(0,2R))$ such that $\|\nabla \eta - \nabla u\|_{W^{1,2}(\V\cap B(0,R))} < \epsilon$. (See \cite[p.~252]{Evans}.)  We may further require that 
\[\|\nabla \eta\|_{L^2(B(0,2R)\setminus B(0,R))}\leq C\|\nabla u\|_{L^2(B(0,2R)\setminus B(0,R))}+C\|u\|_{L^2(B(0,2R)\setminus B(0,R))}/R.\]

But by the weak definition of $\nu\cdot A\nabla u=0$, we have that
\[\int_{\V\cap B(0,2R)} \nabla\bar\eta\cdot A\nabla u=0\]
and so
\begin{align*}
\left|\int_{\V\cap B(0,2R)} \nabla\bar u\cdot A\nabla u\right|
&=\left|\int_{\V\cap B(0,2R)} (\nabla\bar u-\nabla\bar\eta)\cdot A\nabla u\right|
\\&\leq
\int_{\V\cap B(0,R)} |\nabla\bar u-\nabla\bar\eta| |A\nabla u|
\\&\qquad+
\int_{\V\cap B(0,2R)\setminus B(0,R)}(|\nabla \eta|+|\nabla u|)|\nabla u|
\\&
\leq 
C\epsilon \|\nabla u\|_{L^2({\V\cap B(0,R)})}
+ C\zeta^2.
\end{align*}
We first take the limit as $\epsilon\to 0$ to eliminate the first term; we then take the limit as $R\to \infty$, which forces $\zeta\to 0$, which eliminates the second term. Thus by ellipticity of $A$ $\nabla u=0$ almost everywhere, so $u$ must be a constant.

Next, consider the regularity case for $\partial\V$ bounded. If $N(\nabla u)\in L^p$ then $N(\nabla u)$ is finite a.e., and if $N(\nabla u)(X)$ is finite then $\lim_{Z\to X \text{ n.t.}} u(Z)$ exists; so we have that $\lim_{Z\to X \text{ n.t.}} u(Z)=0$ a.e.\ $\partial\V$.

We will again show that $\int_{\V}\nabla\bar u\cdot A\nabla u=0$, but instead of approximating $u$ by smooth functions $\eta$, we will work on subsets of $\V$.

Pick any $\epsilon>0$. I want to show that there is some $\delta>0$ such that if $\dist(X,\partial\Omega)<\delta$ then $|u(X)|<\epsilon$.

Since $N(\nabla u)\in L^p(\partial\V)$, there is some ${\delta'}$ such that \[\int_{B(X_0,C{\delta'})\cap\partial\V} |N(\nabla u)|\,d\sigma <\epsilon/2\] for all $X_0\in\partial\V$, for some $C$ to be chosen later.

Cover $\{X\in\V:\dist(X,\partial\V)<{\delta'}/2\}$ by finitely many balls $B_i=B(X_i,{\delta'})$, and require that $N(\nabla u)(X_i)$ is finite, and that $\lim_{Z\to X_i\text{ n.t.}} u(Z)=0$ for each $X_i$. Assume $\delta'$ is small enough that $B(X_i,2C\delta')\cap\V=B(X_i,2C\delta')\cap\Omega_i$ for some special Lipschitz domains $\Omega_i$.

As in the proof of \lemref{L2loc}, if $X=\psi(x,s)$ is in $B(X_i,\delta')\cap\V$, then
\[|u(X)-u(X_i)|\leq sN(\nabla u)(X_i)+\|N(\nabla u)\|_{L^1({B(X_0,C{\delta'})\cap\partial\V})}.\]
Let $\delta_i=\epsilon/C N(\nabla u)(X_i)$, and let $\delta=\min(\delta',\min_i\delta_i)$.

Pick some $\epsilon>0$. 
Let $\V_\delta=\{X\in\V:\dist(X,\partial\V)>\delta\}$. Note that if $\delta$ is small enough, then $\partial\V_\delta$ lies in the boundary cylinders for $\V$, and so its Lipschitz constants are controlled by those of~$\V$.

Then $u$ is continuous with bounded gradient on $\V_\delta$. So we have that
\[\int_{\V_\delta\cap B(0,R)} \nabla \bar u\cdot A\nabla u =\int_{\partial\V_\delta \cup\partial B(0,R)} \bar u \, \nu\cdot A\nabla u\,d\sigma.\]
But since $N(\nabla u)\in L^1(\partial\V)$, we have that $\nabla u\in L^1(\partial\V_\delta)$ independently of $\delta$. Since $|u|<\epsilon$ on $\partial\V_\delta$, $|u|<\zeta$, $|\nabla u|<C\zeta/R$ on $\partial B(0,R)$, we have that this goes to 0 as $\delta\to 0$ and $R\to\infty$; therefore $\int_{\V} \nabla\bar u\cdot A\nabla u=0$, as desired.

We now consider the case where $\V=\Omega$ is a special Lipschitz domain. Since $N(\nabla u)\in L^p$, for any fixed $\epsilon$, $R_0>0$, there must be some $R>R_0$ such that $N(\nabla u)(\psi(\pm R))\leq \epsilon R^{-1/p}$.

Pick some $R_0$, $\epsilon$, and let $\Omega_R=\psi((-R,R)\times(0,CR))$, where $C$ is large enough that $\partial\Omega_R\subset \gamma(\psi(R))\cup\gamma(\psi(-R))$. Then 
\[\|\nu\cdot A\nabla u\|_{L^p(\partial\V_{R,\epsilon})}\leq C\epsilon
\text{ or }\|\tau\cdot \nabla u\|_{L^p(\partial\V_{R,\epsilon})}\leq C\epsilon\]
(depending on whether $\nu\cdot A\nabla u=0$ or $\tau\cdot\nabla u=0$ on $\partial\Omega$).

Then since uniqueness holds in bounded Lipschitz domains, we must have that if $p$ is small enough, $\|N(\nabla u)\|_{L^p(\partial\Omega_R)}\leq C\epsilon$. So $|\nabla u(X)|\leq C\epsilon \dist(X,\partial\Omega)^{-1/p}$ for all $|X|\leq R_0/C$; by taking the limits as $R_0\to\infty$ and $\epsilon\to 0$, we see that $\nabla u\equiv 0$, as desired.
%
%
%
\end{proof}

\section{$L^p$ uniqueness for the Dirichlet problem}

\begin{thm}\label{thm:Dpunique} Let $\V$ be a good Lipschitz domain. 
Assume that $p<\infty$ is large enough that $(R)^{A^T}_{p'}$ holds in all bounded Lipschitz domains with constants at most $C(k_i)$, where the $k_i$ are the Lipschitz constants of~$\V$.

Then if $\div A\nabla u=0$ in $\V$,
$Nu\in L^p(\partial\V)$, and $u\equiv 0$ on $\partial\V$, then $u\equiv 0$ in~$\V$. 
\end{thm}

\begin{proof}
Suppose first that $\partial\V$ is bounded, and define $\V_\delta=\{X\in\V:\dist(X,\partial\V)>\delta\}$ as before.

By the dominated convergence theorem, for any $\epsilon>0$ we can find some $\delta_0$ such that, if $\delta<\delta_0$, then $\|u\|_{L^p(\partial\V_\delta)}<\epsilon$.

Let $f_\delta=u|_{\partial\V_\delta}$. By continuity of $u$, we know that $f_\delta$ is bounded (if large).

Pick $R$ large enough that $\partial\V\subset B(0,R/2)$. (This will be irrelevant in the case where $\V$ is also bounded.)
Let $v_\delta (X) = u(X)\eta_\delta(X)$, where $\eta_\delta\equiv 1$ on $\V_\delta\cap B(0,R)$ and $\eta_\delta\in C^\infty_0(\V_{\delta/4}\cap B(0,2R))$. We may extend $v_\delta$ by 0; we then have that $v_\delta$ is continuous on $\R^2$, compactly supported, and has a bounded gradient.

Therefore, if $X\in\V_\delta$, then
\begin{align*}
u(X) &= v_\delta(X) = -\int \nabla v_\delta\cdot A^T\nabla\Gamma_X^T
\\&=-\int_{\V_\delta} A\nabla v_\delta\cdot \nabla\Gamma_X^T
-\int_{\V_\delta^C} \nabla v_\delta\cdot A^T\nabla\Gamma_X^T
\\&=-\int_{\partial\V_\delta} \Gamma_X^T \nu\cdot A\nabla v_\delta\,d\sigma
+\int_{\partial\V_\delta} v_\delta \nu\cdot A^T\nabla\Gamma_X^T\,d\sigma
\end{align*}
Since $|\nabla\Gamma_X(Y)|\leq \frac{C}{|X-Y|}$, we have that on $\partial\V_\delta$, $\tau\cdot\nabla\Gamma_X$ is bounded and in~$H^1$. Let $\Phi_X$ be the solution to $(R)_{p'}^{A^T}$ in $\V_\delta$ with boundary data $\Phi_X=\Gamma_X^T$. This means that $N(\nabla \Phi_X)\in L^{p'}$ (and $\lim_{Y\to\infty}|\Phi_X(Y)|$ exists), and so $\Phi_X$ is bounded in $\V_\delta$ and $\nabla\Phi_X\in L^2_{loc}$.

Therefore, 
\begin{align*}
u(X)&=-\int_{\partial\V_\delta} \Phi_X \nu\cdot A\nabla v_\delta\,d\sigma
+\int_{\partial\V_\delta} u \nu\cdot A^T\nabla\Gamma_X^T\,d\sigma
\\&=-\int_{\V_\delta} \nabla\Phi_X \cdot A\nabla v_\delta
+\int_{\partial\V_\delta} u \nu\cdot A^T\nabla\Gamma_X^T\,d\sigma
\\&=-\int_{\partial\V_\delta} v_\delta \nu \cdot A^T\nabla\Phi_X\,d\sigma
+\int_{\partial\V_\delta} u \nu\cdot A^T\nabla\Gamma_X^T\,d\sigma
\\&=\int_{\partial\V_\delta} u \, \nu\cdot A^T\nabla(\Gamma_X^T-\Phi_X)\,d\sigma
\end{align*}
Now, we know that $|\nabla\Gamma_X^T(Y)|\leq \frac{C}{|X-Y|}$, so $\|\nu\cdot A^T\nabla\Gamma_X^T\|_{L^{p'}}\leq C\frac{\sigma(\partial\V_\delta)^{1/p'}}{\dist(X,\partial\V_\delta)}.$
But since $\Phi_X$ is a regularity solution, the same applies to it; therefore, \[|u(X)|\leq C\|f_\delta\|_{L^p}\frac{\sigma(\partial\V_\delta)^{1/p'}}{\dist(X,\partial\V_\delta)}\] and so taking the limit as $\delta\to 0$ yields that $u(X)=0$, as desired.

If $\V=\Omega$ is a special Lipschitz domain, then we may proceed as in the proof of \thmref{NRunique}, by constructing a large finite subdomain with $\|Nu\|_{L^p}$ small on its boundary. If $p$ is large enough that $(D)_{p}^{A}$ holds in all such subdomains, we see that $u$ must be small pointwise, as desired. 
\end{proof}

\section{Mixed uniqueness}

In this section, we show that Dirichlet and regularity solutions are the same.

\begin{thm} \label{thm:mixedunique} Suppose that $\V$ is a bounded or special Lipschitz domain.

Let $p$ be as in \thmref{Dpunique}, $q=p'$.
If $\div A\nabla u=0=\div A\nabla v$ in~$\V$ and  $u|_{\partial\V}=f=v|_{\partial\V}$, and if $Nu\in L^p$, $N(\nabla v)\in L^q$, then $u\equiv v$.

If $\partial\V$ is bounded but $\V$ is not, we as usual must further require that $\lim_{X\to\infty} v(X)$ exists.
\end{thm}

\begin{proof} By \lemref{L2loc}, $v\in L^\infty_{loc}$; if $\V$ is bounded then $v\in L^\infty$, and if $\V^C$ is bounded then $v$ is bounded on $B(0,R)\cap\V$ for some $R$ so large that $|v-v_\infty|<1$ on $B(0,R)^C$, and so $v$ is globally bounded. But then $N v$ is bounded. Since $\partial\V$ is bounded, $v$ is a Dirichlet solution and we need only apply \thmref{Dpunique}.

If $\V=\Omega$ is a special Lipschitz domain, note that for every $R_0>0$, $\epsilon>0$, there is some $R>R_0$ such that $Nu(\psi(\pm R))<\epsilon R^{-1/p}$, $N(\nabla v)(\pm R)<\epsilon R^{-1/q}$.

Let $\Omega_R=\psi((-R,R)\times(0,CR))$, where $C$ is large enough that $\partial\Omega_R\setminus\partial\Omega\subset\gamma(\psi(R))\cup\gamma(\psi(-R))$.

Define $u_R$ as follows: $\div A\nabla u_R=0$ in $\Omega_R$, $u_R=u=v$ on $\partial\Omega_R\cap\partial\Omega$, $u_R(X)=u(\psi(\pm R))(1-\dist(X,\psi(\pm R)))$ on $B(\psi(R),1)\cap\partial\Omega_R\setminus\partial\Omega$, and $u_R=0$ elsewhere on $\partial\Omega_R$. 

Since $\Omega_R$ is bounded, 
\[\|Nu_R\|_{L^p(\partial \Omega_R)}\leq \|u_R\|_{L^p(\partial \Omega_R)},\quad\|N(\nabla u_R)\|_{L^q(\partial \Omega_R)}\leq \|\partial_\tau u_R\|_{L^q(\partial \Omega_R)}.\]

I claim that as $R_0\to\infty$ and $\epsilon\to 0$, $\nabla u_R(X)$ approaches both $\nabla u(X)$ and $\nabla v(X)$ pointwise (not uniformly); this suffices to show that $\nabla u\equiv\nabla v$, and so $u\equiv v$ up to an additive constant (which must be~0).

First, note that 
\[\|u-u_R\|_{L^p(\partial\Omega_R)}
=\|u-u_R\|_{L^p(\partial\Omega_R\setminus\partial\Omega)}
\leq \left(\epsilon^p/R\right)^{1/p}+\|u\|_{L^p(\partial\Omega_R\setminus\partial\Omega)}\leq C\epsilon,\] and so by \thmref{Dpunique}, $\|N(u-u_R)\|_{L^p}\leq C\epsilon$; therefore, if $X\in\Omega_{R/C}$ for $C$ large enough,
\[|u(X)-u_R(X)|\leq C\epsilon\dist(X,\partial\Omega_R)^{-1/p} =C\epsilon\dist(X,\partial\Omega)^{-1/p}.\]

Therefore, by \lemref{PDE1},
\[|\nabla u(X)-\nabla u_R(X)|\leq C\epsilon\dist(X,\partial\Omega)^{-1-1/p}.\]

Next, note that if $q>1$,
\begin{align*}
\|\partial_\tau v-\partial_\tau u_R\|_{L^q(\partial\Omega_R)}
&=\|\partial_\tau v-\partial_\tau u_R\|_{L^q(\partial\Omega_R\setminus\partial\Omega)}
\\&\leq \left(\epsilon^q/R^{q/p}\right)^{1/q}+\|\partial_\tau v\|_{L^q(\partial\Omega_R\setminus\partial\Omega)}\leq C\epsilon.
\end{align*}
So by \thmref{NRunique}, $\|N(\nabla v-\nabla u_R)\|_{L^q}\leq C\epsilon$.

If $q=1$ then $|\partial_\tau u_R|=|u(\psi(\pm R))|<\epsilon$ on $B(\psi(\pm R),1)\cap\partial\Omega_R\setminus\partial\Omega$ and is zero elsewhere on $\partial\Omega_R\setminus\partial\Omega$. $|\partial_\tau v|<\epsilon/R$ on $\partial\Omega_R\setminus\partial\Omega$, a set of size at most $CR$. So
\begin{align*}
\|\partial_{\tau} u_R-\partial_{\tau} v\|_{H^1(\partial\Omega_R)}
&=\|\partial_{\tau} u_R-\partial_{\tau} v\|_{H^1(\partial\Omega_R\setminus\partial\Omega)}
\\&\leq 
\|\partial_{\tau} u_R\|_{H^1(\partial\Omega_R\setminus\partial\Omega)}
+\|\partial_{\tau} v\|_{H^1(\partial\Omega_R\setminus\partial\Omega)}
\leq C\epsilon.
\end{align*}
So $\|N(\nabla u_R-\nabla v)\|_{L^1(\partial\Omega_R)}$.

In either case, 
\[|\nabla u(X)-\nabla u_R(X)|\leq C\epsilon\dist(X,\partial\Omega_R)^{-1/q} =C\epsilon\dist(X,\partial\Omega)^{-1/q}\]
for $X\in \Omega_{R/C}$.

Thus, by letting $\epsilon\to 0$, we see that $\nabla u(X)=\nabla v(X)$, as desired.
\end{proof}

\section{Square-function converse and uniqueness}

In this section, we will prove the following theorem:
\begin{thm}\label{thm:Carlesonmax} Suppose that $\V$ is a bounded Lipschitz domain, $\div A\nabla u=0$ in $\V$, and
\begin{equation}
\label{eqn:squareconverse}
\sup_{Y_0\in\partial\V,R>0}\frac{1}{\sigma(\partial\V\cap B(Y_0,R))}\int_{\V\cap B(Y_0,R)} |\nabla u(Y)|^2 \dist(Y,\partial\V)\,dY\leq \tilde C^2.
\end{equation}

If $\|A-A_0\|_{L^\infty}$ is small enough, $p$ is large enough that $\K$ is invertible on $L^p(\partial\V)$, and if $X_0\in \V$ with $\dist(X,\partial\V)\approx \sigma(\partial\V)$, then
\[\dashint_{\partial\V} N(u-u(X_0))^p\,d\sigma\leq C\tilde C.\]
\end{thm}

To prove this theorem, we need the following lemma:
\begin{lem}\label{lem:squaresubdomain} If (\ref{eqn:squareconverse}) holds in $\V$, and $\U\subset\V$ is a bounded Lipschitz domain, then (\ref{eqn:squareconverse}) holds in $\U$ as well, with constants that may depend on the Lipschitz constants of~$\V$.
\end{lem}

The theorem and lemma together have a corollary:
\begin{cor}\label{cor:BMOconverse} Suppose that $\V$ is a Lipschitz domain, $\div A\nabla u=0$ in $\V$, and $u$ satisfies (\ref{eqn:squareconverse}).

If $\|A-A_0\|_{L^\infty}$ is small enough, then $u|_{\partial\V}\in BMO(\partial\V)$ with $BMO$ norm at most $C\tilde C$.

\end{cor}

\begin{proof}[Proof of \lemref{squaresubdomain}] Let $Y_0\in\partial\U$, and let $R>0$. Note that if $Y\in\U$ then $\dist(Y,\partial\U)\leq\dist(Y,\partial\V)$. Since $\U$ is bounded, we may assume that $R\leq \sigma(\partial\U)$, and so $R\leq C\sigma(\partial\U\cap B(Y_0,R))$.

If $Y_0\in\partial\V$, then 
\[\int_{\U\cap B(Y_0,R)} |\nabla u(Y)|^2 \dist(Y,\partial\U)\,dY
\leq \int_{\V\cap B(Y_0,R)} |\nabla u(Y)|^2 \dist(Y,\partial\V)\,dY.\]
Furthermore, $\sigma(\partial\V\cap B(Y_0,R))\leq CR\leq C\sigma(\partial\U\cap B(Y_0,R))$. So we are done.

If $Y_0\in\partial\U\setminus\partial\V$, then
let $Y_0^*\in\partial\V$ with $|Y_0-Y_0^*|=\dist(Y_0,\partial\V)$, and let $R^*=|Y_0-Y_0^*|+R$.

If $R<\frac{1}{2}\dist(Y_0,\partial\V)$, then for all $Y\in B(Y_0,R)$ we have that
\[\frac{1}{R}\dist(Y,\partial\U)\leq 1
\leq \frac{\dist(Y,\partial\V)}{|Y_0-Y_0^*|-R}
\leq 3\frac{\dist(Y,\partial\V)}{R^*}
\]
and so
\begin{align*}
\int_{\U\cap B(Y_0,R)} |\nabla u(Y)|^2 \dist(Y,\partial\U)\,dY
&\leq\int_{\U\cap B(Y_0,R)} |\nabla u(Y)|^2 \frac{1}{R}\dist(Y,\partial\U)\,dY
\\&\leq \frac{C}{R^*}\int_{\V\cap B(Y_0^*,R^*)} |\nabla u(Y)|^2 \dist(Y,\partial\V)\,dY.
\end{align*}
Since ${\sigma(\partial\V\cap B(Y_0,R))}\leq CR\leq C^2{\sigma(\partial\U\cap B(Y_0,R))}$, we are done.

Conversely, if $R>\frac{1}{2}\dist(Y_0,\partial\V)$, then 
\begin{align*}
\int_{\U\cap B(Y_0,R)} |\nabla u(Y)|^2 \dist(Y,\partial\U)\,dY
&\leq \int_{\V\cap B(Y_0^*,3R)} |\nabla u(Y)|^2 \dist(Y,\partial\V)\,dY
\end{align*}
and so we are done.
\end{proof}

\begin{proof}[Proof of \thmref{Carlesonmax}]
%
%

First, assume that $A$ is smooth and that $u=\D f$ for some $f\in BMO(\partial\V)$. Recall that $\K^t$ is bounded and invertible on $H^1$, so $\|\K f\|_{BMO}\approx \|f\|_{BMO}$.

We know from \cite[Section~3]{Dp'} that if $A_0$ is real, and $\div A_0\nabla v=0$ in a bounded Lipschitz domain $\W$, then
\[\int_{\partial\W} |N (v-v(X_0))|^2\,d\sigma(X)\leq C \int_{\W} |\nabla v(X)|^2 \dist(X,\partial\W)\,dX\]
for any $X_0\in\W$ with $\dist(X_0,\partial\W)\geq \frac{1}{C}\sigma(\partial\W)$.

This lets us bound $\|\K_0 f\|_{BMO}$ as follows. Pick some $\Delta\subset\partial\V$, and let $X_\Delta$ be such that $\dist(X_\Delta,\Delta)\approx\dist(X_\Delta,\partial\V)\approx\sigma(\Delta)$.
Let $\W_\Delta\subset\V$ be such that $\Delta\subset\partial\W_\Delta$, the Lipschitz constants of $\W_\Delta$ are controlled by those of~$\V$, and $X_\Delta\in\W_\Delta$ with $\dist(X_\Delta,\partial\W_\Delta)\geq \sigma(\Delta)/C$. Then
\begin{align*}
\left(\dashint_\Delta |\K_0 f-\D_0 f(X_\Delta)|\,d\sigma\right)^2
&\leq \dashint_\Delta |\K_0 f-\D_0 f(X_\Delta)|^2\,d\sigma
\\&\leq \frac{C}{\sigma(\Delta)}\int_{\partial\W_\Delta} N(\D_0 f-\D_0 f(X_\Delta))^2
\\&\leq \frac{C}{\sigma(\Delta)}\int_{\W_\Delta}|\nabla\D_0 f(Y)|^2\dist(Y,\partial\V)\,dY
\\&\leq \frac{C}{\sigma(\Delta)}\int_{B(X_0,C{\sigma(\Delta)})\cap\V}|\nabla\D_0 f(Y)|^2\dist(Y,\partial\V)\,dY. \end{align*}
This implies that
\[\|\K_0 f\|_{BMO(\partial\V)}^2\leq \sup_{X_0\in\partial\V,R>0}
\frac{C}{R}\int_{B(X_0,R)\cap\V}|\nabla\D_0 f(Y)|^2 \dist(Y,\partial\V)\,dY.\]

Recall from \autoref{chap:square} that
\[\sup_{X_0\in\partial\V,R>0}
\frac{1}{R}\int_{B(X_0,R)\cap\V}|\nabla\D f(Y)|^2\dist(Y,\partial\V)\,dY\leq C\|f\|_{BMO}.\]

So by analyticity, if we let $\alpha=\|A-A_0\|_{L^\infty} \|f\|_{BMO}^2\approx \|A-A_0\|_{L^\infty} \|\K f\|_{BMO}^2$,
\begin{align*}
\|\K f\|_{BMO}^2
&\leq \|\K_0 f\|_{BMO}^2 + C\alpha
\\&\leq
\sup_{X_0\in\partial\V,R>0}
\frac{C}{R}\int_{B(X_0,R)\cap\V}|\nabla\D_0 f(Y)|^2\dist(Y,\partial\V)\,dY
+ C\alpha
\\&\leq
\sup_{X_0\in\partial\V,R>0}
\frac{C}{R}\int_{B(X_0,R)\cap\V}|\nabla\D f(Y)|^2\dist(Y,\partial\V)\,dY
+ C\alpha
\\&=
\sup_{X_0\in\partial\V,R>0}
\frac{C}{R}\int_{B(X_0,R)\cap\V}|\nabla u(Y)|^2\dist(Y,\partial\V)\,dY
+ C\alpha
\\&\leq C\tilde C^2 +C\alpha \leq C\tilde C^2 +C\|A-A_0\|_{L^\infty} \|\K f\|_{BMO}^2.
\end{align*}
Thus if $\|A-A_0\|_{L^\infty}$ is small enough, we may hide the last term. Thus, $\K f=u|_{\partial\V}$ is in $BMO$ with norm $C\tilde C$.

But since $\partial\V$ is compact, by the John-Nirenberg inequality (\cite[p.~144]{stein}), if $1\leq p<\infty$, then for $C_f=\dashint_{\partial\V} \K f$ we have that
\[\|\K f-C_f\|_{L^p(\partial\V)}\leq C\|\K f\|_{BMO(\partial\V)}\sigma(\partial\V)^{1/p} \leq C\tilde C\sigma(\partial\V)^{1/p}.\]
So $\|f-C_f\|_{L^p}\leq C\tilde C\sigma(\partial\V)^{1/p}$, and so by \thmref{NDf}, $\|N(\D f-C_f)\|_{L^p}\leq C\tilde C \sigma(\partial\V)^{1/p}$.

We need this inequality with $C_f$ replaced by $\D f(X_0)$.

We know that
$\dashint_{\partial\V} |\K_0 f-\D_0 f(X_0)|\leq C\tilde C$. Take $\D_0 f(X_0)=0$. Then $\left|\dashint_{\partial\V} \K_0 f\right|\leq\dashint_{\partial\V} |\K_0 f|\leq C\tilde C$, so 
\[C\|f\|_{L^p}\leq C\|\K_0 f\|_{L^p}\leq C\tilde C\sigma(\partial\V)^{1/p}+C\|\K_0f-{\textstyle\dashint_{\partial\V}\K_0 f}\|_{L^p}\leq C\tilde C\sigma(\partial\V)^{1/p}.\] 
But $\dist(X_0,\partial\V)\approx \sigma(\partial\V)$ and so
$|\D f(X_0)|\leq \sigma(\partial\V)^{1/p'} \|f\|_{L^p}/\dist(X_0,\partial\V)
\leq C\tilde C$. Thus,
\[\|\K f-\D f(X_0)\|_{L^p(\partial\V)}\leq C\tilde C\sigma(\partial\V)^{1/p}\]
provided $A$ is smooth.

We now move on to more general $u$ and $A$.
%
%
%
%
%
Let $\V_\delta=\{X\in\V:\dist(X,\partial\V)>\delta\}$. By \lemref{squaresubdomain}, $u$ satisfies (\ref{eqn:squareconverse}) in $\V_\delta$. But $\overline\V_\delta\subset\V$, so $\nabla u$ is bounded on $\overline\V_\delta$; thus $g_\delta=u|_{\partial\V_\delta}\in L^\infty\subset BMO$. Furthermore, $\partial_\tau g_\delta$ is bounded.

Let $A^\eta$ be as in \autoref{chap:apriori}, and let $\D^\eta$, $\K^\eta$ be potentials with $A^\eta$ instead of~$A$. 

Then $g_\delta = \K^\eta_{\V_\delta} f^\eta_\delta$ for some $f^\eta_\delta\in BMO\cap L^p$. Let $u^\eta_\delta = \D^\eta_{\V_\delta} f^\eta_\delta$. 

Then $u^\eta_\delta = u$ on $\partial\V_\delta$. Furthermore, $\partial_\tau g_\delta\in L^{q}$ for some $q$ small enough that $(R)^A_q$ holds in~$\V_\delta$.

So if $X\in\partial\V_\delta$ and $R>0$, then
\begin{multline*}
\int_{B(X,R)\cap\V_\delta} |\nabla u^\eta_\delta(X)|^2\dist(X,\partial\V_\delta)\,dX
\\\begin{aligned}&\leq
\int_{B(X,R)\cap\V_{\delta+\epsilon}} |\nabla u^\eta_\delta(X)|^2\dist(X,\partial\V)\,dX
+\int_{\V_{\delta+\epsilon}\setminus \V_\delta} |\nabla u^\eta_\delta(X)|^2\dist(X,\partial\V)\,dX
\\&\leq
2\int_{B(X,R)\cap\V_{\delta+\epsilon}} (|\nabla u^\eta_\delta(X)-\nabla u(X)|^2+|\nabla u(X)|^2)\dist(X,\partial\V)\,dX
\\&\qquad
+\int_{\V_{\delta+\epsilon}\setminus \V_\delta} 
|\nabla u^\eta_\delta(X)|\dist(X,\partial\V)^{1-1/q}\,dX
\\&\leq
2\int_{B(X,R)\cap\V_{\delta+\epsilon}} |\nabla u^\eta_\delta(X)-\nabla u(X)|^2\dist(X,\partial\V)\,dX
\\&\qquad
+C\tilde C R
+\epsilon(\delta+\epsilon)^{1-1/q}\int_{\partial\V_\delta} 
N(\nabla u^\eta_\delta)\,d\sigma
\end{aligned}
\end{multline*}
By taking $\eta$ and $\epsilon$ small enough, and dealing with the first term as in \autoref{chap:apriori}, we have that
\[\frac{1}{R}\int_{B(X,R)\cap\V_\delta} |\nabla u^\eta_\delta(X)|^2\dist(X,\partial\V_\delta)\,dX\leq C\tilde C.
\]
So we have that
\[\dashint_{\partial\V_\delta} N(u_\delta^\eta-u_\delta^\eta(X_0))^p\,d\sigma
\leq C\tilde C.\]

As in \thmref{apriori:Dp}, $u_\delta^\eta(X_0)-u(X_0)\to 0$ as $\eta\to 0$, so 
\[\dashint_{\partial\V_\delta} N(u_\delta^\eta-u(X_0))^p\,d\sigma
\leq C\tilde C.\]
Thus $u_\delta^\eta-u(X_0)|_{\partial\V}\in L^p$ with norm at most $C\tilde C\sigma(\partial\V)^{1/p}$; but $u_\delta^\eta=u$ on $\partial\V_\delta$.

So since $(D)^A_p$ holds in $\V$,
\[\dashint_{\partial\V_\delta} N(u-u(X_0))^p\,d\sigma
\leq C\tilde C.\]
Letting $\delta\to 0$ completes the proof.
%
%
%
%
%
%
\end{proof}

\begin{cor}\label{cor:BMOunique} Suppose that $u$ satisfies (\ref{eqn:squareconverse}) and that $u|_{\partial\V}$ is a constant. Then $u$ is constant.\end{cor} 

By the previous theorem, we know that $u$ is bounded. We need only prove the following lemma:

\begin{lem}\label{lem:bddmeansconst} If $\V$ is a good Lipschitz domain and $u$ is a bounded solution to $\div A\nabla u=0$ in $\V$,
and if $u\equiv 0$ on $\partial\V$, then $u\equiv 0$ in~$\V$. 
\end{lem}

\begin{proof}
If $\partial\V$ is bounded then the theorem follows trivially from \thmref{Dpunique} as usual. Suppose that $\V=\Omega$ is a special Lipschitz domain. Without loss of generality assume that $\|u\|_{L^\infty}\leq 1$.

If we knew, a priori, that $\nabla u\in L^2_{loc}$, then we could use \lemref{PDE1} to show that $\|\nabla u\|_{L^2(\Omega)}\leq C$. Unfortunately, we do not.

Define $\Q(R)=\psi((-R,R)\times(0,R))$.
For every $\epsilon$, $R>0$, let $u_{R,\epsilon}=u$ on $\partial\Q(R)$, except on $\{X\in\partial\Q(R):0<\dist(X,\partial\Omega)<\epsilon\}$; on this set $u_{R,\epsilon}$ is to increase from $0$ to $u$ smoothly. Take $\div A\nabla u_{R,\epsilon}=0$. By \autoref{chap:maximum}, $\|u_{R,\epsilon}\|_{L^\infty(\Q(R))}\leq C$.

Since $u$ is bounded, we know that $\|N_{\Q(R)}(u-u_{R,\epsilon})\|_{L^p}\leq C\epsilon^{1/p}$; so if we can show that $u_{R,\epsilon}$ is small, then by letting $\epsilon\to 0$ and $R\to\infty$ we will show that $u\equiv 0$.

From \lemref{PDE1} and (\ref{eqn:gradu}), we know that $|\nabla u(X)|\leq C/\dist(X,\partial\V)$. So $|\partial_\tau u_{R,\epsilon}|\leq C/\epsilon$ on $\partial\Q(R)$. So we know from \thmref{mixedunique} that $u_{R,\epsilon}$ must equal the regularity solution, and so by \lemref{L2loc} $\nabla u_{R,\epsilon}\in L^2(\Q(R))$.

Therefore, by \lemref{PDE1}, we know that 
\[\int_{Q(R/2)} |\nabla u_{R,\epsilon}|^2\leq C.\]

Pick some $r>0$ large, $\delta>0$ small, and let $R=2r e^{1/\delta}$,  $\zeta(\rho)=\int_{\Omega\cap\partial Q(\rho)} |\nabla u_{R,\epsilon}|^2\,d\sigma$. Then 
\[\int_{r}^{R/2}\zeta(\rho)\,d\rho =\int_{\Q(R/2)\setminus\Q(\rho)}|\nabla u_{R,\epsilon}|^2\leq C.\]
So since $\int_r^{R/2} \frac{\delta}{\rho}\,d\rho=1$, there must be some $\rho\in (r,R/2)$ such that $\zeta(\rho)\leq C\delta/\rho$. But then 
\begin{align*}
\int_{\Omega\cap\partial \Q(\rho)} |\nabla u_{R,\epsilon}|
&\leq C\rho\dashint_{\Omega\cap\partial \Q(\rho)} |\nabla u_{R,\epsilon}|
\\&\leq C\rho\left(\dashint_{\Omega\cap\partial \Q(\rho)} |\nabla u_{R,\epsilon}|^2\right)^{1/2}
\leq C\sqrt{\rho\zeta(\rho)}\end{align*}
and so $|u_{R,\epsilon}|\leq C\sqrt{\delta}$ on ${\partial\Q(X,\rho)}$. From \autoref{chap:maximum}, this means that $|u_{R,\epsilon}|\leq C\sqrt{\delta}$ on $\Q(r)$, and so by making $\delta$, $\epsilon$ small, we may force $u$ small, as desired.
\end{proof}

We may now use a similar technique to prove a more general version of \thmref{apriori:BMOcpt}.

\begin{thm}\label{thm:apriori:BMO} Suppose that $\Omega$ is a special Lipschitz domain, $A$ satisfies (\ref{eqn:elliptic}), and $f\in BMO(\partial\Omega)$. Then there exists some $\epsilon_0>0$ small such that, if $\|\Im A\|_{L^\infty}<\epsilon_0$, then there is some $u$ a solution to \thmref{BMO} with boundary data~$f$. \end{thm}

\begin{proof} By \thmref{apriori:BMOcpt}, this is true if $f$ is compactly supported. Without loss of generality take $\|f\|_{BMO}=1$.

Let $f_R = f-\dashint_{\psi((-R,R))}f$ on $\psi((-R,R))$, $f_R$ compactly supported with $BMO$ norm at most~$C$. Let $u_R$ be the solution to \thmref{BMO} with boundary data~$f_R$.

Pick some $\delta>0$ small and some $r>0$ large. Let $R$, $S>2r e^{1/\delta}$ with $R<S$. Let $v=u_R-u_S-C$, where $C$ is a constant chosen such that $v\equiv 0$ on $\psi((-R,R))$. Then by \thmref{Carlesonmax}, $|v|\leq C$ on $\psi((-R,R)\times(0,R))$.

Define $Q(R)=\psi((-R,R)\times(0,R))$, and let $v_\epsilon=v$ on $\partial Q(R)$ except on $\{X\in\partial Q(R):0<\dist(X,\partial\Omega)<\epsilon\}$, as in the proof of \lemref{bddmeansconst}.

Then if $X\in Q(R)$, $|v(X)-v_\epsilon(X)|\leq C\epsilon^{1/p}\dist(X,\partial\Q(R))^{-1/p}$ and $|v_\epsilon(X)|\leq C$, so again $|v_\epsilon|\leq C\sqrt{\delta}$ on $\Q(r)$.

By letting $\epsilon\to 0$, we see that $|v(X)|\leq C\sqrt{\delta}$ on $\Q(r)$. Let $w_R=u_R+\dashint_{\psi((-R,R))}f$. Then $\lim_{R\to\infty} w_R$ exists, and $w_R$ converges almost uniformly; so a limit $u$ must exist, and moreover must have $u|_{\partial\Omega}=f$ and satisfy (\ref{eqn:uCarleson}), as desired.
\end{proof}

\section{The $L^1$ converse}

We consider first the Neumann problem.

\begin{lem} \label{lem:BMOhasCarleson} Suppose that $\V$ is a good Lipschitz domain with connected boundary and that $f\in BMO(\partial\V)$. Then there exists some function $F$ in $W^{1,\infty}_{loc}(\V)$ such that $F\to f$ nontangentially a.e.\ in $\partial\V$, and such that
\[\|\nabla F\|_{\mathcal{C}}=\sup_{X_0\in\partial\V,R>0}
\frac{1}{\sigma({B(X_0,R)\cap\partial\V})}\int_{B(X_0,R)\cap\V} |\nabla F| \leq C\|f\|_{BMO}\]
where the constant $C$ depends only on the Lipschitz constants of~$\V$.

Furthermore, if $X_0\in\V$, then
\[\dashint_{B(X_0,C_1\dist(X_0,\partial\V))\cap\partial\V} |f-F(X_0)|\,d\sigma\leq C(C_1)\|f\|_{BMO}\]
and if $f$ is Lipschitz, we may assume that $F$ is as well. If $\partial\V$ is bounded, then $F-\dashint_{\partial\V} f$ is compactly supported; if $\V=\Omega$ is a special Lipschitz domain and $f$ is compactly supported, then $F$ is as well.
\end{lem}


%

\begin{proof} We first consider $\R^2_+$; by \cite{bmo2} the first part of this lemma holds in $\R^2_+$, and furthermore, $|\nabla F(x,t)|\leq C\|f\|_{BMO}/t$.



The construction may be summarized as follows: 
as in \cite[Section~2]{bmo1} if $f\in BMO(\R)$ and $I\subset\R$ is a dyadic interval, then there exists a family $\Omega=\Omega(I)$ of dyadic intervals $\omega\subset I$ and a function $\alpha:\Omega\mapsto \C$ such that 
\begin{itemize}
\item $|\alpha(\omega)|\leq C\|f\|_{BMO}$ for all intervals $\omega\in\Omega(I)$,
\item $\sum_{\omega\subset \check I,\,\,\omega\in\Omega(I)} |\omega|\leq C|\check I|$ for all intervals $\check I\subset I$, and
\item $f=b+\dashint_I f + \sum_{\omega\in\Omega} \alpha(\omega) \chi_\omega$ for some function $b$ such that $\|b\|_{L^\infty(I)}\leq C\|f\|_{BMO}$.
\end{itemize}
If $f$ has compact support, we may assume that we are working in some interval $(-2^k,2^k)\supset\supp f$. Otherwise, let $\Omega_k=\Omega((-2^k,2^k))$. If we have constructed $\Omega_k$, then we may construct $\Omega_{k+1}$ by constructing  $\Omega((2^k,2^{k+1}))$ and $\Omega((-2^{k+1},-2^{k}))$ and taking 
\[\Omega_{k+1}=\Omega_{k} \cup \Omega((2^k,2^{k+1})) \cup \Omega((-2^{k+1},-2^{k}))\cup \{(-2^k,2^k)\}\]
and letting $\alpha({(-2^k,2^k)})=\dashint_{-2^k}^{2^k} f - \dashint_{-2^{k+1}}^{2^{k+1}} f$; by basic $BMO$ theory this is at most $2\|f\|_{BMO}$.

Define
\[F_k(x,t)= \sum_{\omega\in\Omega_k} \alpha_\omega \chi_\omega(x)\chi_{[0,\sigma(\omega)]}(t) + \dashint_{-2^k}^{2^k} f.\]
Then $F_{k+1}=F_k$ on $(-2^k,2^k)\times(0,2^k)$; so $\check F=\lim_{k\to\infty} F_k$ exists.

If $I$ is the smallest dyadic interval of length at least $t_0$ that contains $x_0$, then $F_k(x_0,t_0)=\sum_{I\subseteq\omega\in\Omega_k} \alpha_\omega+\dashint_{-2^k}^{2^k} f$.

Let $\eta\in C^\infty_0(B(0,1))$ have integral $1$. Let $\eta_t(X) = t^{-2} \eta(X/t).$  Let $\tilde F(x,t) = \eta_{t/C}\star\check F (x,t)$, $C\geq 4$. By looking at the smoothing of a single interval, we see that $|\nabla F|$ is a Carleson measure.

We may get $F$ from $\tilde F$ by adding a correction term whose $L^\infty$ norm is controlled by the $BMO$ norm of~$f$ (supplied in \cite{bmo2}).

If $f$ is compactly supported, we may construct $\Omega_k$ such that $F_k$ is as well; since $\tilde F$ is a convolution of $F$ with a smooth cutoff, $\tilde F$ is compactly supported. We get from $\tilde F$ to $F$ by adding a $L^\infty$ function; we may multiply it by a smooth cutoff function (with large support) to get $F$ compactly supported.

We have that $|F(X)-\tilde F(X)|\leq C\|f\|_{BMO}$. I claim that $|\tilde F(X)-F_k(X)|\leq C\|f\|_{BMO}$ for all sufficiently large~$k$.

$\tilde F(x,t)$ is an average of the values $F_k$ takes in $B((x,t),t/C)$. If $C$ is large, this ball is small and so there are at most four such values; I claim that they vary by at most $C\|f\|_{BMO}$.

If $I$ is a dyadic interval and $k$ is large enough, then
\[\dashint_I \Bigl|f-\sum_{I\subseteq\omega\in\Omega_k} \alpha_\omega-{\textstyle\dashint_{-2^k}^{2^k} f}\Bigr|
\leq
\dashint_I \Bigl(|b|+\sum_{\omega\subset I} |\alpha(\omega)| |\omega|\Bigr)
\leq
C\|f\|_{BMO}.
\]
But if $I_1\subset I_2$ and $|I_2|=2|I_1|$, then by standard $BMO$ theory $\left|\dashint_{I_2} f-\dashint_{I_1} f\right|\leq 2\|f\|_{BMO}$. This implies that if $|I_1|=|I_2|$ and $I_1$, $I_2$ share an endpoint, then $\left|\dashint_{I_2} f-\dashint_{I_1} f\right|\leq 4\|f\|_{BMO}$. Using these facts, we see that $F_k$ has oscillation no more than $8\|f\|_{BMO}$ on $B((x,t),t/4)$, and so $|\tilde F(x,t)-F_k(x,t)|\leq 8\|f\|_{BMO}$.

Suppose that $X_0=(x_0,t_0)$, and let $I=B(X_0,C_1 t_0)\cap\partial\R^2_+$. If $I=(x_0-r,x_0+r)$, let $\check I=(x_0-r-t_0,x_0+r+t_0)$. Then for $k$ large enough, 
\begin{align*}
\dashint_{I} |f-F(X_0)| 
&\leq
|\tilde F(X_0)-F(X_0)|+
|F_k(X_0)-\tilde F(X_0)|+
\dashint_{I} |f -F_k(X_0)| 
\\&\leq 
C\|f\|_{BMO} +\|b\|_{L^\infty}
+ \dashint_{I} \Bigl|\sum_{\omega\in\Omega_k}\alpha(\omega)\chi_\omega +{\textstyle\dashint_{-2^k}^{2^k} f} -F_k(X_0)\Bigr|
\\&\leq 
C\|f\|_{BMO}+\sum_{\omega\subset\check I}|\alpha(\omega)| \frac{|\omega|}{|\check I|}\frac{|\check I|}{|I|}
+ \dashint_{I} \Bigl|\sum_{|\omega|>t_0}\alpha(\omega)\chi_\omega +{\textstyle\dashint_{-2^k}^{2^k} f} -F_k(X_0)\Bigr|
\\&\leq 
\left(C+C\frac{|\check I|}{|I|}\right)\|f\|_{BMO}
+ \dashint_{I} \Bigl|\sum_{|\omega|>t_0} \alpha(\omega)
(\chi_\omega - \chi_\omega(x_0)\Bigr|
\end{align*}
Let $j$ be such that $2^{j-1}< t_0 \leq 2^{j}$. Then $I$ is contained in the union of at most $C(C_1)$ intervals of length $2^j$. On each of these intervals, $\sum_{|\omega|>t_0} \alpha(\omega)
(\chi_\omega - \chi_\omega(x_0))$ is a constant; it is zero on the interval containing $x_0$, and so must be bounded by $C(C_1) \|f\|_{BMO}$ throughout.

If $f$ is Lipschitz, let $\mu=\|f\|_{BMO}/\|f'\|_{L^\infty}$. 
%
%

Construct $F_k$, $\tilde F$, $F$ as usual. Let $G(x,t)=f(x)\zeta(t)+F(x,t)(1-\zeta(t))$, where $\zeta=1$ on $[0,\mu]$, $\zeta=0$ on $[2\mu,\infty)$, and $|\zeta'|\leq 2/\mu$.

I claim that $G$ satisfies our desired conditions. Clearly, $G\to f$ nontangentially everywhere; we need only show that $G$ is Lipschitz and $|\nabla G|$ is a Carleson measure.

First, note that
\[\nabla G(x,t) = \Matrix{f'(x)\eta(t)\\(f(x)-F(x,t))\eta'(t)} + (1-\eta(t))\nabla F(x,t).\]
Since $|\nabla F(x,t)|\leq C\|f\|_{BMO}/t$, we have that $|(1-\eta(t))\nabla F(x,t)|\leq C \|f'\|_{L^\infty}$; since $\nabla F$ is a Carleson measure, so is $(1-\eta(t))\nabla F(x,t)$.

$|f'(x)\eta(t)|\leq \|f'\|_{L^\infty}$, and 
\[\frac{1}{T}\int_{x}^{x+T} \int_0^T |f'(x)|\eta(t)\,dt\,dx
\leq \|f'\|_{L^\infty} 2\mu = 2\|f\|_{BMO}.\]
So we need consider only the term $(f(x)-F(x,t))\eta'(t).$
Assume $t<2\mu$, and let $I$ be the smallest dyadic interval containing $x$ with $|I|>t$. Then $|f(x)-\dashint_I f|\leq \frac{1}{2} \|f'\|_{L^\infty}|I|\leq 2\|f\|_{BMO}$, and by the above $\left|F(x,t)-\dashint_I f\right|\leq C\|f\|_{BMO}$, and so
\begin{align*}
|f(x)-F(x,t)|
&\leq \left|f(x)-\dashint_I f\right| + \left|F(x,t)-\dashint_I f\right|
\leq C\|f\|_{BMO}.
\end{align*}

This implies that $|(f(x)-F(x,t))\eta'(t)|\leq C\|f\|_{BMO}/\mu=C\|f'\|_{L^\infty}.$ Furthermore,
\[\frac{1}{T}\int_{x}^{x+T} \int_0^T |(f(x)-F(x,t))\eta'(t)|\,dt\,dx
\leq \|f'\|_{L^\infty} 2\mu = 2\|f\|_{BMO}.\]
So $G\to f$ and $\nabla G$ is both bounded and a Carleson measure, as desired.

We now pass to more interesting Lipschitz domains.
If $f\in BMO(\partial\Omega)$ for $\Omega$ an arbitrary special Lipschitz domain, we may construct such an $F$ by letting $g(x)=f(\psi(x))$, letting $G\to g$ with $\nabla G\in\mathcal{C}$ as before, and by letting 
\[F(\psi(x,t)) = G(x,t)
\Rightarrow F(x\e^\perp + t\e) = G(x,t-\phi(x)).\]
Then $F\to f$ and $|\nabla F| \leq C |\nabla G|$, so 
\[\|\nabla F\|_{\mathcal{C}} \leq C\|f\|_{BMO}.\]


If $\V$ is a Lipschitz domain which is good but not special and $f\in BMO(\partial\V)$, we may proceed as in the proof of \cite[Lemma~2.3]{FK}. $\sigma(\partial\V)$ is finite and so $f\in L^1(\partial\V)$; we may assume without loss of generality that $\int_{\partial\V} f=0$.

Let $\Omega_j$, $R_j$, $X_j$ be as in \dfnref{domain}. Let $\{\eta_j\}_{j=1}^{k_2}$ be a set of smooth, compactly supported functions such that $\sum_j\eta_j(X)= 1$ for all $X$ with $\dist(X,\partial\V)\leq \sigma(\partial\V)/C$, $|\nabla\eta_j|\leq C/\sigma(\partial\V)$ and $\supp\eta_j$ is contained in 
\[\tilde R_j=\{X:|(X-X_j)\cdot \e_j^\perp|<\frac{3}{2}r_j,|(X-X_j)\cdot\e_j|\leq \frac{3}{2}(1+k_1)r_j\}.\]

Then $\|f\eta_j\|_{BMO(\partial\V)}\leq C\|f\|_{BMO(\partial\V)}$. Let $F_j:\Omega_j\mapsto\C$ satisfy $F_j\to f\eta_j$ non-tangentially, $|\nabla F_j|$ a Carleson measure. Note that $\check F_j$ is supported in $R_j$. By taking constants large enough, we may ensure that $\tilde F_j$ is as well; by multiplying our bounded correcting function by a smooth cutoff, we see that $F_j$ is supported in $R_j\cap\Omega_j=R_j\cap\V$.

By simply adding the $F_j$s, we get our desired $F$.
\end{proof}

\begin{thm}\label{thm:N1conv}
Suppose that $\div A\nabla u=0$ in $\V$ for some good Lipschitz domain with connected boundary, and assume $N(\nabla u)\in L^1(\partial\V)$.


If $\V$ is bounded, or if $\V=\Omega$ is a special Lipschitz domain, then 
\[\|\nu\cdot A\nabla u\|_{H^1(\partial\V)}\leq C\|N(\nabla u)\|_{L^1(\partial\V)}\]
in the sense that $\left|\int\eta\,\nu\cdot A\nabla u\right|\leq C\|N(\nabla u)\|_{L^1(\partial\V)} \|\eta\|_{BMO(\partial\V)}$ for all $\eta\in C^\infty_0(\R^2_+)$.

If $\partial\V$ is bounded but $\V$ is not, then $\nu\cdot A\nabla u - C_u$ is in $H^1$ for some constant $C_u$ with $|C_u|\leq C\|N(\nabla u)\|_{L^1}/\sigma(\partial\V)$.
\end{thm}

\begin{proof} 
By \lemref{L2loc}, $\nabla u\in L^2_{loc}$, so $\nu\cdot A\nabla u$ exists in the weak sense.
%
%
%

Suppose that $f\in BMO(\partial\V)$ is Lipschitz and compactly supported. By \lemref{BMOhasCarleson}, if $\V$ is bounded or special then there is a compactly supported Lipschitz function $F$ such that $F\to f$ nontangentially, and such that $\|\nabla F\|_{\mathcal{C}}\leq C\|f\|_{BMO}$. We need only show that \[\left|\int_V\nabla F\cdot A\nabla u\right|\leq C\|\nabla F\|_{\mathcal{C}} \|N(\nabla u)\|_{L^1}.\]

We must review some basic theorems about Carleson measures. Let $G$, $H$ be two functions. It is well known (see, for example, \cite[Section~II.2]{stein}) that 
\[\int_{\R^2_+} |G| |H| \leq C \|G\|_{\mathcal{C}} \|NH\|_{L^1(\partial\R^2_+)}.\]
This clearly extends to special Lipschitz domains. If $\V$ is a good but not special Lipschitz domain, then the inequality holds if we integrate not over all of $\V$ but over the $k_2$ boundary cylinders $R_j\cap\V$ of \dfnref{domain}.

If $X$ is not in one of these boundary cylinders, then $\dist(X,\partial\V)\geq \sigma(\partial\V)/C$.
Therefore, if $R>0$ and $X_0\in\partial\V$ then 
\begin{align*}
\int_{B(X_0,R)\cap\V} |G| |H| 
&\leq
C {\|NH\|_{L^1(\partial\V)}} \|G\|_{\mathcal{C}}+ \int_{X\in B(X_0,R),\dist(X,\partial\V)\geq \sigma(\partial\V)/C} |G| |H| 
\\&\leq 
C {\|NH\|_{L^1(\partial\V)}} \|G\|_{\mathcal{C}}+ C\frac{\|NH\|_{L^1(\partial\V)}}{\sigma(\partial\V)}
\int_{B(X_0,R)\cap\V} |G| 
\\&\leq
C {\|NH\|_{L^1(\partial\V)}} \|G\|_{\mathcal{C}} .
\end{align*}

Applying this to $G=\nabla F$ and $H=\nabla u$, we have that 
\[\int_\V |\nabla F| |\nabla u| \leq C\|f\|_{BMO} \|N(\nabla u)\|_{L^1}\]
as desired.

Finally, we consider domains $\V$ whose complements are bounded. Fix some $C^\infty_0$ function $\eta$ with $\eta\equiv 1$ on $\partial\V$, so that 
\[\int_{\partial\V} \nu\cdot A\nabla u=\int_\V \nabla\eta\cdot A\nabla u.\]
Assume that $\supp\eta\subset\{X:\dist(X,\partial\V)<\epsilon\}$ and that $|\nabla\eta|<C/\epsilon$.

Then $\int_\V |\nabla \eta| |A\nabla u|\leq C\|N(\nabla u)\|_{L^1}$. Let $C_u=\frac{1}{\sigma(\partial\V)}\int_{\partial\V}\nu\cdot A\nabla u$. I claim that $\nu\cdot A\nabla u-C_u\in H^1$, that is, that
\[\left|\int_{\partial\V} f \nu\cdot A\nabla u - f C_u \,d\sigma\right| \leq C \|f\|_{BMO} \|N(\nabla u)\|_{L^1}\]
for all functions $f$ Lipschitz on $\partial\V$.

As before, there is some $F\to f$ Lipschitz with $|\nabla F|$ a Carleson measure and $F-\dashint_{\partial\V} f$ compactly supported.

So
\begin{align*}
\int_{\partial\V}  f \nu\cdot A\nabla u - f C_u \,d\sigma
&=\int_{\partial\V}  (f-{\textstyle\dashint_{\partial\V} f\,d\sigma}) \nu\cdot A\nabla u\,d\sigma
+ \dashint_{\partial\V} f\,d\sigma \int_{\partial\V}\nu\cdot A\nabla u
\\&\qquad
- \int_{\partial\V} f \left(\frac{1}{\sigma(\partial\V)}\int_{\partial\V} \nu\cdot A\nabla u\,d\sigma\right) \,d\sigma
\\&=\int_{\partial\V}  (f-{\textstyle\dashint_{\partial\V} f\,d\sigma}) \nu\cdot A\nabla u\,d\sigma
=\int_\V \nabla F\cdot A\nabla u.
\end{align*}
As before, $\Lambda\int_\V |\nabla F| |\nabla u|\leq C\|f\|_{BMO} \|N(\nabla u)\|_{L^1}$. This completes the proof.
\end{proof}

We now move on to the regularity problem.

\begin{thm}\label{thm:R1conv} Suppose that $\div A\nabla u=0$ in $\V$ and $N(\nabla u)\in L^1(\partial\V)$ for some good Lipschitz domain $\V$ with connected boundary. Then $f(X)=\lim_{Z\to X\text{ n.t.}}u(Z)$ exists for almost every $X\in\partial\V$. Furthermore, $\partial_\tau f$ exists in the weak sense and $\|\partial_\tau f\|_{H^1}\leq C\|N(\nabla u)\|_{L^1}$.
\end{thm}

\begin{proof} If $N(\nabla u)(X)$ is finite, then $\lim_{Z\to X\text{ n.t.}} u(X)$ exists; we need only show $\partial_\tau u\in H^1$.

Recall the conjugate $\tilde u$ of \autoref{sec:conjugate}. If $\V$ is simply connected, then $\tilde u$ exists and is continuous on~$\V$. In this case, define $v=u$, $\tilde v=\tilde u$.

If $\V$ is not simply connected, then $\bar\V^C$ is a bounded, simply connected Lipschitz domain. Let $X\notin \V$ with $\dist(X,\partial\V)\geq\sigma(\partial\V)/C$, and let $R$ be large enough that $\partial B(X,R)\subset\V$.

Let $C_1=\int_{\partial B(X,R)} \nu\cdot A\nabla u\,d\sigma$. Let $v=u-C_1\Gamma_X$. Since $\int_{\partial B(X,R)}\nu\cdot A\nabla \Gamma_X\,d\sigma=1$, and since $\int_{\partial\U}\nu\cdot A\nabla v\,d\sigma=0$ for any simply connected domain $\U\subset\V$, we must have that $\int_\omega \nu\cdot A\nabla v\,d\sigma=0$ for any closed path $\omega\subset\V$; thus, $\tilde v$ is well-defined and continuous on~$\V$.

But $\|\partial_\tau \Gamma_X\|_{H^1(\partial\V)}\leq C$; so we need only show that $\|\partial_\tau v\|_{H^1}\leq C$.

Note that $\div \tilde A\nabla\tilde v=0$ in $\V$ and $\|N(\nabla \tilde v)\|_{L^1}\approx \|N(\nabla v)\|_{L^1}$. So by \thmref{N1conv}, $\nu\cdot \tilde A\nabla\tilde v\in H^1(\partial\V)$ in the weak sense.

I claim that $\partial_\tau v = -\nu\cdot \tilde A\nabla\tilde v$ in the weak sense.

It suffices to prove that, if $I\subset\partial\V$ is connected and has boundary points $X_0$, $X_1$, and $N(\nabla v)(X_i)$ is finite, then for the appropriate ordering of $X_0$, $X_1$,
\[v(X_1)-v(X_0) = -\int \chi_I \nu\cdot\tilde A\nabla\tilde v.\]
In fact, we need only prove this for short intervals; thus, it suffices to prove this for $\V=\Omega$ a special Lipschitz domain.

Let $X_i=\psi(x_i)$; then $x_0<x_1$. Pick some $\epsilon>0$. Let $0<h<\epsilon/N(\nabla v)(X_i)$ for $i=0$,~$1$.
Let $\U=\psi((x_0,x_1)\times(0,h))\subset\V$. Then 
\[\int_{\partial\U} \nu\cdot \tilde A\nabla \tilde v=0,\quad \int_{\psi(\{x_i\}\times(0,h))} |\nabla \tilde v|\leq C\epsilon\]
and so
\[\left|\int \chi_I \nu\cdot\tilde A\nabla\tilde v-
\int_{x_0}^{x_1}\nu(\psi(x))\cdot \tilde A(\psi(x,h))\nabla \tilde v(\psi(x,h))\sqrt{1+\phi'(x)^2}\,dx\right|\leq C\epsilon.\]
But $\psi(x,h)\in\V$, and if $X\in\V$ and $Y\in\partial\V$, then
$\tau(Y)\cdot \nabla v(X) = -\nu(Y)\cdot \tilde A(X)\nabla\tilde v(X)$. So
\begin{multline*}
\int_{x_0}^{x_1}-\nu(\psi(x))\cdot \tilde A(\psi(x,h))\nabla \tilde v(\psi(x,h))\sqrt{1+\phi'(x)^2}\,dx
\\=\int_{x_0}^{x_1}\tau(\psi(x))\cdot \nabla v(\psi(x,h))\sqrt{1+\phi'(x)^2}\,dx
= v(\psi(x_1,h))-v(\psi(x_0,h))
.\end{multline*}
But $|v(X_i)-v(\psi(x_i,h))|\leq\epsilon$. So
\[\left|
\int \chi_I \nu\cdot\tilde A\nabla\tilde v
-(v(X_0)-v(X_1))
\right|\leq C\epsilon.\]
By letting $\epsilon\to 0$ we have our desired result; so $\partial_\tau v=\nu\cdot\tilde A\nabla\tilde v$.

If $\V$ is simply connected, we are done because $\nu\cdot \tilde A\nabla \tilde v\in H^1$. Otherwise, we know that $\nu\cdot \tilde A\nabla \tilde v_1-C_v\in H^1$ for some constant $C_v$. But $\int_{\partial\V}\tau\cdot\nabla v\,d\sigma=0$, and so we see that $C_v$ must be 0; thus $\nu\cdot \tilde A\nabla \tilde v_1\in H^1$, as desired.
\end{proof}

\chapter{\texorpdfstring{Boundary data in $L^\infty$: The maximum principle}{The maximum principle}}
\label{chap:maximum}
\begin{thm} Suppose that $u=f$ on $\partial\V$ for some bounded good Lipschitz domain $\V$, and that $\div A\nabla u\equiv 0$ in $\V$. Then $\|u\|_{L^\infty(\V)}\leq C \|f\|_{L^\infty(\partial\V)}$ for some constant $C$ depending on the ellipticity constants of $A$ and the Lipschitz constants of $\V$, but not on~$\sigma(\partial\V)$.
\end{thm}

This theorem is obviously not true on unbounded domains: consider the harmonic functions $u(x,t)=t$ in $\R^2_+$ or $v(X)=\log |X|$ in $\R^2\setminus B(0,1)$.

For notational convenience, we will instead prove this for $\div A^T\nabla u= 0$. Throughout this section, let $p$ be an exponent such that $(R)^A_p$, $(N)^A_p$, $(D)^{A^T}_{p'}$ hold.

\section{Green's function and a priori bounds}

Let $\V$ be a bounded Lipschitz domain, and assume that $(R)^{A^T}_{p'}$ holds in~$\V$. Choose some $X\in\V$. $\Gamma_X(Y)$ is bounded and continuous with bounded gradient on $\partial\V$, so $\partial_\tau \Gamma_X\in L^{p'}$. Therefore, since $\partial\V$ is compact, there is some $\Phi_X$ with $\div A\nabla\Phi_X=0$ in $\V$ and $\Phi_X=\Gamma_X$ on~$\partial\V$.

Define $G_X=\Gamma_X-\Phi_X$.

\begin{thm} Suppose that $\V$ is a bounded Lipschitz domain, that $(R)^A_{p'}$, $(R)^{A^T}_{p'}$, $(D)^A_p$ hold in~$\V$. Suppose that $\div A\nabla u=0$ in~$\V$, $Nu\in L^p(\partial\V)$, and $u=f$ on $\partial\V$.

Then
\[u(X)=\int_{\partial\V} f\nu\cdot A\nabla G_X\,d\sigma.\]
\end{thm}

Proving this for $f=0$ was the core of our proof of \thmref{Dpunique}.

\begin{proof}
Let $R=\dist(X,\partial\V)$.
Recall that $|\nabla\Gamma_X(Y)|\leq \frac{C}{|X-Y|}$. So
\begin{align*}
\int_{\partial\V}|\nabla\Gamma_X(Y)|^{p'}
&\leq
\sum_{k=0}^\infty \int_{(B(X,2^{k+1}R)\setminus B(X,2^kR))\cap\partial\V}
|\nabla\Gamma_X(Y)|^{p'}
\\&\leq
\sum_{k=0}^\infty C 2^k R \frac{C}{2^{k{p'}} R^{p'}}
\leq C R^{1-{p'}}.
\end{align*}
But since $(R)^A_{p'}$ holds in $\V$, and $\tau\cdot\nabla\Gamma_X = \tau\cdot\nabla\Phi_X$, we must have that
\begin{equation}\label{eqn:greenLp}
\|N(\nabla \Phi_X)\|_{L^{p'}(\partial\V)}\leq C R^{1/{p'}-1}.
\end{equation}

Let $g\in L^p(\partial\V)$ be such that $\|f-g\|_{L^p}<\epsilon$ and $\partial_\tau g \in L^{\infty}(\partial\V)\supset L^{p'}(\partial\V)$. Let $v$ be the $(D)^A_p$ solution to $\div A\nabla v=0$ in $\V$, $v=g$ on $\partial\V$.

Then since $(D)^A_p$ holds in $\V$, 
\begin{multline*}
\left| u(X)-\int_{\partial\V} f\nu\cdot A\nabla G_X\,d\sigma\right|
\\\begin{aligned}&\leq | u(X)-v(X)| 
+\left|v(X)-\int_{\partial\V} g\nu\cdot A\nabla G_X\,d\sigma\right|
+\left|\int_{\partial\V} (f-g)\nu\cdot A\nabla G_X\,d\sigma\right|
\\&\leq C\|N(u-v)\|_{L^p}\min(\sigma(\partial\V),\dist(X,\partial\V)^{-1/p})
\\&\qquad
+\left|v(X)-\int_{\partial\V} g\nu\cdot A\nabla G_X\,d\sigma\right|
+\|f-g\|_{L^p}\|\nabla G_X\|_{L^{p'}(\partial\V)}
\\&\leq C\epsilon\min(\sigma(\partial\V),\dist(X,\partial\V)^{-1/p})
+\left|v(X)-\int_{\partial\V} g\nu\cdot A\nabla G_X\,d\sigma\right|.
\end{aligned}
\end{multline*}
So we need only show that the theorem holds for $v$.

By \lemref{L2loc}, $v$ is bounded and $\nabla v\in L^2(\V)$. By (\ref{eqn:gradu}), $\nabla v$ is bounded near~$X$. By \lemref{Sfcts}, $v$ is H\"older continuous on $\bar\V$. We specified that $\partial_\tau v$ was bounded. So we may extend $v$ to a bounded compactly supported function with gradient in $L^2(\R^2)$ whose gradient is bounded near~$X$.

But then
\begin{align*}
v(X) &=  -\int \nabla v\cdot A^T\nabla\Gamma_X^T
=-\int_{\V} A\nabla v\cdot \nabla\Gamma_X^T
-\int_{\V^C} \nabla v\cdot A^T\nabla\Gamma_X^T
\\&=-\int_{\partial\V} \Gamma_X^T \nu\cdot A\nabla v\,d\sigma
+\int_{\partial\V} v \nu\cdot A^T\nabla\Gamma_X^T\,d\sigma
\\&=-\int_{\partial\V} \Phi_X \nu\cdot A\nabla v\,d\sigma
+\int_{\partial\V} v \nu\cdot A^T\nabla\Gamma_X^T\,d\sigma
\\&=-\int_{\V} \nabla\Phi_X \cdot A\nabla v
+\int_{\partial\V} v \nu\cdot A^T\nabla\Gamma_X^T\,d\sigma
\\&=-\int_{\partial\V} v \nu \cdot A^T\nabla\Phi_X\,d\sigma
+\int_{\partial\V} v \nu\cdot A^T\nabla\Gamma_X^T\,d\sigma
=\int_{\partial\V} v \, \nu\cdot A^T\nabla G_X\,d\sigma
\end{align*}
as desired.
\end{proof}

Thus, the maximum principle is equivalent to bounding $\|\nu\cdot A\nabla G_X\|_{L^1(\partial\V)}$, with a bound independent of $X$ and $\sigma(\partial\V)$.

By (\ref{eqn:greenLp}) and H\"older's inequality, $\|\nabla G_X\|_{L^1}$ is bounded for $\dist(X,\partial\V)>\frac{1}{C}\sigma(\partial\V)$; we need only bound $\|\nu\cdot A\nabla G_X\|_{L^1(\partial\V)}$ for $X$ near the boundary.

\section{Bounding the Green's function}

Let $R=\dist(X,\partial\V)$, and assume that $R<\frac{1}{3}\min_i r_i$, where the $r_i$ are as in \dfnref{domain}. Then $X\in B(X_i, (1+1/3)r_i)\cap\Omega_i$ for some $i$.

Then $\Omega_i = \{\psi_i(x,t):t>0\}=\{X\in\R^2:\phi_i((X-X_i)\cdot\e_i^\perp)<(X-X_i)\cdot\e_i\}$ where $\psi(x,t)=t\e_i + x\e_i^\perp+\phi_i(x)\e_i^\perp$. Without loss of generality take $X=\psi(0,t)$ for some positive number $t\approx R$. Let $X^*=\psi(0)$, and let $\W_r=\psi(-r,r)\times (0,2r)$. Let $Y^r_\pm=\psi(\pm r)$. 

If $r<(1-1/3)r_i$ then $\W_r\subset\V$, and if
$r>R$ then $\dist(X,\partial\W_r)\approx \dist(X,\partial\V) \approx\sigma(\partial\W_r)$. By \lemref{NTM}, we may choose the apertures $a$ of non-tangential cones large enough that $\V\cap\partial\W_r\subset \gamma_{a,V}(Y^r_+)\cup \gamma_{a,V}(Y^r_-)$.

For convenience write $\U_r=\V\setminus\bar\W_r$. We need to show that $\int_{\partial\V} |\nu\cdot A\nabla G_X|\leq C$.
Begin by integrating only over $\partial\V\setminus\partial\U_r=\partial\W_r\cap\partial\V.$ We have that
\begin{align*}
\int_{\partial\V\setminus\partial\U_r}|\nabla G_X|\,d\sigma
&\leq 
\sigma(\partial\V\setminus\partial\U_r)^{1-1/p}
\left(\int_{\partial\V\setminus\partial\U_r}|\nabla G_X|^p\,d\sigma\right)^{1/p}
\\&\leq 
\sigma(\partial\V\setminus\partial\U_r)^{1-1/p}
C R^{1/p-1}
\leq C
\end{align*}
since $\sigma(\partial\V\setminus\partial\U_r)\leq\sigma(\W_r)\approx CR$ and $\|\nabla G\|_{L^p(\partial\V)}\leq C R^{1/p-1}$.

Now, consider $\partial\U_r$. Since $X\notin\U_r$, we have that $G_X$ satisfies $\div A\nabla G_X=0$ in $\U_r$; consequently, $G_X$ satisfies all of our useful theorems in $\U_r$. In particular,
$\|\nabla G_X\|_{L^1(\partial\U_r)} \leq C\|\tau\cdot\nabla G_X\|_{H^1(\partial\U_r)}$. So we need only bound $\|\tau\cdot\nabla G_X\|$ in $H^1$.

But $G_X=0$ on $\partial\V$; therefore, we need only consider $\tau\cdot \nabla G_X$ on $\partial\U_r\setminus\partial\V$.

Let $g=\tau\cdot\nabla G_X$. Note that $\int g=0$, and $\sigma(\supp g)\leq CR$. So we need only show that $\|g\|_{L^\infty(\partial\U_r\setminus\partial\V)}\leq C/R$.

Now, $|g(Y)|\leq |\nabla\Gamma_X(Y)|+|\nabla\Phi_X(Y)|$. 
On $\partial\W_r$, 
\[|\nabla\Gamma_X(Y)|\leq \frac{C}{|X-Y|}\leq \frac{C}{\dist(X,\partial\W_r)}\leq \frac{C}{R}.\]
Since $\partial\U_r\setminus \partial\V\subset\gamma(Y^r_+)\cup \gamma(Y^r_-)$, we have that
\[\|g\|_{L^\infty(\partial\U_r\setminus \partial\V)}\leq C/R + N(\nabla\Phi_X)(Y^r_+)+N(\nabla\Phi_X)(Y^r_-).\] 
So we have that $\|g\|_{H^1}\leq C + RN(\nabla\Phi_X)(Y^r_+)+ RN(\nabla\Phi_X)(Y^r_-),$ and therefore that
\[\|\nu\cdot A\nabla G_X\|_{L^1}\leq C + RN(\nabla\Phi_X)(Y^r_+)+ RN(\nabla\Phi_X)(Y^r_-).\]

As in \autoref{sec:realh1:bddiffinite}, we take the average from $r=R$ up to $r=2R\leq (1-1/3) r_i$:
\begin{align*}
\|\nu\cdot A\nabla G_X\|_{L^1}
&\leq 
C + \dashint_{\psi((R,2R))}RN(\nabla\Phi_X)\,d\sigma
 + \dashint_{\psi((-2R,-R))}RN(\nabla\Phi_X)\,d\sigma
\\&\leq 
C + C\dashint_{\psi((-2R,2R))}RN(\nabla\Phi_X)\,d\sigma
\\&\leq 
C 
+ C\left(\dashint_{\psi((-2R,2R))}R^p N(\nabla\Phi_X)^p\,d\sigma\right)^{1/p}
\\&\leq 
C 
+ CR^{1-1/p}\left(\int_{\psi((-2R,2R))} N(\nabla\Phi_X)^p\,d\sigma\right)^{1/p}
\\&\leq C R^{1-1/p} \|N(\nabla \Phi_X)\|_{L^p(\partial\V)}\leq C
\end{align*}
as desired.







\singlespacing

\end{document}